\documentclass[reqno]{amsart}
\usepackage[all]{xy}
\usepackage{xcolor}
\usepackage{hyperref}
\usepackage{enumerate}
\usepackage{amsmath}
\usepackage{amsthm}
\usepackage{amsmath}
\usepackage{amssymb}
\usepackage{yfonts}
\usepackage{amsfonts}
\usepackage{amscd}
\usepackage{graphicx}
\usepackage{epsfig}
\usepackage{mathrsfs}
\usepackage[english]{babel}
\usepackage[stable]{footmisc}
\usepackage{tensor}
\usepackage{stmaryrd}
\usepackage{tabularx}
\usepackage{combelow}

\DeclareMathAlphabet{\mathscr} {U}{BOONDOX-cal}{r}{n}

\definecolor{dgreen}{rgb}{0.0, 0.5, 0.0}

\let\oldmarginpar\marginpar
\renewcommand\marginpar[1]
{\oldmarginpar{\tiny\bf \begin{flushleft} {\color{red} #1 }\end{flushleft}}}

\newcommand{\rmap}{\longrightarrow}

\newcommand{\fproduct}[1]{\tensor{\times}{_{#1}}}

\renewcommand{\d}{\mathrm d}               
\newcommand{\Lie}{\boldsymbol{\pounds}}    
\newcommand{\X}{\ensuremath{\mathfrak{X}}} 
\newcommand{\red}{{\mathrm{red}}} 
\renewcommand{\top}{{\mathrm{top}}} 
\newcommand{\can}{{\mathrm{can}}} 
\renewcommand{\ss}{{\mathrm{ss}}}
\newcommand{\lin}{{\mathrm{lin}}} 
\renewcommand{\inf}{{\mathrm{inf}}}
\newcommand{\interior}{{\mathrm{int}}}
\newcommand{\Gr}{{\mathrm{Gr}}}
\newcommand{\hol}{{\mathrm{hol}}}
\newcommand{\reg}{{\mathrm{reg}}}
\newcommand{\princ}{{\mathrm{princ}}}
\newcommand{\subreg}{{\mathrm{subreg}}}
\newcommand{\aff}{{\mathrm{aff}}}
\newcommand{\Aff}{{\mathrm{Aff}}}
\newcommand{\tr}{{\mathrm{tr}}}
\newcommand{\sing}{{\mathrm{sing}}}
\newcommand{\im}{{\mathrm{Im}}}
\newcommand{\id}{{\mathrm{id}}}
\newcommand{\pr}{{\mathrm{pr}}}
\newcommand{\inv}{{\mathrm{inv}}}
\newcommand{\res}{{\mathscr{R}}}

\newcommand{\g}{\gamma}

\newcommand{\w}{\omega}
\newcommand{\eps}{\varepsilon}
  
\newcommand{\hM}{\widehat{M}}
\newcommand{\hU}{\widehat{U}}
\newcommand{\hF}{\widehat{\cF}_{\pi}}

\newcommand{\hphi}{\widehat{\phi}}
\newcommand{\hLD}{\widehat{L}}

\newcommand{\hQ}{\widehat{Q}}
\newcommand{\hq}{\widehat{q}}

\newcommand{\hG}{\widehat{\cG}}
\newcommand{\hS}{\widehat{S}}
\newcommand{\hx}{{\widehat{x}}}
\newcommand{\hxx}{{\widehat{y}}}
\newcommand{\hqq}{\widehat{r}}
\newcommand{\hy}{{\widehat{y}}}
\newcommand{\hi}{{\widehat{i}}}

\newcommand{\hT}{{\widehat{\cT}}}
\newcommand{\hL}{\widehat{L}_\pi}
\newcommand{\hSi}{{\widehat{\Si}}}
\newcommand{\hSc}{{\widehat{\Sc}}}
\newcommand{\hOmega}{\widehat{\Omega}}
\newcommand{\hgg}{\widehat{\gg^*}}
\newcommand{\htt}{\widehat{\tt}}

\newcommand{\hH}{\mathscr{H}}
\newcommand{\hLL}{\mathscr{L}}
\newcommand{\hnu}{\widehat{\nu}}

\newcommand{\oM}{{\hM}^{\orb}}
\newcommand{\oU}{\widehat{U}^{\orb}}
\newcommand{\oB}{\widetilde{B}^{\orb}}
\newcommand{\orb}{{\mathrm{orb}}}
\newcommand{\VV}{{\mathscr{V}}}
\newcommand{\VVO}{{\mathscr{V}}_0}

\newcommand{\R}{\mathbb{R}}
\newcommand{\C}{\mathbb{C}}
\renewcommand{\S}{\mathbb{S}}
\newcommand{\T}{\mathbb{T}}
\newcommand{\Z}{\mathbb{Z}}

\newcommand{\cB}{\mathcal{B}}
\newcommand{\cC}{\mathcal{C}}
\newcommand{\cD}{\mathcal{D}}
\newcommand{\cE}{\mathcal{E}}
\newcommand{\cF}{\mathcal{F}}
\newcommand{\cG}{\mathcal{G}}

\newcommand{\cK}{\mathcal{K}}

\newcommand{\cH}{\mathcal{H}}
\newcommand{\cI}{\mathcal{I}}

\newcommand{\cM}{\mathcal{M}}

\newcommand{\cP}{\mathcal{P}}

\newcommand{\cS}{\mathcal{S}}
\newcommand{\cT}{\mathcal{T}}

\newcommand{\cV}{\mathcal{V}}
\newcommand{\hW}{\widehat{\mathcal{W}}}
\newcommand{\cW}{\mathcal{W}}

\newcommand{\Bgood}{B'}

\newcommand{\Pol}{{\mathrm{Pol}}}
\newcommand{\laction}{\curvearrowright}

\newcommand{\corank}{\operatorname{corank}}

\newcommand{\G}{\cG}            
\renewcommand{\O}{\mathcal{O}}             
\DeclareMathOperator{\Mon}{Mon}         
\newcommand{\s}{\mathbf{s}}             
\renewcommand{\t}{\mathbf{t}}           
\renewcommand{\aa}{\mathfrak{a}}        
\renewcommand{\gg}{\mathfrak{g}}        
\newcommand{\su}{\mathfrak{su}}     %

\newcommand{\kk}{\mathfrak{k}}          
\renewcommand{\tt}{\mathfrak{t}}        
\newcommand{\pp}{\mathfrak{p}}          
\newcommand{\zz}{\mathfrak{z}}        

\newcommand{\Refl}{\mathfrak{R}}

\newcommand{\Sc}{\mathfrak{S}}
\newcommand{\Si}{\Sigma}
\newcommand{\cSi}{\cS^\inf}

\newcommand{\act}{\mathscr{a}}

\newcommand{\tto}{\rightrightarrows}    

\DeclareMathOperator{\Ker}{Ker}           
\renewcommand{\ker}{\Ker}
\DeclareMathOperator{\ad}{ad}           
\DeclareMathOperator{\Ad}{Ad}           
\DeclareMathOperator{\Aut}{Aut}         
\DeclareMathOperator{\Rep}{Rep}         
\DeclareMathOperator{\GL}{GL}           
\DeclareMathOperator{\Hol}{Hol}         
\DeclareMathOperator{\codim}{codim}           

\DeclareMathOperator{\vol}{vol}           

\DeclareMathOperator{\dev}{dev}      

\newcommand{\SU}{\mathrm{SU}}
\newcommand{\SO}{\mathrm{SO}}

\newcommand{\helpI}{\mathrm{var}}
\newcommand{\Iaff}{\helpI_{\varpi}^{\aff}} 
\newcommand{\Ilin}{\helpI_{\varpi}^{\lin}} 

\newcommand{\helpVar}{\mathrm{var}} 
\newcommand{\Var}{\helpVar_{\varpi}} 
\newcommand{\Varb}{\helpVar_{0}} 

\newcommand{\diffto}{\xrightarrow{\raisebox{-0.2 em}[0pt][0pt]{\smash{\ensuremath{\sim}}}}} 

\numberwithin{equation}{section}
\allowdisplaybreaks

\newtheorem{theorem}{Theorem}[section]
\newtheorem{theorem*}{Theorem}
\newtheorem{problem}{Open Problem}
\newtheorem{lemma}[theorem]{Lemma}
\newtheorem{proposition}[theorem]{Proposition}
\newtheorem{corollary}[theorem]{Corollary}

\theoremstyle{definition}
\newtheorem{definition}[theorem]{Definition}
\newtheorem{example}[theorem]{Example}
\newtheorem{remark}[theorem]{Remark}

\begin{document}
\title[PMCTs]{Poisson Manifolds of Compact Types\\ 
\MakeLowercase{with an appendix by} Joshua Mundinger}



\author[M. Crainic]{Marius Crainic}
\address{Department of Mathematics, Utrecht University, P.O. Box 80010 
3508 TA Utrecht, The Netherlands}
\email{m.crainic@uu.nl}

\author[R.L. Fernandes]{Rui Loja Fernandes}
\address{Department of Mathematics, University of Illinois at Urbana-Champaign,  214 Harker Hall, 1305 W Green St, Urbana, IL 61801 USA}
\email{ruiloja@illinois.edu}

\author[D. Martinez-Torres]{David Mart\'inez Torres}
\address{Department of Applied Mathematics, ETSAM Section, Universidad Polit\'ecnica de Madrid, 
Avda. Juan de Herrera 4, 28040 Madrid, Spain}
\email{df.mtorres@upm.es}

\makeatletter
\let\@wraptoccontribs\wraptoccontribs
\makeatother
\contrib[appendix by]{Joshua Mundinger}
\address{Department of Mathematics, University of California, Berkeley\\ 970 Evans Hall\\ Berkeley, CA 94720 USA}
\email{mundinger@berkeley.edu}

\thanks{The authors were partially supported by NWO Vici grant no. 639.033.312, NSF grant DMS-2303586,  MCIN-AEI grant PID2022-139069NB-I00 and CNPq grant 304049/2018-2.}

\begin{abstract}
We develop the theory of Poisson and Dirac manifolds of compact types, a broad generalization in Poisson and Dirac geometry of compact Lie algebras and Lie groups. We establish key structural results, including local normal forms, canonical stratifications, and a Weyl type resolution, which provides a way to resolve the singularities of the original structure. These tools allow us to show that the leaf space of such manifolds is an integral affine orbifold and to define their Weyl group. This group is a Coxeter group acting on the orbifold universal cover of the leaf space by integral affine transformations, and one can associate to it Weyl chambers, reflection hyperplanes, etc. We further develop a Duistermaat-Heckman theory for Poisson manifolds of s-proper type, proving the linear variation of cohomology of leafwise symplectic form and establishing a Weyl integration formula. As an application, we show that every Poisson manifold of compact type is necessarily regular. We conclude the paper with a list of open problems.
\end{abstract}

\maketitle
\setcounter{tocdepth}{1}
\tableofcontents
\section{Introduction}\label{sec:intro}

This paper concerns an important class of Poisson and (possibly twisted) Dirac manifolds, generically called of proper type. These objects are analogous to compact Lie algebras and Lie groups in Lie theory, or foliations whose leaf spaces are orbifolds. We show that, in many respects, the theory of Poisson and Dirac manifolds of proper type closely parallels the classical theory of compact Lie algebras and Lie groups. In fact, one of the key takeaways of this paper is that many fundamental aspects  from the theory of compact Lie groups are purely Poisson geometric.

For the purposes of this introduction, we will focus primarily on Poisson manifolds of proper type, though many results also hold for twisted Dirac structures of proper type. A Poisson manifold $(M,\pi)$ is said to be of \emph{proper type} if there is a proper symplectic groupoid $(\cG,\Omega)\tto M$ with connected source fibers integrating $(M,\pi)$. This class was first introduced in \cite{CFM-I} under the broader designation of \emph{Poisson manifolds of compact types} (PMCT). That work established fundamental properties of PMCTs and provided various examples and constructions. Further examples and constructions of PMCTs were obtained by \'Alvarez \cite{AlvarezTese}, Mart\'inez Torres \cite{Mar} and Zwaan \cite{Zwaan23b}.

In \cite{CFM-II}, we focused on regular PMCTs, i.e., those whose symplectic leaves have the same dimension. We showed that regular PMCTs possess a rich transverse geometry, where various structures -- both classical and novel -- interact in intricate ways. These include orbifold structures, integral affine structures, symplectic gerbes, and more. In this paper we consider arbitrary PMCTs, regular or not.

For a Poisson manifold $(M,\pi)$ of proper type, we will begin by showing that:
\begin{enumerate}[$\bullet$]
    \item It admits two \textbf{canonical stratifications}: one of an infinitesimal nature, and a second one, associated with a choice of proper integration, which refines the former. Both are smooth stratifications by Poisson submanifolds, and these strata exhibit remarkable properties (Section \ref{sec:canonical:stratifications}).
    \item It has a canonical {\bf Weyl resolution}, defined by
	\begin{align*}
	\hM&:= \{ (x,\tt)\,|\, x\in M,\, \tt\subset \gg_x\,  \mathrm{maximal\ torus}\},\\	
	\res&: \hM\to M,\quad \ \res(x, \tt):= x,
	\end{align*}
	where $\gg_x$ is the isotropy Lie algebra of $(M, \pi)$ at $x$. The resolution $(\hM,\hL)$ is a \emph{regular} Dirac manifold and $\res$ is a forward Dirac map that is a diffeomorphism over the regular locus $M^\reg$ (Section \ref{sec:Weyl:resolution}). 
\end{enumerate}
For instance, when $M=\gg^*$ the dual of a compact Lie algebra, its Weyl resolution is diffeomorphic to $G/T\times_W \tt^*$, where $T\subset G$ is a maximal torus with Lie algebra $\tt$, and $W$ is the classical Weyl group. The resolution map is then $\res([gT,\xi])=\Ad_g\xi$.

The Weyl resolution allows to reduce questions about the \emph{singular} Poisson geometry of $(M,\pi)$ to questions about the \emph{regular} Dirac geometry of $(\hM,\hL)$. Moreover, the map $\res$ induces an identification of the leaf spaces of $(\hM,\hL)$ and $(M,\pi)$, along with their algebras of smooth functions. This generalizes the well-known identification $\tt^*/W\cong \gg^*/G$ and a related classical theorem due to Chevalley, and leads to the following (see Section \ref{sec:leaf:space}):

\begin{theorem*}
\label{thm:intro:int:affine:leaf:space}
Each proper symplectic integration  $(\cG,\Omega)$ of a  Poisson manifold $(M,\pi)$ induces an integral affine orbifold structure on its leaf space $B$. Moreover: 
\begin{enumerate}[(i)]
    \item The underlying classical orbifold structure is independent of the choice of integration and is a good orbifold;
    \item The smooth stratifications of $M$ induce orbifold stratifications of $B$.
\end{enumerate}
\end{theorem*}

We recall that a classical good orbifold is one that can be realized as the quotient of a manifold by a discrete group acting properly and effectively. In particular, for any such orbifold, its orbifold universal covering space is a manifold. In our case, much more holds, as demonstrated by the following result (Section \ref{sec:Weyl:group}):

 \begin{theorem*}
\label{thm:intro:orbifold:fundamental:group}
The orbifold universal covering space $\oB$ of the leaf space $B$ of a Poisson manifold of proper type $(M,\pi)$ is an integral affine manifold. The orbifold fundamental group $\pi_1^\orb(B)$ acts on $\oB$ by integral affine transformations and fits into a split short exact sequence:
\begin{equation}
\label{eq:intr:short:sequence} 
\xymatrix{1\ar[r] & \cW\ar[r] & \pi_1^\orb(B)\ar[r] & \pi_1^\orb(B^\reg)\ar[r] & 1},
\end{equation}
where $B^\reg$ is the leaf space of the regular locus $M^\reg$. 
\end{theorem*}

We call the kernel $\cW(M,\pi):=\cW$ of the sequence \eqref{eq:intr:short:sequence} the \textbf{Weyl group} of $(M,\pi)$. We also call the connected components of the regular locus of the action of $\cW(M,\pi)$ on $\oB$ the \textbf{Weyl chambers} of $(M,\pi)$. In the next statement, to be proved in Section \ref{sec:Weyl:group:geometric}, by a \emph{geometric reflection} we mean an involution whose fixed point set is an integral affine submanifold that separates $\oB$:

\begin{theorem*}
\label{thm:intro:Weyl:group}
The Weyl group $\cW(M,\pi)$ of a Poisson manifold of proper type is generated by geometric reflections and acts transitively and freely on the set of Weyl chambers. In particular, $\cW(M,\pi)$ is a Coxeter group.
\end{theorem*}

When $M$ is the dual of a compact Lie algebra $\mathfrak{g}$, equipped with the canonical linear Poisson structure, the group $\mathcal{W}(\mathfrak{g}^*,\pi_{\gg^*})$ coincides with the classical Weyl group (Section \ref{ex:regular-coadjoint-orbits}). When $M$ is a compact connected Lie group $G$, equipped with the Cartan-Dirac structure, the group $\mathcal{W}(G,L_G)$ coincides with the affine Weyl group of $G$ (Section \ref{ex:conjugacy-classes}). These facts, along with the many similarities between $\mathcal{W}(M,\pi)$ and the classical Weyl groups, justify our use of the term ``Weyl group". It should be noted, however, that the definition of the Weyl group of a Poisson manifold of proper type does not rely on classical Lie theory, and there exist many examples of Poisson manifolds whose Weyl group is not a classical Weyl group (Section \ref{sec:examples:parts:removed}).  

All the results stated so far have analogs for twisted Dirac structures. In fact, the main tool used in their proofs is the local normal form around leaves for symplectic, and more generally, twisted presymplectic groupoids, as discussed in Section \ref{sec:nform}. Next, we turn to results that are specific to Poisson manifolds.

We recall from \cite{CFM-I} that a Poisson manifold $(M,\pi)$ is of \textbf{s-proper type} if it admits a symplectic integration $(\cG,\Omega)\tto M$ whose source map is proper and has connected fibers. It is of \textbf{compact type} if it admits a symplectic integration $(\cG,\Omega)\tto M$ that is compact. One has obvious implications
\[ \text{compact type} \quad\Rightarrow\quad \text{s-proper type} \quad\Rightarrow\quad \text{proper type}. \]
A Poisson manifold of s-proper type has compact symplectic leaves. The classical \textbf{Duistermaat-Heckman theory} generalizes to these type of Poisson manifolds as follows.

First, using the Weyl resolution and the integral affine structure on the leaf space, we will show that for an s-proper Poisson manifold $(M,\pi)$, the cohomology class of the symplectic forms $\omega_b$ of the leaves $S_b$ varies linearly. More precisely, one has a flat integral orbivector bundle over the leaf space $B$ with fibers the second cohomology groups of the symplectic leaves:
\[ \cH\to B,\quad \cH_b:=H^2(S_b), \]
and we have (Section \ref{sec:variation:volume}):

 \begin{theorem*}
\label{thm:intro:linear:variation}
For any s-proper Poisson manifold $(M,\pi)$ with leaf space $B$, the map
\[B\ni b\mapsto [\omega_b]\in H^2(S_b) \]
defines an integral affine section of $\cH\to B$. Moreover, the (signed) leafwise symplectic volumes define a polynomial function $\VV_0:\oB\to\R$ on the integral affine manifold $\oB$, whose square descends to a smooth function $\VV_0^2:B\to\R$.
\end{theorem*}

Next, an s-proper integration $(\cG,\Omega)$ induces a \emph{Duistermaat-Heckman measure} $\boldsymbol{\mu^{DH}_{B}}$ on $B$ by first pushing forward the Liouville measure $\frac{\Omega^\top}{\top!}$ along the source map, and then along the quotient map $p:M\to B$. On the other hand, the integral affine structure on $B$ mentioned previously gives rise to a \emph{Lebesgue measure} $\boldsymbol{\mu^{\aff}_{B}}$. 
In Section \ref{sec:DH:formula} we will show that the classical Duistermaat-Heckman and the Weyl integration formulas extend to our setting (we refer to that section for explanations of notation).

\begin{theorem*} 
\label{thm:introd:Weyl:theory}
For any s-proper integration $(\G,\Omega)$ of $(M, \pi)$ one has
\[
\boldsymbol{\mu^{DH}_{B}}= \VV^2\cdot \boldsymbol{\mu^{\aff}_{B}}.
\]
Moreover, for all $f\in \cC_c^{\infty}(M)$ one has
\[ \int_M f(x) \, \boldsymbol{\d \mu^{DH}_{M}} (x) =  \int_B \left(\int_{\cB(\hx, -)}f(\res(t(g)))\,\d\mu_{\cB(\hx, -)}^{\aff}(g) \right) \VV^2(b) \,  \boldsymbol{\d \mu^{\aff}_{B} } (b).\]
\end{theorem*}  

In Section \ref{sec:compact:completeness}, we deduce structural results about Poisson manifolds of s-proper type and compact type. For example, just as the cohomology of a symplectic form on a compact manifold is constrained, the existence of a multiplicative symplectic form on a compact groupoid forces the groupoid to be regular. In fact, we will show the following.

\begin{theorem*}
Any Poisson manifold of compact type is regular. Equivalently, every compact symplectic groupoid is regular.
\end{theorem*}

We emphasize that the proof of this result is non-trivial, involving virtually all the previous theorems, along with some new facts about polynomial functions on integral affine orbifolds.

Finally, in Section \ref{sec:open}, we conclude the paper by enumerating several open problems that arise naturally from the theory developed here.
\medskip

\textbf{Acknowledgments.} The authors enthusiastically acknowledge the support of the Oberwolfach Research in Pairs program. We are also encouraged by the fact that several of our former graduate students, Ioan M\u{a}rcu\cb{t}, Jo\~ao Nuno Mestre, Maarten Mol, Joel Villatoro and Luka Zwaan, have been pursuing various research projects that originated from this work.


\
\medskip

\section{Local Normal Forms}
\label{sec:nform}
In this section we describe two normal forms around symplectic leaves for Poisson manifolds of proper type and their associated proper symplectic groupoids. These normal forms yield important tools for our study of non-regular PMCTs.

The first normal form is the \emph{linear normal form} which we have already encountered in the first paper of this series \cite[Section 8]{CFM-I}. It is a normal form which holds on any small enough tubular neighborhood of a symplectic leaf of a Poisson manifold of proper type. It linearizes, in an appropriate sense, the Poisson structure around the leaf.

The second normal form is the so-called \emph{Hamiltonian normal form} which holds on an entire saturated neighborhood of a symplectic leaf of any Poisson manifold of proper type. This saturated neighborhood may fail to retract on the leaf and the Poisson structure in this model is not linear. Still, it exhibits the Poisson structure as a symplectic quotient of a Hamiltonian $G$-space and this will turn out to be useful for us.

We also discuss suitable generalizations of these normal forms to the Dirac setting.
These will be needed later when we study the desingularization of PMCTs. 

\subsection{The Hamiltonian normal form}
\label{sec:ex:simple-local-model}
Consider a (right) Hamiltonian $G$-space 
\[ \mu: (Q, \omega)\to \gg^*. \] 
If the action is free and proper, then $Q/G$ inherits a {\it reduced Poisson structure $\pi_\red$} uniquely characterized by requiring $p^{*}: C^{\infty}(Q/G)\hookrightarrow C^{\infty}(Q)$ to be a morphism of Poisson algebras, where $p:Q\to Q/G$ is the quotient map. Geometrically, $\pi_\red$ is the Poisson structure with symplectic leaves the connected components of the symplectic quotients $\mu^{-1}(\xi)/G_\xi$ (see, e.g., \cite[Section 1.5]{CFM21}). In our set up, $0$ is in the image of $\mu$ and we focus on and around the symplectic quotient $\mu^{-1}(0)/G$, which we denote by $(Q//G, \omega_\red)$.

\begin{theorem}\label{thm-lf-Poisson-v2} 
Let $(S, \omega_S)$ be a symplectic leaf in a Poisson manifold $(M,\pi)$ of proper type. Then there exists a saturated open neighborhood $U$ of $S$ in $M$ that is Poisson diffeomorphic to the reduced space $Q/G$ associated to some connected, free, Hamiltonian $G$-space 
\begin{equation}\label{eq:Ham:space} 
\mu: (Q, \omega)\to \gg^*,
\end{equation}
where $G$ is a compact Lie group, $\mu$ has connected fibers and $0\in \mu(Q)$. Under this diffeomorphism $(S, \omega_S)$ corresponds to $(Q//G, \omega_\red)$. 
\end{theorem}

Letting $V= \mu(Q)\subset \gg^*$,  the Hamiltonian normal form will also be pictured by the diagram 
\begin{equation}\label{eq:diagram:dual-pair}
\vcenter{\xymatrix{
 &  & (Q,\omega) \ar[dll]^{p}\ar[drr]_{\mu}  & &  \\
(U,\pi)&  & & & (V,\pi_\gg^*)}}
\end{equation}
This indicates that we think of $(Q, \omega)$ as defining, locally around $S$, a \emph{dual pair} (see, e.g., \cite[Section 6.4]{CFM21}) between $(M, \pi)$ and the linear Poisson structure $(\gg^*,\pi_\gg^*)$.

\begin{proof}[Proof of Theorem \ref{thm-lf-Poisson-v2}]
The Hamiltonian space (\ref{eq:Ham:space}) arises from a proper symplectic groupoid $(\G,\Omega)$ integrating $(M,\pi)$. Namely, $G$ is the isotropy group $\G_x$, $Q$ is the preimage $\s^{-1}(T)$ of a transversal $T$ to $S$ through a base point $x\in S$, $\omega$ is the restriction $\Omega|_{\s^{-1}(T)}$ and $U$ is the saturation of $T$. 
In order to describe the $G$-action on $Q$, by general facts about symplectic groupoids (see Remark \ref{remark:Morita-general} below), we have a symplectic Morita equivalence:
\[
\xymatrix{
 (\G|_U,\Omega|_U) \ar@<0.25pc>[d] \ar@<-0.25pc>[d]  & \ar@(dl, ul) & (Q,\omega) \ar[dll]^{\t}\ar[drr]_{\s}  & \ar@(dr, ur) & (\G|_T,\Omega|_T) \ar@<0.25pc>[d] \ar@<-0.25pc>[d]  \\
(U,\pi)&  & & & (T,\pi_T)}
\]
where $\pi_T$ is the Poisson structure induced by $\pi$ on the transversal. Since $(\G|_T,\Omega|_T)$ is proper, by the Weinstein-Zung Linearization Theorem \cite{We3,Zu} if $T$ is taken sufficiently small it follows that $(\G|_T,\Omega|_T)$ is isomorphic to $(T^*G|_{V},\Omega_\can)$, where $V\subset \gg^*$ is a $G$-invariant neighborhood of $0$. Since $T^*G|_{V}$ can be identified with the coadjoint action groupoid $G\ltimes V$, we obtain an action of $G$ on $Q$.

Under this isomorphism, $(T,\pi_T)$ is identified with $(V,\pi_{\gg^*})$, where $\pi_{\gg^*}$ denotes the linear Poisson structure on the dual of the Lie algebra. One then obtains the groupoid version of the diagram (\ref{eq:diagram:dual-pair})  
\[
\xymatrix{
 (\G|_U,\Omega|_U) \ar@<0.25pc>[d] \ar@<-0.25pc>[d]  & \ar@(dl, ul) & (Q,\omega) \ar[dll]^{p=\t}\ar[drr]_{\s}  & \ar@(dr, ur) & (T^*G|_{V},\Omega_\can) \ar@<0.25pc>[d] \ar@<-0.25pc>[d]  \\
(U,\pi)&  & & & (V,\pi_{\gg^*})}
\]
showing that $(Q, \omega)$ is a symplectic Morita equivalence between $\G|_{U}$ and $T^*G|_{V}$. Using the identification $T^*G|_{V}\simeq G\ltimes V$, the statement follows setting $\mu:=-\s$.
\end{proof}

Notice that the proof also gives a local normal form for the symplectic groupoid $(\cG,\Omega)$. The symplectic Morita equivalence yields an identification between $(\G|_U,\Omega|_U)$ and the ``gauge groupoid''
\begin{equation}\label{eq:Gauge-groupoid} 
((Q\times_{V} Q)/G, \underline{\Omega})\tto Q/G
\end{equation}
where $\underline{\Omega}$ is the symplectic form corresponding to the basic form
\begin{equation}\label{eq:Gauge-form} 
\textrm{pr}_1^*\omega- \textrm{pr}_2^*\omega\in \Omega^2(Q\times_{V} Q).
\end{equation}
See \cite[Proposition 5.6]{CFM-I} for more details.

\begin{remark}[Morita equivalences]\label{remark:Morita-general} Throughout the paper we will make use of Morita equivalences and their basic properties in various contexts. For a brief introduction and further references please see \cite[Section 3.1 and Appendix A3]{CFM-II}. For later use we recall here that any Morita equivalence between groupoids, 
\[
\xymatrix{
 \G_1 \ar@<0.25pc>[d] \ar@<-0.25pc>[d]  & \ar@(dl, ul) & Q\ar[dll]^-{p_1}\ar[drr]_-{p_2} & \ar@(dr, ur)   & \cG_2 \ar@<0.25pc>[d] \ar@<-0.25pc>[d]\\
M_1 & & & & M_2}
\]
gives rise to a homeomorphism between the the spaces $M_1/\cG_1$ and $M_2/\cG_2$ and isomorphisms between the isotropy groups of the two groupoids and their isotropy representations. We need here a more precise description of these isomorphisms, as well as of the general philosophy of using Morita equivalences to pass from one side to the other. 

First of all, we say that that two saturated subspaces $N_1\subset M_1$ and $N_2\subset M_2$ are \textbf{$Q$-related} if $p_1^{-1}(N_1)= p_{2}^{-1}(N_2)$. The homeomorphism  $M_1/\cG_1\to M_2/\cG_2$ is the correspondence that assigns to an orbit  $\mathcal{O}_1$ of $\cG_1$ the $Q$-related orbit of $\cG_2$. 

Secondly, we will say that a point \textbf{$x_1\in M_1$ is $Q$-related to $x_2\in M_2$} if the orbits through them are $Q$-related. That is equivalent to the existence of $q\in Q$ with $p_1(q)= x_1$, $p_2(q)= x_2$. The choice of such a point $q$ gives rise to:
\begin{enumerate}
\item[(i)] an isomorphism between the isotropy groups, $\psi_q: \G_{x_1}\to G_{x_2}$, uniquely determined by 
\[ g\cdot q= q\cdot \psi_q(g), \quad \forall\ g\in \G_{x_1}.\]
\item[(ii)] an $\psi_q$-equivariant isomorphism $\nu_{x_1}(\O_{x_1})\to \nu_{x_2}(\O_{x_2})$ between the representations of the isotropy groups on the normal spaces to the orbits. This isomorphism is uniquely determined by the condition that, for any $v\in T_qQ$, the class of $\d p_1(v)$ is sent to the one of $\d p_2(v)$.
\end{enumerate}
In the case of symplectic groupoids and symplectic Morita equivalences, we have $\nu_{x_i}=\gg_{x_i}^*$, and the normal isotropy representations of $G_{x_i}$ are identified with the coadjoint representations of $G_{x_i}$. In this case, the isomorphisms from item (ii) are dual to the ones from item (i). 
\end{remark}

\subsection{The linear normal form}
\label{sec:ex:local-model}
Given \emph{any} Poisson manifold $(M,\pi)$ and fixing a symplectic leaf $(S,\omega_S)$, Vorobjev \cite{Vorobjev01,Vorobjev05} has constructed a \emph{linear local model} which is a certain Poisson structure $\pi_\lin$ on an open neighborhood of the zero section in the normal bundle $\nu(S)$ with the following two properties: 
\begin{itemize}
\item $(S,\omega_S)$, identified with the zero section of $\nu(S)$, is a symplectic leaf of $\pi_\lin$;
\item on each fiber $\nu(S)_x\simeq \gg^*_x$, $\pi_\lin$ restricts to the linear Poisson structure $\pi_{\gg^*_x}$. 
\end{itemize}
In general, the germ of $\pi$ around $S$ may fail to be isomorphic to the germ of $\pi_\lin$ around the zero section. When they are isomorphic one says that $(M,\pi)$ is \emph{linearizable} around $S$. If, additionally, one can find arbitrary small saturated neighborhoods of $S$ where $\pi$ and $\pi_\lin$ are isomorphic then one says that $(M,\pi)$ is \emph{invariantly linearizable} around $S$. 

For PMCTs we have (see \cite[Section 8]{CFM-I}):

\begin{theorem}
\label{thm:linearization:PMCT}
If $(M,\pi)$ is a Poisson manifold of proper (respectively, s-proper) type then it is linearizable (respectively, invariantly linearizable) around any symplectic leaf.
\end{theorem}

For a general Poisson manifold the local linear model $(\nu(S),\pi_\lin)$ is not so simple to describe and depends on an auxiliary choice of an IM-connection (see \cite{FM22} for a modern approach). However, for a PMCT (in fact for any integrable Poisson manifold) one can describe the linear local model very explicitly. This was already observed in \cite[Section 5.8]{CFM-I} but we review it here in greater detail, since it will be needed in the sequel.

The starting data consists of:
\begin{enumerate}[(i)]
\item A symplectic manifold $(S, \omega_S)$; 
\item A principal $G$-bundle $p: P\to S$, with $P$ connected and $G$ compact;
\item An auxiliary principal connection 1-form:
\[
\theta\in \Omega^1(P, \gg).
\]
\end{enumerate}
The group $G$ acts diagonally on $P\times \gg^*$, where on the second factor one considers the coadjoint action. One has a 2-form
\begin{equation}\label{loc-mod-2-form} 
\omega^{\theta}_\lin:= p^*\omega_S-\d\langle \theta,\cdot\rangle \in \Omega^2(P\times \gg^*),
\end{equation}
where $\langle \theta,\cdot\rangle\in \Omega^1(P\times \gg^*)$ is the form given by 
\[ \langle \theta,\cdot\rangle_{(p,\zeta)}(v,w)= \langle \theta_p(v),\zeta\rangle. \]
The 2-form $\omega^{\theta}_\lin$ is closed, $G$-invariant, and it is non-degenerate at points of $P\times\{0\}$. 

Hence, there is a $G$-invariant open set $P\times\{0\}\subset \cV\subset P\times \gg^*$ such that $\omega^{\theta}_\lin$ restricts to a symplectic form on $\cV$.
It follows that $\omega^{\theta}_\lin$ induces a Poisson bivector $\pi^{\theta}_\lin$ on the quotient:
\[ \mathcal{M}:=\cV/G. \]
The pair $(\cM,\pi^{\theta}_\lin)$ is the desired {\bf linear local model}. 
Notice that when $S$ is compact the space $P$ is compact, and one can take $V$ of the form 
\[ \cV=P\times V, \]
where $V\subset \gg^*$ is a $G$-invariant neighborhood of $0$. In this case  the projection in the second factor defines a Hamiltonian $G$-space with connected fibers
\begin{equation}\label{eq:model-H-presympl}
\textrm{pr}_2: (P\times V, \omega^{\theta}_\lin)\to \mathfrak{g}^*,
\end{equation}
hence, the corresponding symplectic quotients 
\[ S_{\xi}:= P\times_{G} \mathcal{O}_{\xi} \simeq P/G_{\xi},\]
form the symplectic foliation of $(\cM,\pi^{\theta}_\lin)$.

For a PMCT $(M,\pi)$ the data needed to construct the linear local model around a symplectic leaf $(S,\omega)$ is obtained from a symplectic integration $(\G,\Omega)$. Fixing a base point $x\in S$, the relevant principal $G$-bundle $p:P\to S$ is
\begin{equation}\label{s-fiber:principal:bundle}  
\G_x\laction \s^{-1}(x)\stackrel{\t}{\rmap} S.
\end{equation}

\begin{proof}[Proof of  Theorem \ref{thm:linearization:PMCT}]
The proof consists of 4 steps:
\smallskip

\emph{Step 1:} Consider the  groupoid
\begin{equation}\label{l-mod-gpd-eq}  
\G_\lin:= (P\times P)\times_{G} \gg^*\tto P \times_{G} \gg^*,
\end{equation}
the quotient of the product of the pair groupoid $P\times P\tto P$ with the identity groupoid $\gg^*\tto \gg^*$. It carries a closed,
multiplicative form $\Omega^{\theta}_\lin$ induced by the basic 2-form  
\begin{equation}\label{l-mod-gpd-form-eq}  
\widetilde{\Omega}^{\theta}_\lin:=p_1^*\omega_S-p_2^*\omega_S-\d\langle \theta_1,\cdot\rangle+\d\langle \theta_2,\cdot\rangle\in \Omega^2(P\times P\times \gg^*)
\end{equation}
where  $p_i: P\times P\times \gg^*\to S$ denotes the composition of the projection on the i-th factor with the bundle projection $p: P\to S$, while $\langle \theta_i,\cdot\rangle$ denotes the 1-form
\[ \langle \theta_i,\cdot\rangle_{(p_1,p_2,\zeta)}(v_1,v_2,w)= \langle \theta_{p_i}(v_i),\zeta\rangle. \] 
The form $\Omega^{\theta}_\lin$ restricts to a symplectic form on $\cG_\lin|_{\cM}$, and this yields a symplectic groupoid integrating the linear local model $(\cM,\pi^\theta_\lin)$

\smallskip

\emph{Step 2:} If $(\G,\Omega)$ integrates $(M,\pi)$ and we let $p:P\to S$ be the principal bundle \eqref{s-fiber:principal:bundle}, the groupoid $\G_\lin$ is isomorphic to the action groupoid formed from the action of the restricted groupoid $\G_S:=\G|_S$ on the normal bundle $\nu(S)=P\times_G\gg^*$:
\[ \G_\lin \simeq \G_S\ltimes \nu(S). \]
The right-hand side is the linearization of $\G$ at $S$, see \cite{CS,HoFe,We3}.

\smallskip

\emph{Step 3:}  
If $\G\tto M$ is a proper groupoid, then by step 2 and the Weinstein-Zung linearization theorem \cite{CS,HoFe,We3,Zu} there are open sets $S\subset U\subset M$ and $S\subset \cM'\subset \nu(S)$ such that we have a groupoid isomorphism
\[ \G|_U  \simeq \G_\lin|_{\cM'}. \] 

If $\G\tto M$ is a s-proper groupoid, then we can take $U$ saturated and $\cM'=P\times_GV'$, with $V'\subset\gg^*$ a $G$-invariant open neighborhood of $0$.
\smallskip

\emph{Step 4:} The isomorphism in the previous step provides a second multiplicative symplectic form on $\G_\lin|_{\cM'}$, besides $\Omega^{\theta}_\lin$. Since these two forms agree on $\G_S$, by a multiplicative version of Moser's method, possibly after possibly shrinking $U$ and $\cM'$, one obtains a groupoid automorphism which pulls back the second symplectic form to $\Omega^{\theta}_\lin$, so that we have an isomorphism of symplectic groupoids 
\[ (\G|_U,\Omega|_U) \simeq (\G_\lin|_{\cM'},\Omega^{\theta}_\lin|_{\cM'}). \]
See also \cite[Section 8]{CFM-I}. 

A different proof of this result, using the Hamiltonian local form, follows from the proof of Theorem \ref{thm-lf-Dirac} (the case $\beta=0$).
\end{proof}

\begin{remark}
The construction of the linear local model depends on the choice of connection $\theta$. Indeed, both $\omega^{\theta}_\lin$ and the invariant open set $\cV\subset P\times \gg^*$ will depend on this choice, and hence also $(\cM,\pi^{\theta}_\lin)$. However, two different choices of connection lead to isomorphic local models on a possibly smaller invariant neighborhood of $S$ in $P\times_{G}\gg^*$. In fact, if $\theta$ and $\theta'$ are two connections with open invariant sets $\cV,\cV'\subset P\times\gg^*$, then on the overlap $\cV\cap \cV'$ equation \eqref{l-mod-gpd-form-eq}  shows that the symplectic forms  $\Omega^{\theta}_\lin$ and $\Omega^{\theta'}_\lin$ differ by $\d\beta$, where $\beta$ is a multiplicative 1-form vanishing on $\cG_S$. Hence, a multiplicative version of Moser's method implies that the corresponding symplectic groupoids are isomorphic around $S$. It follows that the corresponding linear local models $(\cM,\pi^{\theta}_\lin)$ and $(\cM',\pi^{\theta'}_\lin)$ are also isomorphic around $S$.
\end{remark}

\subsection{Local models in the Dirac setting} 
\label{sec:Dirac:local-model}
The previous constructions of the Hamiltonian local model and the linear local model extend to Dirac manifolds and presymplectic groupoids. We will be interested in this more general setting because we will introduce later a resolution of non-regular PMCTs and proper symplectic groupoids that produce regular DMCTs and proper presymplectic groupoids. 

The extension of the Hamiltonian local model to the Dirac setting is as follows:

\begin{theorem}\label{thm-lf-Dirac} 
Let $(S, \omega_S)$ be a presymplectic leaf in a proper Dirac manifold $(M,L)$. Then there exists a saturated open neighborhood $U$ of $S$ in $M$ that is Dirac diffeomorphic to the reduced space $Q/G$ associated to some connected, free, Hamiltonian presymplectic $G$-space 
\begin{equation}\label{eq:Ham:space:Dirac} 
\mu: (Q, \omega)\to \gg^*,
\end{equation}
where $G$ is a compact Lie group, $\mu$ is a submersion with connected fibers and $0\in \mu(Q)$. Under this diffeomorphism $(S, \omega_S)$ corresponds to $(Q//G, \omega_\red)$. 
\end{theorem}

Here by a Hamiltonian presymplectic $G$-space we mean a $G$-manifold $Q$ together with a $G$-invariant, closed 2-form $\omega$, and a $G$-equivariant map $\mu:Q\to\gg^*$ satisfying the usual moment map condition. We will always assume that the action if free and proper, and the moment map is a submersion -- in the presymplectic setting the latter is not implied by the former. These conditions guarantee that $Q/G$ inherits a reduced Dirac structure $L_\red$ uniquely determined by imposing that the quotient map is forward Dirac (see, e.g., \cite{BF14}).

The Hamiltonian presymplectic space in the theorem is constructed similarly to the Poisson setting. Using a proper presymplectic groupoid $(\G,\Omega)$ integrating $(M,L)$, one sets $Q:=\s^{-1}(T)$ and $\omega:=\Omega|_T$, where $T$ is a small enough transversal $T$ to the presymplectic leaf $(S, \omega_S)$. 

The proof is also similar to the Poisson setting where $(Q,\omega)$ gives a presymplectic Morita equivalence with  $(\G|_T,\Omega|_T)$. Here we need to apply the theory of presymplectic groupoids and presymplectic Morita equivalences (see \cite[Section 3.2]{CFM-I}, \cite[Appendix A]{CFM-II} and \cite{Xu04}). Note that the Dirac structure induced by $L$ on $T$ is actually the graph of a Poisson structure. Hence, the restriction $(\G|_T,\Omega|_T)$ is a symplectic groupoid and the Weinstein-Zung linearization theorem can still be applied.

\begin{remark}
\label{rem:presymplectic:groupoid:hamiltonian:form}
Similarly to what we observed after the proof of Theorem \ref{thm-lf-Poisson-v2}, one also deduces a normal form for the proper presymplectic groupoid $(\G,\Omega)$. One obtains a presymplectic groupoid isomorphism with the gauge groupoid \eqref{eq:Gauge-groupoid}, equipped with the presymplectic form induced by \eqref{eq:Gauge-form}, formed from $(Q,\omega)$ above.
\end{remark}

\medskip

Let us now turn to the linear local model. The Dirac setting actually makes the construction of the model simpler, since one does not need to restrict to a $G$-invariant neighborhood of the origin in $\gg^*$. Furthermore, the linear model also works for non-compact Lie groups, so we momentarily drop that assumption. 

For the starting data one now assumes that $(S, \omega_S)$ is just a presymplectic manifold. Then, given a principal $G$-bundle $p: P\to S$ with $P$ connected, and an auxiliary principal connection 1-form $\theta\in \Omega^1(P, \gg)$,
the coadjoint bundle
\[ \cM_\lin:= P\times_{G} \mathfrak{g}^*, \]
carries a natural Dirac structure $L^{\theta}_\lin$. Namely, the quotient Dirac structure induced by the graph of the closed form $\omega^{\theta}_\lin\in\Omega^2(P\times\gg^*)$ given by \eqref {loc-mod-2-form}. Explicitly, this Dirac structure is given by
\[ L^{\theta}_\lin:= \big\{ ( \d q(v), \eta): q^*\eta= i_{v}\omega^{\theta}_\lin \big\},\] 
where $q:P\times\gg^*\to \cM_\lin$ denotes the quotient map. Here, just as in the Poisson case, the projection 
\[ \textrm{pr}_2: (P\times \gg^*, \omega^{\theta}_\lin)\to \mathfrak{g}^*, \]
defines a Hamiltonian presymplectic $G$-space, and the presymplectic foliation of $L^{\theta}_\lin$ consists of the presymplectic quotients 
\[ S_{\xi}:= P\times_{G} \mathcal{O}_{\xi} \simeq P/G_{\xi}\subset  \cM_\lin,\]
where now $\xi\in\gg^*$ is arbitrary.

The construction of a canonical presymplectic groupoid integrating $(\cM_\lin, L^{\theta}_\lin)$ is as in Step 1 in the proof of Theorem \ref{thm:linearization:PMCT}.  
The Dirac manifold $(\cM_\lin, L^{\theta}_\lin)$ will be called the \textbf{(Dirac) linear local model} associated 
to $p: P\to (S, \omega)$ and $\theta$, while $(\G_\lin,  \Omega^{\theta}_\lin)$ will be called the \textbf{linear groupoid local model}. We will often drop the connection from the notation, for the reason explained in the next remark, and write simply $(\cM_\lin, L_\lin)$ and $(\G_\lin,\Omega_\lin)$.

\begin{remark} $\quad$
\begin{enumerate}[1)]

\item The choice of connection $\theta$ affects the linear local model as follows. Given two principal connections $\theta$ and $\theta'$, the closed 2-form
\[ \d\langle \theta,\cdot\rangle-\d\langle \theta',\cdot\rangle\in \Omega^2(P\times \gg^*) \]
is a basic form, so induces a closed 2-form $\beta\in\Omega^2(\cM_\lin)$. The Dirac structures are related by a gauge transformation in the sense of \cite{BR03}:
\[ L^{\theta'}_\lin=e^\beta L^{\theta}_\lin. \]
Similarly, using \ref{l-mod-gpd-form-eq}, at the groupoid level the presymplectic forms change by a gauge transformation
\[ \Omega^{\theta'}_\lin=\Omega^{\theta}_\lin+\t^*\beta-\s^*\beta. \] 

\item Assume that in the preceeding discussion the principal bundle is over a symplectic base $(S, \omega_S)$. Then $S$ admits a \emph{Poisson neighborhood} $\mathcal{M}$, i.e., there is an open neighborhood 
\[  S\subset \cM\subset \cM_\lin,\]
where $L^{\theta}_\lin$ is defined by a bivector $\pi^{\theta}_\lin$. Indeed, the set of points where a Dirac structure is the graph of a bivector is open and, since $\omega_S$ is non-degenerate, the ``Poisson support'' of $L^{\theta}_\lin$ contains $S$. The restriction of $\G_\lin$ to $\cM$ is then a symplectic groupoid and one recovers in this way the Poisson linear local model and its integration.

\item We will be interested in the case where $G$ is compact. In that case, $\G_\lin$ is a proper groupoid, so the Dirac linear local model is always a Dirac manifold of proper type. It is of strong proper type iff $P$ is 1-connected, 
and it is of s-proper type iff $P$ is compact (i.e., $S$ is compact).
\end{enumerate}
\end{remark}

Finally, let us turn to the Dirac analogue of Theorem \ref{thm:linearization:PMCT}. Linearization, in general, can only be achieved up to gauge transformation. This is in agreement with the fact that in the Dirac category an isomorphism consists of a diffeomorphism composed with a gauge transformation.



\begin{theorem}
\label{cor:nform} 
Let $(\G, \Omega)$ be a proper presymplectic groupoid which integrates the Dirac manifold $(M,L)$ and let $(S, \omega_S)$ be a presymplectic leaf. Let $G$ be the isotropy group of $\cG$ at some base point $x\in S$, let $p=\t: P=\s^{-1}(x)\to S$ be the isotropy-principal bundle at $x$ and let $(\G_\lin,\Omega_\lin)$ be the resulting linear Dirac local model. Then there are arbitrarily small neighborhoods $S\subset U \subset M$ and $S\subset U_\lin\subset \cM_\lin$ and an isomorphism of presymplectic groupoids:
\[ 
\xymatrix@=0.3pc{ 
(\G|_U,\Omega|_{U}) \ar@<0.25pc>[dd] \ar@<-0.25pc>[dd] & & (\G_\lin|_{U_\lin},\Omega_\lin+\t^*\beta-\s^*\beta) \ar@<0.25pc>[dd] \ar@<-0.25pc>[dd]  \\
& \simeq &     \\ 
U & & U_\lin
}
\]
for some exact 2-form 
\[ \beta=\d\gamma \in \Omega^2(U_\lin). \] 
Under this isomorphism, $S$ is mapped diffeomorphically to  $\mu^{-1}(0)/G$, and the following can be arranged:
\begin{enumerate}[(i)]
\item If $\G$ is source proper, then $U$ and $U_\lin$ can be chosen saturated;
\item If $\G$ is source 1-connected, then one can choose $\gamma$ whose pullback to $S$ is zero.
\end{enumerate}
\end{theorem}

\begin{proof}[Proof of Theorem \ref{cor:nform}]

By Theorem \ref{thm-lf-Dirac} we may assume that:
\begin{enumerate}[1)]
    \item we have a Hamiltonian presymplectic $G$-space $\mu: (Q, \omega)\to \gg^*$, with $\mu$ a submersion;
    \item the Dirac manifold and presymplectic leaf are the quotients
    \[ M= Q/G, \quad S=Q//G= \mu^{-1}(0)/G;\] 
    \item The integrating groupoid $\cG$ is the associated gauge groupoid 
    \[ \cG=Q\times_G Q\tto M. \]
\end{enumerate} 
After possibly replacing $Q$ by a smaller $G$-invariant neighborhood of $P:= \mu^{-1}(0)$ in $Q$, one may assume that we have a $G$-equivariant submersive retraction $r: Q\to P$, and that $(r, \mu): Q\to P\times \gg^*$ is a diffeomorphism onto an open subspace of $P\times \gg^*$. In other words, we may assume that $Q\subset P\times \gg^*$ is an $G$-invariant open subset and $\mu$ is the projection $\textrm{pr}_2$.

We then have two multipicative presymplectic forms $\Omega$ and $\Omega_{\lin}^\theta$ on the same gauge groupoid, and both can be recovered from their pullbacks to $Q$,  $\omega $ and $\omega_{\lin}^\theta$ (see Remark \ref{rem:presymplectic:groupoid:hamiltonian:form}).
These forms are $G$-invariant and their difference $\omega-\omega^{\theta}_\lin$ is $G$-basic since (i) each makes $\mu:Q\to\gg^*$ into a presymplectic $G$-Hamiltonian space and (ii) the moment map conditions give, for any $v\in \gg$, 
\[i_{v_{Q}}(\omega-\omega^{\theta}_\lin)=\d \langle\mu, v\rangle-\d \langle\mu, v\rangle=0,\]
where $v_Q$ is the fundamental vector field of the action of $v$ on $Q$. Moreover, 
 the difference $\omega-\omega^{\theta}_\lin$ pulls back to zero along $P$. Therefore, there  exist $\beta\in\Omega^2(U_\lin)$ satisfying
\[\omega-\omega^{\theta}_\lin=\t^*\beta, \]
where $\t:\cG_\lin\to U_\lin$ is the target map, 
and such that its pullback to $S=S\times \{0\}\subset U_\lin $ vanishes.
Since we can assume that $S$ is a deformation retract of $U_\lin$, we conclude that  $\beta=\d\gamma$, $\gamma\in \Omega^1(U_\lin)$. This proves the first half of the theorem.

If $\cG$ is source proper, then in the Weinstein-Zung linearization theorem we can take arbitrary small saturated neighborhoods, and item (i) follows.

If $\cG$ is source 1-connected, then the long exact homotopy sequence for the projection $p:P\to S$ implies that $S$ has finite fundamental group, and therefore trivial first de Rham cohomology group. Hence, the pullback of $\gamma$ to $S$ is exact with primitive $f\in C^\infty(S)$. If we let $F\in C^\infty(U_\lin)$ be an extension of $f$, then 
\[\gamma-\d F\in \Omega^1(U_\lin)\]
is a primitive for $\beta$ whose pullback to $S$ vanishes. This proves item (ii).

\end{proof}

\begin{remark}\label{rmk-nf-tDirac}
Theorems \ref{thm-lf-Dirac} and \ref{cor:nform} have $\phi$-twisted versions. If on a $\phi$-twisted Dirac manifold $(M,L)$ one fixes a leaf $(S,\omega_S)$, then one obtains:
\begin{enumerate}[(i)]
    \item A Hamiltonian local model, as in Theorem \ref{thm-lf-Dirac}, where now we have a Hamiltonian twisted presympletic $G$-space $(Q,\omega)$, with $\d\omega=p^*\phi$;
    \item A linear local model, where in the starting data one has a twisted presymplectic manifold $(S,\omega_S)$. The resulting linear local model has a twisted Dirac structure $L_\lin$ with twist $\d(\pr^*\omega_S)$, where $\pr:P\times_G\gg^*\to S$ is the (tubular neighborhood) projection. Theorem \ref{cor:nform} still holds, where the integrating groupoid is $(\cG_\lin,\Omega^\theta_\lin)$, with $\d(\pr^*\omega_S)$-twisted presymplectic form 2-form still given by formula \eqref{l-mod-gpd-form-eq}. 
\end{enumerate}

There are also linear local models and linearization results for Poisson (respectively, Dirac) manifolds along Poisson (respectively, Dirac) submanifolds, generalizing the ones along leaves (see \cite{FM22}). 
\end{remark}

\section{The Canonical Stratifications}
\label{sec:canonical:stratifications}


This section is the starting point of our investigation of the geometry of the leaf spaces of possibly non-regular PMCTs. We shall define two canonical stratifications, we will analyze the Poisson geometric properties of their strata, and we will describe the structure induced on the leaf spaces of the strata by a choice of a proper symplectic integration.

The stratifications can be introduced right away. Recall that any Poisson manifold $(M, \pi)$ induces an algebroid structure on $T^*M$, with anchor $\pi$ interpreted as a map $\pi^{\sharp}: T^*M\to TM$. 
The corresponding symplectic foliation is 
\begin{equation}\label{eq:Poisson:fix} 
\cF_{\pi}:= \im(\pi^{\sharp}),
\end{equation}
and the isotropy Lie algebras coincide with the conormal spaces to the foliation
\begin{equation}\label{eq:Poisson:fix2} 
\gg_x(M, \pi):= \ker(\pi^{\sharp}_{x})=\nu^{*}_{x}(\cF_{\pi}):=(\cF_{\pi})^{o}\subset T^*_xM.
\end{equation}
When no confusion arises, we will use the simpler notation $\gg_x$ for the isotropy Lie algebras.

We define a first equivalence relation on $(M,\pi)$ by setting
\begin{equation}\label{eq:inf-equiv}
x \underset{\inf}{\sim} y \quad\textrm{iff}\quad \gg_x \simeq \gg_y.
\end{equation}
We define a second equivalence relation, using the linear holonomy representation of the symplectic leaf $S$ through $x\in M$ on the center of the isotropy $\zz(\gg_x)$ -- see \eqref{eq:rho:lin:hol} -- by setting
\begin{equation}\label{eq:equiv-inv} 
x \underset{\hol}{\sim} y \quad\textrm{iff}\quad 
\begin{cases}
\gg_x \simeq \gg_y \\  
\zz(\gg_x)^{\Hol}\simeq \zz(\gg_y)^{\Hol},
\end{cases}
\end{equation} 
where the superscript denotes the fixed-point set. 

\begin{definition} 
\label{def:canonical:stratifications:PMCT}
Let $(M,\pi)$ be a Poisson manifold of proper type. 
\begin{enumerate}[(i)]
 \item Its {\bf canonical infinitesimal stratification} $\cSi(M, \pi)$ is the partition of $M$ 
 by the connected components of the equivalence classes w.r.t $\underset{\inf}{\sim}$.
 \item Its {\bf canonical stratification} $\cS(M, \pi)$  is the partition of $M$ 
 by the connected components of the equivalence classes w.r.t  $\underset{\hol}{\sim}$.
 \end{enumerate}
\end{definition}

For any proper Lie groupoid there exist two canonical stratifications (see \cite{CraMe,PPT14}), which generalize the well-known {\it orbit type} and {\it infinitesimal orbit type} stratifications of a proper Lie group action (see, e.g., \cite{DK,Pflaum}). We will see that the canonical stratifications of a Poisson manifold of proper type coincide with the stratifications obtained from \emph{any} proper symplectic integration. 

We emphasize that the partitions in Definition \ref{def:canonical:stratifications:PMCT} make sense for any Poisson manifold, integrable or not. However, these partitions may be ill-behaved and fail to satisfy the requirements to fit into a stratification. For a Poisson manifold of proper type the situation is much nicer as they do form stratifications which enjoy remarkable Poisson-theoretic properties. We will show that:

\begin{theorem} 
Let $(M,\pi)$ be a Poisson manifold of $\cC$-type. Then the strata of both the canonical and the infinitesimal stratifications are regular Poisson submanifolds of $(M,\pi)$ of $\cC$-type. Furthermore, the infinitesimal strata are core Poisson submanifolds of $(M,\pi)$.
\end{theorem}
 
The notion of (regular) core Poisson submanifold, mentioned in the theorem, arises in connection with the transverse geometry of the ambient Poisson manifold. More precisely,  core Poisson submanifolds retain all the information about monodromy of the ambient Poisson manifold (see Definition \ref{def:core}).

A more detailed version of the previous theorem states that, from a proper integration of $(M, \pi)$, one obtains proper integrations of each member of either canonical stratification. Since the strata are regular Poisson manifolds, the results in the second paper in this series \cite[Chapter 3]{CFM-II} show that the leaf space of each stratum $\Sc\in\cS(M,\pi)$ and $\Si\in\cSi(M,\pi)$ is an integral affine orbifold. Moreover, the strata are saturated submanifolds of the ambient Poisson manifold so if $\Sc\subset\Si$, one obtains inclusions of leaf spaces:
\[ \Sc/\cF_\pi \subset \Si/\cF_\pi\subset M/\cF_\pi. \]
It is natural to wonder if the integral affine orbifold structures on $ \Sc/\cF_\pi$ and $\Si/\cF_\pi$ are induced by an integral affine orbifold structure on $M/\cF_\pi$. However, the ambient Poisson manifold is not regular, so \cite{CFM-II} does not apply to it. We will be able to address this issue once we have introduced in Section \ref{sec:leaf:space} the desingularization of $(M,\pi)$. Hence, we postpone to that section the discussion of leaf spaces of PMCTs.




\subsection{Stratifications and orbit types} 
Before we discuss in detail the canonical stratifications of  PMCTs, we will start by providing a short overview of the basics on stratifications and we will recall the orbit type stratification of a proper Lie groupoid.

\subsubsection{General stratifications} 

We consider stratifications of a space $M$ in the following sense.

\begin{definition}\label{def:str:man}
A {\bf stratification} $\cS$ of a topological space $M$ is a locally finite partition into locally closed, connected, subspaces -- called \emph{strata} -- which are manifolds and satisfy the following frontier condition: the closure of each strata $S\in \cS$ satisfies
\[ \overline{S}=S\cup \bigcup_i S_i , \]
where $S_i$ are strata satisfying $\dim S_i<\dim S$. When $M$ is a manifold and each stratum is an embedded submanifold we call $\cS$ a  {\bf smooth stratification}.
\end{definition}



\begin{remark}[connected strata]\label{rem:starufications:connectedness}
The condition that the strata be connected is more important than it may seem at first sight. This condition makes the notion of stratification into a local one, which agrees with Mather's germ-viewpoint on stratifications but without reference to germs (see, e.g., \cite{Pflaum}). In practice, one usually starts with a locally finite partition $\cP$ of $M$ by manifolds and one passes to the partition $\cP^{c}$ consisting of the connected components of the members of $\cP$. In many examples (e.g., for the partition by orbit types), the frontier condition holds only after this passage. That implies, in particular, that a stratification may be induced by several different (but interesting) partitions. To handle such situations, one has the following simple lemma (see, e.g., \cite{CraMe}).

\begin{lemma}
\label{stratif-lemma-1}
Let $\cP_{i}$, $i\in \{1, 2\}$ be two partitions of a space $M$ by manifolds with the subspace topology (whose connected components may have different dimensions). Let $\cP_{i}^{c}$ be the partition obtained by taking the connected components of the members of $\cP_i$. Then $\cP_{1}^{c}= \cP_{2}^{c}$ if and only if, for each $x\in M$, there exists an open neighborhood $U$ of $x$ in $M$ such that
\[ P_1\cap U= P_2\cap U,\]
where $P_i\in \cP_i$ are the members containing $x$.
\end{lemma}
\end{remark}

The strata of a stratification $\cS$ of $M$ are partially ordered by
\[ S\preccurlyeq T \quad \Longleftrightarrow \quad S\subset \overline{T}.\]
The strata that are maximal w.r.t. $\preccurlyeq$ are precisely the open strata. The union of 
all the maximal strata form a dense open subset of $M$, called the \textbf{$\cS$-regular part} of $M$ and denoted
\[ M^{\cS-\reg}\subset M. \] 
A stratum $S$ is called {\bf subregular} if it is maximal among the non-regular strata.

The $\cS$-regular part is the first member of a filtration
\[ C_0 \subset C_1 \subset \ldots \subset C_n= M,\]
where $C_{k}$ consists of the disjoint union of all strata of codimension at most $k$. If $\cS$ is a smooth stratification, this codimension is relative to the ambient manifold $M$. For a general stratification the codimension of a stratum is relative to the dimension of the maximal stratum whose closure contains it. The stratification $\cS$ is recovered from the filtration by considering the connected components of the codimension $k$-locus:
\begin{equation}\label{sigma-k} 
\Sigma_{k}:= C_k\setminus C_{k-1} .
\end{equation}
The axioms on $\cS$ translate into properties of the filtration -- e.g., that $C_0$ is dense or, as a consequence of the frontier axiom, that each $\Sigma_k$ is a disjoint union of manifolds closed in $C_k$, so all the members of the filtration are open in $M$. One can also consider a similar filtration by dimension, but for us codimension arises more naturally. The previous discussion also implies the following result.

\begin{lemma}\label{stratif-lemma-aux} Let $\cS_1$ and $\cS_2$ be two stratifications of $M$ and assume that for every $x\in M$ the strata of $\cS_1$ and $\cS_2$ through $x$ have the same codimension. Then $\cS_1=\cS_2$. 
\end{lemma}


The following result should be standard, but since we could not find a proof in the literature we include one.

\begin{lemma}
\label{lem:stratif-lemma-4}  Let $\cS$ be a smooth stratification of a manifold $M$. 
If $\cS$ has no strata of codimension $k\in \{1, \ldots , p\}$ 
then the maps induced in homotopy by the inclusion, 
\[ i_*: \pi_{j} (M^{\cS-\reg})\to \pi_j(M),\]
are isomorphisms in all degrees $j\in \{0, 1, \ldots, p-1\}$. 
\end{lemma} 


\begin{proof}
If $S\subset M$ is a submanifold of codimension $k$, then the inclusion $M\setminus S\hookrightarrow M$ induces an isomorphism 
\[ \pi_j(M\setminus S)\simeq \pi_j(M), \quad (0\le j< k-1).\] 
This is a folklore result which can be proved representing elements of homotopy groups by smooth maps and using  transversality theory.

Now apply this result to the codimension $k$ submanifold $\Sigma_k$ of $C_k$. Since $C_{k-1}=C_k\setminus\Sigma_k$, we conclude that if $\cS$ has no strata of codimension $k\in \{1, \ldots , p\}$, then
\[ \pi_j(M)=\pi_j(C_n)=\cdots=\pi_j(C_p)=\pi_j(M^{\cS-\reg})\quad (0\le j < p). \]
\end{proof}


\subsubsection{The stratifications induced by a proper action}
\label{sec:stratifications proper action}
An important class of examples of stratifications comes from proper actions.

At the global level, an action of a Lie group $G$ on a manifold $M$ determines the following equivalence relations on $M$:
\begin{itemize}
\item {\bf Isotropy types}:  $x\sim y$ if the isotropy groups $G_x$ and $G_y$ are isomorphic.
\item  {\bf Orbit types}: $x\underset{\textrm{orb}}{\sim} y$ if the isotropy groups $G_x$ and $G_y$ are conjugate.
\item {\bf Local orbit types:}  $x\underset{\textrm{loc}}{\sim} y$ if $G_y= \textrm{Ad}_{g} G_y$ and there is an isomorphism of the normal isotropy representations $G_x\laction\nu_x$ and $G_y\laction\nu_y$ compatible with $\textrm{Ad}_{g}$.
\item {\bf Morita types}:  $x\underset{\textrm{Mor}}{\sim} y$ if the normal isotropy representations $G_x\laction\nu_x$ and $G_y\laction\nu_y$ are isomorphic (this is equivalent to the existence of a Morita equivalence between neighborhoods of the orbits through $x$ and $y$, preserving these orbits).
\end{itemize}
There are some obvious implications
 \begin{equation}\label{eq:diagram:4:stratif} 
 \xymatrix{ 
 \underset{\textrm{orb}}{\sim}   \ar@{=>}[r]  & \sim \\
 \underset{\textrm{loc}}{\sim}  \ar@{=>}[r]  \ar@{=>}[u] & \underset{\textrm{Mor}}{\sim}  \ar@{=>}[u] 
 }
 \end{equation}
For a proper action, all these equivalence relations determine partitions of $M$ by submanifolds. Upon passing to connected components, these partitions induce the same smooth stratification of $M$. We call it the \textbf{canonical stratification} induced by the proper action, denoted $\cS_{G}(M)$. We also have infinitesimal versions of the equivalence relations above, such as:
\begin{itemize}
\item {\bf Infinitesimal isotropy types}:  $x\underset{\inf}{\sim} y$ if $\gg_x$ and $\gg_y$  are isomorphic. 
\item {\bf Infinitesimal Morita types}: $x\overset{\textrm{Mor}}{\underset{\inf}{\sim}} y$  if $(\gg_x, \nu_x)$ and $(\gg_y, \nu_y)$ are isomorphic as pairs consisting of a Lie algebra and a representation.  
\end{itemize}
Again one finds that, upon passing to connected components, the partitions defined by them induce the same smooth stratification. We call it the \textbf{canonical infinitesimal stratification} induced by  the proper action, denoted $\cSi_{G}(M)$. We refer the reader to \cite[Section 4]{CraMe} for details. 

The strata of both $\cS_G(M)$ and $\cSi_G(M)$ are saturated by orbits of the action, and therefore give rise to two partitions of the orbit space $M/G$, denoted $\cS_G(M/G)$ and $\cSi_G(M/G)$. The first partition is a stratification of $M/G$ in the sense of Definition \ref{def:str:man}. The second one would also be a stratification if in the definition  we would  allow the members of the partition to be orbifolds. Furthermore, the two stratifications can be interpreted as ``smooth stratifications of $M/G$''. Indeed, although $M/G$ may fail to be a manifold, there are various known frameworks to make sense of ``smooth structures'' on $M/G$, of smooth subspaces and, ultimately, of smooth stratifications of $M/G$ (see Remark \ref{rmk:smooth:functions}).

\subsubsection{The stratifications induced by a proper Lie {groupoid}}
\label{sec:The stratifications induced by a proper Lie groupoid}
Most (but not all) of the previous equivalence relations for proper $G$-actions extend to general proper Lie groupoids $\G\tto M$. They give rise to the groupoid analogues of the canonical stratifications above. 

First of all, one can define right away the isotropy type equivalence relations $\sim$ and $\underset{\mathrm{inf}}{\sim}$ using the isotropy Lie group $\G_x$ and the associated Lie algebra $\gg_x$, at any point $x\in M$.  Then, since there are also isotropy representations $\nu_x\in \textrm{Rep}(\G_x)$ one can also introduce the Morita type equivalence relation $\underset{\textrm{Mor}}{\sim}$ and its infinitesimal version. Passing to connected components, and using that the slice theorem holds in this generality, one ends up with two smooth stratifications of $M$, induced by $\sim$ and $\underset{\mathrm{inf}}{\sim}$, and denoted
\[ \cS_{\G}(M) \quad\text{ and } \quad \cSi_{\G}(M).\]
We call them the \textbf{canonical stratification} and 
the \textbf{canonical infinitesimal stratification} induced by $\cG$, respectively. These are discussed in greater detail in \cite{CraMe}. 

The regular part of the stratifications $\cS_\cG(M)$ and $\cSi_\cG(M)$ are called the {\bf principal locus} and the {\bf regular locus} of $M$, respectively:
\[
M^{\princ}:= M^{\cS_{\cG}(M)-\reg},\quad
M^{\reg}:= M^{\cSi_{\cG}(M)-\reg}.
\]
Recalling that a point is regular w.r.t.~to a stratification if and only if the corresponding stratum is open, and using the slice theorem, one finds:
\begin{align*}
M^{\princ}&= \{x\in M: \textrm{the action of $\cG_x$ on $\nu_x$ is trivial}\},\\
M^{\reg}&= \{x\in M: \textrm{the action of $\gg_x$ on $\nu_x$ is trivial}\}.
\end{align*}
We will also call the \textbf{subregular locus} of $M$ the subregular part of the infinitesimal stratification:
\[ M^\subreg:=M^{\cSi_{\cG}(M)-\subreg}.\]

Similarly, at the level of the orbit space $B=M/\cG$, one obtains partitions
\[ 
\cS_{\G}(B) \quad\text{ and } \quad \cSi_{\G}(B).
\]
Again, these are ``smooth stratifications" in the more general sense mentioned before for orbit spaces of proper actions.

\begin{remark}\label{rem:correspondence-stratifications}
The 1-1 correspondence between subspaces of $B= M/\G$ and invariant subspaces of $M$ gives rise to a 1-1 correspondence
between stratifications of $B$ and $\G$-invariant stratifications of $M$ (provided the orbit spaces of the $\cG$-invariant strata are manifolds). The fact that the canonical projection $p: M\to B$ is an open map implies that the frontier condition is preserved.

\end{remark}

\begin{remark}[Morita equivalences II]\label{remark:Morita-general-stratifications} It follows from  Remark \ref{remark:Morita-general} that for Morita equivalent proper groupoids $Q$-related points have isomorphic isotropic groups. Hence, both canonical stratifications for proper groupoids are Morita invariant. 
\end{remark}

\subsubsection{Stratification of foliations of proper type}\label{sec:proper-foliations}
One class of groupoids that plays an important role for us is that of {\it foliation groupoids}, i.e., Lie groupoids $\cB\tto M$ that integrate a foliation $(M,\cF)$ (not necessarily s-connected). Equivalently, these are Lie groupoids
with all isotropy groups discrete. In particular this implies that the infinitesimal stratification on the base of such a Lie groupoid is trivial.

We recall from \cite[Chapter 2]{CFM-II} that a foliation $\cF$ is said to be $\cC$-type if it admits a s-connected integration that is of $\cC$-type.


%
\begin{proposition}
\label{prop:stratification:foliation}
Two proper s-connected foliation groupoids integrating a foliation $(M,\cF)$ induce the same canonical stratification.
\end{proposition}
\begin{proof} 
Recall that any foliation groupoid $\cB\tto M$ integrating $\cF$ covers the holonomy groupoid $\Hol(M,\cF)$, and the properness of $\cB$ forces the properness of $\Hol(M,\cF)$. Therefore it suffices to compare the stratifications induced by $\cB$ and $\Hol(M,\cF)$. 

In short, the argument is based on a combination of Lemma \ref{stratif-lemma-1} and the Morita equivalence underlying the normal form for $\cB$ from \cite[Theorem 2.5.5]{CFM-II}. More precisely, to check the condition from Lemma \ref{stratif-lemma-1} at a point $x\in M$ for the partitions associated to $\cB$ and $\Hol(M,\cF)$, we make use of \cite[Theorem 2.5.5]{CFM-II} applied at $x$ and we consider the resulting Morita equivalence
\[
\xymatrix{
 \cB|_{U} \ar@<0.25pc>[d] \ar@<-0.25pc>[d]  & \ar@(dl, ul) & Q\ar[dll]^-{p_1}\ar[drr]_-{p_2} & \ar@(dr, ur)   & \Gamma\ltimes V \ar@<0.25pc>[d] \ar@<-0.25pc>[d]\\
U & & & & V  
}
\]
where $\Gamma= \cB_x$, $V= \nu_x(S)$ and $x$ is $Q$-related to $0$. The holonomy group $\Hol_x$ is the image of the representation map $\Gamma \to \GL(V)$ and the Morita equivalence above induces one between  the holonomy groupoid and the action groupoid $\Hol_x\ltimes V$. By Remark \ref{remark:Morita-general-stratifications}, to check the equality from Lemma \ref{stratif-lemma-1} at $x$ for the partitions associated to $\cB$ and $\Hol(M, \cF)$, we have to check the similar equality on the right hand side at $0$ for the partitions associated to $\Gamma\ltimes V$ and $\Hol_x\ltimes V$. This boils down to the fact that a vector $v\in V$ is $\cB_x$-invariant if and only if it is $\Hol_x$-invariant. 
\end{proof}

\begin{definition} \label{def:foliations:hol:stratif} 
Given a foliation of proper type $\cF$ of $M$, the {\bf holonomy stratification  induced by $\cF$}, denoted $\cS_{\Hol}(M,\cF)$, is the canonical stratification induced by any s-connected proper foliation groupoid integrating $\cF$. 
\end{definition}

Since each strata is saturated, $\cS_{\Hol}(M,\cF)$ induces a partition of the leaf space $M/\cF$, denoted 
$\cS_{\Hol}(M/\cF)$. This turns out to be a stratification in the sense of definition \ref{def:str:man}. The smoothness of the strata is obtained as a corollary of the previous proof.
 We call a leaf {\bf principal} if it is contained in the principal part of the holonomy stratification.

\begin{corollary}\label{corolarry-normal} For any stratum $\Sigma\in \cS_{\Hol}(M,\cF)$ one has
\[ \nu_{\Sigma}(\cF)= \nu_{M}(\cF)^{\Hol}\]
and $(\Sigma, \cF|_{\Sigma})$ is a simple foliation.  In particular a leaf is principal if and only if it has trivial holonomy.
\end{corollary}

\begin{proof} 
The equality follows from the previous proof. The second part follows from the fact that proper foliations with trivial holonomy must be simple. 
\end{proof}


\subsection{The canonical infinitesimal stratification of PMCTs}
\label{sec:The canonical infinitesimal stratification of PMCTs}

Finally, we move to the setting of PMCTs. We will show that the canonical and the canonical infinitesimal stratifications 
induced by the corresponding proper symplectic groupoids do not depend on the integrations themselves but only
on the actual Poisson bivector, and that their strata will inherit Poisson-geometric structures. In this subsection we discuss the canonical infinitesimal stratification and in the next one the canonical stratification.

Besides the equivalence relation $x \underset{\inf}{\sim} y$, defined by \eqref{eq:inf-equiv}, we also consider
the equivalence relation
\[ x\underset{\textrm{cod}}{\sim} y  \quad\textrm{iff}\quad{\codim}_{\gg_{x}}\,\zz(\gg_{x})={\codim}_{\gg_{y}}\,\zz(\gg_{y}),\]
that partitions $M$ into the subsets
\begin{equation}\label{eq:cod-k-locus}
\Sigma_k^{\inf}(M, \pi)=\{x\in M\,|\,\mathrm{codim}_{\gg_{x}}\,\zz(\gg_{x})=k\}, \quad k=0,1,\dots.
\end{equation}
In the case of PMCTs the isotropy Lie algebras are of compact type, and for these we have
\[ \codim_{\gg}\,\zz(\gg)= \mathrm{dim}\, \gg^{\ss},\]
where $\gg^\ss$ is the semi-simple part of $\gg$. Indeed, this follows from the usual decomposition
\begin{equation}\label{eq:cpctLie-decomp}
\gg=\zz(\gg)\oplus \gg^{\ss},\quad \gg^{\ss}:=[\gg,\gg].
\end{equation}

\begin{theorem}
\label{thm:inf-strat} 
Let $(M, \pi)$ be a Poisson manifold of proper type.  Then the partition $\cSi(M, \pi)$ is a stratification of $M$ and its codimension $k$ strata are precisely the connected components of $\Sigma_k^{\inf}(M, \pi)$. This stratification satisfies:
\begin{align}
\Sigma_0^{\inf}(M, \pi)&=M^{\reg}; \label{eq:compute:Sigma0}\\
\Sigma_1^{\inf}(M, \pi)&=\Sigma_2^{\inf}(M, \pi)=\emptyset, \label{eq:compute:Sigma12}
\end{align}
where $M^{\reg}$ consists of points where the rank of $\pi$ is locally constant. 
In particular, 
the set $M^{\reg}$ is connected, so all regular leaves have the same dimension, and
the inclusion $M^{\reg}\subset M$ induces an isomorphism of fundamental groups.
\end{theorem}

\begin{remark}
    Equation \eqref{eq:compute:Sigma0} says that the regular locus of the Poisson bivector coincides with the regular locus of the infinitesimal stratification, so we can unambiguous call this locus the \textbf{regular locus} of $(M,\pi)$. We define the \textbf{subregular locus} of $(M,\pi)$ as the subregular locus of the infinitesimal stratification
    \[  M^{\subreg}:= M^{\cSi(M,\pi)-\subreg}.\]
    According to the theorem, this is the subset of $M^\sing:=M\backslash M^\reg$ where the rank $\pi|_{M^\sing}$ is maximal.
\end{remark}

\begin{proof} 
Let $(\G, \Omega)\tto M$ be a proper symplectic groupoid integrating $(M, \pi)$. The fact that
$\cSi(M, \pi)$ is a stratification follows from the general discussion on proper Lie groupoids; the discussion below provides further insight.

To prove that $\underset{\textrm{cod}}{\sim}$ and $\underset{\inf}{\sim}$ induce the same stratifications, we proceed like in the proof of Proposition \ref{prop:stratification:foliation}. We use Lemma \ref{stratif-lemma-1} and the normal form  of Theorem \ref{thm-lf-Poisson-v2}, which yields a Morita equivalence (see Remark \ref{remark:Morita-general})
\[
\xymatrix{
 \G|_U \ar@<0.25pc>[d] \ar@<-0.25pc>[d]  & \ar@(dl, ul) & Q \ar[dll]^{\alpha}\ar[drr]_{\mu}  & \ar@(dr, ur) & G\ltimes V \ar@<0.25pc>[d] \ar@<-0.25pc>[d]  \\
U&  & & & V}
\]
with $V\subset \gg^*$ invariant under the coadjoint action, and with $x_0$ $Q$-related to $0\in V$. By Remark \ref{remark:Morita-general-stratifications} and the fact that $\underset{\textrm{cod}}{\sim}$ only depends on isotropy Lie algebras, to prove the condition from Lemma \ref{stratif-lemma-1} at $x_0$ it suffices to check that
\begin{equation}\label{eq:analogue:for:proof}
V\cap {\Sigma_0}= V\cap \Sigma_k^{\inf}(\gg^*,\pi_{\gg^*}), 
\end{equation}
where $\Sigma_0$ is the stratum of the canonical stratification of $\gg^*$ through the origin and $k$ equals the codimension of $\zz(\gg)$ in $\gg$. 

Notice that $V\cap {\Sigma_0}=V^{\gg}$, the fixed-point set of the restriction of the coadjoint action to $V$. On the other hand, we have that $\Sigma_k^{\inf}(\gg^*,\pi_{\gg^*})\cap V$ consists of the points $\xi\in V$ such that $\zz(\gg_\xi)$ has the same dimension as $\zz(\gg)$. Since $\gg$ is compact, this happens if and only if $\xi\in\zz(\gg)^*$, which is precisely $V^{\gg}$. 

For later reference, we note that the previous argument gives the following description of the stratum $\Sigma\in\cSi(M,\pi)$ through $x_0$ in the Hamiltonian local normal form
\begin{equation}
\label{eq:stratum:local:model}
U\cap \Sigma= 
\mu^{-1}(\zz(\gg)^*)/G. 
\end{equation}

The equality (\ref{eq:compute:Sigma0}) is a restatement of the fact that a point $x\in M$ is regular for $\pi $ if and only if
its isotropy Lie algebra is abelian. The equalities (\ref{eq:compute:Sigma12}) are consequences of the fact that 
there are no semisimple Lie algebras of compact type of dimension $1$ and $2$. 
Lemma \ref{lem:stratif-lemma-4} implies that the inclusion $M^{\mathrm{reg}}=\Sigma_0\subset M$ induces
an isomorphism on $\pi_0$ and $\pi_1$, proving also the remaining claims about $M^{\reg}$. 
\end{proof}

\begin{corollary}\label{corollary: infinitesimal-stratum-conormal} For any stratum $\Sigma$ of the infinitesimal stratification of a Poisson manifold of proper type $(M, \pi)$ and any $x\in \Sigma$ one has
\begin{equation}
\label{eq:strata:conormal}
(T_x\Sigma)^0= 
((\gg^*_x)^{\gg_x-\inv})^{0}= \gg^\ss_x.
\end{equation}
In particular, 
\begin{equation}
\label{eq:strata:normal}
T_x\Si/T_xS_x=\zz_x(\gg)^*.
\end{equation}
Moreover, the following are equivalent:
\begin{enumerate}[(a)]
\item $\Sigma$ is a subregular stratum;
\item the codimension of $\Sigma$ is $3$;
\item $\gg_x^\ss\cong \mathfrak{su}(2)$ for a/any $x\in \Si$. 
\end{enumerate}
\end{corollary}

\begin{proof}
The first part of the result follows from \eqref{eq:stratum:local:model}. The equivalences (a)-(c) 
hold when $(M,\pi)$ is the dual of a Lie algebra of compact type. Hence, as in the proof of the theorem, they hold for any Poisson manifold of proper type.
\end{proof}


\begin{remark}
Theorem \ref{thm:inf-strat}  and Corollary \ref{corollary: infinitesimal-stratum-conormal} admit generalizations to the infinitesimal stratification of any proper groupoid $\cG$. Namely, one can show that $\cSi(\cG)$ coincides with the stratification induced by the codimension of the fixed point set of the actions of the isotropy Lie algebras on the normal spaces. The identity \eqref{eq:compute:Sigma0} remains valid, but not the identities \eqref{eq:compute:Sigma12}. Also, the identity \eqref{eq:strata:conormal} now becomes
\[ (T_x\Sigma)^0=(\nu_x(\O_x)^{\gg_x})^0. 
\]
\end{remark}

We now move to the Poisson geometric properties of the canonical infinitesimal stratification. Notice that, by the definition of $ \underset{\inf}{\sim}$, the strata are saturated with respect to the symplectic foliation, hence are Poisson  submanifolds. As we shall see shortly, they belong to the following very special class of Poisson submanifolds. 

\begin{definition}
\label{def:core}
A \textbf{core Poisson submanifold} of $(M, \pi)$ is  a Poisson submanifold $\Sigma$ with  symplectic foliation $\cF_{\Sigma}$ such that:
\begin{enumerate}[(i)]
\item $\Sigma$ is saturated and $\cF_{\Sigma}$ is regular;
\item for each $x\in \Sigma$, the canonical map  
\begin{equation}
    \label{map:core:condition}
    \gg_x= \nu_{x}^{*}(\cF_{\pi}) \to \nu_{x}^{*}(\cF_{\Sigma}), \quad \xi\mapsto \xi|_{T\Sigma}, 
\end{equation} 
restricts to an isomorphism 
\begin{equation}\label{restr-to-sigma-new}
r: \zz(\gg_x)\overset{\sim}{\rmap} \nu_{x}^{*}(\cF_{\Sigma}). 
\end{equation}
\end{enumerate}
\end{definition}

Core Poisson submanifolds allow one to reduce the study of certain invariants of Poisson manifolds to the regular case, where its geometry is more transparent and they become easier to compute.

\begin{theorem}\label{thm:inf-strat-type} 
If $(M,\pi)$ is a Poisson manifold of (strong) $\mathcal{C}$-type, then so are all the strata of $\cSi(M, \pi)$. Moreover, every strata is a core Poisson submanifold of $(M, \pi)$.
\end{theorem}

Before proceeding with the proof, we first recall the following construction of Hamiltonian reduction of symplectic groupoids along Poisson submanifolds. 
Let $(\cG, \Omega)\tto M$ be a symplectic groupoid. Its restriction $\cG_N\tto N$ to any saturated submanifold $N\subset M$ is a coisotropic subgroupoid of $(\cG, \Omega)$ and one applies the classical symplectic reduction. Namely, the kernel 
\[ \mathfrak{K}:= \Ker(\Omega|_{\cG_N})\subset T\cG_N \]
can be interpreted as a foliation on $\cG_N$ and, if the leaf space 
\begin{equation}\label{eq:red:gpd}
\cG_{\textrm{red}}:= \cG_{N}/\mathfrak{K}
\end{equation}
is a smooth manifold, then $\Omega|_{\cG_N}$ descends to symplectic structure $\Omega_{\textrm{red}}$ on $\cG_{\textrm{red}}$. In our context, it follows that 
$\left(\cG_{\textrm{red}}, \Omega_{\textrm{red}}\right)\tto N$
is a symplectic groupoid (cf. \cite{BCWZ}). The various groupoids appearing in this reduction are represented by the following diagram:
\[ \vcenter{\xymatrix{
 & (\cG_N, \Omega_N)\ar@{->>}[d]  \ar@{^{(}->}[r] & (\G, \Omega) \\
 & (\cG_\red, \Omega_\red) & 
}}\]
We say
that the saturated submanifold $N$ is \textbf{suited for smooth reduction of $\G$} if the leaf space (\ref{eq:red:gpd}) is, indeed, a smooth Hausdorff manifold. The resulting symplectic groupoid 
$\left(\cG_{\textrm{red}}, \Omega_{\textrm{red}}\right)$ will be called the \textbf{reduction of $\cG$ along $N$}. 
It is clear that, in this case, if $\cG$ is of $\cC$-type, then so is its reduction along $N$. 

%
%

\medskip 
The foliation $\mathfrak{K}$ can be further described by taking advantage of the groupoid context. First of all, $\mathfrak{K}$ can be reconstructed from its restriction $\mathfrak{K}|_{N}$ to $N$ via right (or left) translations.  Secondly, $\mathfrak{K}|_{N}$ is a smooth sub-bundle of the bundle of isotropy Lie algebras of $\cG_{N}$ which, via the isomorphism $\textrm{Lie}(\G)\simeq T^*M$ given by $\Omega$, is identified with the Lie algebra bundle 
\[  \mathfrak{K}|_{N}\simeq (TN)^0\subset T^*_N M.\]
Thirdly, each fiber $\mathfrak{K}_{x}$ integrates to a connected subgroup $\cK_{x}\subset \cG_x$ and one obtains a bundle of groups $\cK\to N$ that acts on $\cG_{N}$. 
Finally, the leaf space (\ref{eq:red:gpd}) can 
now be described as the quotient of this action. This leads to the following result.

\begin{lemma}\label{lemma:suited-for-reduction}
Given a symplectic groupoid  $(\cG, \Omega)\tto M$, a saturated submanifold $N\subset M$ is suited for smooth reduction of $\G$ if and only if  the bundle of connected groups $\cK\subset \cG$ integrating $\mathfrak{K}|_{N}= (TN)^0$ is closed in $\cG_{N}$.

If $\cG$ is of (strong) $\cC$-type, then $N$ is suited for smooth reduction if and only if each connected subgroup $\cK_x$ is closed  in $\cG_x$, in which case $\cG_{\textrm{red}}$ is also of (strong) $\cC$-type.
\end{lemma}

\begin{proof}
The first part of the lemma follows from the previous discussion. Assume then that $\cG$ is of $\cC$-type. If all $\cK_x$ are closed, they are also compact since $\cG$ is proper. It follows that the bundle $\cK$ is closed in $\cG_N$.
 The reduced groupoid $\G_{\textrm{red}}$ is proper since it is the restriction of $\cG$ to the saturated submanifold $N$, followed by the quotient by a bundle of compact Lie groups $\cK$. Similarly, the s-fiber of $\G_{\textrm{red}}$ at $x$ is the quotient of the s-fiber $\G$ at $x$ by the compact connected subgroup $\cK_{x}$. It follows that if $\cG$ is of (strong) $\mathcal{C}$-type so is $\G_{\textrm{red}}$.
\end{proof}



\begin{proof}[Proof of Theorem \ref{thm:inf-strat-type}] 
Fix a stratum $\Sigma\in\cSi(M, \pi)$. This is a regular Poisson submanifold since, by the definition of $\cSi(M, \pi)$, points in $\Sigma$ have isomorphic isotropy.
Notice also that, by Corollary \ref{corollary: infinitesimal-stratum-conormal}, the kernel of the canonical map  
\[ \gg_x= \nu_{x}^{*}(\cF_{\pi}) \to \nu_{x}^{*}(\cF_{\Sigma}), \quad \xi\mapsto \xi|_{T\Sigma}, \]
is precisely $\gg^\ss_x$. Therefore, the restriction of this map to the center $\zz(\gg_x)$ gives an isomorphism onto $\nu_{x}^{*}(\cF_{\Sigma})$, i.e.,  $\Sigma$ is a core Poisson submanifold.

To prove that $\Sigma$ is of (strong) $\cC$-type we show that $\Sigma$ is suited for smooth reduction of $\G$ by appealing to Lemma \ref{lemma:suited-for-reduction}. We are left with proving that for each $x\in \Sigma$ the connected Lie group $\cK_x$ integrating $(T\Sigma)^0\subset T^*_\Sigma M$ is closed in $\cG_x$. By \eqref{eq:strata:conormal}  $(T\Sigma)^0=\gg_x^{\textrm{ss}}$ and thus its integration 
is a closed subgroup of $\cG_x^0$ (see, e.g., pp. 165 in \cite{DK}).

\end{proof}

 The following result shows that core Poisson submanifolds retain all the information about two important invariants, that we recall after its statement.

\begin{theorem}\label{thm:inf-strat-monodromy} 
For any core Poisson submanifold $\Sigma$ of a Poisson manifold $(M, \pi)$:
\begin{enumerate}[(i)]
\item the holonomy representations of $(M, \pi)$ and $(\Sigma, \pi_{\Sigma})$ at any $x\in \Sigma$ are isomorphic, i.e. one has a commutative diagram: 
\[
\xymatrix{
&   \GL(\zz(\gg_x))\ar[dd]_-{\sim}^-{r_*}\\
\pi_1(S, x) \ar[ru]^-{\rho^{\Hol}_M} \ar[rd]_-{\rho^{\Hol}_{\Sigma}}   \\
&    \GL(\nu_{x}^{*}(\cF_{\Sigma}))
} \]
where $r$ is the isomorphism \eqref{restr-to-sigma-new};


\item the monodromy map of  $(\Sigma, \pi_{\Sigma})$ at any $x\in \Sigma$ factors through the monodromy map of $(M,\pi)$:
\[ \xymatrix@1{
&   Z(G(\gg_x))\ar[dd]^{\Phi}\\
\pi_2(S, x) \ar[ur]^-{\partial^{M}}\ar[dr]_-{\partial^{\Sigma}}   & \\
&  \zz(\gg_x)  
} \]
where we use the isomorphism $ \nu_{x}^{*}(\cF_{\Sigma})\simeq\zz(\gg_x)$ given by \eqref{restr-to-sigma-new}, and $\Phi$ is the restriction to the center of the group morphism $G(\gg_x)\to \zz(\gg_x)$ integrating the map (\ref{map:core:condition}) in the core condition.


\end{enumerate}
\end{theorem}

For a Poisson manifold $(M, \pi)$ its {\bf linear holonomy representation} at $x\in M$ is a homomorphism
\begin{equation}\label{eq:rho:lin:hol} 
\rho^{\Hol}: \pi_1(S, x)\to \GL(\zz(\gg_x)),
\end{equation}
where $S$ is the symplectic leaf containing $x$. It is the parallel transport map for the canonical flat connection on the bundle of centers of the isotropy $\zz(\gg_S)\to S$ defined by
\begin{equation}\label{eq:canonical-flat-on-zz-gg} 
\nabla: \X(S)\times \Gamma(\zz(\gg_S))\to  \Gamma(\zz(\gg_S)),\quad \nabla_{\pi^{\sharp}(\alpha)}\xi=[\alpha,\xi].
\end{equation}
In the regular case this is just the Bott connection and \eqref{eq:rho:lin:hol} is the dual of the usual linear holomomy representation of a foliation. When $(M,\pi)$ is integrable, the linear holonomy also has the following interpretation.

\begin{proposition}
\label{prop:holonomy:groupoid:version} 
Let $(\G,\Omega)$ be any symplectic integration of $(M,\pi)$, fix $x\in M$ and let $S$ be the symplectic leaf containing $x$. Then the linear holonomy representation factors through the adjoint representation of $\cG_x/\cG_x^0$ on $\zz(\gg_x)$
\[
\xymatrix{
\pi_1(S, x) \ar[rd]^-{\rho^{\Hol}_M} \ar[dd] \\
& \GL(\zz(\gg_x))\\
\cG_x/\cG_x^0\ar[ru]_{\Ad}
} \]
where the vertical arrow arises from the homotopy exact sequence of $\t:\s^{-1}(x)\to S$.
\end{proposition}

The proof is deferred to the end of the section.

Next we briefly recall the other invariant present in the statement of Theorem \ref{thm:inf-strat-monodromy}. For any Poisson manifold $(M, \pi)$, its {\bf monodromy map} at $x$ is a group homomorphism
\begin{equation}\label{eq:rho:lin:mon} 
\partial_x: \pi_2(S, x)\to Z(G(\gg_x)).
\end{equation}
In the regular case, this map has a nice geometric interpretation as the variation of symplectic areas of spheres in directions transverse to $S$ (see \cite{CF2} or  \cite[Chapter 14]{CFM21}). 
In the non-regular case, assuming that $(M, \pi)$ is integrable with the source 1-connected integration denoted $\Sigma= \Sigma(M, \pi)$, then \eqref{eq:rho:lin:mon} can be defined as the boundary map in the long homotopy sequence associated to $\t:\s^{-1}(x)\to S$, combined with the canonical inclusion 
of $\pi_1(\Sigma_x)$ into $Z(G(\gg_x))$ (see \cite{CF2} or  \cite[Chapter 14]{CFM21}). We also recall that one defines the {\bf monodromy group} of $(M, \pi)$ at $x$ as the image of the monodromy map
\[ \mathcal{N}_x:= \textrm{Im}(\partial_x)\subset Z(G(\gg_x)).\]
The integrability of a Poisson manifold is controlled by the groups $\mathcal{N}_x\cap Z(G(\gg_x))^0$.

In general, passing to a Poisson submanifold, one loses information about the holonomy and/or the monodromy of the ambient Poisson manifold. Theorem \ref{thm:inf-strat-monodromy} asserts that core Poisson submanifolds have the remarkable property of keeping the relevant part of this information and allows to reduce the computation of these groups to the regular case. More precisely, they encode the entire linear holonomy and, since $\Phi$  restricts to an isomorphism on $Z(G(\gg_x))^0$, they also detect the obstructions to integrability.

\begin{proof}[Proof of Theorem \ref{thm:inf-strat-monodromy}] 
The representation \eqref{eq:adj:rep:algebroid} is natural with respect to a surjective morphism $\phi: A\to B$ of transitive Lie algebroids over $S$. Similarly, at the groupoid level with the adjoint representation. Upon restriction to the centers, the naturality translates to the commutativity of the diagram
 \[\xymatrix{
 \pi_1(S, x) \times \zz(\gg^A_x) \ar[r]^-{\rho^{\Hol}_A} \ar[d]^-{(\id, \phi)} & \zz(\gg^A_x) \ar[d]^-{\phi}\\
 \pi_1(S, x) \times \zz(\gg^B_x) \ar[r]^-{\rho^{\Hol}_B}& \zz(\gg^B_x)
} 
\]

A similar naturality argument applies to the monodromy maps of transitive algebroids, and one obtains another commutative diagram
 \[\xymatrix{
 & Z(G(\gg^A_x)) \ar[dd]^-{\Phi} \\
 \pi_2(S, x) \ar[ru]^-{\partial_A} \ar[rd]_-{\partial_B} \\
 & Z(G(\gg^B_x))
} 
\]
where $\Phi:G(\gg^A_x)\to G(\gg^B_x)$ is the morphism integrating $\phi$.

The result follows from the previous commutative diagrams applied to the algebroid morphism 
\[
T^*_S M \to T^*_S\Sigma,
\]
where $S$ is a symplectic leaf inside $\Sigma$.
\end{proof}

 \begin{proof}[Proof of Proposition \ref{prop:holonomy:groupoid:version}] 
We start by remarking that the linear holonomy representation (\ref{eq:rho:lin:hol}) of $(M, \pi)$ at $x$ depends only on the restricted algebroid $T^*_SM$, and in fact makes sense for a general transitive Lie algebroid $(A_S, [\cdot, \cdot]_{A_S}, \rho_S)\to S$. For such an algebroid, one has a flat $A_S$-connection on the isotropy Lie algebra bundle $\gg_S:= \mathrm{Ker}(\rho_S)$ given by
\begin{equation}
\label{eq:adj:rep:algebroid}
\nabla: \Gamma(A_S)\times \Gamma(\gg_S) \to \Gamma(\gg_S), \quad \nabla_{\alpha}(\beta)= [\alpha, \beta]_{A_S}. 
\end{equation}
For any source connected Lie groupoid $\G_S$ with algebroid $A_S$, this integrates to the adjoint representation of $\G_S$ on $\gg_S$. Upon restriction to the centers $\zz(\gg_S)$, the $A_S$-connection factors via the anchor map $\rho_S:A_S\to TS$ to a flat connection 
\[ \nabla: \mathfrak{X}(S)\times  \Gamma(\zz(\gg_S))\to \Gamma(\zz(\gg_S)). \]

Using this remark, let $\G_S$ be the restriction  of $\G$ to $S$ and let $\cI_S$ be the bundle of groups consisting of the unit connected components of the isotropy groups of $\G_S$. Then $\cG_S$ and $\cI_S$ have Lie algebroids $A_S= T^*_SM$ and $\gg_S= \Ker (\rho_S=\pi^\sharp|_S: A_S\to TS$). The proof is a combination of several ingredients:
\begin{enumerate}[1)]
\item The representation $\Ad$ from the statement is simply the one induced by the adjoint representation of the Lie group $\G_x$ on its Lie algebra. The key remark is that the adjoint action makes sense as an action of the entire $\G_S$ on $\gg_S$ and, when acting just on $\zz(\gg_S)$, it descends to an action 
\[ \Ad: \G_S/\cI_S\to GL(\zz(\gg_S)).\]
\item The infinitesimal counterpart of $\Ad$ is the the action of $A_S$ on $\gg_S$ induced by the Lie bracket, i.e., the flat $A_S$-connection \eqref{eq:adj:rep:algebroid}.
As remarked above, upon restriction to the center, it descends to an action of $A_S/\gg_S$. Under the isomorphism $A_S/\gg_S\simeq TS$ this becomes the flat connection on 
$\zz(\gg_S)$ whose parallel transport defines the representation $\rho^{\Hol}_M$. At the global level, one can consider the entire parallel transport as a representation
\[ \rho^{\Hol}_M: \Pi(S)\to GL(\zz(\gg_S)).\]
\item Notice that 
$\G_S/\cI_S$ is a $s$-connected Lie groupoid integrating $A_S/\gg_S$ while $\Pi(S)$ is a source 1-connected integration of $TS$. Therefore, we have a groupoid morphism
\begin{equation}\label{eq:Psi:Pi-to-quotient} 
\Psi: \Pi(S)\to \G_S/\cI_S,
\end{equation}
integrating the isomorphism $TS\simeq A_S/\gg_S$. Restricting to the isotropy group at a point $x\in M$ provides the desired vertical group homomorphism in the statement. Explicitly, a class $[\gamma]\in\pi_1(S,x)$ is sent to the class of $u(1)\in \G_x/\G^0_x$, where $u: I\to \s^{-1}(x)$ is any lift of $\gamma$ along the target map starting at the unit. This is precisely the boundary map from the statement. 
\end{enumerate}
We now see that the representations $\Ad\circ \Psi$ and $\rho^{\Hol}_M$ have the same linearization and, therefore, they must coincide. 
The argument can summarized by the following diagrams
\[
\xymatrix{
TS \ar[rd]^-{\nabla} \ar@{<->}[dd]^{\sim}_{\rho_S} \\
& \mathrm{gl}(\zz(\gg_S))\\
A_S/\gg_S\ar[ru]_{\ad}
} 
\quad \quad \quad 
\xymatrix{
\Pi(S) \ar[rd]^-{\rho^{\Hol}_M} \ar[dd]_{\Psi} \\
& \GL(\zz(\gg_S))\\
\cG_S/\cI_S\ar[ru]_{\Ad}
} \]
where the left diagram is the infinitesimal counterpart of the right diagram. At $x$, one obtains the diagram from the statement. 
\end{proof}
\subsection{The canonical stratification of PMCTs}
\label{sec:The canonical stratification of PMCTs}

We now discuss the canonical stratification $\cS(M, \pi)$, originally introduced in Definition \ref{def:canonical:stratifications:PMCT}.
 We will give two analogues of the result we proved for proper foliations, Proposition \ref{prop:stratification:foliation}.
The first one is a simple consequence for regular PMCTs of the definition of $\underset{\textrm{hol}}{\sim}$.

\begin{proposition}\label{prop:regular-canonical-is-foliation} For any regular 
 Poisson manifold of proper type $(M,\pi)$ its canonical stratification $\cS(M,\pi)$ coincides with the canonical stratification of its foliation $\cF_\pi$. 
\end{proposition}

\begin{proof}
We note that we even have an equality at the level of partitions. The linear holonomy of the algebroid $T^*M$ that is used in the definition of $\underset{\textrm{hol}}{\sim}$ is nothing but the (dual of) the linear holonomy of the foliation $\cF_\pi$.
\end{proof}

The second Poisson geometric analogue of Proposition \ref{prop:stratification:foliation} holds for arbitrary PMCTs, also justifying  our terminology for $\cS(M, \pi)$.

\begin{proposition}\label{prop:can-str-equiv-hol}
Let $(M,\pi)$ be a  Poisson manifold of proper type. Then for any proper integration $(\cG,\Omega)$ the canonical stratification $\cS_{\cG}(M)$ coincides with the partition $\cS(M, \pi)$ induced by the equivalence relation $\underset{\hol}{\sim}$ defined by \eqref{eq:equiv-inv}.
\end{proposition}

\begin{proof}
Proposition \ref{prop:holonomy:groupoid:version} implies that the equivalence relation $\underset{\textrm{hol}}{\sim}$ agrees with the following one:
\[ x\sim'_{\cG} y \quad \text{iff}\quad
\begin{cases}
\gg_x \simeq \gg_y \\  
\zz(\gg_x)^{\G_x}\simeq \zz(\gg_y)^{\G_y}.
\end{cases}
\]
So we have to prove that $\sim_{\cG}$ and $\sim'_{\cG}$ produce the same partition of $M$. We do that by an argument similar to the one in the proof of Theorem \eqref{thm:inf-strat}, making use of 
Lemma \ref{stratif-lemma-1} and the normal form  of Theorem \ref{thm-lf-Poisson-v2}. We are left with proving the analogue of (\ref{eq:analogue:for:proof}):
\[ 
V\cap \Sc_0= V\cap \Sc'_{0}, 
\]
Writing out the two sides, we are left with proving that, for $\xi\in \gg^*$:
\[ \xi\in V^G \quad \Longleftrightarrow \quad \xi\in V^{\gg}\ \textrm{and}\ \zz(\gg_{\xi})^{G_{\xi}}\simeq
\zz(\gg)^{G}.  \]
To show this, fix first $\xi$ satisfying the conditions in the right-hand side. The first condition guarantees that  $\gg_{\xi}= \gg$ which, in turn, implies that the second isomorphism is actually an equality. On the other hand, if we choose a $G$-invariant inner product on $\gg$, then the dual $v_{\xi}\in \gg$ of $\xi$ is always in 
$\zz(\gg_\xi)^{G_\xi}$ and hence it has to be $G$-invariant. Therefore $\xi$ is $G$-invariant as well. The opposite implication is clear.
\end{proof}

\begin{remark} The previous theorem implies that $\cS_{\cG}(M)$ does not depend on the choice of proper integration of $(M, \pi)$ and can be defined in terms of infinitesimal data. This, together with an appropriate modification of the theorem, holds true for proper integrations of arbitrary algebroids. 
The more general version of  $\underset{\textrm{hol}}{\sim}$ is: $x\underset{\textrm{hol}}{\sim} y$ if $\gg_x\simeq \gg_y$ and the invariant parts of $\nu^*_x(\mathcal{O}_x)$ and $\nu^*_y(\mathcal{O}_y)$ are isomorphic. 
\end{remark}


The {\bf principal part} of $(M,\pi)$ is defined as the regular part of this stratification
\[  M^{\princ}:= M^{\cS(M,\pi)-\reg}.\]
Also, we will denote by $\Sc_k(M,\pi)$ the union of the codimension $k$ strata of $\cS(M, \pi)$. Similar to Theorem \ref{thm:inf-strat} for the infinitesimals stratification, we have the following.

\begin{theorem}\label{thm:can-strat} 
Let $(M, \pi)$ be a  Poisson manifold of proper type.  Then the codimension of the strata of $\cS(M,\pi)$ through $x\in M$ is precisely the codimension of 
$\zz(\gg_x)^{\Hol}$ in $\gg_x$, i.e., 
\[ \Sc_k(M,\pi)=  \left\{x\in M: \codim_{\gg_{x}}(\zz(\gg_x)^{\Hol})=k\right\}. \]
In particular,  $M^{\princ}$ consists of those points $x\in M^{\reg}$ with the property that the linear holonomy representation $\rho^{\Hol}_M:\pi_1(S_x)\to \GL(\nu_x)$ is trivial. 
\end{theorem}

Notice however that, unlike the case of the infinitesimal stratification, the canonical stratification of a Poisson manifold of proper type may have codimension one strata. See Example \ref{example:codimenision-one-do-exist}.

\begin{proof}
Notice that, by Proposition \ref{prop:holonomy:groupoid:version}, $\zz(\gg_x)^{\Hol}= \zz(\gg_x)^{\G_x}$ so the codimension partition is invariant under Morita equivalences. Therefore we can apply again the same type of argument used in the proofs of Theorem \eqref{thm:inf-strat} and Proposition \ref{prop:can-str-equiv-hol}.
We will be left with proving that:
\[  \xi\in (\gg^*)^{G} \quad\Longleftrightarrow \quad
{\codim}_{\gg_{\xi}}\left(\zz(\gg_{\xi})^{G_{\xi}}\right)= 
{\codim}_{\gg}\left(\zz(\gg)^{G}\right).\]
Notice that the codimension of  $\zz(\gg_{\xi})^{G_{\xi}}$ in $\gg_{\xi}$ can can only increase when moving away from $0\in \gg^*$. Hence, if $\xi$ satisfies the conditions on the right hand side, then $\gg_{\xi}= \gg$ and $\zz(\gg_{\xi})^{G_{\xi}}= \zz(\gg)^{G}$, from which it follows that $\xi\in (\gg^*)^{G}$. The opposite implication is clear.
\end{proof}

\noindent Similar to Corollary \ref{corollary: infinitesimal-stratum-conormal} we have:

\begin{corollary}\label{corollary: stratum-conormal}
For any stratum $\Sc$ of the canonical stratification of a Poisson manifold of proper type $(M,\pi)$ 
and any $x\in \Sigma$ one has:
\begin{equation}
\label{eq:strata:conormal2}
(T_x\Sc)^0= 
((\gg_{x}^*)^{\G_x})^{0} = \gg^\ss_x\oplus \mathrm{Coinv}_{\pi_1(S_x)}(\zz(\gg_{x})).
\end{equation}
In particular, 
\begin{equation}
\label{eq:strata:normal2}
T_x\Sc/T_xS_x=(\zz(\gg)^*)^{\G_x/\G_x^0}.
\end{equation}
\end{corollary}

\begin{example} \label{example:codimenision-one-do-exist}
Similarly to what happens for proper Lie group actions, the canonical stratification of PMCTs may have codimension one strata. For example, let $(S,\w)$ be a symplectic manifold with a connected double cover $q:\widetilde{S}\to S$. Also, let $\Z_2$ act on $\S^1$ by inversion, so we have the principal $\S^1\rtimes \Z_2$-bundle 
 \[P:=\widetilde{S}\times_{\Z_2} (\S^1\rtimes \Z_2)\to S.\]
The gauge construction for $P$ and $(S,\w)$ produces a  Poisson manifold of proper type (see \cite[Section 5.8]{CFM-I})
 \[M_\lin=:\widetilde{S}\times_{\Z_2}\R, \]
 with Poisson structure $\pi_\lin$ induced by the leafwise symplectic form on $\widetilde{S}\times \R$.
 The coadjoint action of $\Z_2$ on $\R$ is by reflection and it follows that $\cS(M_\lin,\pi_\lin)$ has the codimension one strata
 \[\Si_1=\widetilde{S}\times \{0\}/\Z_2\equiv S.\]
\end{example}

\noindent Next, we discuss the interaction between the canonical and the infinitesimal strata.

 
\begin{theorem}
\label{thm:can-strat:inf-strat} 
For any  Poisson manifold of proper type $(M, \pi)$  the canonical stratification $\cS(M, \pi)$ is a refinement of $\cSi(M, \pi)$. Its restriction to a stratum $\Sigma\in\cSi(M,\pi)$ coincides with the canonical stratification of $(\Sigma,\pi_\Sigma)$, i.e.,
\[ \cS(M,\pi)|_{\Sigma}=\cS(\Sigma,\pi_\Sigma). \]
Furthermore, the principal part of $(\Sigma, \pi_{\Sigma})$ consists of those points $x\in \Sigma$ for which the representation $\rho^{\Hol}_M:\pi_1(S_x)\to \GL(\zz(\gg_x))$ is trivial. 

\end{theorem}

\begin{proof}
If $\Sigma\subset M$ is a stratum of $\cSi(M,\pi)$, since $(\Sigma,\pi_\Sigma)$ is a core Poisson submanifold, by Theorem \ref{thm:inf-strat-monodromy} its linear holonomy coincides with the linear holonomy of $(M,\pi)$. This implies the statement in the theorem. 
The last part also follows from the fact that $\Si$ is a core Poisson submanifold and from the description in \eqref{eq:strata:normal} 
of $\nu_\Sigma(S_x)$. 
\end{proof}

Next, we have the analogue of Theorem \ref{thm:inf-strat-type}:

\begin{theorem}
\label{thm:can-strat-type}
If $(M,\pi)$ is a  Poisson manifold of proper type each strata of $\cS(M,\pi)$ is a saturated, regular, Poisson submanifold of $(M,\pi)$, of the same (strong) $\mathcal{C}$-type as $(M,\pi)$. 
\end{theorem}

\begin{proof}
A stratum $\Sc\in\cS(M,\pi)$ is saturated by symplectic leaves and therefore it is a Poisson submanifold of $(M,\pi)$. It is regular because it is a Poisson submanifold of a regular Poisson manifold $(\Si,\pi_\Si)$, for some $\Si\in\cSi(M,\pi)$.

Let $(\cG,\Omega)$ be an integration of $(M,\pi)$ of $\cC$-type and fix $\Sc\in\cS(M,\pi)$. We have $\Sc\subset \Sigma$ for some $\Sigma\in \cSi(M,\pi)$, and we already know that $\Si$ is suited for smooth reduction of $(\cG,\Omega)$. So we may assume without loss of generality that $(\cG,\Omega)$ is an integration of $\cC$-type of the regular Poisson manifold $(\Si,\pi_\Si)$. We make use of Lemma \ref{lemma:suited-for-reduction} and of the notations therein. We need to show that $\Sc$ is suited for smooth reduction of $(\cG,\Omega)$, i.e., that the integration $\cK\subset\cG$ of the bundle of abelian subalgebras
\[ (T\Sc)^0\subset T^*\Sigma\]
is a closed subgroupoid.

The bundle of groups $\cK$ is contained in the connected component of the isotropy $\cT$ of $\cG\tto \Sigma$ which, since $(\Si,\pi_\Si)$ is regular, is a bundle of tori with fibers 
\[ \cT_x=\nu^*(\cF_\Si)_x/\Lambda_x. \] 
Here $\Lambda$ denotes the transverse integral affine structure associated to the regular  Poisson manifold of proper type $(\Sigma, \pi_{\Sigma})$ \cite[Section 3.3]{CFM-II}, i.e., the kernel of the exponential map
\[ \exp:\gg_x=\nu_x^*(\cF_\Si)\to \cG_x. \] 
To show that $\cK$ is closed it is enough to show that its fibers are subtori of the fibers of $\cT$, i.e., that each $(T_x\Sc)^0$ is an integral affine subspace  of $(\gg_x,\Lambda_x)$. It follows from \eqref{eq:strata:conormal2} that we have a canonical $\cG_x$-invariant splitting
\[ \gg_x=\gg_x^{\cG_x}\oplus  (T_x\Sc)^0. \]
Since $\gg_x$ is abelian, the action of the connected component $\cG^0_x$ is trivial and we obtain a $\cG_x/\cG^0_x$-invariant splitting
\[ \gg_x=\gg_x^{\cG_x/\cG^0_x}\oplus  (T_x\Sc)^0. \]
Observing that the action of $\cG_x/\cG^0_x$ on $\gg_x$ is by integral affine transformations, 
to conclude our proof it suffices to invoke the following lemma.

\begin{lemma}
\label{lem:invariant:decomposition}
Let $V$ be an integral affine vector space and let $\Gamma$ be a finite group acting by integral affine transformations on $V$. Then there is a canonical $\Gamma$-invariant splitting by integral affine subspaces
\[ V=V^\Gamma\oplus ((V^*)^\Gamma)^0. \]
\end{lemma}

\begin{proof}[Proof of the Lemma]
Let $T:V\to V$ be the linear integral affine map
\[ T(v):= \sum_{\gamma\in\Gamma}\gamma\cdot v. \]
This map preserves the fixed point $V^\Gamma$, has image $V^\Gamma$ and kernel $((V^*)^\Gamma)^0$, so the lemma follows.
\end{proof}
\end{proof}

 \subsection{Canonical stratifications of DMCTs}
 \label{sec:stratifications:DMCT}

Twisted Dirac manifolds of $\cC$-type also have canonical stratifications, with the same features as in the Poisson case. Namely, Theorems \ref{thm:inf-strat} and \ref{thm:can-strat}, as well as Corollaries \ref{corollary: infinitesimal-stratum-conormal} and \ref{corollary: stratum-conormal}, hold as stated for (possibly twisted) DMCTs. This follows because the Hamiltonian normal form (cf. Theorem \ref{thm-lf-Dirac} and Remark \ref{rmk-nf-tDirac}) still gives, around each leaf, a Morita equivalence with a $G$-invariant neighborhood of 0 in the dual of $\gg^*$.

The definition of core submanifold $\Sigma$ of a twisted Dirac manifold $(M,L)$ (cf. Definition \ref{def:core}) is the same, since the conormal bundle still coincides with the isotropy of the Dirac structure. Also, Theorems \ref{thm:inf-strat-type} and \ref{thm:can-strat-type} still hold, because:
\begin{enumerate}[(i)]
    \item The notion of saturated submanifold $i:N\hookrightarrow M$ suited for smooth reduction in the Dirac setting, as well as Lemma \ref{lemma:suited-for-reduction}, carry through. To that end, one sets $\mathfrak{K}:=\ker(\Omega|_{\cG_N})\cap\ker(\d\s)\cap\ker(\d\t)$. Under the isomorphism $\mathrm{Lie}(\cG)\simeq L$ we still have $\mathfrak{K}|_N=(TN)^0$;
    \item Theorem \ref{thm:inf-strat-monodromy} still holds. In its proof, one replaces the algebroid morphism between cotangent Lie algebroids by the analogous one for twisted Dirac structures, namely $L|_S\to (i^*L)|_S$, where $i^*L$ denotes the pullback Dirac structure to $N$.
\end{enumerate}
From these it follows also that the relationship between the canonical and (i) the canonical infinitesimal stratification (Theorem \ref{thm:can-strat:inf-strat}) and (ii) the holonomy stratification of its characteristic foliation in the regular case (Proposition \ref{prop:can-str-equiv-hol}), also hold in the twisted Dirac setting.

Finally, Proposition \ref{prop:holonomy:groupoid:version} also holds for twisted presymplectic groupoids, as can easily be seen from its proof.

\begin{remark}
\label{rem:suited:IA}
The concept of submanifold suited for smooth reduction has an interesting interpretation in the case of presymplectic torus bundles $(\cT,\Omega_\cT)\tto M$. By \cite[Proposition 3.2.8]{CFM-II}, such bundles encode transverse integral affine structures $(\cF,\Lambda)$ on $M$. Using the same arguments as in \cite{CFM-II}, it follows that a submanifold $N\subset M$ is suited for smooth reduction of $(\cT,\Omega_\cT)$ if and only if it is a transverse integral affine submanifold of $(\cF,\Lambda)$.

Suppose $(\cG,\Omega)$ is a proper, regular, symplectic groupoid  with associated short exact sequence
\[
\xymatrix{1\ar[r]& (\cT,\Omega_{\cT})\ar[r]& (\cG,\Omega)\ar[r]& \cB(\cG)\ar[r] & 1.}
\]
Then one checks that given a saturated  submanifold $N\subset M$, the symplectic groupoid $(\cG,\Omega)$ is suited for reduction along $N$ if and only if the presymplectic torus bundle $(\cT,\Omega_{\cT})$ is suited for reduction along $N$.
\end{remark}

\section{The Weyl Resolution}
\label{sec:Weyl:resolution}

In this section we show that any Poisson manifold of proper type can be ``desingularized'', in the sense that the non-regular locus can be eliminated. Such a resolution of singularities will force us to consider the more general setting of Dirac geometry. However, these Dirac resolutions will still be Poisson on a dense open subspace and of proper type. For simplicity of the presentation, we will assume that the Poisson manifold $(M,\pi)$ is connected. For a general Poisson manifold one may apply these results to each component.

The definition of the resolution can be given straight away as follows.

\begin{definition}\label{def:resolution}
The Weyl resolution of a Poisson manifold $(M, \pi)$ is defined as 
\[ \hM:= \{ (x,\tt)\,|\, x\in M,\, \tt\subset \gg_x\,  \textrm{maximal\ torus}\},\]
together with the map
\[ \res: \hM\to M,\quad \ \res(x, \tt)= x.\]
\end{definition}

If $(\cG,\Omega)\tto M$ is a symplectic groupoid integrating $(M,\pi)$ then it acts naturally on the the Weyl resolution along the resolution map
\begin{equation}
\label{eq:action:resolution} 
g\cdot (x,\tt)=(y,\Ad_g(\tt)),
\end{equation}
where $g$ is an arrow from $x$ to $y$ and $\Ad_g$ is the differential of conjugation by $g$.

For a general Poisson manifold, $\hM$ may even fail to be smooth and the Weyl resolution can be of little use.  However, for PMCTs the situation is quite different and in this section we will discuss many nice properties of the resolution for these Poisson manifolds. For example, the pullback by $\res$ of $\pi$, as a Dirac structure, is a Dirac structure 
\begin{equation}
    \label{eq:Dirac:structure}
    \hL:=\res^*L_\pi.
\end{equation}
We will denote the associated presymplectic foliation by 
\begin{equation}
    \label{eqdefn:foliation:resolution}
    \hF:=\pr_{T\hM}(\hL).
\end{equation}  
Its leaves are the preimages by $\res$ of the leaves of the symplectic foliation $\cF_\pi$. We have the following result which summarizes the contents of this section:

%
%
%
%

\begin{theorem}\label{thm:Weyl-intro} 
For any  Poisson manifold of proper type $(M, \pi)$:
\begin{enumerate}[(i)]
\item $(\hM,\hL)$ is a regular Dirac manifold of proper type, and its Poisson locus is precisely the preimage $\res^{-1}(M^{\reg})$ of the regular part of $(M, \pi)$; 
\item $\res$ is a forward Dirac map which restricts to a Poisson diffeomorphism from $\res^{-1}(M^{\reg})$ to $M^{\reg}$;
\item $\res$ descends to a homemorphism between the leaf spaces of $(\hM,\hL)$ and $(M, \pi)$.  
\item For any proper integration $(\cG,\Omega)\tto (M, \pi)$, the action groupoid and the pullback of $\Omega$
\[
    \hG:=\G\ltimes\hM\tto\hM,\quad \hOmega:=\pr_\cG^*\Omega,
\]
is a proper presymplectic integration of $(\hM,\hL)$.
\end{enumerate}
\end{theorem}



If $(M,\pi)$ is a Poisson manifold of proper type, then Theorem \ref{thm-lf-Poisson-v2} implies that the rank of the isotropy Lie algebras is locally the rank of the isotropy Lie algebras of the coadjoint action, which is constant. Hence, since we assume $M$ to be connected, the common rank of the isotropy Lie algebras equals codimension of a regular leaf, which is usually called the {\bf corank} of $(M,\pi)$. Denoting this corank by $r$, the Weyl resolution sits canonically inside the Grassmann bundle:
\[ \hM \subset \Gr_{r}(T^*M).\]
In particular $\hM$ comes equipped with a natural topology. In fact, $\hM$ will turn out to be a closed submanifold of the Grassmannian.

\begin{example}[The case $M= \gg^*$]
\label{ex:linear:case}
Let us consider the linear Poisson manifold $M= \gg^*$, where $\gg$ is a Lie algebra of compact type. Denote
\begin{equation}\label{eg:all-tori-in-gg} \cT(\gg):= \{ \tt\subset \gg: \mathfrak{t}\  \textrm{is\ a\ maximal\ torus} \} \subset \Gr_r(\gg),
\end{equation}
where $r$ is the rank of $\gg$. The standard decomposition $\gg= \tt\oplus [\gg, \tt]$ allows one to write 
\begin{equation*}
\tt^*\subset \gg^*
\end{equation*}
Using this inclusion we find $\cT(\gg_{\xi})= \{ \tt\in \cT(\gg): \xi\in \tt^* \}$.
Hence, the Weyl resolution is
\begin{equation}\label{eq:g-star-hat}
\widehat{\gg^*}= \{(\xi, \tt): \tt\in \cT(\gg), \xi\in \tt^* \}.
\end{equation}

Let $G$ be a compact Lie group integrating $\gg$ and fix a maximal torus $T\subset G$ with Lie algebra $\tt\subset \gg$. Then all maximal tori in $\gg$ are conjugate to $\tt$ so we have an isomorphism
\begin{equation}
\label{eq:t:g} 
G/N(T)
\overset{\sim}{\rmap} \cT(\gg),\quad g\mapsto \Ad_{g}(\tt),
\end{equation}
where $N(T)$ is the normalizer of $T$ in $G$. Using \eqref{eq:g-star-hat} we obtain a bijection
\begin{equation}\label{eq:Weyl-smooth}
\phi:(G\times \tt^*)/N(T)\overset{\sim}{\rmap} \widehat{\gg^*}, \quad (gT,\xi)\mapsto (\Ad_g^*(\xi),\Ad_g(\tt)).
\end{equation}
Note that we also have a natural identification
\[ 
(G\times \tt^*)/N(T)\simeq G/T\times_W \tt^*, 
\]
where $W= N(T)/T$ is the Weyl group of $G$ relative to $T$. Viewing $\phi$ as a map into $\Gr_r(T^*\gg^*)=\gg^*\times \Gr_r(\gg)$, one checks that the first component is proper and $\phi$ is an immersion. It follows that $\phi$ is a proper immersion or, equivalently, a closed embedding. We conclude that $\widehat{\gg^*}$ is a submanifold of $\Gr_r(T^*\gg^*)$.

Notice that in the model given by \eqref{eq:Weyl-smooth} the resolution map is simply:
\begin{equation}
\label{res:linear:case}
\res:G/T\times_W \tt^*\to \gg^*, \quad (gT,\xi)\mapsto \Ad_g^*(\xi),
\end{equation} 
and this maps the $W$-classes with representatives in $G/T\times\{\xi\}$ to the coadjoint orbit through $\xi$. In other words, the symplectic foliation of $\gg^*$ is resolved into a regular foliation. If we pullback the symplectic form on the coadjoint orbit through $\xi$, we obtain a $G$-invariant closed form $\omega_\xi$ on $G/T\times\{\xi\}$ which at the point $[(T,\xi)]$ is given by the familiar formula
\[ \omega_\xi((u,0),(v,0))=\xi([u,v]). \]
These leafwise presymplectic forms  descend to $W$-classes and together define the Dirac structure $\hL$ on $\hgg$ given by \eqref{eq:Dirac:structure}.

For later use, notice that we have not assumed $G$ to be connected. The construction of $\widehat{\gg^*}$ does not depend on the choice of Lie group integrating $\gg$. We have a short exact sequence
\begin{equation}
\label{eq:short:seq:Weyl:grp}
\xymatrix{1\ar[r] & W^0\ar[r] & W\ar[r] & G/G^0\ar[r] & 1}
\end{equation}
where $W^0=(N(T)\cap G^0)/T$, so the we can replace $G$ by $G^0$ in the construction above. This agrees with the fact that Definition \ref{def:resolution} only uses infinitesimal data.
\end{example}

\subsection{Smoothness of the resolution}

\begin{theorem}
\label{thm:resolution:map}
Let $(M, \pi)$ be a  Poisson manifold of proper type of corank $r$. Then $\hM$ is a closed embedded submanifold of $\Gr_r(T^*M)$ and $\res:\hM\to M$ is a proper, smooth surjection with connected fibers. Moreover, for any integration $(\cG,\Omega)$ of $(M, \pi)$ the action \eqref{eq:action:resolution} is smooth and its orbit foliation is regular of codimension $\corank\pi$.
\end{theorem}

\begin{proof}
One can cover $M$ by saturated open sets where the Hamiltonian local model holds and it is enough to prove the various statements on each such domain. So we assume that  $(M,\pi)=(Q/G,\pi_\red)$, where $\mu: (Q, \omega)\to \gg^*$ is as in Theorem \ref{thm-lf-Poisson-v2}. We denote by
\[ \pp: Q\to M\]
the canonical projection, we pull back along $\pp$ the map $\hM\to\Gr_r(T^*M)$ covering the identity of $M$, and  we consider the following diagram:
\[
\xymatrix{
Q\fproduct{M} \hM\ar@{-->}[dr]\ar@{-->}[dd] \ar@{-}[rr]^{\sim} & & Q\fproduct{\gg^*} \hgg \ar@{-->}[dl]\ar[dd]\\
 & \Gr_r(T^*Q) \\
Q\fproduct{M}  \Gr(T^*M) \ar[ur]_{\pp^*} & & Q\fproduct{\gg^*}  \Gr(T^*\gg^*)\ar[ul]^{\mu^*}
}
\]
where for $(q,x,\tt)\in Q\fproduct{M} \hM$ the remaining maps behave as follows:
\[
\xymatrix{
(q,x,\tt)\ar@{-->}[dr]\ar@{-->}[dd] \ar@{-}[rr]^{\sim} & & (q,\mu(q),\phi_q(\tt)) \ar@{-->}[dl]\ar[dd]\\
 & (\d_q\pp)^*(\tt)=(\d_q\mu)^*(\phi_q(\tt)) \\
(q,\tt) \ar[ur]_{\pp^*} & & (q,\phi_q(\tt))\ar[ul]^{\mu^*}
}
\]
Here 
\begin{equation}\label{eq:isotropy-maps:phi-q} 
\phi_q:  \gg_{\mu(q)}\to \gg_{\pp(q)}^{M}
\end{equation}
is the differential at unit of the map $\psi_q$ from Remark \ref{remark:Morita-general}.
In the diagram above, the top horizontal arrow is a bijection, inducing a smooth structure on $Q\fproduct{M} \hM$. The solid arrows represent proper embeddings, and therefore it follows in sequence that the dash arrows are proper embeddings.

Now consider the diagram
\[
\xymatrix{
Q\fproduct{M} \hM\ar[r]\ar[d] & Q\fproduct{M}  \Gr(T^*M)\ar[d]\\
 \hM\ar[r] & \Gr(T^*M)
}
\]
The vertical arrows are the quotient maps for the free and proper action of the compact Lie group $G$ on $Q$, while the top horizontal arrow is a $G$-equivariant proper embedding. Hence, $\hM$ inherits a quotient smooth structure such that the bottom horizontal arrow is a proper embedding.

Since the projection $\Gr_r(T^*M)\to M$ is proper, the map $\res:\hM\to M$ is smooth and proper. The connecteness of the fibers follows from the definition of the topology on $\hM$.

For the smoothness of the action we use again the Hamiltonian local model. In this model the action becomes:
\[
\xymatrix@C=15pt{
(Q\fproduct{\gg^*}Q)/G \ar@<0.25pc>[dr] \ar@<-0.25pc>[dr]& {}\save[]+<-40pt,0cm>*\txt{\Large $\circlearrowright$}\restore (Q\fproduct{\gg^*}\hgg)/G \ar[d] \\
  & Q/G  }
\] 
where the action is the obvious one. 

Finally, from the definition of the action, its isotropy groups have constant dimension equal to $\corank\pi$, so the result follows.
\end{proof}

\begin{remark}
\label{rem:life:easy}
From the proof above we deduce the following property of the smooth structure of $\hM$. If $M\supset U\simeq Q/G$ is a domain where the Hamiltonian local model holds, then $\widehat{U}:=\res^{-1}(U)$ is diffeomorphic to $\hQ/G$ where $\hQ:=Q\fproduct{\gg^*}\hgg$. Moreover, we have the following morphism of Morita equivalences,

\[
\xymatrix{
  & & \hQ{}\save[]+<32pt,1 pt>*\txt{$=Q\times_{\gg^*}\widehat{\gg^*}$}\restore\ar[dll]\ar[drr]^-{\pr_{\widehat{\gg}^*}} \ar[d]_{\pr_Q}& & \\
\hM {}\save[]+<-22pt,+2 pt>*\txt{$\hQ/G=$}\restore \ar[d]_{\res} & & Q\ar[dll]^-{p}\ar[drr]_-{\mu} & & \widehat{\gg^*}{}\save[]+<35pt,0 pt>*\txt{$=G/T\times_W \tt^*$}\restore \ar[d]^{\res}\\
M {}\save[]+<-22pt,1 pt>*\txt{$Q/G=$}\restore
& & & & \gg^*
}
\]
where the bottom two-leg diagram is an equivalence between $\cG$ and $G\ltimes\gg^*$, and the top two-leg diagram is an equivalence between the action groupoids $\cG\ltimes \hM$ and $G\ltimes \widehat{\gg^*}$. We will be using this local model for $\hM$ in the sequel. 
\end{remark}

Recall that $y\in Y$ is a clean value of a smooth map $\phi:X\to Y$  if the preimage $\phi^{-1}(y)$ is a submanifold of $X$ and $T_x\phi^{-1}(y)=\ker\d_x\phi$ for all $x\in \phi^{-1}(y)$. 

\begin{proposition} 
\label{prop:range:resolution}
Let $(M, \pi)$ be a  Poisson manifold of proper type. The differential at $\hx:=(x,\tt_x)\in\hM$ of the resolution map $\res$ has range the subspace $\mathcal{R}_{\hx}\subset T_xM$ whose annihilator is
\[ (\mathcal{R}_{\hx})^0=[\tt_x,\gg_x]\subset T_x^*M. \]
Furthermore,
\[ \dim(\ker\d_{\hx}\res)=\dim\gg_x-\dim\tt_x, \]
and all values of $\res$ are clean.
\end{proposition}

\begin{proof}
Since this result is local, we can assume that $M=Q/G$ for a Hamiltonian local model $\mu:(Q,\omega)\to\gg^*$. We also observe that we can reduce the result to the linear case due to the following diagram:
\[
\xymatrix{
T^*_{(x,\tt_x)}\hM \ar@{^{(}->}[r] & T_{(q,x,\tt_x)}^*(Q\fproduct{M} \hM) \ar[r]^{\sim} & T_{(q,\xi,\tt_\xi)}^*(Q\fproduct{\gg^*} \hgg) & \ar@{_{(}->}[l]  T^*_{(\xi,\tt_\xi)}\hgg\\
T^*_x M \ar@{^{(}->}[r]_{\pp^*} \ar[u]^{\res^*}& T_q^*Q\ar@{=}[r]\ar[u]_{\pr_Q^*} & T^*_qQ\ar[u]^{\pr_Q^*} & \ar@{_{(}->}[l]^{\mu^*}  T^*_\xi\gg^*\ar[u]_{\res^*}
}
\]
where we fixed a point $(q,x,\tt_x)\in Q\fproduct{M} \hM$, $\xi=\mu(q)$, and $\phi_q(\tt_\xi)=\tt_x$ where $\phi_q: \gg_\xi\to \gg^M_x$ is the isomorphism \eqref{eq:isotropy-maps:phi-q}. The kernels of the vertical arrows are contained in various isotropy Lie algebras that sit inside the bottom row and correspond to each other
\[
\xymatrix{
\gg^M_x\ar@{<->}[r] & p^*(\gg^M_x)\ar@{=}[r] & \mu^*(\gg^M_x) & \ar@{<->}[l]  \gg_\xi
}
\]
In other words, we have:
\[ (\mathcal{R}_{(x,\hx)})^0 =\phi_q((\mathcal{R}_{\xi, \tt_\xi})^0), \]
where on the right side we have the range of the differential of $\res:\hgg\to\gg^*$.

We are left with proving the statement for the linear case. We use the model \eqref{res:linear:case} for the resolution map
\[
\xymatrix{
G\times \tt^*\ar[d]\ar[dr]^{\widetilde{\res}} \\
G/T\fproduct{W}\tt^*\ar[r]_-{\res} & \gg^*
}
\]
By $G$ equivariance, it suffices to describe the image of $\d\widetilde{\res}$ at elements $(1,\xi)$. Since
\[ \d_{(1,\xi)}\widetilde{\res}(u,\eta)=\ad_u\xi+\eta, \]
it follows that its image is annihilated precisely by $[\tt,\gg_\xi]$ (note that $\tt\subset \gg_\xi$), proving the statement about the range.

The last formula in the statement now follows from $\gg_x=\tt_x\oplus [\tt_x,\gg_x]$. Also, since the preimage $\res^{-1}(x)=\cT(\gg_x)\subset \Gr_r(T_x M)$ is a submanifold and $T\res^{-1}(x)\subset \ker\d\res$, by a dimension count it follows that $x$ is a clean value of $\res$.
\end{proof}

\begin{corollary}
\label{cor:resolution:action}
Let $(M, \pi)$ be a  Poisson manifold of proper type and $(\cG,\Omega)\tto M$ a  proper integration. Then
\begin{equation}
\label{eq:resolution:action}
(\mathcal{R}_{\hx})^0\cap \ker\widehat{\act}=\{0\},
\end{equation}
where $\widehat{\act}:\res^*T^*M\to T\hM$ is the infinitesimal action 
induced by the action (\ref{eq:action:resolution}).
\end{corollary}


\subsection{Dirac versus Poisson geometry}
The local model for the resolution $\hM$, as in Remark \ref{rem:life:easy}, can be enhanced with a presymplectic form. If $M=Q/G$ for a Hamiltonian $G$-space $\mu:(Q,\omega)\to\gg^*$, then  $\widehat{\omega}=\pr^*_Q\omega$ gives the following presymplectic Hamiltonian $G$- space:
\begin{equation}
\label{eq:resolution:not:local:model}
\xymatrix{
(\hQ,\widehat{\omega})\ar[r]^{\pr_Q}
\ar@/^2.0pc/[rr]^{
\widehat{\mu}}
\ar[d] & (Q,\omega)\ar[r]^{\mu}\ar[d]& \gg^*\\
\hM {}\save[]+<-22pt,+2 pt>*\txt{$\hQ/G=$}\restore \ar[r] & M{}\save[]+<22pt,0cm>*\txt{$=Q/G$}\restore
}
\end{equation}
Since $\widehat{\omega}$ has kernel, this brings us to the Dirac geometry of the resolution.


\begin{theorem}
\label{thm:resolution:Dirac}
Let $(M, \pi)$ be a  Poisson manifold of proper type. Then $\hM$ carries a unique Dirac structure $\hL$ making $\res:\hM\to M$ a forward Dirac map. Furthermore:
\begin{enumerate}[(i)]
\item $\hL$ is a regular Dirac structure;
\item the presymplectic leaves $( \widehat{S},\omega_{\widehat{S}})$ of $\hL$ take the form
\[ \widehat{S}= \res^{-1}(S),\quad \omega_{\widehat{S}}= \res^*(\omega_S), \]
where $(S,\omega_S)$ is a symplectic leaf of $(M,\pi)$;
\item the Poisson locus of $\hL$ is $\res^{-1}(M^{\reg})$ and $\res:\res^{-1}(M^{\reg})\to M^\reg$ is a diffeomorphism.
\end{enumerate}
Moreover,  $\res$ induces a homeomorphism between the leaf spaces associated to $\hL$ and $\pi$.
\end{theorem}

\begin{proof}
We define $\hL$ as in \eqref{eq:Dirac:structure} by pulling back the Poisson structure $\pi$ viewed as a Dirac structure. Hence, $\hL$ is given pointwise by
\[ \hL|_{\hx}=\{(V,\res^*\beta):\pi^\sharp(\beta)=\d_{\hx}\res(V)\}. \]
In order to obtain a Dirac structure, we only need to check that this is a smooth vector subbundle of $T\hM\oplus T^*\hM$ (see, e.g., \cite[Section 7.3]{CFM21}). To that end, we observe that $\hL$ has rank equal to $\dim M$ and contains the image of the injective bundle map
\begin{equation}
\label{eq:injective:IM:form} 
(\widehat{\act},\d\res^*): \res^*T^*M\to T\hM\oplus T^*\hM. 
\end{equation}
Here $\widehat{\act}:\res^*T^*M\to T\hM$ is the infinitesimal action corresponding to the action of $\cG$ on $\hM$, and injectivity follows from Corollary \ref{cor:resolution:action}. In conclusion, $\hL$ is the image of the previous injective bundle map.

The description of the presymplectic leaves in item (ii) follows from the fact that $\res$ is a backward Dirac map. 

To prove item (iii), observe that $\res:\res^{-1}(M^{\reg})\to M^\reg$ is a bijection and by the construction of the smooth structure it follows that it is a diffeomorphism. Since $M^\reg$ is open and dense in $M$ and $\res$ is proper, we have that $\res^{-1}(M^{\reg})$ is open and dense in $\hM$. From this it follows that $\res$ is also a forward Dirac map, and that the Dirac structure is unique. 

Since $\res$ is a proper map, it follows that the continuous bijection between leaf spaces is also a closed map, hence a homeomorphism.
%
%
%
\end{proof}

\begin{corollary}
Let $(M, \pi)$ be a  Poisson manifold of proper type. For $y\in \hM$ the following are equivalent:
\begin{enumerate}[(a)]
   \item $\res(y)$ is a regular point of $(M, \pi)$;
   \item $\res(y)$ is a regular value of $\res$;
   \item $y$ belongs to the Poisson support of $\hL$.
\end{enumerate}
\end{corollary}

\begin{corollary}
The restriction $\res:\widehat{S}\to S$ is a fiber bundle.
\end{corollary}

\begin{proof}
    Proposition \ref{prop:range:resolution} shows that $T\cF_\pi\subset\im(\d\res)$. Since the resolution map $\res$ is a backward Dirac map, it follows that its restriction to a presymplectic leaf is automatically a submersion. The result then follows from the fact that $\res$ is proper.
\end{proof}

\begin{example}
One can obtain a Hamiltonian local model for $(\hM, \hL)$ from one for $(M,\cF_\pi)$. Assume that 
 $M= Q/G$ is a  Hamiltonian $G$-space $\mu:(Q,\omega)\to\gg^*$. Fixing a maximal torus $T\subset G$, one has
\begin{equation} 
\label{eq:resolution:local:model}
\hM\simeq\mu^{-1}(\tt^*)/N(T),
\end{equation} 
for the presymplectic Hamiltonian $N(T)$-space 
\[ \mu:(\mu^{-1}(\tt^*),i^*\omega)\to \tt^*, \]
where $i:\mu^{-1}(\tt^*)\hookrightarrow Q$ is the inclusion. The identification \eqref{eq:resolution:local:model} is given by
\[ qN(T)\mapsto (\pp(q),\phi_q(\tt)),\]
and arises from applying $Q\fproduct{\gg^*} -$ to the transversal $\tt^*\hookrightarrow \hgg$, $\xi\mapsto (\xi,\tt)$.
\end{example}


\begin{example}\label{ex:linear local model resolution}
Consider the linear local model $(M,\pi^\theta_\lin)$ associated with the data of a principal $G$-bundle $p:P\to (S,\w_S)$ and a connection 1-form $\theta\in \Omega^1(P,\gg)$, described in Section \ref{sec:ex:local-model}. One has $M=(P\times\gg^*)/G$ and the Dirac structure $L_{\pi^\theta_\lin}$ is the reduction of the form (cf.~\eqref{loc-mod-2-form}):
\[ p^*\omega_S-\d\langle \theta,\cdot\rangle \in \Omega^2(P\times \gg^*).\]

Fixing a maximal torus $T\subset G$ one obtains data for a linear local model for $(\hM,\hL)$. Namely:
\begin{itemize}
 \item the principal $N(T)$-bundle $q:P\to \hS=P/N(T)$;
 \item the presymplectic form $\res^*\w_S$ on $\hS$;
\item the connection 1-form $\theta^T:=\pr_\tt\circ\theta\in \Omega^1(P, \tt)$ which is the $\tt$-component of $\theta$ with respect to the decomposition $\gg=\tt\oplus [\tt,\gg]$.
\end{itemize}
The construction in Section \ref{sec:ex:local-model} gives the local model for $\hM$ together with a map induced by the inclusion $\tt^*\hookrightarrow\gg^*$,
\[ \res:\hM=(P\times\tt^*)/N(T)\to M=(P\times\gg^*)/G,\quad \quad [(p,\xi)]_{N(T)}\mapsto [(p,\xi)]_{G}. \]
The Dirac structure $\widehat{L}_{\pi^\theta_\lin}$ for $\hM$ is obtained as a quotient of the $N(T)$-invariant form (see 
\[
     p^*\omega_S-\d\langle \theta^T,\cdot\rangle \in \Omega^2(P\times \tt^*),
\]
where we used that $q^*\res^*\omega_s=p^*\omega_S$.
The map 
\[\res:(\hM,\widehat{L}_{\pi^\theta_\lin})\to (M,{L}_{\pi^\theta_\lin})\]
is forward and backward Dirac, and it is the  resolution for $(M,\pi^\theta_\lin)$ in our sense over the Poisson support of ${L}_{\pi^\theta_\lin}$.
\end{example}

\begin{proposition}
Let $(\G, \Omega)$ be an integration of $(M, \pi)$ of (strong) $\cC$-type and consider the $\G$-action on $\hM$. The associated action groupoid 

\[\hG:=\G\ltimes \hM,\quad \widehat{\Omega}:=\pr_{\cG}^{*}\Omega, \] 
is a presymplectic integration of $(\hM, \hL)$ of (strong) $\cC$-type.
\end{proposition}

\begin{proof}
Since $\pr_{\cG}:\G\ltimes \hM\to\cG$ is a groupoid morphism, the form $\widehat{\Omega}$ is a closed, multiplicative 2-form. The only things to be checked are that (i) the non-degeneracy condition
\[ \ker\d\s\cap\ker\d\t\cap \ker\widehat{\Omega}=\{0\} \]
holds and (ii) $\hL$ is the induced Dirac structure on the base. Both of these follow from \cite[Corollary 4.8]{BCWZ} and the fact that \eqref{eq:injective:IM:form} is bundle isomorphism onto $\hL$. 

Since $\res:\hM\to M$ is a proper map, $\G\ltimes \hM$ has the same $\cC$-type as $\cG$.
\end{proof}

\subsection{A characterization of the Weyl resolution}

We provide a way to detect the Weyl resolution 
among backward Dirac maps.

\begin{theorem}
\label{thm:characterization:resolution}
Let $\phi:(N,L)\to (M,\pi)$ be a backward Dirac map from a Dirac manifold to a Poisson manifold of proper type with the following properties:
\begin{enumerate}[(i)]
\item $L$ is regular;
\item $\phi$ is proper and restricts to a diffeomorphism
\[\phi:\phi^{-1}(M^\reg)\to M^\reg.\]
\item There is an action $\widehat{\act}$ of $T^*M$ on $\phi$ (necessarily unique)
yielding a Lie algebroid isomorphism $(\widehat{\act},\phi^*):\phi^*T^*M\to L$;
\end{enumerate}
 Then there exist a diffeomorphism
 \[\psi:(N,L)\to (\hM,\hL)\]
 such that $\phi=\res\circ \psi$.
\end{theorem}

\begin{remark}
The previous theorem implies that following global statement. Assume that $(\cG, \Omega)\tto (M, \pi)$ is a proper groupoid integrating a Poisson manifold $(M, \pi)$, acting along a proper map $\phi: N\to M$, such that:
\begin{enumerate}[(i)]
\item the orbits of the action are regular;
\item $\phi$ restricts to a diffeomorphism
\[\phi:\phi^{-1}(M^\reg)\to M^\reg;\]
\item the resulting action groupoid, endowed with the pull-back of $\Omega$ by the projection, is a presymplectic groupoid.
\end{enumerate}
Then the same conclusion of the theorem holds.
\end{remark}

\begin{proof}
The map $\phi$ is backward Dirac, hence $L=\phi^!L_\pi$. In order to define the map $\psi:N\to \hM$ we denote by $\cI\to N$ the bundle of isotropy Lie algebras of $\phi^!L_\pi$. This is a bundle of abelian Lie algebras, because $\phi^!L_\pi$ is regular, and we  consider the composition of Lie algebroid morphisms
\begin{equation}\label{eq:family of maximal tori}
\xymatrix{
\psi_0:\cI\, \ar@{^{(}->}[r] & \phi^!L_\pi\ar[r]^{{(\widehat{\act},\phi)}^{-1}} & \phi^*T^*M\ar[r] ^{\pr}& T^*M}.
\end{equation}
The last projection is injective on isotropy Lie algebras. Hence, $\psi_0$ is a fiberwise injective vector bundle map covering $\phi:N\to M$. Therefore we have an induced  smooth map between the Grassmannians of rank $r=\corank(\pi)$:
\[ \psi: N=\Gr_r(\cI)\to \Gr_r(T^*M),\quad y\mapsto \psi_0(\cI_y).\]
Over the open dense set $\phi^{-1}(M^\reg)$ the map $\psi_0$ is an isomorphism which takes values in $\hM\subset \Gr_r(T^*M)$. Hence, applying also Theorem \ref{thm:resolution:map}, one has that:
\begin{enumerate}[(a)]
    \item $\hM$ is closed in $\Gr_r(T^*M)$, so $\psi$ has image contained in $\hM$;
    \item $\hM$ is embedded in $\Gr_r(T^*M)$, so $\psi$ can be viewed has a smooth map into $\hM$.
\end{enumerate}
Therefore 
\[\psi:N\to \hM,\quad y\mapsto \psi_0(\cI_y)\in \Gr_r(T^*_xM)\]
is a smooth map that satisfies 
\begin{equation}\label{eq:factor resolution}
\res\circ \psi=\phi.
\end{equation}

%


Next we show that $\psi$ is a diffeomorphism. For that it is enough to prove that it is a local diffeomorphism. Indeed, note that since $\phi$ is a proper map, the factorization \eqref{eq:factor resolution} implies that also $\psi$ is  proper. A proper local diffeomorphism is necessarily a covering map and $\psi$ being both a closed and an open map, must be surjective. Since $\psi$ is a diffeomorphism over $M^\reg$, we conclude that $\psi:N\to \hM$ is a diffeomorphism.

Finally, to show that $\psi$ is a local diffeomorphism we note that, by \eqref{eq:factor resolution}, we have
\[ \ker \psi_*\subset \ker\phi_*.\]
Hence, injectivity of $\psi_*$ follows from the identity
\begin{equation}\label{eq:trivial kernel intersection}
\ker\phi_*\cap\ker\psi_*=\{0\}.
\end{equation}
To prove this identity, let $y\in N$ and let $S'\subset N$ and $S\subset M$ denote the presymplectic leaves through $y$ and $x=\phi(y)$, respectively. Since $\phi$ is a backward Dirac map, we have that $(X,0)\in\phi^!L_\pi$ for any $X\in\ker \phi_*$, so 
\[{(\ker\phi_*)}|_{S'}=\ker{(\phi|_{S'})}_*.\]

We now use the algebroid isomorphism $(\widehat{\act},\phi):\phi^*T^*M\to \phi^!L_\pi$. Since the algebroid $T^*M$ acts along the proper map $\phi$, the action integrates to an action of the source 1-connected Lie groupoid $\Sigma(M)$ \cite[Thm 5.3]{MoerdijkMrcun02}. It follows that for any $x\in S$
 we obtain an action of the isotropy group $\Sigma(M)_x$ on the fiber 
 \[ \phi^{-1}(x)\cap S'=(\phi|_{S'})^{-1}(x)\] 
 whose orbits are tangent to $ \ker{(\phi|_{S'})}_*$. This also shows that we have a \emph{local} action of the isotropy group $G_x$ of a (possible smaller) proper integration of $(M,\pi)$ on this fiber. This provides a diffeomorphism
\[ G_x/T_x\supset U\diffto V\subset {(\phi|_{S'})}^{-1}(x)\]
from  a neighborhood $U$ of the coset of the identity to a neighborhood $V$ of $y$ in $(\phi|_{S'})^{-1}(x)$. Under this isomorphism, the map $\psi|_{V}$ becomes the restriction of the canonical covering map 
\[G_x/T_x\supset U\to G_x/N(T_x)=\mathcal{T}(\gg_x).\]
Therefore the differential of $\psi$ along $\ker\phi_*$ is injective, and this proves \eqref{eq:trivial kernel intersection}.
\end{proof}

\subsection{Resolution and the infinitesimal stratification}

In the sequel, we will say that the intersection of a map $\phi:X\to Y$ with a (embedded) submanifold $S\subset Y$ is {\bf very clean} if it satisfies the following properties:
\begin{enumerate}[(a)]
   \item $\phi^{-1}(S)$ is a (embedded) submanifold of $X$;
   \item  $T\phi^{-1}(S)=(\d \phi)^{-1}(TS)$;
   \item $\phi|_{\phi^{-1}(S)}:\phi^{-1}(S)\to S$ is a submersion.
\end{enumerate}
Notice that the usual notion of cleanness only requires the first two items which, unlike the notion of very clean, do not imply that points in $S$ are clean values. Also, (c) guarantees that the rank of $\phi$ is constant along $\phi^{-1}(S)$.

Cleanness is relevant also when trying to pull-back  stratifications $\cS_Y$ of $Y$ along $\phi: X\to Y$: if $\phi$ intersects all the strata of $\cS_Y$ cleanly and $\phi$ has connected fibers, then 
\[ \phi^*\cS_Y:=\left\{\phi^{-1}(S):S\in \cS_Y\right\}\]
is a stratification of $X$. Requirement (b) is relevant when addressing the Whitney conditions. Very clean ensures that the conclusion holds not only for $\cS_Y$ but for all other refinements of $\cS_Y$. 


\begin{theorem}\label{thm:compatibility with inf-stratification} 
Let $(M, \pi)$ be a  Poisson manifold of proper type. Then for each stratum $\Sigma$ of $\cSi(M,\pi)$, one has:
\begin{enumerate}[(i)]
   \item The intersection of $\res$ with $\Sigma$ is very clean;
   \item $\res^{-1}(\Sigma)$ is an embedded submanifold of $\hM$ along which $\res$ has constant rank equal to $\dim M+\corank(\pi)-\dim\gg_x$, for any $x\in\Sigma$.
\end{enumerate}
The partition of $\hM$ given by the rank of $\d\res$ induces the stratification $\res^{-1}(\cSi(M,\pi))$, after passing to connected components.
\end{theorem}

\begin{proof}
Let us assume, for the moment, that 
\begin{enumerate}
\item[(C)] $\res^{-1}(\Sigma)$ is embedded in $\hM$ and the restriction $\d\res:T\res^{-1}(\Sigma)\to T\Sigma$ is surjective.
\end{enumerate}
Then we have an inclusion of short exact sequences
\[ 
\xymatrix{
T_{\hx}\res^{-1}(x)\ar[r] \ar@{_{(}->}[d]& T_{\hx}\res^{-1}(\Sigma)\ar[r]\ar@{_{(}->}[d] & T_x\Sigma\ar@{=}[d]\\
\ker(\d_{\hx}\res)\ar[r] & (\d_{\hx}\res)^{-1}(T_x\Sigma)\ar[r] & T_x\Sigma
}
\]
Since all values of $\res$ are clean, the left vertical arrow is an equality, hence, so is the middle one. To complete the proof of the very clean statement (no pun intended) it remains to check that (C) holds.

In order to prove (C) we can use the Hamiltonian local model, so we assume $M=Q/G$ for a Hamiltonian $G$-space $\mu:(Q,\omega)\to\gg^*$. According to \eqref{eq:resolution:local:model}  the resolution map reads
\[ \res:\mu^{-1}(\tt^*)/N(T)\to Q/G, \]
where $T\subset G$ is a a fixed maximal torus. Then $\Sigma$ is identified with (see \eqref{eq:stratum:local:model})
\[ \Sigma= \mu^{-1}(\zz(\gg)^*)/G, \]
while its pre-image becomes
\[ \res^{-1}(\Sigma)=\mu^{-1}(\zz(\gg)^*)/N(T). \]
From this (C) is immediate.

To check that the preimage of $\cSi(M,\pi)$ is a stratification of $\hM$, one uses the local model, as in Remark \ref{rem:life:easy}. This reduces  the problem to the linear case where $\res:G/T\fproduct{W}\tt^*\to\gg^*$ and the infinitesimal stratification of $\gg^*$ has strata the $G$-saturations of the open faces $\interior(\Delta_I)$ of a Weyl chamber $\Delta$ of $\tt^*$ (see Example \ref{ex:regular-coadjoint-orbits} for more details). One observes that their pullbacks under the resolution map, $\res^{-1}(\interior(\Delta_I))=G/T\fproduct{W}(W\cdot \interior(\Delta_I))$, satisfy the frontier condition.

The formula for the dimension in item (ii) now follows. From this formula and the characterization \eqref{eq:inf-equiv} of the infinitesimal stratification, applying Lemma \ref{stratif-lemma-1}, the last statement also follows.
\end{proof}

\begin{corollary}\label{cor:codimension-rank-stratification}
For $\Si\in \cSi(M,\pi)$ the following holds 
\begin{align*}
    \codim(\Si)- \codim(\res^{-1}(\Si))&= \dim (\gg_x)- \dim (\tt_x),\\
    \codim (\res^{-1}(\Si))&= \dim (\tt_x)- \dim (\zz(\gg_x)).
\end{align*}
In particular, for any subregular stratum $\Si$,
$\res^{-1}(\Si)$ is of codimension one.
\end{corollary}

\subsection{Resolution and the canonical stratification}

A similar result holds for the pullback of the canonical stratification. 
This time however, $(\hM, \hL)$ comes with its own 
canonical stratification, denoted by $\cS(\hM,\hL)$,  as follows from the results in Section \ref{sec:stratifications:DMCT}. 
However, since $\hL$ is regular, Proposition \ref{prop:regular-canonical-is-foliation} allows us to be more direct and adopt the following: 
\begin{equation}\label{eq:cS-on-hM}
\cS(\hM,\hL):= \cS(\hM, \hF),
\end{equation}
the canonical stratification of the underlying regular foliation, as discussed in section \ref{sec:proper-foliations}.
More transparently, the analogue of Proposition \ref{prop:can-str-equiv-hol} ensures that  
$\cS(\hM,\hL)$ coincides with the canonical stratification induced by its integrating groupoid $\hG\tto \hM$, 
as discussed in section \ref{sec:The stratifications induced by a proper Lie groupoid}.

\begin{theorem}\label{thm:compatibility with stratification} Let $(M, \pi)$ be a Poisson manifold of proper type. The resolution map 
$\res:(\hM,\hL)\to (M,\pi)$
is compatible with the canonical stratifications. More precisely, $\res$ intersects every canonical stratum very cleanly and
\[ \cS(\hM,\hL)= \left\{\res^{-1}(\Sc): \Sc\in \cS(M,\pi)\right\}\]
\end{theorem}

\begin{proof}
%
We claim that $\res^{-1}(\cS(M,\pi))$ is a stratification of $\hM$ with the property that
\begin{equation}\label{eq:equal:tgts} 
T_y \widehat{\Sc}= T_y \res^{-1}(\Sc), 
\end{equation}
for every $ \widehat{\Sc}\in \cS(\hM,\hL)$, $\Sc\in \cS(M,\pi)$ and $y\in \widehat{\Sc}\cap \res^{-1}(\Sc)$. Applying Lemma \ref{stratif-lemma-aux}, the result follows.


To prove the claim we first observe that that $\res^{-1}(\cS(M,\pi))$ is a stratification because $\cS(M,\pi)$ refines $\cSi(M,\pi)$, 
and $\res$ is a map with connected fibers which intersects very cleanly the strata of $\cSi(M,\pi)$ -- see the remarks before Theorem \ref{thm:compatibility with inf-stratification}. 

To prove \eqref{eq:equal:tgts} fix $\hx= (x,\tt_x)$ and consider the stratum $\Sc\in \cS(M,\pi)$ through $x$ and the stratum $\widehat{\Sc}\in \cS(\hM,\hL)$ through $\hx$. We also fix the symplectic leaf $S\subset M$ through $x$ and the presymplectic leaf $\hS\subset \hM$ through $\hx$. Since $\hS=\res^{-1}(S)$,  \eqref{eq:equal:tgts}  will follow by showing that:
\begin{equation}
\label{eq:claim:normal:spaces} 
T_{\hx} \widehat{\Sc}/T_{\hx} \hS= T_{\hx} \res^{-1}(\Sc)/T_{\hx} \hS.
\end{equation}
For this, observe that $\res$ intersects $S$ cleanly, so we have an injection of normal spaces:
\[ \d_{\hx}\res: T_{\hx} \hM/T_{\hx} \hS\to T_x M/T_x S=\gg^*_x. \]
By Proposition \ref{prop:range:resolution}, the image is $\tt^*_x$, and we will identify the normal space to $\hS$ with this image:
\begin{equation}
\label{eq:res:normal:spaces}
\d_{\hx}\res: T_{\hx} \hM/T_{\hx}\hS\diffto \tt^*_x  \subset  \gg^*_x=T_x M/T_x S. 
\end{equation} 
Now we fix a integration $(\cG,\Omega)$ of $(M,\pi)$, we let $G=G_x$ and $T\subset G$ be the maximal torus of $G$ with Lie algebra $\tt=\tt_x$. From now on, we also denote $\gg_x$ by $\gg$. For the associated action groupoid $\hG=\cG\ltimes\hM$ the isotropy group at $(x,\tt)$ is $N(T)$. Its isotropy Lie algebra $\tt$ is canonically identified with the conormal to $\hS$ at $(x,\tt)$. This is precisely the dual identification to \eqref{eq:res:normal:spaces}. We need the following lemma which identifies each side of \eqref{eq:claim:normal:spaces}.

\begin{lemma}
Under the identification \ref{eq:res:normal:spaces}, one has that:
\begin{enumerate}[(i)]
\item the right-hand side of \eqref{eq:claim:normal:spaces} becomes
\[
 T_{(x, \tt)} \res^{-1}(\Sc)/T_{(x, \tt)} \hS=(\zz(\gg)^*)^{G/G^0}\subset\tt^*;
\]
\item  the left-hand side of \eqref{eq:claim:normal:spaces} becomes
\[
T_{(x, \tt)} \widehat{\Sc}/T_{(x, \tt)} \hS=(\tt^*)^{N(T)/T}\subset\tt^*.
\]
\end{enumerate}
\end{lemma}

\begin{proof}
We note the following the inclusions:
\[ T_x\Sc/T_xS\subset T_x\Si/T_xS\subset T_x M/T_x S, \]
which have the Lie theoretical interpretation (see \eqref{eq:strata:normal} and \eqref{eq:strata:normal2}):
\[(\zz(\gg)^*)^{G/G^0}\subset \zz(\gg)^*\subset \gg^*. \]
So item (i) follows.

As for item (ii), we consider the holonomy action of $\pi_1(\hS)$ on the normal space to the presymplectic leaf. By Proposition \ref{prop:holonomy:groupoid:version}, this action coincides with the coadjoint action of the group of connected components of the isotropy of the groupoid $\hG$, i.e., $N(T)/T$. With the identifications described before the proposition, (ii) follows.
\end{proof}

To conclude the proof of \eqref{eq:claim:normal:spaces} and the proof of the theorem, we only need to show that:
\[ (\tt^*)^{N(T)/T}=(\zz(\gg)^*)^{G/G^0}. \]
This follows from the short exact sequence \eqref{eq:short:seq:Weyl:grp}, the decomposition \eqref{eq:cpctLie-decomp} for $\tt^*$, and the fact that $W^0$ acts without fixed points on $\tt^{\ss}:=[\gg,\gg]\cap\tt$.
\end{proof}

\subsection{Weyl resolution of DMCTs}
 
The definition of the Weyl resolution of a $\phi$-twisted DMCT $(M,L)$ is exactly the same as for PMCTs, using maximal tori in the isotropy Lie algebras of $L$ (see Definition \ref{def:resolution}). We shall denote the resolution by
\[ \res:(\hM,\hLD)\to (M,L).\]
where $\hLD=\res^*L$ is a twisted Dirac structure with twist $\hphi=\res^*\phi$.

\begin{example}
\label{ex:Cartan:Dirac:baby}
Let $G$ be a compact connected Lie group, let $\langle\cdot, \cdot\rangle$ be an $\Ad$-invariant inner product on $\gg^*$, and let $L_G$ be the corresponding
Cartan-Dirac structure on $G$ with twisting $\phi$ the Cartan 3-form \cite{AAM98,BCWZ,SW01}. Recall that its leaves are the conjugacy classes of $G$. An s-proper integration is provided by the conjugacy action groupoid $G\ltimes G$ endowed the with the multiplicative 2-form
 \[\Omega_G(g, h) =\frac{1}{2}\left(\langle\Ad_h \pr_1\theta^L , \pr_1\theta^L\rangle+\langle \pr_1\theta^L, \pr_2(\theta^L +\theta^R\rangle\right),\]
where $\theta^L$ and $\theta^R$ are the left and right-invariant Maurer-Cartan forms.

Similarly to the case of coadjoint orbits, recalling that the maximal torus of $G_g$ are precisely the maximal torus of $G$ that contain $g$, one finds that the resolution
\[\widehat{G}=\{(h,T)\,|\, h\in T,\,T\subset G\,\,\mathrm{maximal}\,\,\mathrm{torus}\},\quad \res=\mathrm{pr}_1:\widehat{G}\to G. \]
Fixing a maximal torus $T$, one has an identification 
\[G\times_{N(T)}T=G/T\times_W T \cong  \hM,\quad [(g,h)]\mapsto (ghg^{-1},gTg^{-1}), \]
so under this identification the resolution map is given by:
\[ \res([g,h])=ghg^{-1}. \]
\end{example}

The resolution enjoys many of the properties we saw for the case of PMCTs. For example, we have:

\begin{theorem}\label{thm:Weyl-intro:Dirac} 
For any  twisted Dirac manifold of (strong) $\cC$-type $(M,L)$:
\begin{enumerate}[(i)]
\item $(\hM,\hLD)$ is a regular Dirac manifold of (strong) $\cC$-type; 
\item $\res$ is a forward Dirac map which restricts to a diffeomorphism from $\res^{-1}(M^{\reg})$ to $M^{\reg}$;
\item $\res$ descends to a homemorphism between the leaf spaces of $(\hM,\hLD)$ and $(M,L)$.  
\item For any twisted presymplectic integration $(\cG,\Omega)\tto (M,L)$ of (strong) $\cC$-type, the action groupoid and the pullback of $\Omega$
\begin{equation}
    \label{eq:grp:resolution}
    \hG:=\G\ltimes\hM\tto\hM,\quad \hOmega:=\pr_\cG^*\Omega,
\end{equation}
is a twisted presymplectic integration of $(\hM,\hLD)$ of (strong) $\cC$-type.
\end{enumerate}
\end{theorem}

The proof of these facts is the same as in the Poisson case using the Hamiltonian local form for DMCTs (cf.\,Theorem \ref{thm-lf-Dirac}). More specifically, one shows that the Weyl resolution $\hM$ again is canonically a closed embedded submanifold of $\mathrm{Gr}_r(T^*M)$, where $r$ is the codimension of a regular leaf on $(M,L)$. The (regular) Dirac structure $\hLD$ is also obtained as the image of the injective bundle map
\begin{equation*}
(\widehat{\act},\d\res^*\mathrm{pr}_{T^*M}): \res^*L\to T\hM\oplus T^*\hM, 
\end{equation*}
where $\widehat{\act}$ is the action map.

The resolution map enjoys properties analogous to the ones of the resolution map in the Poisson case. Namely, one finds that:
\begin{enumerate}[(i)]
    \item $\res:\hM\to M$ is a proper smooth surjection with connected fibers;
    \item All values of $\res$ are clean and the regular values are exactly the points in $M^\reg$;
    \item The range and kernel of $\d\res$ satisfy:
    \begin{align*}
        &\im(\d_{x,\tt_x}\res)^0=[\tt_x,\gg_x],\\
        &\dim(\ker\d_{x,\tt_x}\res)=\dim\gg_x-\dim\tt_x;
    \end{align*}
    \item Upon restriction to the twisted presymplectic leaves,
    \[ \res:\hS\to S\]
    is a fiber bundle.
\end{enumerate}
Moreover, the characterization of the Weyl resolution in Theorem \ref{thm:characterization:resolution} extends to DMCTs replacing the algebroid $T^*M$ by $L$. 

Finally, the resolution interacts with the canonical stratifications of $(M,L)$ and $(\hM,\hLD)$ in an similar fashion as in the Poisson case. 

\begin{theorem}\label{thm:compatibility with inf-stratification:Dirac} 
Let $(M,L)$ be a twisted Dirac manifold of proper type. Then for each stratum $\Sigma$ of $\cSi(M,L)$, one has:
\begin{enumerate}[(i)]
   \item The intersection of $\res$ with $\Sigma$ is very clean;
   \item $\res^{-1}(\Sigma)$ is an embedded submanifold of $\hM$ along which $\res$ has constant rank equal to $\dim M+\corank(L)-\dim\gg_x$, for any $x\in\Sigma$;
   \item The stratification $\res^{-1}(\cSi(M,L))$ is induced by the partition of $\hM$ given by the rank of $\d\res$, after passing to connected components;
   \item One has
   \begin{align*}
    \codim(\Si)- \codim(\res^{-1}(\Si))&= \dim (\gg_x)- \dim (\tt_x),\\
    \codim (\res^{-1}(\Si))&= \dim (\tt_x)- \dim (\zz(\gg_x)),
    \end{align*}
    and in particular each subregular infinitesimal strata is of codimension 1.
\end{enumerate}
Moreover, $\res$ intersects every canonical stratum very cleanly and
\[ \cS(\hM,\hLD)= \left\{\res^{-1}(\Sc): \Sc\in \cS(M,L)\right\}\]

\end{theorem}

\section{The Geometry of the Leaf Space} 
\label{sec:leaf:space}

In this section we fix a Poisson manifold of proper type $(M, \pi)$ and we look at the structure present on its space of symplectic leaves $M/\cF_\pi$. One of the main ingredients in this discussion is the resolution map $\res$. By Theorem \ref{thm:resolution:Dirac}, it induces a homeomorphism 
\[\res:\hM/\hF\to M/\cF_\pi. \] 
We will soon identify the two leaf spaces, 
but we would like to point out that this identification is not entirely obvious and it becomes relevant when we take into account the various structures present on them. For example, when $M= \gg^*$, the dual of the Lie algebra of a compact Lie group $G$, this amounts to the identification of $\tt^*/W$ with $\gg^*/G$.

\begin{remark}[Leaf spaces of DMCTs]
All results in this section are valid for any twisted Dirac manifold of proper type. We leave it to the reader to check that the proofs hold in this more general setting.
\end{remark}

\subsection{The smooth structure and a Chevalley-type theorem}
\label{sec:The smooth structure}
The quotient map $M\to M/\cF_{\pi}$ can be used to endow $M/\cF_{\pi}$ not only with a topology, but also to make sense of ``smooth functions'': we say that $f: M/\cF_{\pi}\to \R$ is smooth if its composition with the quotient map is smooth. We denote by $\cC^{\infty}(M/\cF_{\pi})$ the collection of all such smooth functions, and we define $\cC^{\infty}(\hM/\hF)$ similarly. The identification  between the two leaf spaces is compatible with these smooth structures.

\begin{theorem} The resolution map induces an isomorphism
\[ \res^*: \cC^{\infty}(M/\cF_{\pi})\cong \cC^{\infty}(\hM/\hF).\]
\end{theorem}

\begin{proof}
Denote $B=M/\cF_{\pi}$. The concept of smoothness defined above gives rise to a subsheaf $\cC^{\infty}_{B}$ of the sheaf $\cC_B$ of continuous functions on $B$. Similarly for $\widehat{B}= \hM/\hF$. The homeomorphism $\res: \widehat{B}\to B$
gives rise to an isomorphism of sheaves
\[ \res^*: \res^*\cC_{B}\to \cC_{\widehat{B}}.\]
The statement of the theorem follows from the stronger claim that this map restricts to an isomorphism 
\[ \res^*:\res^*\cC^{\infty}_{B}\to \cC^{\infty}_{\widehat{B}}.\] 
This more stronger statement is sheaf theoretic and, therefore, it suffices to prove it locally. For that we can use the Hamiltonian local form to reduce it to the case where $M=\gg^*$, where $\gg$ is a Lie algebra of compact type. 

We are left to show that we have an isomorphism
\begin{equation}
    \label{eq:res:iso:functions}
    \res^*:C^\infty(\gg^*)^G\to C^\infty(\hgg)^G,
\end{equation}
where $G$ is a compact Lie group with Lie algebra $\gg$, the resolution of $\gg^*$ is (see Example \ref{ex:linear:case}):
\[ \hgg=G/T\times_W \tt^*,\]
and the resolution map is (cf.~\eqref{res:linear:case})
\[ \res:G/T\times_W \tt^*\to \gg^*, \quad (gT,\xi)\mapsto \Ad_g^*(\xi). \]
Here the $G$-action on $\hgg=G/T\times_W \tt^*$ is by translations on the first factor
\[ g\cdot (g'T,\xi)=(gg'T,\xi). \]
By classical results of Chevalley (\cite[Section 3.5]{Humphreys90} and Schwarz (\cite{Schwarz75}), the inclusion $i:\tt^*\hookrightarrow\gg^*$ gives an isomorphism
\[ i^*:C^\infty(\gg^*)^G\to C^\infty(\tt^*)^W.\] 
Now we just have to observe that the composition
\[ \res^*\circ (i^*)^{-1}:C^\infty(\tt^*)^W\to C^\infty(G/T\times_W\tt^*)^G, \]
is clearly an isomorphism since it
maps a $W$-invariant function $f\in C^\infty(\tt^*)$ to the $G$-invariant function $\widehat{f}\in C^\infty(G/T\times_W\tt^*)$ given by
\[ 
\widehat{f}([gT,\xi]):=f(\xi).
\]
\end{proof}

\begin{remark} 
\label{rmk:smooth:functions}
The notion of smoothness on the leaf space introduced above makes sense whenever we deal with quotients of manifolds. In general, the resulting concept is a poor one. However, we are in a very special situation: both $M/\cF_{\pi}$ and $\hM/\hF$ arise as orbit spaces of proper groupoids. That implies that that $C^{\infty}(M/\cF_{\pi})$ behaves very similarly to the algebra of smooth functions on a manifold, including the existence of partitions of unity. 
Actually, for such orbit spaces even more is true: there are several existing frameworks that allow us to treat them as ``smooth spaces'', such as that of ringed spaces, differentiable spaces, locally $C^{\infty}$-schemes, Sikorski spaces or  subcartesian spaces (for an overview, and more references, see \cite{CraMe}). All these settings produce the same notion of smoothness as the one described above. However, as we shall see in the next section, the case of Poisson manifolds of proper type and their proper symplectic groupoids is even more special: the resulting leaf spaces are actually orbifolds. 
\end{remark}

\begin{definition} 
Given a a Poisson manifold of proper type, we denote its leaf space by 
\begin{equation}\label{eq:the-leaf-space} 
B= B(M, \pi):= M/\cF_{\pi}= \hM/\hF.
\end{equation}
where the last equality is the identification induced by $\res$.
\end{definition} 
\medskip


\subsection{Orbifolds and stratifications}
\label{sec:The case of orbifolds}
Before  continuing our discussion on the leaf space $B$ we recall some basic notions and facts concerning orbifolds and their stratifications. 


As in \cite[Chp. 2.6]{CFM-II}, by an {\bf orbifold} we mean a pair $(B, \cB)$ where $B$ is a topological space and $\cB$ is an orbifold atlas for $B$, i.e., $\cB\tto M$ is a proper foliation groupoid and $p:M\to B$ is a map inducing a homeomorphism $M/\cB\to B$. By a \textbf{classical orbifold} we mean an orbifold presented by an effective orbifold atlas. Two atlases induce the same orbifold structure if they are Morita equivalent. By the discussion in Section \ref{sec:The stratifications induced by a proper Lie groupoid}, an orbifold $B$ carries an induced stratification, independent of the atlas, known as the \textbf{canonical stratification of the orbifold $B$}.

A connected orbifold will be said to be \textbf{of smooth type} if its stratification has a single strata. This condition can be rephrased by requiring $B$ to carry a smooth structure, necessarily unique, making $p: M\to B$ into a smooth submersion. The same condition is equivalent also to the requirement that $\cB$ is \textbf{totally ineffective}, meaning that the normal actions of the isotropy groups $\cB_x$ are all trivial. Notice however that, even if $B$ is of smooth type, it may still carry interesting non-trivial orbifold structures. 

More generally, assume that a space $B$ carries two different orbifold structures defined by atlases $\cB_i\tto N_i\stackrel{p_i}{\to} B$. We will say that $\cB_1$ covers $\cB_2$ if there exist a surjective, submersive, groupoid morphism fibered over $B$, as in the following diagram:
\[
\xymatrix{
\cB_1\ar[rr] \ar@<0.25pc>[d] \ar@<-0.25pc>[d]  & & \cB_2 \ar@<0.25pc>[d] \ar@<-0.25pc>[d]  \\
N_1\ar[rr]\ar[dr]_{p_1} & & N_2\ar[dl]^{p_2}\\
 & B 
}
\]
We will also say that the second orbifold structure is an {\bf isotropy reduction} of the first one. In this situation, the canonical stratifications induced by $\cB_1$ and $\cB_2$ coincide.

Given an orbifold $B$,  with orbifold atlas $\cB\tto M$, by a {\bf suborbifold} we mean any subspace $C\subset B$ 
with the property that $p^{-1}(C)\subset M$ is a submanifold of $M$. We say that $C$ is {\bf embedded} if $p^{-1}(C)\subset M$ is an embedded submanifold. In this case we endow $C$ with the orbifold structure presented by the restriction of $\cB\tto M$ to $p^{-1}(C)$. Similarly, given a classical orbifold $B$ with effective atlas $\cB\tto M$ we say that $C\subset B$ is a  \textbf{classical embedded suborbifold} if $p^{-1}(C)\subset M$ is an embedded submanifold and the restriction of $\cB$ to $p^{-1}(C)$ is effective.

\begin{remark}\label{remark:for:reference:to:embedded:suborbifolds}
We warn the reader that the notion of suborbifold considered here is one suited for our purposes. There are various notions of suborbifold in the literature with subtle relationships -- see, e.g., \cite{MestreMartin} and references therein.
\end{remark}

The previous discussion allows us to talk about \textbf{orbifold stratifications}: in Definition \ref{def:str:man} the strata are allowed to be orbifolds, but they are required to be embedded suborbifolds.

\begin{example}
For any orbifold $B$ the members of its canonical stratification are embedded suborbifolds of smooth type. So this stratification can be consider either as:
\begin{itemize}
    \item a stratification in the weaker sense of Definition \ref{def:str:man} (so we regard its members as manifolds), or
    \item an orbifold stratification (so we now regard its members as embedded suborbifolds).
\end{itemize}
From now on, given an orbifold we always consider its canonical stratification as an orbifold stratification.
\end{example}

Finally, here is one more concept on orbifolds that we need: 
an {\bf integral affine orbifold} is described by an orbifold atlas $\cB\tto M\stackrel{p}{\to} B$ 
together with a transverse integral affine structure on the orbit foliation of $\cB$. The latter can also be described by a presymplectic torus bundle $(\cT,\Omega_\cT)\to M$ -- see \cite[Prop. 3.2.8]{CFM-II} and Remark \ref{rem:suited:IA}.
An  {\bf integral affine suborbifold} is any suborbifold $C$ with the property that $p^{-1}(C)\subset M$ is 
a transverse integral affine submanifold. This allows us to talk about \textbf{integral affine orbifold stratifications} of an integral affine orbifold: the strata are required to be integral affine suborbifolds.

\subsection{The integral affine orbifold structure}
\label{sec:The integral affine structure on the leaf space}
Using the identification \eqref{eq:the-leaf-space}, we now make actual use of the geometry of the resolution $\hM$, a \emph{regular} Dirac manifold. Recall that a choice of proper integration $(\cG,\Omega)$ of $(M, \pi)$ gives rise to a \emph{regular}, proper, presymplectic groupoid $(\hG,\hOmega)$ over $\hM$ (see \eqref{eq:grp:resolution}). This allows us to apply the Dirac version of Theorem 3.0.1 of \cite{CFM-II}, obtaining the following result.

\begin{theorem}
\label{thm:int:affine:leaf:space}
Each proper symplectic integration  $(\cG,\Omega)$ of a  Poisson manifold of proper type $(M,\pi)$ induces an integral affine orbifold structure on the leaf space $B= B(M, \pi)$.
\end{theorem}

Let us explain briefly how the integral affine orbifold structure on the leaf space arises, using the concepts recalled in Section \ref{sec:The case of orbifolds}.
The orbifold atlas for $B$ is the proper foliation groupoid $\cB(\hG)$ defined by the short exact sequence 
\begin{equation}\label{eq:Weyl-short-exact-sequence}
\xymatrix{1\ar[r]& (\hT,\widehat{\Omega}_{\hT})\ar[r]& (\hG,\widehat{\Omega})\ar[r]& \cB(\hG)\ar[r] & 1.}
 \end{equation}
Here $\hT$ is the bundle of groups consisting of the unit connected components of the isotropy groups of $\hG$. Explicitly, it is the tautological torus bundle whose fiber at $(x,\tt_x)\in\hM$ is the torus $T_x\subset \cG_x$ integrating  $\tt_x$. The restriction of $\hOmega$ to $\hT$ makes it into a presymplectic torus bundle which induces an integral affine structure on the orbifold.

From now on we always consider $B$ with the integral affine orbifold structure on $B$ induced by an integration $(\cG,\Omega)$.

\subsection{
Stratifications of the leaf space of a Poisson manifold of proper type
}
\label{sec:stratifications:leaf:space}

Throughout this section $(M, \pi)$ is a Poisson manifold of proper type, with a fixed choice of proper integration $(\G, \Omega)\tto (M, \pi)$. Our aim is to explain how the stratifications  discussed in Section \ref{sec:canonical:stratifications} interact with the integral affine orbifold structure that we just described.

First of all, the canonical and the infinitesimal stratifications of $(M, \pi)$ descend to similar partitions on the leaf space $B$, denoted
\[ \cS(B), \quad \cSi(B) .\]
Via the identification (\ref{eq:the-leaf-space}) these can also be seen as induced  by the canonical and the rank stratifications of the resolution $(\hM, \hL)$,
 respectively -- see Theorems \ref{thm:compatibility with stratification} and  \ref{thm:compatibility with inf-stratification}. Recalling the notions of ``(integral affine) orbifold stratification'' and ``canonical stratification of an orbifold'' (see Section \ref{sec:The case of orbifolds}) one has:

\begin{theorem}
\label{thm:canonical-orbifold-stratifications} 
    The members of $\cS(B)$ and $\cSi(B)$ carry natural orbifold structures so that $\cS(B)$ and $\cSi(B)$ become integral affine orbifold stratifications on the leaf space $B$. 
\end{theorem}

Applying Corollary \ref{cor:codimension-rank-stratification} we deduce:

\begin{corollary}
\label{cor:subregular:leaf:space}
For any subregular strata $\Sigma\in \cSi(M, \pi)$, $B_\Sigma:=\Si/\cF_\pi\in\cSi(B)$ is a connected codimension one integral affine embedded suborbifold of $B$.
\end{corollary}

From now on, unless otherwise specified, we consider $\cS(B)$ and $\cSi(B)$ as integral affine orbifold stratifications, as in Theorem \ref{thm:canonical-orbifold-stratifications}. These have the following properties:

\begin{proposition}
Fix $\Si\in\cSi(B)$. Then:    \label{prop:canonical-orbifold-stratifications}
\begin{enumerate}[(i)]
\item $\cS(B)$ coincides with the canonical stratification of the orbifold $B$; 
\item The restriction $\cS(B)|_{B_{\Si}}$ is the canonical stratification of the orbifold $B_{\Si}$.
\end{enumerate}
\end{proposition}

Each member of $\cS(B)$ and $\cSi(B)$ carries \emph{another} natural orbifold structure, which we now describe. First, the members can be described as:
\begin{align*}
    B_{\Sc}:= B(\Sc, \pi|_{\Sc})= \Sc/\cF_{\pi}= \hSc/\hF, \quad &\textrm{with}\ \Sc\in \cS(M, \pi),\ \hSc= \res^{-1}(\Sc),\\
   B_{\Si}:= B(\Sc, \pi|_{\Si})= \Si/\cF_{\pi}= \hSi/\hF,\,  \quad &\textrm{with}\ \Si\in \cSi(M, \pi),\ \hSi= \res^{-1}(\Si). 
\end{align*}
Each member $\Sc\in \cS(M, \pi)$ sits inside a member $\Si\in \cSi(M, \pi)$, so that 
\[ B_{\Sc}\subset B_{\Si}\subset B.\]
Theorem \ref{thm:inf-strat-type} and Theorem \ref{thm:can-strat-type} ensure that $B_{\Sc}$ and $B_{\Si}$ are themselves leaf spaces of regular PMCTs. Furthermore, as we shall also recall below, the proofs of those theorems show that the original groupoid $(\cG, \Omega)$ induces proper symplectic integrations $\G_{\Si}$ and $\G_{\Sc}$ of the strata $\Si$ and $\Sc$. Therefore both $B_{\Si}$ and $B_{\Sc}$ inherit integral affine orbifold structures from $\G$. These are \emph{different} from the integral affine orbifold structures appearing in Theorem \ref{thm:canonical-orbifold-stratifications}, which were inherited from $B$. To explain the relationship between the two integral orbifold structures on $B_{\Si}$ and $B_{\Sc}$ we make use of the notion of ``isotropy reduction'' described in Section \ref{sec:The case of orbifolds}. 


\begin{proposition}
\label{prop:canonical-orbifold-stratifications:G} 
Let $\Sc\in \cS(M)$ and $\Si\in \cSi(M)$. Then:
\begin{enumerate}[(i)] 
\item The orbifold structures of $B_{\Si}$ and $B_{\Sc}$ induced from $\cG$ are  isotropy reductions of the orbifold structures inherited from $B$;
\item If $B_{\Sc}$ and $B_{\Si}$ are endowed with the orbifold structures induced from $\cG$ and $\Sc\subset \Si$, then $B_{\Sc}$ is an integral affine suborbifold of $B_{\Si}$.
\item If we endow the members of both $\cS(B)$ and $\cSi(B)$ with the orbifold structures induced from $\cG$, then $\cS(B)|_{B_\Si}$ is the canonical orbifold stratification of $B_\Si$.
\end{enumerate}
\end{proposition}

\begin{remark}
    Notice that Theorem \ref{thm:canonical-orbifold-stratifications} implies that for each $\Sc\in \cS(B)$ the orbifold structure on $B_{\Sc}$ induced from $B$ is of smooth type (see Section \ref{sec:The case of orbifolds}). The previous proposition then implies that the orbifold structure on $B_{\Sc}$ induced from $\cG$ is also of smooth type. 
\end{remark}

\begin{proof}[Proof of Theorem \ref{thm:canonical-orbifold-stratifications} and Propositions \ref{prop:canonical-orbifold-stratifications} and \ref{prop:canonical-orbifold-stratifications:G}]
We start by showing that both $\cS^\inf(B)$ and $\cS(B)$ are integral affine orbifold stratifications of $B$.
The members of both partitions satisfy the frontier condition as in Remark \ref{rem:correspondence-stratifications}. So it suffices to prove that each
\[B_{\Si}=\hSi/\hF,\quad  B_{\Sc}=\hSc/\hF\] is an  embedded integral affine suborbifold of $B$. 

Checking that they are embedded suborbifolds is immediate: the preimage of $B_{\Si}$ and $B_{\Sc}$ by 
$p:\hM\to B$ are the embedded submanifolds $\hSi$ and $\hSc$, respectively.

Checking the integral affine suborbifold condition is equivalent, by Remark \ref{rem:suited:IA}, to the presymplectic torus bundle $(\hT,\hOmega)$ being suited for smooth reduction along both $\hSi$ and $\hSc$. The latter conditions are the same as $(\hG,\hOmega)$ being suited for smooth reduction along these strata. For this observe that:
\begin{enumerate}[(a)]
    \item[-] By Theorem \ref{thm:compatibility with stratification} the strata $\hSc$ belongs to the canonical stratification of $(\hM,\hL)$. The Dirac version of Theorem \ref{thm:can-strat-type} (see Section \ref{sec:stratifications:DMCT}) then shows that $(\hG,\hOmega)$ is suited for smooth reduction along $\hSc$.
    \item[-] We claim that $(\hG,\hOmega)$ is suited for smooth reduction along $\hSi$. For that, we fix $\hx=(x,\tt_x)\in\hSi$ and we need to show that the connected Lie subgroup $\cK_{\hx}\subset\hG_{\hx}$ integrating $(T_{\hx}\hSi)^0$ is closed (see Lemma \ref{lemma:suited-for-reduction}). Notice that 
    \[ \hG_{\hx}=N(T)\subset \cG_x, \]
    where $T\subset \cG_x$ is the maximal torus corresponding to $\tt_x$, and we claim that $\cK_{\hx}$ has Lie algebra 
    \begin{equation}
        \label{eq:Lie:algebra:aux}
        \kk_{\hx}=\gg_x^\ss\cap \tt_x\subset \gg_x.
    \end{equation}
    It follows that $\cK_{\hx}$ is the intersection of a closed subgroup of $\cG_x$ with $N(T)$, hence it is closed.
    To check \eqref{eq:Lie:algebra:aux} we recall that the Lie algebroid of $\hG=\cG\ltimes\hM$ is identified with $\hL$ via the map
    \[ (\act,\res^*):\res^*(T^*M)\to \hL. \]
    This restricts to an isomorphism
    \[ \res^*:\tt_x\to (T_{\hx}\hS)^0,\]
    where $\tt_x\subset\gg_x=(T_x S)^0$. Now we note that the linear map 
    \[ \res^*: (T_xS)^0\to (T_{\hx}\hS)^0, \]
    is surjective with kernel $[\tt_x,\gg_x]$ (c.f. Proposition \ref{prop:range:resolution}). Since $\res$ has very clean intersection with $\Sigma$, one finds
    \[
    \res^*((T_x\Sigma)^0)=(T_{\hx}\hSi)^0.
    \]
    On the other hand, by \eqref{eq:strata:conormal}, we have
    \[ (T_x\Sigma)^0=[\gg_x,\gg_x]=\gg_x^\ss. \]
    so \eqref{eq:Lie:algebra:aux} follows.
\end{enumerate}
This completes the proof of Theorem \ref{thm:canonical-orbifold-stratifications}.
\medskip

Next we prove item (i) of Proposition \ref{prop:canonical-orbifold-stratifications}. Using again that $\res$ pulls back $\Sc(M,\cF_\pi)$
to $\Sc(\hM,\hL)$, by the Dirac version of Proposition \ref{prop:regular-canonical-is-foliation} (see Section \ref{sec:stratifications:DMCT}), the canonical stratification $\Sc(\hM,\hL)$ equals the holonomy stratification of the regular foliation $\hF$, and hence item (i) of Proposition \ref{prop:canonical-orbifold-stratifications} follows.
\medskip

To prove item (ii) of Proposition \ref{prop:canonical-orbifold-stratifications}, we note that $\hSc$ and $\hSi$ are both saturated for the regular foliation $\hF$, so whenever $\hSc\subset\hSi$ the holonomy of a leaf $\hS\subset\hSc$ for $\hF|_{\hSc}$ is the restriction of the holonomy of $\hF|_{\hSi}$. Since the the canonical stratification $\cS(\hM)$ equals the holonomy stratification of $\hF$, the result follows.
\medskip

We now turn to the proof of Proposition \ref{prop:canonical-orbifold-stratifications:G}. To prove item (ii), we observe that by the proof of Theorem \ref{thm:can-strat-type}, one has:
\begin{enumerate}[(a)]
\item for any  $\Si\in\cSi(M)$, the reduction of $(\cG,\Omega)$
is a proper symplectic integration $(\cG(\Sigma),\Omega(\Sigma))$ of $(\Si,\pi_{\Si})$: 
\[
\xymatrix{
1\ar[r] & (\cT({\Si}),\Omega(\Si)|_{\cT(\Si)})\ar[r] & (\cG(\Si),\Omega(\Si))\ar[r] & \cB({\Si}) \ar[r] & 1},
\]
\item and, if $\Sc\in\cS(M)$ with $\Sc\subset \Si$, then $(\cG(\Si),\Omega(\Si))$ is suited for smooth reduction along $\Sc$, resulting in
\[
\xymatrix{
1\ar[r] & \left(\cT(\Sc),\Omega(\Sc)|_{\cT(\Sc)}\right)\ar[r] &\left(\cG(\Sc),\Omega(\Sc)\right)\ar[r] & \cB({\Sc}) \ar[r] & 1},
\]
where $\cB({\Sc})=\cB({\Si})|_{\Sc}$. Also, this coincides with the reduction of $\cG$ along $\Sc$.
\end{enumerate}
Therefore, by Remark \ref{rem:suited:IA}, the transverse integral affine foliation $(\Si,\cF_{\pi}|_{\Si})$ contains $\Sc$ as a transverse integral affine submanifold.
\medskip

To prove item (i) of Proposition \ref{prop:canonical-orbifold-stratifications:G} we have to compare $\cB(\hSc)$ and $\cB(\Sc)$, and, similarly, $\cB(\hSi)$ and $\cB(\Si)$. This follows from two general remarks concerning the orbifold atlas $\cB(\cG)$ on the orbit space $M/\cG$ of a proper regular Lie groupoid with associated exact sequence
\[ 
\xymatrix{1\ar[r] & \cT(\cG)\ar[r] & \cG \ar[r] & \cB(\cG)\ar[r] & 1}. 
\]
Namely:
\begin{enumerate}[1)]
\item For any two such proper regular groupoid $\cG_i\tto M$, if there is a submersive groupoid morphism $\cG_1\to \cG_2$ covering the identity, then the orbifold structure $\cB(\cG_2)$ is an isotropy reduction of $\cB(\cG_1)$; 
\item Secondly, if a proper regular groupoid $\cG\tto M$ acts transitively along a submersion $\hM\to M$ with connected fibers, then the orbit space of $\cG\ltimes \hM\tto \hM$ can be identified with $M/\cG$ and $\cB(\cG)$ is an isotropy reduction of $\cB(\cG\ltimes \hM)$.
\end{enumerate}
\medskip

Finally,to complete the proof, item  (iii) of Proposition \ref{prop:canonical-orbifold-stratifications:G}  follows from Proposition \ref{prop:regular-canonical-is-foliation} and item (ii).
\end{proof}

\begin{remark}
    Item (iii) in Proposition \ref{prop:canonical-orbifold-stratifications:G} is a special case of a general fact for proper Lie groupoids. Namely,
    let $\cG\tto M$ be a proper Lie groupoid with leaf space $B$. Then each member of $\cS_\cG(B)$ and $\cSi_\cG(B)$ inherits a natural orbifold structure from $\cG$: for each such member $N\subset M$, the restriction $\cG|_N$ is a regular proper groupoid and its effectivization defines an orbifold atlas for the leaf space of $N$ (see, e.g., \cite{CraMe,Moerdijk03}). Now, given $\Si\in\cSi(M)$, one finds that the restriction $\cS(B)|_{B_\Si}$ is the orbifold stratification of $B_\Si$. Moreover, the members of $\cS(B)|_{B_\Si}$ -- namely $B_\Sc$ for $\Sc\in \cS(M)$ with $\Sc\subset\Si$ -- are embedded suborbifolds of smooth type.
\end{remark}

\subsection{Coverings and the main diagram}
\label{sec:main:diagram}
In this section we collect from \cite{CFM-II} various facts about coverings of the resolution and of the leaf space of a PMCT, which arise from the integral affine geometry. This will be heavily used in the remainder of this paper and can be summarized in the following result.

\begin{proposition}[The main diagram]
Let $(M,\pi)$ be a Poisson manifold with a proper integration $(\cG, \Omega)\tto (M, \pi)$ and fix a base point $\hx_0= (x_0, \tt_0)\in \hM$. The resolution map extends to a commutative diagram
\begin{equation}\label{eq:Weyl:main-diagram} 
\vcenter{
\xymatrix{
\oM \ar[d]\ar[r] &  \hM^{\aff} \ar[d]\ar[r]&  \hM^{\lin} \ar[d]\ar[r]&  \hM \ar[d]\ar[r]^{\res} &  M \ar[d] \\
\oB \ar[r] &  B^{\aff} \ar[r]&  B^{\lin} \ar[r]&  B= \hM/\hF  \ar@{=}[r] &  M/\cF_{\pi} 
}}
\end{equation}
with the following properties:
\begin{enumerate}[(i)]
    \item In the top row the spaces are covering spaces of $\hM$. In particular, they inherit pullback Dirac structures from $\hL$.
    \item In the bottom row $B^\lin$ is an 
    integral affine smooth manifold, while $B^\aff$ and $\oB$ are covering spaces of $B^\lin$, with $\oB$ simply connected. So they are integral affine smooth manifolds as well.
    \item The fibers of the vertical rows are the presymplectic leaves of the pullback Dirac structures.
\end{enumerate}
\end{proposition}

\begin{remark}$\quad$
    \begin{enumerate}[(a)]
        \item The spaces in the diagram depend on the choice of proper integration.
        \item We will see in Section \ref{sec:The orbifold fundamental group} that the bottom row of the diagram consists of orbifold covering spaces of $B$, viewed as a classical orbifold. In particular, $\oB$ is the orbifold universal covering space of $B$, so $B$ is a good orbifold. 
    \end{enumerate}
\end{remark}

In the sequel we will explain the spaces in the diagram and how this proposition follows from the results of \cite[Appendix B]{CFM-II}. We will make use of the tautological bundle 
\[ \htt\to \hM, \quad \htt_{(x, \tt)}:=\tt. \]
Notice that this provides a concrete description of the normal bundle to the foliation $\hF$ since we have a canonical identification
\[ \nu(\hF)\cong \htt^*.\]
We recall that this bundle carries an integral affine structure $\Lambda\subset \htt^*$ induced from $(\hG,\hOmega)$, the dual of the lattice $\ker(\exp:\htt\to \hG)$. We will denote by 
\[ \GL_\Lambda(\,\htt^*)\tto\hM, \]
the transitive groupoid whose arrows are the invertible integral affine transformations between the fibers of $\htt^*$.  
\subsubsection{The space $\hM^\lin$}
\label{sec:hMlin}
The bundle $\htt^*$ carries a canonical flat connection, uniquely determined by the condition that the local sections of $\Lambda$ are flat. The corresponding parallel transport gives rise to the \textit{linear holonomy action} of $\Pi_1(\hM)$ on $\htt^*$ denoted
\begin{equation}\label{lin-action-transversal} 
h^{\lin}: \Pi_1(\hM) \to \GL_\Lambda(\,\htt^*\,).
\end{equation}
It is proved in \cite[Appendix B]{CFM-II} that its image is an embedded subgroupoid
\[ \Pi_1^{\lin}(\hM)\subset \GL_\Lambda(\,\htt^*\,).\]
Since $\Pi_1(\hM)$ and $\Pi^{\lin}_1(\hM)$ are transitive, these groupoids are encoded in their isotropy groups at the base point $\hx_0$, acting on
the s-fiber $\s^{-1}(\hx_0)$. For $\Pi_1(\hM)$ we obtain the fundamental group at $\hx_0$ and the universal cover of $\hM$
\[ \pi_1(\hM):=\pi_1(\hM, \hx_0), \quad \widetilde{\hM}:= \Pi_1(\hM)(\hx_0, -).\]
For $\Pi_1^{\lin}(\hM)$, notice that (\ref{lin-action-transversal}) induces at $\hx_0$ a group homomorphism
\begin{equation}\label{lin-action-transversal2} 
h^{\lin}_{\hx_0}: \pi_1(\hM) \to \GL_{\Lambda_0}(\tt^*_{0}).
\end{equation}
The kernel and image of this homomorphism will be denoted 
\[ K^{\lin}=\ker(h^{\lin}_{\hx_0})\subset \pi_1(\hM), \quad \Gamma^{\lin}=\im(h^{\lin}_{\hx_0})\subset \GL_{\Lambda_0}(\tt^*_{0}).\]
The group $\Gamma^\lin$ is precisely the isotropy group of $\Pi_1^{\lin}(\hM)$ at $\hx_0$, while its $s$-fiber at $\hx_0$ can be canonically identified with 
\begin{equation}
\label{eq:hM-linear-covering} 
\hM^{\lin}=\widetilde{\hM}/K^{\lin},
\end{equation}
i.e., the covering space of $\hM$ with $\Gamma^{\lin}$ as group of deck transformations. We call $\hM^{\lin}$ the \emph{linear holonomy cover} of $\hM$.

\subsubsection{The space $B^\lin$} 
The presymplectic foliation $\hF$ of the Dirac structure $\hL$ is a proper foliation $\hM$. Hence, its holonomy and linear holonomy groupoids coincide. Moreover, the Bott connection of $\hF$ coincides with the restriction of the flat connection on $\htt^*$, and one has a groupoid embedding (see \cite[Appendix B]{CFM-II})
\[ \Hol(\hF)\hookrightarrow \GL_\Lambda(\,\htt^*\,), \]
whose image is contained in $\Pi^{\lin}_1(\hM)$. 
The holonomy and the monodromy groupoids of the foliation $\hF$ then fit into a diagram
\[ 
\xymatrix@C=40pt{
\Mon(\hF)\ar[d]_-{i_*}  \ar@{->>}[r]^-{\hol^{\lin}} & \Hol(\hF) \quad\ar@{_{(}->}[d]  \ar@{^{(}->}[rd] & \\
\Pi_1(\hM) \ar@{->>}[r]_-{ h^{\lin}} &\Pi^{\lin}_1(\hM)  \ar@{^{(}->}[r] & \GL_\Lambda(\,\htt^*\,)
}\]
where $i_*:\Mon(\hF)\to\Pi_1(\hM)$ is the map induced by inclusion. Fixing the base point $\hx_0$, one obtains a similar diagram for the isotropy groups
\[ \xymatrix@C=40pt{
\pi_1(S_{\hx_0})  \ar[d]_-{i_*}  \ar@{->>}[r]^-{\hol^{\lin}} & \Hol(\hF)_{\hx_0}\quad\ar@{_{(}->}[d]  \ar@{^{(}->}[rd] & \\
\pi_1(\hM) \ar@{->>}[r]^-{ h^{\lin}} &\Gamma^{\lin} \ar@{^{(}->}[r] &  \GL_{\Lambda_0}(\tt^*_{0})
}\]

Pulling back the Dirac structure $\hL$ along the covering (\ref{eq:hM-linear-covering}) yields a presymplectic foliation $\hF^{\lin}$ so one has a covering of foliated spaces 
\[ (\hM^{\lin}, \hF^{\lin})\to (\hM, \hF).\]
The remarkable property of $\hF^\lin$ proved in \cite[Appendix B]{CFM-II} is that it is simple and its leaf space
\[ B^{\lin}:= \hM^{\lin}/\hF^{\lin}\]
is a smooth integral affine manifold. Moreover, the action of $\Gamma^{\lin}$ on $\hM^{\lin}$ descends to a proper action on 
$B^{\lin}$ by integral affine transformations and $\hM^{\lin}$ yields a Morita equivalence of Lie groupoids
\begin{equation}
    \label{eq:Morita:lin}
    \vcenter{
\xymatrix{
 \Hol(\hF) \ar@<0.25pc>[d] \ar@<-0.25pc>[d]  & \ar@(dl, ul) & \hM^{\lin}\ar[dll]\ar[drr] & \ar@(dr, ur)   & B^{\lin}\rtimes \Gamma^{\lin} \ar@<0.25pc>[d] \ar@<-0.25pc>[d]\\
\hM & & & & B^{\lin}  }}
\end{equation}
One obtains an isomorphism of classical integral affine orbifolds
\[ B\cong \hM/\hF\cong B^{\lin}/\Gamma^{\lin}. \]
At this point we have obtained the 2 right squares in the main diagram \eqref{eq:Weyl:main-diagram}.

\subsubsection{The spaces $\hM^\aff$ and $B^\aff$} 
The affine part of the diagram arises from the developing map associated to the transverse integral affine structure on $\hF$, denoted 
\begin{equation}
    \label{eq:dev:cocycle:groupoid}
    \dev: \Pi_1(\hM)\to \htt^*.
\end{equation}
This is the groupoid cocycle that integrates the projection on the normal bundle to $\hF$, $T\hM\to \htt^*$, interpreted as an algebroid 1-cocycle. Here we are using the representation given by the linear holonomy action \eqref{lin-action-transversal}. 

The developing map together with the linear holonomy action give rise to the {\bf affine holonomy action} where an element $[\gamma]\in \Pi_1(\hM)$ acts by the affine transformation
\[ h^\aff([\gamma]):\htt^*_{\gamma(0)}\to \htt^*_{\gamma(1)}, \quad v\mapsto h^\lin([\gamma])(v)+\dev([\gamma]). \]
As shown in \cite[Appendix B]{CFM-II}, the affine holonomy action gives an embedding of Lie groupoids 
\begin{equation}\label{aff-action-transversal} 
h^{\aff}: \Pi_1(\hM)\hookrightarrow \Aff_\Lambda(\,\htt^*\,).
\end{equation}
Its image $\Pi_1(\hM)^{\aff}\subset \Aff_\Lambda(\,\htt^*\,)$ fits into a commutative diagram 
\[ \xymatrix@C=40pt@R=10pt{
 &  \Pi_1^{\lin}(\hM)\ar@{^{(}->}[r] & \GL_\Lambda(\,\htt^*\,)\\
\Hol(\hF)\overset{i_*}{\longrightarrow}\Pi_1(\hM)\ar@{->>}[rd]_-{h^{\aff}} \ar@{->>}[ru]^-{h^{\lin}}  & & \\
 & \Pi_1^{\aff}(\hM)\ar@{^{(}->}[r]\ar[uu] & \Aff_\Lambda(\,\htt^*\,) \ar[uu]_-{}
}\]
Fixing the base point, ones obtains a corresponding diagram of isotropy groups, as in Section \ref{sec:hMlin}, giving rise to the groups $\Gamma^{\aff}$ and $K^{\aff}$ as in the following commutative diagram
\[ \xymatrix@C=40pt@R=10pt{
K^{\aff}\ar@{^{(}->}[rd]\ar@{_{(}->}[dd] & & \Gamma^{\lin}\ar@{^{(}->}[r] & \GL_{\Lambda_0}(\tt^*_{0}) \\
& \pi_1(\hM)\ar@{->>}[ru]^-{h^{\lin}} \ar@{->>}[rd]_-{h^{\aff}}  & & \\
K^{\lin}\ar@{^{(}->}[ru]& & \Gamma^{\aff}\ar@{^{(}->}[r] \ar[uu]& \Aff_{\Lambda_0}(\tt^*_{0})\ar[uu]_-{}
}\]
Also, we can identify the s-fiber of $\Pi(\hM)^{\aff}$ with the covering space of $\hM$ with $\Gamma^{\aff}$ as group of deck transformations, namely
\begin{equation}
\label{eq:hM-affine-covering} 
\hM^{\aff}=\widetilde{\hM}/K^{\aff},
\end{equation}
We call $\hM^{\aff}$ the \emph{affine holonomy cover} of $\hM$. Pulling back the Dirac structure $\hL$ along the covering map yields a presymplectic foliation $\hF^{\aff}$ so one has sequence of foliated coverings
\[ (\hM^{\aff}, \hF^{\aff}) \to (\hM^{\lin}, \hF^{\lin}) \to (\hM, \hF).\]

The foliation $\hF^\aff$ is simple and its leaf space
\[ B^{\aff}:= \hM^{\aff}/\hF^{\aff}\]
is a smooth integral affine manifold. Moreover, the action of $\Gamma^{\aff}$ on $\hM^{\aff}$ descends to a proper action on $B^{\aff}$ by integral affine transformations and $\hM^{\aff}$ yields another Morita equivalence of Lie groupoids
\begin{equation}
    \label{eq:Morita:aff}
\vcenter{\xymatrix{
 \Hol(\hF) \ar@<0.25pc>[d] \ar@<-0.25pc>[d]  & \ar@(dl, ul) & \hM^{\aff}\ar[dll]\ar[drr] & \ar@(dr, ur)   & B^{\aff}\rtimes \Gamma^{\aff} \ar@<0.25pc>[d] \ar@<-0.25pc>[d]\\
\hM & & & & B^{\aff}  }}
\end{equation}
where the left action is by translations via $h^\aff\circ i_*:\Hol(\hF)\to\Pi_1^\aff(\hM)$. This gives one further isomorphism of classical integral affine orbifolds
\[ B\cong \hM/\hF\cong B^{\lin}/\Gamma^{\lin}\cong B^{\aff}/\Gamma^{\aff}. \]
Altogether, we obtain the 3 right squares in the main diagram \eqref{eq:Weyl:main-diagram}. 

\subsubsection{The spaces $\hM^\orb$ and $\oB$} 
Finally, there is the left-hand square of the main diagram. For that, we let $\oB$ be the universal covering space of $B^\aff$ (which is also the universal covering space of $B^\lin$). Then we can pullback the submersion $\hM^\aff\to B^\aff$ along the covering projection, obtaining $\hM^\orb\to\oB$. Noticing that the square involving $\hM^{\aff}$ and $\hM^{\lin}$ is also a pull-back diagram, we have
\[ \hM^\orb=\oB\times_{B^{\aff}} \hM^{\aff}= \oB\times_{B^{\lin}} \hM^{\lin}.\]
With this definition the left-hand square of the main diagram is commutative, and the induced map $\hM^\orb\to\hM^\aff$ is a cover.

Furthermore, the principal actions of $\Pi_1(\hF)$ in either $ \hM^{\lin}$ or $\hM^{\aff}$, give a principal action of $\Pi_1(\hF)$ on $\hM^\orb$ whose quotient is $\oB$, so we have a Morita equivalence
\begin{equation}\label{eq:diagram:?:in:diagram}
\vcenter{\xymatrix{
 \Hol(\hF) \ar@<0.25pc>[d] \ar@<-0.25pc>[d]  & \ar@(dl, ul) & \hM^{\orb}\ar[dll]\ar[drr] & \ar@(dr, ur)   & ? \ar@<0.25pc>[d] \ar@<-0.25pc>[d]\\
\hM & & & & \oB }}
\end{equation}
The groupoid on the right-hand side, apriori, is the gauge groupoid. We will see later that it can be identified with an action groupoid $\oB\ltimes\pi_1^\orb(B)\tto \oB$, where $\pi_1^\orb(B)$ is the orbifold fundamental group of $B$ (viewed as a classical orbifold).

\subsubsection{The developing map} 
For later use, we observe that developing map \eqref{eq:dev:cocycle:groupoid} factors through the groupoid morphism $h^\aff:\Pi_1(\hM)\to\Pi_1^\aff(\hM)$. We will denote by the same letter the resulting groupoid cocycle
\[ \dev:\Pi_1^\aff(\hM)\to\htt^*. \]
The restriction of this cocycle to the source fiber of the base point $\hx_0$, gives the \emph{developing map based at $\hx_0$} denoted
\[ \dev_0:\hM^\aff\to\tt^*_{0}. \]
This map is constant along the leaves of $\hF^\aff$, and the resulting map
\[ \dev_0:B^\aff\to\tt^*_{0}, \]
is an integral affine local diffeomorphism. 

\begin{definition}
    A Poisson manifold $(M,\pi)$ of proper type is called \textbf{complete} if its developing map $\dev_0:B^\aff\to\tt^*_{0}$ is a global diffeomorphism.
\end{definition}

This definition is equivalent to say that the canonical flat connection on $B^\lin$ (or any of its covering spaces) is complete. In the $\s$-proper case, it is equivalent to say that the canonical flat connection on the normal bundle to the foliation on $\hM$ (or any of its covering spaces) is complete.

\subsection{The main diagram for the Hamiltonian local model} 
\label{ex:model main diagram}
As in the local model for a Poisson manifold of proper type, we let $\mu: (Q, \omega)\to \gg^*$ be a $G$-Hamiltonian space for a free and proper $G$-action, where $\mu$ has connected fibers, we let $V=\mu(Q)$ and we consider the quotient 
\[ M:= Q/G, \]
equipped with the quotient Poisson structure $\pi$. This Poisson manifold has proper integration $\cG=(Q\times_V Q)/G\tto M$ and fixing a maximal torus $T\subset G$ we claim that the main diagram in this case becomes
\begin{equation}
\label{eq:Weyl:main-diagram:local:model} 
\vcenter{
\xymatrix{
\oM=\hM^{\aff}=\hM^{\lin}\ar@{=}[d] \ar[r]&  \hM \ar@{=}[d]\ar[r]^{\res} &  M \ar@{=}[d] \\
\mu^{-1}(\tt^*)/Z_G(T) \ar[d]\ar[r]&    \mu^{-1}(\tt^*)/N(T) \ar[d]\ar[r]^{\res} &   Q/G \ar[d] & \\
\tt^*\cap V \ar@{=}[d] \ar[r]&    (\tt^*\cap V)/W \ar@{=}[d]\ar[r]^{\simeq} &  (\gg^*\cap V)/G  \ar@{=}[d] &  \\
\oB=B^{\aff}=B^{\lin} \ar[r]&  B= \hM/\hF  \ar@{=}[r] &  M/\cF_{\pi} 
}}
\end{equation}
Here $W= N(T)/T$ and $Z_G(T)\subset N(T)$ are the elements in the normalizer with trivial coadjoint action on $\tt^*$ and we also claim that
\[ \Gamma^\aff=\Gamma^{\lin}\cong N(T)/Z_G(T).\]

To show that these claims hold we fix as base point $\hx_0=(x_0,\tt)$ where $x_0\in M$ is any point that is $Q$-related to $0\in \gg^*$. We also use the identification $\hM= \mu^{-1}(\tt^*)/N(T)$ (see (\ref{eq:resolution:local:model})). We will first compute the representation \eqref{lin-action-transversal}
\[ h^\lin:=h^{\lin}_{\hx_0}: \pi_1(\hM) \to \GL(\tt^*). \]
The flat bundle $\htt^*$ whose holonomy was used to define this representation is easier to describe after pulling it along the cover 
\[ \overline{M}:= \mu^{-1}(\tt^*)/T \to \mu^{-1}(\tt^*)/N(T).\]
Indeed, the foliation $\hF$ now becomes the one associated to the submersion $\overline{M}\to \tt^*$ induced by $\mu$, so the normal bundle becomes the trivial bundle with fiber $\tt^*$ and the connection becomes the flat one. 
Conversely, $\hM$ is the quotient of the action of $W$ on $\overline{M}$ and this group also acts on $\tt^*$ via the standard representation
\begin{equation}\label{eq:W(T)-inside-GL}
W\to \GL(\tt^*).
\end{equation}
It follows that $\htt^*=(\overline{M}\times\tt^*)/W$ and that $h^{\lin}$ is the composition of (\ref{eq:W(T)-inside-GL}) with the boundary map 
\begin{equation}\label{ref:partial-hM-W} 
\partial: \pi_1(\hM)\to W
\end{equation}
associated to the homotopy sequence for $\overline{M}\to \hM$. Since the image of (\ref{eq:W(T)-inside-GL}) is precisely $\Gamma^{\lin}$ and its kernel is $Z_G(T)/T$, we deduce that 
\[ \Gamma^{\lin}\cong N(T)/Z_G(T).\]
This also implies that the holonomy cover is 
precisely the quotient of $\overline{M}$ modulo the kernel
$Z_G(T)/T$, i.e., 
\[ \hM^{\lin}= \mu^{-1}(\tt^*)/Z_G(T). \]
One obtains also $B^{\lin}= \tt^*$. Since this space is 1-connected, it follows that $\oB=B^\aff=B^\lin$ and then that $\hM^\orb=\hM^\aff=\hM^\lin$, which proves all the claims.
\medskip

For the sequel, we will need to restrict to an invariant open subset $U\subset M$. This will allow, for example, the reduction of fact-checking for a general Poisson manifold to the Hamiltonian local model. The main diagram is functorial for such restrictions.

\begin{proposition}
    \label{rk:functoriality-main-diagram}
    Let $(M,\pi)$ be a Poisson manifold with a choice of proper integration $(\cG,\Omega)\tto M$. Let $U\subset M$ be an invariant open subset and consider the restriction  $(\cG|_U,\Omega|_U)\tto U$. Fixing a base point $\hx_0=(x_0,\tt)$ with $x_0\in U$, the corresponding main diagrams \eqref{eq:Weyl:main-diagram} are related by 
\begin{equation}\label{eq:Weyl:main-diagram-open}
\vcenter{
\xymatrix@C=10pt@R=10pt{
\hU^\orb\ar[rr] \ar[dd] \ar[dr]&  &\hU^{\aff} \ar[rr] \ar@{-}[d] \ar[dr]&  & \hU^{\lin} \ar@{-}[d] \ar[rr] \ar[dr]&  & \hU \ar[rr]^{\res} \ar@{-}[d] \ar[dr]&  &  U  \ar@{-}[d] \ar[dr]\\
& 
\oM\ar[rr]\ar[dd] & \ar[d] & 
\hM^\aff\ar[rr]\ar[dd] & \ar[d] & 
\hM^\lin\ar[rr]\ar[dd] & \ar[d] &
\hM\ar[rr]^(.3){\res}\ar[dd] & \ar[d] &   M\ar[dd] \\
\oB_U \ar[dr] \ar@{-}[r] &\ar[r] & B^\aff_U \ar[dr] \ar@{-}[r] &\ar[r] & B^\lin_U \ar[dr] \ar@{-}[r] &\ar[r] & B_U \ar[dr] \ar@{=}[r] &\ar@{=}[r]  & B_U\ar[dr]\\
 & \oB \ar[rr] &   & B^\aff\ar[rr] &  & B^\lin\ar[rr] & & B\ar@{=}[rr] &  &B
}}
\end{equation}
where the diagonal maps are open embeddings, except possibly for the ones on the left face of the diagram. The map $\hU^\orb\to\hM^\orb$ is an open embedding if $B^\aff_U\to B^\aff$ is one.
\end{proposition}

\begin{proof}
    The construction of the resolution shows that we have a commutative diagram
    \[
    \xymatrix{
    \hU \ar[d]\ar[r]^{\res} &  U \ar[d] \\
    \hM  \ar[r]^\res &  M 
    }
    \]
    where the left vertical arrow is an open embedding. The diagonal maps in \eqref{eq:Weyl:main-diagram-open}, with the exception of the ones on the left face of the diagram, arise from the identification of $\hM^{\lin}$ as the s-fiber of $\hx_0$ of the image $\Pi^{\lin}_1(\hM)\subset \GL(\htt^*)$ of the linear holonomy map (\ref{lin-action-transversal}). The bottom diagonal  map on the left face of the diagram is the map induced on universal covering spaces from the embedding $B_U^\aff\to B^\aff$, while the first top diagonal map is the one induced between fiber products 
    \[\oU=\oB_U\times_{B_U^{\aff}}\hU^{\aff}\to 
    \oM=\oB\times_{B^{\aff}}\hM^{\aff}.
    \]
    This map is an embedding if the map $\oB_U\to \oB$ is an embedding.
\end{proof}

\newpage

\section{The Weyl Group}
\label{sec:Weyl:group}

In this section we analyse further the structure of the leaf space of a PMCT $(M,\pi)$. We saw before that its leaf space $M/\cF_\pi$, can be identified with the leaf space $\hM/\hF$ of its resolution, and we denoted it by $B$, so we have
\[
\xymatrix{
\hM \ar[d]\ar[r]^{\res} &  M \ar[d] \\
B \ar@{=}[r] &  B 
}
\]
The space $B$ inherits an integral affine orbifold structure for each choice of proper integration. Any such choice induces the same classical orbifold structure on $B$, which coincides with the orbifold structure induced by the proper groupoid $\Hol(\hM,\hF)$. In this section \emph{we will work exclusively with this classical orbifold structure}. We will still need to fix an integration to specify the integral affine structure on $B$.

\subsection{The orbifold fundamental group}
\label{sec:The orbifold fundamental group}

The orbifold fundamental group(oid) will play a crucial role for us, so let us recall its definition. It is convenient to adopt a more general point of view and define the fundamental groupoid of a Lie groupoid $\cG$. For details we refer to \cite{BH99,MM05}. First, one defines a Haeflieger path on $\cG\tto X$ to be a sequence
\[ 
\eta=(\gamma_0,g_1,\gamma_1,\dots,g_k,\gamma_k)
\]
where $\gamma_i$ are paths in $X$, $\g_i\in\cG$, and
\[ \t(g_i)=\gamma_{i-1}(0),\quad \s(g_i)=\gamma_i(1). \]
One defines the source and target of $\eta$ by
\[ \s(\eta)=\gamma_k(0),\quad \t(\eta)=\gamma_0(1). 
\]
There is an obvious concatenation operation on the set $P(\cG)$ of all Haeflieger paths. Also, given two Haeflieger paths $\eta_1,\eta_2\in P(\cG)$ with the same end-points one defines a homotopy between them to be the equivalence relation $\sim$ generated by the following operations:
\begin{enumerate}[(i)]
    \item \emph{multiplication:} If $\gamma_i$ is a constant path, then
    \[ (\gamma_0,g_1,\gamma_1,\dots,g_i,\gamma_i,g_{i+1},\dots,g_k,\gamma_k)\sim(\gamma_0,g_1,\gamma_1,\dots,g_ig_{i+1},\dots,g_k,\gamma_k);\]
    \item \emph{concatenation:} If $g_i=1_{\gamma_i(1)}$, then
    \[ (\gamma_0,g_1,\gamma_1,\dots,\gamma_{i-1},g_i,\gamma_i,\dots,g_k,\gamma_k)\sim (\gamma_0,g_1,\gamma_1,\dots,\gamma_{i-1}\gamma_i,\dots,g_k,\gamma_k);\]
    \item \emph{deformation:} If there is a continuous family of Haeflieger paths $\eta_\varepsilon$ with fix end-point $\s(\eta_\varepsilon)=x_0$ and $\t(\eta_\varepsilon)=x_1$, then $\eta_0\sim \eta_1$.
\end{enumerate}
The fundamental groupoid of $\cG$ has arrows the equivalence classes of Haeflieger paths and will be denoted
\[
\Pi_1(\cG)\tto X.
\]
Multiplication is induced by concatenation. It is proved in \cite{MM05} that this is a Lie groupoid. \textbf{The fundamental group of $\cG$ based at $x$} is, by definition, the resulting isotropy group at $x$, 
\[ \pi_1(\cG,x):=\Pi_1(\cG)_x.\]
Since $\Pi_1(\cG)$ is transitive, the fundamental groups of $\cG$ based at different points are isomorphic. 

Notice that we have natural Lie groupoid morphisms
\begin{equation}\label{eq:fund-gpd-canonical-maps} 
\cG\to \Pi_1(\cG), \qquad \Pi_1(X)\to \Pi_1(\cG). 
\end{equation}
The second one is an instance of the fact that a groupoid morphism $\phi:\cH\to\cG$ induces a morphism between the fundamental groupoids $\phi_*:\Pi_1(\cH)\to \Pi_1(\cG)$. The later is actually an isomorphism whenever $\phi$ is a weak equivalence, so Morita equivalent groupoids have isomorphic fundamental groups and groupoids.

Let $B$ be an orbifold, with a choice of an orbifold atlas $\cB\tto X$. We will fix a base point $x_0\in X$ and we will call 
\[ \pi_1^\orb(B,b_0):=\pi_1(\cB,x_0)\] 
the \textbf{orbifold fundamental group} of $B$ based at $b_0=p(x_0)$. We have the map induced by the inclusion of the units
\begin{equation}\label{eq:p*-again} 
\pi_1(X,x_0) \to \pi_1^\orb(B,b_0). 
\end{equation} 
which can also be thought as being induced by the orbit space projection $p: X\to B$.

We recall also that any (effective) orbifold $B$ has a universal covering space 
\[ \oB\]
which is itself an (effective) orbifold. If $B$ is presented by an orbifold atlas $\cB\tto X$ and we fix $x_0\in X$, then the model for $\oB$, based at $x_0$, is the orbifold atlas given by the action groupoid
\begin{equation}
\label{eq:universal:abstract}
\cB\ltimes \Pi_1(\cB)(x_0,-) \tto \Pi_1(\cB)(x_0,-), 
\end{equation}
together with the (right) action of $\pi_1^\orb(B,b_0)=\pi_1(\cB,x_0)$.

We will be interested in two more specific situations that we now describe.

\subsubsection{Orbifolds presented by a holonomy groupoid}
Assume $B$ has orbifold atlas the holonomy groupoid of a proper foliation 
\[\cB=\Hol(X,\cF)\tto X. \]
Note that in this case every arrow in $\cB$ is represented by a path, from which it follows that \eqref{eq:p*-again} is a surjective morphism. This allows us to reinterpret $\pi^\orb_1(B,x_0)$ as the quotient of $\pi_1(X,x_0)$ modulo the equivalence relation, denoted $\sim_{\cF}$ and called \textbf{$\cF$-homotopy}, induced by the kernel of (\ref{eq:p*-again}):
\begin{equation}\label{eq:F-homotopy} K_{x_0}=\big\{[\gamma]\in \pi_1(X,x_0)\,|\, \gamma\sim_{\cF} x_0\big\}.
\end{equation}
Explicitly, a loop is $\cF$-homotopic to the constant path $x_0$ if it is homotopic to a concatenation of loops $\gamma_1\cdots \gamma_k$, each of the form
\[\gamma_i=\alpha_i \beta_i \alpha_i^{-1},\]
where $\beta_i$ is a leafwise loop with trivial holonomy, based at some $y_i$, and $\alpha_i$ is an arbitrary path connecting $y_i$ to $x_0$ (note that the $\beta_i$'s will be in general contained in different leaves).

In this case, the construction \eqref{eq:universal:abstract} of the universal covering space yields a connected space
\[ X^\orb:=\Hol(X,\cF)(x_0,-)\longrightarrow X\]
This is precisely the covering space of $X$ with group of deck of transformations $\pi_1^{\orb}(B)$,  i.e., the quotient of the universal covering space of $X$ modulo the kernel of (\ref{eq:p*-again}). 
The construction of the orbifold universal cover yields an orbifold atlas consisting of the holonomy groupoid of the pull-back foliation $\cF^{\orb}$ along $X^{\orb}\to X$. In particular, the underlying topological space is $\oB= X^{\orb}/\cF^{\orb}$ with atlas
\[ \Hol(X^{\orb}, \cF^{\orb})\tto X^{\orb}
\overset{p}{\longrightarrow} X^{\orb}/\cF^{\orb}.\]

For the next lemma recall (Corollary \ref{corolarry-normal}) that principal leaves of a proper foliation are the leaves with trivial holonomy.

\begin{lemma}
\label{lem: orbifold fundamental group isomorhism} 
Let $(X,\cF)$ be a proper foliation and on the leaf space $p:X\to B$ consider the orbifold defined by the atlas $\Hol(X,\cF)\tto X$. If all principal leaves are 1-connected then one has an isomorphism
\[ \pi_1^\orb(B,p(x_0))\simeq \pi_1(X,x_0). \]
\end{lemma}


\begin{proof}
Let $K_{x_0}$ be the kernel of the morphism \eqref{eq:p*-again}. An element $[\gamma]\in K_{x_0}$ is a concatenation of loops 
of the form 
\[l=\sigma \beta \sigma^{-1},\]
where $\beta$ is a loop with trivial holonomy in the leaf $L$ containing the base point $x_0$ and $\sigma$ connects $x_0$ with the starting point of $\beta$. Let $\nu(L)\subset X$ be a tubular neighborhood of $L$. Since $\beta$ has trivial holonomy, pulling back $(\nu(L),\cF)$ along $\beta$ yields a foliation which is trivial near the zero section, so we can homotope $\beta$ to any close enough leaf. Since the principal locus is dense, we can assume that $\beta$ is in a principal leaf. By assumption, principal leaves are 1-connected so $\beta$ is contractible in its leaf. Therefore we conclude that $l$ is contractible and thus that $K_{x_0}$ is the trivial group. 
\end{proof}

\subsubsection{Good orbifolds}
\label{sec:good:orbifold}
Assume now that $B$ has orbifold atlas the action groupoid
\[ \cB= \Bgood\rtimes \Gamma\tto \Bgood \]
associated with a proper and effective action of a discrete group $\Gamma$ on a manifold $\Bgood$. The orbifold fundamental group has the explicit model $\pi_1(\Bgood\ltimes \Gamma, b'_0)$ defined as the collection of pairs $(g, \sigma)$ with $g\in \Gamma$ and where $\sigma$ is a path homotopy class in $\Bgood$ from $b'_0$ to $b'_0g $, with the composition
\begin{equation}\label{eq:fund-group-discrete-actions} 
(g_1, \sigma_1)(g_2, \sigma_2)= (g_1g_2, g_2^*(\sigma_1)\sigma_2).
\end{equation}
where $g_2^*(\sigma_1)$ is the $g_2$-translate of $\sigma_1$.
It follows right away that it fits into a short exact sequence
\begin{equation}\label{eq:pi1-orb-B-gamma} 
\xymatrix{
1\ar[r] & \pi_1(\Bgood, b'_0)\ar[r] & \pi_{1}^{\orb}(B,b_0)\ar[r] & \Gamma \ar[r]& 1
}
\end{equation}
In this case, the action groupoid \eqref{eq:universal:abstract} presenting the orbifold universal cover has base space a smooth manifold, namely the universal covering space of $\Bgood$ based at $b'_0$:
\[ \widetilde{B}^{\orb}= \widetilde{\Bgood}, \]
which we think of as consisting of path homotopy classes of paths in $\Bgood$ starting at $b'_0$.  
The resulting (right) action of $\pi_1^\orb(B,b_0)$ on this space can be described explicitly as:
\begin{equation}\label{eq:right action orbifold} \tau\cdot (g, \sigma)= g^*(\tau)  \sigma
\end{equation}
This action is proper and effective, and this can also be checked directly using (\ref{eq:pi1-orb-B-gamma}). However, the action is free if and only if the original action of $\Gamma$ on $\Bgood$ is free.

\subsubsection{Orbifold fundamental group of the leaf space of PMCTs}
We can now apply the previous discussion to our setting of PMCTs. Recall we are viewing the leaf space $B=\hM/\hF=M/\cF_\pi$ as a classical orbifold. 
The next result clarifies the missing piece in diagram (\ref{eq:diagram:?:in:diagram}).

\begin{proposition}
\label{prop:Morita:oB}
    Let $(M,\pi)$ be a Poisson manifold with a proper integration $(\cG, \Omega)\tto (M, \pi)$ and fix a base point $\hx_0= (x_0, \tt_0)\in \hM$. In the main diagram \eqref{eq:Weyl:main-diagram} $\oB$ is the orbifold universal covering of the leaf space $B$. In particular, $\pi_1^\orb(B)$ acts properly and effectively on $\oB$ and there is a Morita equivalence
    \begin{equation}
    \label{eq:Morita:orb}
\vcenter{\xymatrix{
 \Hol(\hM, \hF) \ar@<0.25pc>[d] \ar@<-0.25pc>[d]  & \ar@(dl, ul) & \hM^{\orb}\ar[dll]\ar[drr] & \ar@(dr, ur) &  \oB\rtimes\pi_1^{\orb}(B) \ar@<0.25pc>[d] \ar@<-0.25pc>[d]\\
\hM & & & &\, \oB }}
\end{equation}
\end{proposition}

Note also that the kernel $K=K_{x_0}$ of the homomorphism \ref{eq:p*-again} coincides with the normal subgroup of $\pi_1(\hM,\hx_0)$ associated with the covering space $\oM\to\hM$.

\subsection{The Weyl group}
\label{sec:Weyl:group:PMCT}
Given a Poisson manifold of proper type  $(M, \pi)$ and chosing $\hx=(x,\tt_x)\in\hM$, the kernel of the map
\[ \res_*: \pi_1(\hM,\hx)\to\pi_1(M,x)\]
will be denoted $\hW_{\hx}$. We recall that we view the leaf space $B$ of $(M,\pi)$ as a classical orbifold presented by $\Hol(\hF)$.

\begin{definition}
The Weyl group $\cW_{\hx}$ of $(M,\pi)$ at $\hx$  is the image of $\hW_{\hx}$ under the map 
\begin{equation}\label{eq:p*-again2} 
\pi_1(\hM,\hx) \to \pi_1^\orb(B,b).
\end{equation} 
\end{definition}


By passing from $\hM$ to $B$ we are able to make use of the action of $\pi_1^{\orb}(B,b)$ on the orbifold universal covering space $\oB$. In particular, we obtain a proper and effective action of $\cW_\hx$ on $\oB$. Our ultimate aim is to show that $\cW_{\hx}$ is a Coxeter group acting as a reflection group on $\oB$, in the sense of \cite{Davis83}.



Note that the Weyl group $\cW_\hx$ is a normal subgroup of $\pi_1^\orb(B,b)$ because (\ref{eq:p*-again2}) is surjective. The following commutative diagram explains the relationship between the various groups that appear in the definition: 
\begin{equation}
    \label{eq:splitting-weil-pi1M}
    \vcenter{
    \xymatrix{
1\ar[r] & \hW_{\hx}\ar[r]\ar@{-->>}[d] & \pi_1(\hM,\hx)\ar[r]_{\res_*}\ar@{->>}[d]^-{p_*} & \pi_1(M,x)\ar[r]\ar@/_1.2pc/[l] & 1\\
 & \cW_{\hx}\,  \ar@{^{(}->}[r] & \pi_1^{\orb}(B,b) & 
}}
\end{equation}
The top row splits canonically: we have an isomorphism $\res: \hM^{\reg}\overset{\sim}{\longrightarrow} M^{\reg}$ and $M^{\reg}\hookrightarrow M$
induces an isomorphism in $\pi_1$. Hence, we have
\[ \pi_1(\hM,\hx)\cong \pi_1(M,x)\ltimes \hW_{\hx}. \]
We will see later that the orbifold fundamental group of $B$ is also a semi-direct product with kernel $\cW_{\hx}$.

\begin{remark}[Base-points]\label{rem:base-point change}
All the covering spaces appearing in the main diagram are constructed using paths based at some $\hx=(x,\tt_x)$. Note that concatenation by a path $\sigma$ gives isomorphisms between the spaces in the main diagram based at $\sigma(0)=\hx$ and at $\sigma(1)=\hy$. Conjugation by $\sigma$ gives isomorphisms $\phi$ between the (orbifold) fundamental groups at $\hx$ and at $\hy$, that carry the groups $\cW_{\hx}$ and $\hW_{\hx}$ to $\cW_{\hy}$ and $\hW_{\hy}$. The isomorphisms between the spaces are $\phi$-equivariant. In other words, the path $\sigma$ provides an equivariant isomorphism between the main diagram based at $\hx$ and at $\hy$.

With that in mind, from now on we fix a base point $\hx_0= (x_0, \tt_0)$ with $x_0\in M^{\reg}$. We endow all the spaces with the resulting base points, denoted 
\[ b_0= p(\hx_0)\in B,\ x_0\in M,\ \widetilde{b}_0\in \oB, \textrm{etc}.\]
and we denote
\[ \hW=\hW_{\hx_0}, \quad \cW= \cW_{\hx_0}. \]
Furthermore, when we use fundamental groups, we omit base-points from the notation unless necessary. 
\end{remark}

\begin{example}
\label{ex:Weyl-group:gg*}
For $M= \gg^*$ the Lie algebra of a compact connected Lie group $G$, we claim that
\[ \hW\cong \cW\cong\pi_1^\orb(B) \cong W, \]
where $W=N(T)/T$ is the classical Weyl group of $G$ relative to the maximal torus $T\subset G$.

To prove this claim, recall from Example \ref{ex:linear:case} that the resolution map for $\gg^*$ is
\[ \res: G/T\times_{W} \tt^*\to \gg^*, \quad (gT, \xi)\mapsto \Ad_g(\xi). \]
Since the image is contractible, by  (\ref{eq:splitting-weil-pi1M}) we have 
\[ \hW\cong \pi_1(G/T\times_{W} \tt^*),\quad \cW\cong \pi_1^\orb(B), \] 
where $\pi_1(G/T\times_{W} \tt^*)\cong W$. 
Since the principal leaves are diffeomorphic to $G/T$, they are simply connected and Lemma \ref{lem: orbifold fundamental group isomorhism} implies that $\pi_1(G/T\times_{W} \tt^*)\cong \pi_1^\orb(B)$, and so the claim follows.

Observe that the Weyl group $\cW$ in this example appears as the fundamental group of the fiber of the resolution $\res^{-1}(0)=\cT(\gg)$:
\begin{equation}
\label{eq:pi1:tori-Grassmannian} 
\pi_1\left(\mathcal{T}(\gg)\right) \cong  N(T)/T= W.
\end{equation}
This fact will be generalized appropriately in Proposition
\ref{pro:non-trivial Weyl group elements}.

For a more concrete example let 
$M= \mathfrak{su}(2)^*$, which we can 
identify with $\R^3$. The coadjoint action of $\SU(2)$ is given, via the covering morphism $\SU(2)\to \SO(3)$, by  rotations. We let $\tt\subset \mathfrak{su}(2)$ be the span of the infinitesimal generator of the rotations around an axis and set $W$ to be the Weyl group of $\SU(2)$ relative to $T$.  We have the diffeomorphism
\begin{equation}\label{eq:models so(3)}\widehat{\mathfrak{su}(2)^*}\cong \left(\SU(2)/T\times \tt\right)/W\cong (S^2\times \R)/\Z_2,
\end{equation}
where $S^2\subset \R^3$ is the unit sphere and the non-trivial element of $\Z_2$ act as 
\begin{equation}\label{eq:action in so(3)}(\theta,z)\mapsto (-\theta,-z).
\end{equation} 
Under the isomorphism \eqref{eq:models so(3)} the resolution map becomes
\[\res:(S^2\times \R)/\Z^2\to \R^3,\quad [\theta,z]\mapsto z\theta, \]
its fiber over the origin is
\[(S^2\times \{0\})/\Z_2\cong \mathbb{RP}^2. \]
The Weyl group in this case is $\hW\cong \cW\cong\pi_1^\orb(B)\cong
\pi_1(\mathbb{RP}^2)\cong
\Z_2$.
\end{example}

\begin{example}
\label{ex:Weyl-group:G-Dirac}
Consider now the Cartan-Dirac structure $L_G$ on a compact simply connected Lie group $G$ determined by a fixed $\Ad$-invariant inner product on $\gg$. Recall from Example \ref{ex:Cartan:Dirac:baby} that, fixing a maximal torus $T\subset G$, the resolution can be identified with
\[ \widehat{G}=G/T\times_W T. \]
We claim that this twisted Dirac structure has Weyl group
\[\cW\cong\hW\cong \pi_1^\orb(B)\cong \Lambda_G\rtimes W,\]
where $W=N(T)/T$ acts on $\Lambda_G=\ker(\exp:\tt\to T)$ via the standard representation $W\to \GL(\tt)$. 
Since, by assumption, the base of the resolution map
\[
\res:G/T\times_W T\to G
\]
is simply connected, by (\ref{eq:splitting-weil-pi1M}) we have 
\[ \hW\cong \pi_1(G/T\times_W T),\quad \cW\cong \pi_1^\orb(B). \] 
Since the principal leaves are diffeomorphic to $G/T$, they are simply connected and Lemma \ref{lem: orbifold fundamental group isomorhism} implies that $\pi_1(G/T\times_W T)\cong \pi_1^\orb(B)$.

We are left to prove that
\[ \pi_1(G/T\times_W T)\simeq \Lambda_G\rtimes W. \]
This follows by observing that we have an isomorphism
\[ 
(G/T\times \tt)/(\Lambda_G\rtimes W)\cong G/T\times_W T, \quad
[(gT,\xi)]\mapsto [(gT,\exp(\xi))],
\]
where on the left side, $\Lambda_G$ acts trivially on $G/T$ and by translations on $\tt$.

For a concrete example, the case $G=\SU(2)$. To obtain an explicit expression of the resolution map we let $T$ be the diagonal matrices in $\SU(2)$ and we regard $\SU(2)$ as the unit sphere in $\C^2$ with coordinates $z_1,z_2$. Then $T$ is identified with the unit circle $\S^1$ with coordinate $u$. The resolution map becomes 
\[\res:\left(\SU(2)\times \S^1\right)/N(\S^1)\to \SU(2),\quad [z_1,z_2,u]\mapsto (uz_1\bar{z}_1+\bar{u}z_2\bar{z}_2,(\bar{u}-u)z_1z_2).\]
Its fiber at $\pm I$ is 
\[ \left(\SU(2)\times \{\pm 1\}\right)/N(\S^1)\cong \S^2/{\Z_2}\cong \mathbb{RP}^2. \]
Note that the Weyl group in this case is
\[\cW\cong\hW\cong \pi_1^\orb(B)\cong \Z\rtimes\Z_2,\]
where $-1\in\Z_2$ acts on $\Z$ by multiplication.

Observe that we can also compute the fundamental group of
\[ \left(\SU(2)\times \S^1\right)/N(\S^1)\cong \S^2\times_{\Z_2} \S^1,\]
using the Seifert-Van Kampen theorem, and this recovers the well-known isomorphism 
\[\Z\rtimes \Z_2\cong \Z_2*\Z_2.\]

\end{example}

\begin{example}
\label{ex:Weyl-group:G-Dirac-minus}
    Let us consider the twisted Dirac structure $(M,L)$ obtained by removing from the Cartan-Dirac structure on $\SU(2)$ the point $-I$. Then it follows from the previous example that the resolution is 
    \[ \res:\hM=\S^2\times_{\Z_2} (\S^1\backslash \{-1\})\to M=\SU(2)\backslash \{-I\}.\]
    Since the base of the resolution and the principal fiber -- still $\S^2$ -- are simply connected, we have as in the previous examples, that
    \[ \cW\cong\hW\cong \pi_1^\orb(B)\cong \pi_1(\hM)\cong\Z_2. \]

    Note that this example does not fit into classical Lie theory, so one does not have a classical notion of Weyl group.
\end{example}

\subsection{Abstract reflections}
\label{sec:abstract:reflections}
Our next aim is to show that $\cW$ is generated by certain ``reflections'', which will be introduced below abstractly as certain elements of order two. In the next section we will show that they can be realized as actual geometric reflections on $\oB$. 

The most obvious way to exhibit elements of the Weyl group is by focusing on the fundamental groups of the fibers of $\res$. 

\begin{proposition}\label{pro:non-trivial Weyl group elements}
Fix a point  $\hx= (x,\tt_x)\in \hM$ and let $i:\res^{-1}(x)\hookrightarrow \hM$ be the inclusion. Then there is a canonical diffeomorphism 
\[ \res^{-1}(x)\cong G/N(T),\] 
where $G=\cG_x^0$ and $T\subset G$ is the torus corresponding to $\tt_x$. The composition
\[
\xymatrix{\pi_1(\res^{-1}(x),\hx)\ar[r]^-{i_*}& \pi_1(\hM,\hx) \ar[r]^{p_*}& \pi_1^{\orb}(B,b)}\]
gives a canonical injective homomorphism 
\[ W\hookrightarrow \cW_{\hx}\, ,\]
where $W=N(T)/T$ is the Weyl group.
\end{proposition}

\begin{proof} For simplicity of notations let $\tt= \tt_x$, $\gg= \gg_x$, $G= \G_x^{0}$, and $T= T_x$. By definition, $\res^{-1}(x) =\mathcal{T}(\gg)$ and the canonical isomorphism in the statement is
\[ G/N(T)\diffto \cT(\gg),\quad gN(T)\mapsto \Ad_g(\tt). \]

We claim that the map $i_*$ 
is injective. For that we use the holonomy representation associated to the flat tautological bundle $\mathscr{T}$ over $\hM$,
\[ h^{\lin}: \pi_1(\hM, \hx)\to \GL(\tt).\]
More precisely, we will show that $h^{\lin}\circ i_*$ is injective. This composition is, of course, the holonomy map of the flat bundle $\mathscr{T}|_{\res^{-1}(x)}$. In turn, 
$\mathscr{T}|_{\res^{-1}(x)}$ is identified with  the tautological bundle over $G/N(T)$, call it $E$. This flat bundle can be trivialized by pulling it back to the universal cover  
\[ G/T\to G/N(T).\]
We see that, conversely, $E\to G/N(T)$ can be described, as a flat bundle, as the quotient of the trivial bundle $G/T\times \tt$ modulo the diagonal action of $W$, where the action on the second factor is via the standard representation $W\to \GL(\tt)$. 
This is similar to the computation of the linear holonomy in Section \ref{ex:model main diagram}.
The holonomy of $E\to G/N(T)$ is the the standard representation of $W$.
Hence $i_*$ is injective. 

On the other hand, recall that $h^{\lin}$ factors through $p_*$:
\[ \xymatrix{
\pi_1(\hM,\hx) \ar[r]^-{h^{\lin}} \ar[d]_{p_*} & \Gamma^{\lin}\subset \GL(\tt) \\
\pi_{1}^{\orb}(B,b) \ar[ru] & 
} \]
Therefore, the previous argument implies also that $p_* \circ i_*$ is injective. 
\end{proof}

\begin{example}\label{ex:W-for-Hamiltonian-model} When $M= Q/\gg^*$ is the Hamiltonian local model, if one uses points $Q$-related to $0\in \gg^*$, then the previous proposition gives a copy of the Weyl group $W^0:=N(T)\cap G^0/T$ inside the Weyl group of $M$. Later, in Proposition \ref{prop:Morita:Weyl} (see also Remark \ref{rem:connecteness}), we will show that Weyl groups are Morita invariant, so one finds $\cW\cong W^0$.
\end{example}

\medskip

When $\hx\in\hM$ is a subregular point one has $\gg_x\cong \mathfrak{su}(2)\oplus \zz_x$ with $\zz_x$ abelian,
 the fiber over $x$ is
\[ \res^{-1}(x)=\mathcal{T}(\su(2))\cong \mathbb{RP}^2,\]
so has fundamental group $\Z_2$. The last proposition then shows that at any such subregular point one obtains a non-trivial involution
\[\tau_{\hx}\in \cW_{\hx}.\]
This will be called the \textbf{subregular reflection} at $\hx$. 


\begin{definition}\label{def:abstract-reflections}
An  \textbf{abstract reflection} in $\cW$ is any element obtained from a subregular reflection $\tau_{\hx}\in \cW_{\hx}$ by transporting it along a path from $\hx$ to $\hx_0$.
\end{definition}

Note that each subregular reflection $\tau_{\hx}$ does not give rise to a unique abstract reflection in $\cW$. Rather, it produces a canonical conjugacy class of involutions which only depends on the subregular strata $\Sigma$ through $\hx$. 

Similarly, one shall make use of abstract reflections in $\hW$, induced by the order two elements $\widehat{\tau_{\hx}}\in \hW_{\hx}$ in the image of $i_*:\pi_1(\res^{-1}(x),\hx)\to \pi_1(\hM,\hx)$.

\begin{example}
\label{ex:local:model:Weyl:continued}
In the case $M= \gg^*$, fixing a maximal torus $\tt\subset\gg$, we saw in Example \ref{ex:Weyl-group:gg*} that the Weyl group $\cW$ is isomorphic to the classical Weyl group $W$ of $G$ relative to $\tt$. The abstract reflections are exactly the standard reflections associated to the root hyperplanes in $\tt$. Indeed, if $\xi$ is a subregular point on a root hyperplane $\mathcal{H}$, then $\res^{-1}(\xi)$  embedds canonically in $\res^{-1}(0)\cong G/N(T)$ as the set of maximal tori of $\gg$ contained in $\gg_\xi$. The generator of the fundamental group of this embedded $\mathbb{RP}^2$ is the standard reflection on $\mathcal{H}$. The straight path connecting  $\widehat{\xi}:=(\xi,\tt)$ with $(0,\tt)$ conjugates  $\tau_{\widehat{\xi}}$ to the standard reflection on $\cH$.

The fact that the standard reflections generate the classical Weyl group generalizes as follows.
\end{example}

\begin{proposition}\label{prop:generation-by-reflections}
For any Poisson manifold of proper type $(M,\pi)$ the Weyl group $\cW$ is generated by the abstract reflections.
\end{proposition}

The proposition follows from the following lemma.

\begin{lemma}\label{lemma:prop:generation-by-reflections}
Any element in $\pi_1(\hM)$ can be written as a product of abstract reflections in $\hW$ and an element in the image of the splitting in (\ref{eq:splitting-weil-pi1M}). 
\end{lemma}

\begin{proof}[Proof of the Lemma]

Let $[\gamma]\in \pi_1(\hM)$.
We can assume that $\gamma$ is a smooth loop which, as in the proof of Lemma \ref{lem:stratif-lemma-4}, intersects at most the codimension one strata of $\cS^ \textrm{rk}(\widehat{M},\widehat{\cF})=\res^{-1}(\cSi(M,\pi))$ and any such intersection is transverse. By Corollary \ref{cor:codimension-rank-stratification} these are subregular points.



Next, we claim that we can eliminate successively the finitely many intersection points one by one: 
for the first subregular intersection point we construct an abstract reflection $\widehat{\tau}$ in $\hW$ such that $[\gamma]=\widehat{\tau} \cdot[\gamma']$ where $\gamma'$ has one less subregular intersection point than $\gamma$.


Let $t_1\in [0,1]$ be the fist value such that  $\gamma(t_1)=\hx$ is a subregular point on the stratum $\hSi\in  \cS^ \textrm{rk}(\widehat{M},\widehat{\cF})$. We can assume that the reflection $\widehat{\tau}_{\hx}$ is represented by an embedded loop $\kappa$
and we decompose $\gamma=\sigma_1\sigma_2$ according to the subdivision $[0,1]=[0,t_1]\cup [t_1,1]$. Then we have the factorization
\[[\gamma]=[\sigma_1\kappa\kappa\sigma_2]=[\sigma_1\kappa\sigma_1^{-1}] [\sigma_1 \kappa\sigma_2],\] 
where, by construction, $[\sigma_1\kappa\sigma_1^{-1}]$ is an abstract reflection. We claim  that $\sigma_1 \kappa\sigma_2$ is homotopic to a loop which intersects 
the codimension one strata $\hSi$ in one point less than $\gamma$.

Since $\widehat{\tau}_{\hx}= [\kappa]$ is nontrivial, the restriction of the normal bundle of $\hSi$ in $\hM$ to $\kappa$ is the non-trivial line bundle over $\mathbb{RP}^2$. It further follows that there is an embedded Moebius band   
\[\Phi:\R\times (-1,1)/(\theta,s)\sim (\theta+2\pi,-s)\to \hM,\]
\[ \kappa(t)=\Phi([2\pi t,0]),\,t\in [0,1],\quad  \hx=\Phi([0,0]),\]
so that $\hSi\cap \Phi= \kappa$. 
We can arrange that $\gamma\cap \Phi$
be the normal fiber over $\hx$ and that points in the Moebius band other than the zero section are regular points. 
Then the normal fibers  can be used to perform a homotopy supported in the Moebius band which moves the curve 
$\sigma_1\kappa\sigma_2$ away form the zero section, i.e., away from $\hSi$ (see Figure \ref{fig:homotopy}).

\begin{figure}
\centering
\includegraphics[scale=1]{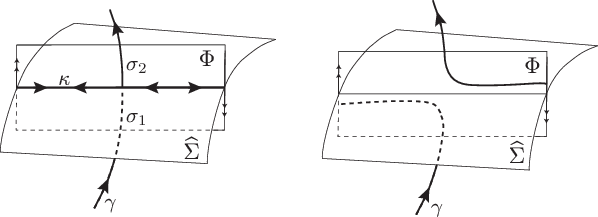}
\caption{Homotopy supported on Moebius band}
\label{fig:homotopy}
\end{figure}
\end{proof}


To close this section we will show that $\pi_1^\orb(B)$ is a $\cW$-semi-direct product. In other words, that 
there is an analogue for $\cW$ of the split exact sequence in the diagram \ref{eq:splitting-weil-pi1M}.

For that, note that for any Poisson manifold of proper type $(M,\pi)$, the regular part $(M^\reg,\pi)$ is also of proper type. The restriction of a proper integration $(\cG,\Omega)$ to   $M^\reg$ endows the leaf space $B^\reg$ with an integral affine orbifold  structure. For this structure, the inclusion of the regular part $M^\reg\subset M$
descends to an embedding of classical integral affine 
orbifolds (see Section \ref{sec:The case of orbifolds}) 
\[B^\reg\hookrightarrow B.\]
We also define the regular part of any of the orbifold covering spaces of $B$, denoted
\[ (B^\lin)^\reg,\,(B^\aff)^\reg, \, (\oB)^\reg,\]
as the pre-image of $B^\reg$ by the corresponding covering maps.

\begin{proposition}
\label{prop:split:orbi:fundamental:group}
The map in homotopy induced by the inclusion $B^{\reg}\subset B$ gives a splitting of the short exact sequence
\[
\xymatrix{1\ar[r] & \cW \ar[r] & \pi_1^\orb(B)
\ar[r] &
\pi_1^\orb(B^\reg)\ar[r] & 1}.\]
Furthermore, the image of this splitting consists of the $\pi_1^\orb(B)$-stabilizer of the connected component of $(\oB)^{\reg}$ containing the base-point. 
\end{proposition}

\begin{proof}
We complete the diagram \ref{eq:splitting-weil-pi1M}  with the orbifold fundamental group of the regular part $M^\reg$. The orbifold embedding $B^\reg\hookrightarrow B$ induces a map between the  orbifold fundamental groups that makes the outer square of the following diagram commutative 
\[
\xymatrix{
1\ar[r] & \hW\ar[d]^{p_*} \ar[r] &\pi_1(\hM)\ar[d]^{p_*} \ar[r]^{\res_*} & \pi_1(M^\reg) \ar[r]\ar[d]_{p_*} \ar@/_1.2pc/[l]_{\hi} & 1
\\
1\ar[r] & \cW \ar[r] & \pi_1^\orb(B)\ar@{-->}[r]^{\mathscr{r}_*} &\pi_1^\orb(B^\reg) \ar@/^1.2pc/[l]^{i}  
}
\]
We claim that there exists  dotted arrow $\mathscr{r}_*$ making the inner square commutative. By diagram chasing, using that vertical arrows are surjective, this implies the exactness of the sequence
\[ \xymatrix{
1\ar[r] & \cW \ar[r] & \pi_1^\orb(B)\ar[r] &\pi_1^\orb(B^\reg)  \ar[r] & 1
}
\]
Next, again by a diagram chasing, it follows that the existence of $\mathscr{r}_*$ is equivalent to the fact that $\res_*(K)\subset \Ker (p_*)$, where $K$ is the kernel of the middle vertical arrow (described in \eqref{eq:F-homotopy}). 
This will follow by showing that
\begin{equation}\label{eq: kernel splitting}
\hi\circ\res_*(K)\subset K.
\end{equation}
To prove this, we show that any $[\gamma]\in K$ can be factorized as
\begin{equation}
    \label{eq:factorization:K}
    [\gamma]=[\gamma_1]\cdots [\gamma_k],\quad \textrm{ with }\quad [\gamma_i]=w_i[\gamma'_i]w_i^{-1}, 
\end{equation}
where $\gamma_i'\subset \hM^\reg$, $[\gamma_i']\in K$, and $w_i\in\hW$. Using this factorization, we find that if we let $\gamma'=\gamma'_1\cdots \gamma'_k$, then 
\[ \hi\circ\res_*[\gamma]=\hi\circ\res_*[\gamma']=[\gamma']\in K, \]
where the last equality holds because of the definition of $\widehat{i}$ and $\gamma'\subset 
\hM^\reg$. Hence, \eqref{eq: kernel splitting} follows.

It remains to prove the factorization \eqref{eq:factorization:K}. By \eqref{eq:F-homotopy}, we note that any $[\gamma]\in K$ is a product loops
of type 
\begin{equation}\label{eq:fact:we:want0}  \mathscr{l}= \sigma \beta \sigma^{-1},
\end{equation} with $\beta$ a null-holonomic leafwise loop and $\beta$ is a path between the base point of $\beta$ and $x_0$. We will show that for such $\mathscr{l}$ we have
\begin{equation}\label{eq:fact:we:want} \mathscr{l}= w \mathscr{l}' w^{-1}
\end{equation}
with $w\in \hW$ and $\mathscr{l}'\in K$ sitting inside $\hM^{\reg}$, this proving \eqref{eq:factorization:K}.
To prove \eqref{eq:fact:we:want}, starting with \ref{eq:fact:we:want0}, we make the following modifications:
\begin{enumerate}[(i)]
\item Since $\beta$ is null-holonomic, as in the proof of Lemma \ref{lem: orbifold fundamental group isomorhism}, we may write $\mathscr{l}= \alpha\beta'{\alpha^{-1}}$
with $\beta'\subset \hM^\reg$ a leafwise null-holonomic path. 
\item We observe that the factorization result of Lemma \ref{lemma:prop:generation-by-reflections} applies not just to loops, but to paths with end-points in $\hM^\reg$. By applying it to $\alpha$ we obtain that our $\mathscr{l}$ can be written as claimed in (\ref{eq:fact:we:want}). 
\end{enumerate}

The last part of the statement can be seen as a particular case of a property of  classical orbifolds. In order to state it, we use the notion of classical embedded orbifold (see Section \ref{sec:The case of orbifolds}).

\begin{lemma}
Let $B$ be a good orbifold, assume that $B_0\subset B$ is a classical embedded suborbifold with $b_0\in B_0$.
Then the image of the map 
\[ i_*: \pi_1^{\orb}(B_0)\to  \pi_1^{\orb}(B)\]
induced by the inclusion is the $\pi_1^{\orb}(B)$-stabilizer of the connected component $C$ of $p^{-1}(B_0)$ containing the base point $\widetilde{b}_0$,
where $p^{-1}(B_0)$ is the pre-image of $B_0$ by the canonical projection $\oB\to B$
\end{lemma}

\begin{proof} We use the notations and the discussion from Section \ref{sec:good:orbifold}. 
Since $B$ is a good orbifold and $B_0$ is a classical embedded suborbifoldwe can write 
\[ (B, b_0)= (\Bgood, b_0')/\Gamma, \quad  (B_0, b_0)= (\Bgood_0, b_0')/\Gamma,\]
with $\Bgood_0\subset \Bgood$ a $\Gamma$-invariant embedded submanifold and the $\Gamma$-action on $B'_0$ remains effective. 
Given an element 
\[ a= (g, \sigma)\in \pi_1^{\orb}(B)= \pi_1(\Bgood\ltimes \Gamma)\]
(see (\ref{eq:fund-group-discrete-actions}) and the discussion preceding it) the condition $C\cdot a= C$ is equivalent to the existence of a path in $p^{-1}(B_0)$ connecting $\widetilde{b}_0$ with $\widetilde{b}_0 a$. Recalling that $\widetilde{b}_0= 1_{b_{0}'}$ is represented by the constant path at $b_{0}'$, (\ref{eq:right action orbifold}) gives 
\[ \widetilde{b_0}\cdot a= g^{*}(1_{b_{0}'})\cdot \sigma= 1_{b_{0}' g}\cdot  \sigma= \sigma,\]
where $\cdot$ stands for the concatenation. 
Hence there exists a path in  $p^{-1}(B_0)$ between 
$\widetilde{b}_0$ and $\sigma$. That means precisely that $\sigma$ is represented by a path in $\Bgood_0$, i.e., that $a$ is in the image of the splitting. 
\end{proof}

\end{proof}

\subsection{Morita invariance of the Weyl group}
Next we prove Morita invariance of the Weyl groups.


\begin{proposition}\label{prop:Morita:Weyl}
Let $(M_i, \pi_i)$ be two PMCTs with integrating proper symplectic groupoids $(\G_i, \Omega_i) \tto (M_i, \pi_i)$, and let $(Q, \Omega)$ be a Morita equivalence between them. 
Then the homeomorphism induced by $Q$ between the leaf spaces $B_i= M_i/\cF_{\pi_i}$, 
\[ \phi: B_1\overset{\sim}{\longrightarrow} B_2,\]
is an isomorphism between the orbifold integral affine structures induced by $(\G_i, \Omega_i)$. Furthermore, the induced map in the orbifold fundamental groups, 
\[ \varphi: \pi_1^{\orb}(B_1)\to \pi_1^{\orb}(B_2)\]
restricts to an isomorphism between the corresponding Weyl groups and makes the induced isomorphism $\phi_*:\oB_1\to\oB_2$ a $\varphi$-equivariant map. Under this isomorphism abstract reflections correspond to each other.
\end{proposition}

\begin{remark}
The statement of the proposition can be refined by specifying the base points $\hx_i= (x_i, \tt_i)\in \hM_i$ appropriately. The points $x_i\in M_i$ should be $Q$-related via a point $q\in Q$. Such a $q$ gives rise to an identification between the Lie algebras $\gg_{x_i}$ (cf. Remark \ref{remark:Morita-general}) and we require that the $\tt_i$'s correspond to each other. With these choices in place, we obtain concrete models for the orbifold fundamental groups, for the $\cW_{\hx_i}$ sitting inside them, as well as for $\phi_*$ mapping $\cW_{\hx_1}$ isomorphically to $\cW_{\hx_2}$.
\end{remark}

\begin{proof}
Observe that for a Morita equivalence $\xymatrix{M_1 & \ar[l]_{\pp_1} Q \ar[r]^{\pp_2}& M_2}$ as in the statement the pull-back of the graphs of the Poisson structures, 
$L_1:=\pp_1^*L_{\pi_1}$ and $L_2:=\pp_2^*L_{\pi_2}$, are related to each other via a gauge transformation w.r.t.~$\Omega$ (see \cite[page 154]{CFM21}). Moreover, $(Q,L_1)$ and $(Q,L_2)$ are DMCTs with integrating proper presymplectic groupoids the pullback groupoids $\pp^*\G_i= Q\times_{M_i} \G_i\times_{M_i} Q\tto Q$ endowed with presymplectic forms that differ by $\t^*\Omega-\s^*\Omega$. Therefore, the leaf spaces of $(Q,L_i)$, their integral affine orbifold structures and their Weyl groups are identical .

We are left with proving the following claim. Given a submersion with connected fibers $\pp: Q\to M$ into a Poisson manifold of proper type $(M, \pi)$ and a proper symplectic groupoid $(\G, \Omega)$, consider the Dirac manifold $(Q,L= \pp^*L_{\pi})$ together with the proper presymplectic integration $\pp^*\G= Q\times_M \G\times_M Q\tto Q$. Then there is a canonical isomorphism between the leaf space of $(M,\pi)$ and $(Q,L)$ which preserves their integral affine structures and identifies their Weyl groups.

To prove this claim we pass to the level of the resolutions. Their construction shows that $\pp:Q\to M$ induces a submersion with connected fibers
\[ \widehat{\pp}:\hQ\to\hM. \]
This submersion pulls back the foliation of $\hM$ to the foliation of $\hQ$. Hence their leaf spaces are canonical isomorphic classical orbifolds. Also, observe that the natural groupoid morphisms $\pp^*\cG\to \cG$ yields a groupoid morphism at the level of the action groupoids over the resolutions:
\[ \pp^*\cG\ltimes \hQ\to \cG\ltimes \hM. \]
This restricts to an isomorphism between isotropy groups so the transverse integral affine structures of $\hQ$ and $\hM$ also correspond to each other, and we have an isomorphism of the integral affine structures on the leaf spaces.

For the claim on the Weyl groups, we consider the following commutative diagram:
\[ \xymatrix@C=15 pt{
1\ar[rr] & & \hW_{\hq} \ar[rr] \ar@{-->}[dd]_{\widehat{\pp}_*} \ar[rd]^{i} & & \pi_1(\hQ, \hq)\ar[ld]_{i} \ar[dd]_{\widehat{\pp}_*} \ar[rr] & & \pi_1(Q, q)\ar[dd]_{\pp_*} \ar[r] & 1 \\
& & & \pi^{\orb}_{1}(B, b) & &  \\
1\ar[rr] &  & \hW_{\hx} \ar[rr] \ar[ru]_{j}& & \pi_1(\hM, \hx)\ar[rr] \ar[lu]^{j}& &\pi_1(M, x) \ar[r] & 1 
}
\]
where $\hq\in \hQ$ determines all the other points $q\in Q$, $\hx\in \hM$, $x\in M$, $b\in B$. 
We have to prove that 
\begin{equation}
    \label{eq:proof:Weyl:Morita}
    i(\hW_{\hq}) =j(\hW_{\hx}),
\end{equation}
coincide. The inclusion
$i(\hW_{\hq})\subset j(\hW_{\hx})$ is clear. For the reverse, we start with $u\in j(\hW_{\hx})$. By Proposition \ref{prop:generation-by-reflections}, we may assume that $u$ arises from a reflection 
corresponding to a point $\hxx$ in a subregular strata: $u= j(\sigma \tau_{\hxx} \sigma^{-1})$, with $\tau_{\hxx}\in \hW_{\hxx}$ as above, and where $\sigma$ is a path in $\hM$ between $\hx$ and $\hxx$. We also choose a point $\hqq\in \hQ$ above $\hxx$ and a path $\widetilde{\sigma}$ in $\hQ$ between $\hq$ and $\hqq$. Since the groups that we are comparing are normal in $\pi_1^{\orb}(B)$ we may assume that $\widetilde{\sigma}$ covers $\sigma$. Observe now that: 
\begin{enumerate}[(i)]
\item $\hqq$ is necessarily subregular for $(Q, L)$, so it has an associated element  $\tau_{\hqq}\in \hW_{\hqq}$. 
\item The map $\widehat{\pp}_*$ sends $\tau_{\hqq}$ to $\tau_{\hxx}$. This is due to the fact that $p:\hQ\to \hM$ restricts to a diffeomorphism between the fibers of the resolution maps above $q$ and $x$. 
\end{enumerate}
It follows that $j(\tau_{\hxx})= i(\tau_{\hqq})$
and then 
\begin{align*}
u & = j(\sigma \tau_{\hxx} \sigma^{-1})\\
  &= j(\widehat{\pp}_*(\widetilde{\sigma})\widehat{\pp}_*(\tau_{\hqq})\widehat{\pp}_*(\widetilde{\sigma})^{-1})\\
  &=j\widehat{\pp}_*(\widetilde{\sigma}\tau_{\hqq}\widetilde{\sigma}^{-1})\\
  &= i(\widetilde{\sigma}\tau_{\hqq}\widetilde{\sigma}).    
\end{align*}
This shows that \eqref{eq:proof:Weyl:Morita} 
holds and completes the proof of the proposition.
\end{proof}


\begin{remark}[s-connecteness]
\label{rem:connecteness} 
Our symplectic groupoids are assumed to be source connected. However, in the Hamiltonian local model we will encounter non-source connected groupoids whenever the isotropy $G=\cG_x$ is disconnected. For that, it is important to note that Morita invariance of the Weyl groups remains valid for Morita equivalent symplectic groupoids that are not necessarily s-connected, provided the Morita bi-bundle is connected. 

For this, note that for a proper symplectic groupoid $(\cG,\Omega)\tto (M,\pi)$ which is not necessarily s-connected, the relevant orbit space is $B=M/\cG$, and that may be different from the leaf space $M/\cF_\pi$. To define the Weyl group $\cW\subset \pi_1^\orb(B)$ one still uses the resolution $(\hG,\hOmega)\to (\hM,\hL)$, where $\hG=\hM\rtimes\cG$ may fail to be s-connected. The classical orbifold structure on $B=M/\cG\cong\hM/\hG$ is presented by the effectivization $\cE(\hG)$ of the proper foliation groupoid $\cB(\hG)$ of the associated short exact sequence
\[ 
\xymatrix{1 \ar[r] & \cT \ar[r] & \hG \ar[r] &\cB(\hG)\ar[r] & 1 }
\]
The effective groupoid $\cE(\hG)\tto \hM$ integrates the foliation $\hF$ -- although it is possibly not s-connected --  so its source connected component is $\Hol(\hM,\hF)$. Hence, $\pi_1^\orb(B)$ is defined as the fundamental group of $\cE(\hG)$ and the Weyl group is defined as before by the diagram
\[ 
\xymatrix{
1\ar[r] & \hW_{\hx}\ar[r]\ar@{-->>}[d] & \pi_1(\hM,\hx)\ar[r]_{\res_*}\ar[d]^-{p_*} & \pi_1(M,x)\ar[r]\ar@/_1.2pc/[l] & 1\\
 & \cW_{\hx}\,  \ar@{^{(}->}[r] & \pi_1^{\orb}(B,b) & 
}
\]
We remark that now the vertical arrow $p_*$ is not surjective in general.

The proof of Proposition \ref{prop:Morita:Weyl} can be adapted to show that 
a connected symplectic Morita equivalence between (not necessarily s-connected) symplectic groupoids $(\cG_1,\Omega_1)\tto (M_1,\pi_1)$ and $(\cG_2,\Omega_2)\tto (M_2,\pi_2)$ induces a presymplectic Morita equivalence between the groupoids $(\hG_1,\hOmega_1)\tto (\hM_1,\widehat{L}_{\pi_1})$ and $(\hG_2,\hOmega_2)\tto (\hM_2,\widehat{L}_{\pi_2})$. The latter induces a Morita equivalence between the effectivizations $\cE(\hG_1)$ and $\cE(\hG_2)$ and hence an isomorphism between orbifold fundamental groups. This isomorphism restricts to an isomorphism of Weyl groups.
\end{remark}

\begin{example}
\label{ex:local:model:Morita}
We can now provide a Morita-equivalence interpretation of Example \ref{ex:W-for-Hamiltonian-model},  
and find the orbifold fundamental group and Weyl group of the Hamiltonian local model $M=Q/G$. 
The space $Q$ provides a symplectic Morita equivalence between the symplectic groupoids $(\cG,\Omega)\tto M$ and $(\gg^*\rtimes G,\omega_\can)\tto \gg^*$, where the latter is not s-connected whenever $G$ is disconnected. Fixing a maximal torus $T\subset G$, one finds that
\[ \cB(\widehat{\gg}^*\rtimes G)\cong \tt^*\rtimes N(T) \]
and so we find for the effectivization the action groupoid
\[ \cE(\widehat{\gg}^*\rtimes G)\cong\tt^*\rtimes N(T)/Z_G(T). \]
Then the previous remark gives
\begin{enumerate}[(a)]
\item a computation of the orbifold fundamental group of $B$ and the identification of the Weyl group:
\[\pi^\orb_1(B)\cong N(T)/Z_G(T),\quad \cW(M,\pi)\cong W^0.
\]
where $W^0:=(N(T)\cap G^0)/T$.
Notice that one also obtains $\Gamma^\lin=\pi^\orb_1(B)$.
\item an identification of the orbifold universal cover $\oB$ with $\tt^*$, as $\pi_1^\orb$-equivariant integral affine manifolds.
\end{enumerate}
%
In particular, we obtain an identification of the action of $\cW(M,\pi)$ on $\oB$ with the classical action of $W^0$ on $\tt^*$.
\end{example}

\begin{corollary}\label{cor:didn't-loose-any-reflection}
Let $(M, \pi)$ be a Poisson manifold of proper type and choose any $\hx= (x,\tt_x)\in \hM$.
The image of any standard hyperplane reflection under the morphism $W\to \cW_{\hx}$ (see Proposition \ref{pro:non-trivial Weyl group elements}) is an abstract reflection.
\end{corollary}

\begin{proof} 
The inclusion of any saturated open subset $(U,\pi|_U)\to (M,\pi)$ maps abstract reflections of $\cW_{\hx}(U,\pi|_U)$ to abstract reflections of $\cW_{\hx}(M,\pi)$. Therefore it is enough to prove the result for the Hamiltonian local model $M=Q/G$.

By Example 
\ref{ex:local:model:Weyl:continued}
the standard reflections in $W$ are exactly the abstract reflections. By Proposition \ref{prop:Morita:Weyl}, the symplectic Morita equivalence between $(\gg^*,\pi_\lin)$ and $(M,\pi)$ provided by $Q$, maps isomorphically $W$ into $\cW(M,\pi)$, identifying abstract reflections.
\end{proof}

\section{The Weyl Group as a Geometric Reflection Group}
\label{sec:Weyl:group:geometric}

Here, as in the previous section, the leaf space $B=\hM/\hF=M/\cF_\pi$ is endowed with its classical orbifold structure. 

\subsection{Geometric reflections}
\label{sec:geometric:reflections}
Our next aim is to show that the abstract reflections can be realised as geometric ones. By a  \textbf{geometric reflection} on a manifold $N$ we mean any  smooth involution $r: N\rightarrow N$ with the property that the fixed-point set
\[ N^r:= \{x\in N: r(x)= x\}\]
is a codimension one submanifold that separates $N$, i.e., $N\setminus N^\tau$ has two connected components.

\begin{definition}
\label{def:integral affine reflection group}
A \textbf{reflection group} on a manifold $N$ is a discrete group acting properly and effectively on $N$ and generated by geometric reflections.

When $N$ is an integral affine manifold and the group acts by integral affine transformations, we call it an \textbf{integral affine reflection group} on $N$.
\end{definition}

Given a reflection group $\Gamma$ acting on an orientable manifold $N$, we define the \textbf{parity character} to be the group homomorphism
\[ \delta:\Gamma\to\Z_2, \]
which associates to $\gamma\in\Gamma$ the element $1$ (respectively, $-1$) if $\gamma$ preserves (resp.~reverses) orientation. We note that $\delta(\gamma)$ coincides with the parity of the decomposition of $\gamma$ into a product of reflections.

We will now show that.

\begin{theorem}
\label{thm:Weyl:reflection:group}
The Weyl group $\cW$ of a Poisson manifold of proper type is an integral affine reflection group on $\oB$.
\end{theorem}

First we recall the following folklore result regarding the separation property.

\begin{lemma}\label{lem:separation} If $N$ is a simply connected manifold, then  any  closed connected codimension one submanifold $Y$ is co-orientable and separates $N$.
\end{lemma}
\begin{proof} The fact that the normal bundle of $Y$ must be trivial is proved in \cite{Samelson69}. By the excision property, the reduced homology group $\widetilde{H}_0(N,N\setminus Y)$ is  isomorphic to $\widetilde{H}_0(\nu(Y),\nu(Y)\setminus 0)\cong \Z$, so the result follows.
\end{proof}

We have already observed that the action of $\cW$ on $\oB$ is proper and effective. Theorem \ref{thm:Weyl:reflection:group} now follows from the following proposition and the fact that the Weyl group is generated by abstract reflections -- see Proposition \ref{prop:generation-by-reflections}.

\begin{proposition}\label{pro: abstract is geometric reflection}
The action of any abstract reflection $\tau\in \cW$ on $\oB$ is a geometric reflection on $\oB$. 
\end{proposition}

\begin{proof}
Given $\tau\in \cW$ an abstract reflection  and $\widetilde{b}\in \oB$ a fixed point of $\tau$, we claim that:
\begin{enumerate}
\item[(C)] there exists a contractible neighborhood $V$ of $\widetilde{b}\in \oB$ such that 
\[V\cap (\oB)^\tau\]
is a codimension one submanifold.
\end{enumerate}
Assuming this claim, the result follows. To see this we first notice that $(\oB)^\tau\subset \oB$ is a closed embedded submanifold, with connected components possibly of different dimensions. Let $\cH$ be the component that contains $\widetilde{b}$, which is assumed to be of codimension one. Since $\oB$ is simply connected we can apply the previous lemma,  so $\oB\setminus \cH$ has two connected components and $\cH$ is the common boundary. Furthermore, $\cH$ would be orientable. Then since the action of $\tau$ on $V$ interchanges the two connected components of $V\setminus (V\cap \cH)$, we conclude that $\tau$ interchanges the connected components of $\oB\setminus \cH$. Therefore there cannot be any other fixed point away from $\cH$.


We are left with proving (C). For that note that $\tau$ comes from a fiber $\res^{-1}(x)$ inside a subregular stratum, i.e. $\tau= \sigma \tau_{\hx}\sigma^{-1}$ where $\sigma$ joins the base point $\hx_0$ with $\hx$ and $\tau_{\hx}$ is a non-contractible loop in $\res^{-1}(x)$. We divide the proof into several steps, that reduce the problem to the local Hamiltonian model.
\smallskip

\emph{Step 1:} We may assume that $\hx_0=\hx$ and $\tau=\tau_{\hx}$. 

This follows from Remark \ref{rem:base-point change} and the fact that conjugation by $\sigma$ takes $\tau$ to $\tau_{\hx}$.
\smallskip

\emph{Step 2:} It suffices to prove that the claim holds for $(U,\pi|_U)$ where $U$ is a saturated neighborhood of $\hx$ invariant by $\tau$. 


This is a consequence of the functoriality of the main diagram with respect to restriction to saturated opens (see Remark \ref{rk:functoriality-main-diagram}). The inclusion $U\hookrightarrow M$ gives group homomorphisms $\phi$ between the (orbifold) fundamental groups corresponding to $(U, \pi|_U)$ and $(M, \pi)$. 
The maps relating the main diagrams are $\phi$-equivariant. Notice that $\phi$ maps the groups $\cW(U,\pi|_U)$ to $\cW(M, \pi)$ and sends $\tau_{\hx}\in \cW(U,\pi|_U)$ to $\tau_{\hx}\in \cW(M, \pi)$.
\smallskip

\emph{Step 3:} The claim holds when $M= Q/G$ is the Hamiltonian local model with $\hx=(x,\tt_x)$ $Q$-related to $(0,\tt)$. 

Since $\hx$ a subregular point we have $\gg^\ss=\mathfrak{su}(2)$. The discussion in Example \ref{ex:local:model:Morita} shows that the Weyl group of $Q/G$ is isomorphic to $\Z_2$ and acts on 
$\tt^*=\zz(\gg)^*\oplus {(\tt^\ss)}^*$ fixing the first summand. Therefore, the claim holds.  
\end{proof}

\subsection{The Weyl group as a Coxeter group}
\label{sec:Coxeter}

The fact that the Weyl group of $(M,\pi)$ is an integral affine reflection group on $\oB$ has several interesting consequences. To describe them we use several notions from the theory of reflection groups and Coxeter systems (see, e.g., \cite{Davis08}). Namely, we will be using the following: 
\begin{itemize}
\item \textbf{set of geometric reflections} $\Refl\subset \cW$.
\item \textbf{hyperplane} $\cH_{r}$ of $(\oB,\cW)$: the fixed point set of a reflection $r\in \Refl$.
Any $\cH_{r}$ is an integral affine hyperplane that separates $\oB$;
\item \textbf{$\cW$-regular point}: any point $u\in \oB$ which is not fixed by any $w\in \cW$. We denote by $(\oB)^{\cW-\reg}$ the set of regular points. 
\item \textbf{chamber} of $(\oB,\cW)$: the closure of a connected component of the subset of regular points. 
\item \textbf{simple reflection} associated to  the chamber $\Delta$: a reflection $r\in \Refl$ with the property that there is a point $u\in \Delta\cap \cH_{r}$ such that $r$ is the only reflection that fixes $u$. We denote by $\Refl_{\Delta}$ the collection of all such reflections. 
\end{itemize}

\smallskip

We recall also that a Coxeter system is a group $\cW$ with a set of generators $\Refl_0=\{r_i:i\in I\}$ satisfying relations
\[ 
(r_{i}r_{j})^{m_{ij}}=1,\quad (i,j\in I), 
\]
where $m_{ii}=1$ and $m_{ij}=m_{ji}\geq 2$ is either an integer or $\infty$ for $i\neq j$. The condition $m_{ij}=\infty$ means that no relation $(r_{i}r_{j})^{m}=1$ for any integer $m\geq 2$ is imposed.

Using the general results in \cite{Davis83} concerning reflection groups one deduces the following result.

\begin{theorem}\label{thm:Davis}
For a Poisson manifold of proper type, one has:
\begin{enumerate}[(i)]
\item $(\cW, \Refl_{\Delta})$ is a Coxeter system.
\item $\cW$ acts freely and transitively on the set of chambers.
\item $\Delta$ is a fundamental domain 
in the sense that the canonical projection restricts to a homeomorphism $\Delta \cong \oB/\cW$. In particular, 
\[ \Delta\cdot \cW= \oB. \] 
\item the $\cW$-regular part is the complement of the hyperplanes: 
\[ (\oB)^{\cW-\reg}=\oB\setminus \bigcup_{r\in \Refl} \cH_{r}. \]
\item the isotropy group $\cW(u)$ at any $u\in\Delta$ is a finite group generated by the reflections from $\Refl_{\Delta}$ that fix $u$. Moreover, if $U$ is a $\cW(u)$-linear neighborhood of $u$, one has:
\begin{enumerate}[(1)]
\item $U\cap \Delta$ is a chamber for the action of $\cW(v)$ on $U$; 
\item $\Refl_{U\cap \Delta}$ coincides with the reflections in $\Refl_{\Delta}$ which fix $u$.
\end{enumerate}
\item for $r\in \Refl_{\Delta}$, $\Delta$ is contained in the closure of one of the two components of $\oB\setminus \cH_r$. Denoting by $H_r^{+}(\Delta)$ that component, one has
\[ \Delta= \bigcap_{r\in \Refl_{\Delta}} H_r^{+}(\Delta).\]
\end{enumerate}
\end{theorem}

\medskip

In particular, we have: 

\begin{corollary} The Weyl chamber $\Delta$ is both a manifold with corners as well as a (classical) orbifold, with the orbifold atlas given by $\oB\rtimes \cW\tto \oB$.
\end{corollary}
 
Next, the action of the entire $\pi_1^{\orb}(B)$ descends to an action of the group
\[ \pi:= \pi_1^{\orb}(B)/\cW\]
on the chamber $\Delta$, with orbit space 
\[ \Delta/\pi\cong B.\] 
Since the action of $\pi$ is still proper and effective, one may think of $\Delta \rtimes \pi\tto \Delta$ as an atlas making $B$ into an ``orbifold with corners''. 
\medskip

At this point we have two notion of regular points in $\oB$: the points in $(\oB)^\reg$ coming from from the infinitesimal stratification (see the comments preceding Proposition \ref{prop:split:orbi:fundamental:group}), and 
the $\cW$-regular points described in item (iv) of the last theorem. We now show that the two notions coincide. 


\begin{proposition}\label{pro: equality regulars} The following equality holds:
\begin{equation}\label{eq: equality regulars}
(\oB)^\reg= (\oB)^{\cW-\reg}.
\end{equation}
Moreover,
\begin{enumerate}[(i)]
    \item the chambers of $(\oB,\cW)$ are the connected components of $(\oB)^\reg$;
    \item $\Refl$ is precisely the set of all abstract reflections of $\cW$.
\end{enumerate}
\end{proposition}

\begin{proof}
We first prove the inclusion $\subset$.
Since $\cW$ acts transitively on connected components, if an element $w\in \cW$  fixes a regular point, a conjugate $w'$ of it fixes a regular point in the connected component $C$ of ${(\oB)}^\reg$ containing the base point, hence $C\cdot w'= C$. Then Proposition \ref{prop:split:orbi:fundamental:group} implies that $w'$ must be  the identity, and therefore  the inclusion holds. 

At this point we have:
\[(\oB)^\reg \subset (\oB)^{\cW-\reg}=\oB\setminus \bigcup_{r\in \Refl} \cH_{r}\subset \oB\setminus \bigcup_{s\in \Refl^{\mathrm{abs}}} \cH_{s}
\]
where $\Refl^{\mathrm{abs}}$ is the collection of all abstract reflections. Next we prove that the last term is included in the first one. In other words, we will show that any point  $\widetilde{b}\in (\oB)^\sing$ is fixed by some abstract reflection $w\in \cW$. 

As in the proof of Proposition \ref{pro: abstract is geometric reflection} we may assume that $\widetilde{b}$ is the base-point of $\oB$ and, using the functoriality, that $M$ is the Hamiltonian local model. So let
$M= Q/G$ with $\hx=(x,\tt_x)$ $Q$-related to $(0,\tt)$ and let $W$ denote the Weyl group of $G^0$ relative to $\tt$. Any standard reflection $w\in W$ on a root hyperplane fixes the origin $0\in \tt^*$. By Corollary \ref{cor:didn't-loose-any-reflection}, its image under the inclusion $W\to \cW$ is the desired abstract reflection.
We conclude that (\ref{eq: equality regulars}) holds and that 
\begin{equation}\label{eq:inclusion:hyperplanes} \cH_r\subset \bigcup_{s\in \Refl^{\mathrm{abs}}} \cH_{s}\ \quad \quad \forall\ r\in \Refl.
\end{equation}

Item (i) follows immediately from the (\ref{eq: equality regulars}). As for item (ii), let $r\in \Refl$ be any geometric reflection and consider a neighborhood of a point $\widetilde{b}\in\cH_r$. Since the action of $\cW$ can be linearized around $\widetilde{b}$ and all the hyperplanes through $\widetilde{b}$ become linear hyperplanes, it follows from \eqref{eq:inclusion:hyperplanes} that a neighborhood of $\widetilde{b}$ in $\cH_r$ is inside a 
neighborhood of $\widetilde{b}$ in some $\cH_{s}$ with $s\in \Refl^{\mathrm{abs}}$. In the linearised neighborhood, $r$ and $s$ will be linear reflections on the same hyperplane and, therefore, $rs^{-1}$ will be either the identity or will have infinite order. Using properness and effectiveness of the action, the isotropy group $\cW_{\widetilde{b}}$ is finite and we must have $rs^{-1}= 1$.
\end{proof}

\medskip
The previous results justify the following definition.

\begin{definition}
    Given a Poisson manifold $(M,\pi)$ of proper type, a \textbf{Weyl chamber} of $\oB$ is the closure of a connected component of the regular locus of $\oB$.
\end{definition}

It follows from Propositions \ref{prop:split:orbi:fundamental:group} and \ref{pro: equality regulars} that the orbifold fundamental group of the embedded integral affine orbifold $B^\reg\subset B$ can  identified with the subgroup of  $\pi_1^\orb(B^\reg,b_0)$ which stabilizes the Weyl chamber $\Delta$ containing $\tilde{b}_0$. The next result identifies the orbifold universal covering space of $B^\reg$

\begin{proposition}\label{pro:chamber-covering}
An open Weyl chamber $\interior(\Delta)\subset \oB$ is an orbifold universal covering space of $B^\reg$. In particular $\interior(\Delta)$ is simply connected. 
\end{proposition}

\begin{proof}
We can assume that $\Delta$ is the Weyl chamber containing the base point $\tilde{b}_0$. Let us represent $B^\reg$ by the orbifold atlas 
\[{(B^\reg)}^\lin\rtimes \Gamma_0^\lin\tto (B^\reg)^\lin\]
where $\Gamma_0^\lin$ is the linear holonomy group of $(\hM^\reg,\hF)$. Note that it sits as a subgroup of $\Gamma^\lin$.

The inclusion $i:Y:=(B^\reg)^\lin\hookrightarrow B^\lin$ (see Prop. \ref{rk:functoriality-main-diagram}) induces a map $\widetilde{i}:\widetilde{Y}\to \widetilde{B^\lin}=\oB$
between orbifold universal covering spaces which is $\pi_1^\orb$-equivariant. We have the following diagram
\[
\xymatrix{\widetilde{Y}\ar[d]\ar[dr]^{\widetilde{i}}\\
\widetilde{i}(\widetilde{Y})\ar@{^{(}->}[r]\ar[d]^{p} & \widetilde{B^\lin}\ar[d]^{p} \\
Y \ar@{^{(}->}[r]_{i} & B^\lin}
\]
where the left column is a sequence of covering maps among (connected) manifolds. This has two consequences:
\begin{enumerate}[(i)]
    \item the open subset $\widetilde{i}(\widetilde{Y})\subset \widetilde{B^\lin}$ is the connected component of $p^{-1}(Y)$ which contains the base-point. By Proposition \ref{pro: equality regulars} this  connected component is exactly the interior of the chamber $\Delta$;
    \item $\tilde{i}$ is a diffeomorphism onto its image if and only if the induced map on fundamental groups $i_*:\pi_1(Y)\to \pi_1(B^\lin)$ is injective. If this holds, the proof of the proposition would follow.
\end{enumerate}
To prove the injectivity of $i_*$ we use the short exact sequences 
relating fundamental groups and orbifold fundamental groups (see \ref{eq:pi1-orb-B-gamma}): 
\[
\xymatrix{
1\ar[r] & \pi_1(B^\lin)\ar[r] & \pi_1^\orb(B)\ar[r] & \Gamma^\lin \ar[r]& 1\\
1\ar[r] & \pi_1((B^\reg)^\lin)\ar[r]\ar[u]_{i_*} & \pi_1^\orb(B^\reg)\ar[r]\ar[u] & \Gamma^\lin_0 \ar[r]\ar[u] & 1
}\]
By Proposition \ref{prop:split:orbi:fundamental:group} the middle vertical arrow is injective, so the injectivity of $i_*$ follows.
\end{proof}

In general, the Weyl group $\cW$ of a PMCT (or a DMCT) is not a classical (affine) Weyl group. However, as another application of Theorem \ref{thm:Davis}, we show that its isotropy subgroups are classical Weyl groups. Namely, the isotropy of $\cW$ at a point $\hx=(x,\tt_x)$ coincides with the classical Weyl group of the isotropy Lie algebra $\gg_x$ relative to $\tt_x$. This is made more precise in the following result which improves Proposition \ref{pro:non-trivial Weyl group elements}.

\begin{proposition} 
\label{prop:Weyl:group:isotropy at point}
Let $\tilde{b}\in \oB$, fix any $\hx=(x,\tt_x)\in \hM$ mapping to the image of $\tilde{b}$ in $B$ and let $W$ be the classical Weyl group relative to $\tt_x$. 

Then there exists open neighborhoods $V_b$ of $b\in \oB$ and $V$ of $0\in \tt_x^*$, a diffeomorphism
$\varphi:V\to V_{\widetilde{b}}$
and a monomorphism $\phi:W\to \cW_{\hx}$,
with the following properties:
\begin{enumerate}[(i)]
    \item $\varphi$ is $\phi$-equivariant;
    \item $\phi$ yields an isomorphism between the regular parts  ${(\tt_x^*)}^\reg$ and $V_b\cap {(\oB)}^\reg$. 
\end{enumerate}
In particular, $\phi$ is an isomorphism from  $W$ to the stabilizer $\cW_{\hx}(\widetilde{b})$.
\end{proposition}


\begin{proof}
We first prove that the existence of $\varphi$ and $\phi$ as in the statement implies the equality
\begin{equation}
\label{eq:stabilizers}\cW_{\hx}(\widetilde{b})=\phi(W).
\end{equation}
To see this, by item (v) in Theorem \ref{thm:Davis}, the stabilizer $\cW_{\hx}(\tilde{b})$ is generated by all reflections which fix $\widetilde{b}$. In turn, each reflection is characterized by its fixed point set. By Proposition \ref{pro: equality regulars} the regular and $\cW_{\hx}$-regular parts agree. Therefore, every reflection hyperplane containing $\widetilde{b}$ is the $\varphi$-image of a classical reflection hyperplane in $\tt^*_x$. Hence, $\varphi$ identifies the corresponding reflection of $\cW_{\hx}$ with the image by $\phi$ of a classical reflection of $W$. The equality  \eqref{eq:stabilizers} follows.

Moreover, since the action of $W$ on $\tt_x^*$ is effective, it follows that $\phi$ is a monomorphism. 

To complete the proof we show the existence of a homomorphism $\phi:W\to\cW_{\hx}$ and a diffeomorphism $\varphi:V\to V_{\widetilde{b}}$ which is $\phi$-equivariant. Since the action of $W$ on $\tt_x^*$ is effective, it follows that $\phi$ is a monomorphism. This is divided into 3 steps, that reduce the problem to the local model:
\begin{enumerate}[1)]
\item We may assume that $\hx_0=\hx$ (and therefore $\widetilde{b}=\widetilde{b}_0$). This is a consequence of   the fact that conjugation by a path in $\hM$ joining the base point $\hx_0$ with $\hx$ produces a $\cW_{\hx}$-equivariant diffeomorphism between based orbifold universal covering spaces.  


\item It suffices to contruct the equivariant diffeomorphism for $(U,\pi|_U)$, where $U$
is an invariant neighborhood of $\hx$ for which $B_U^\lin$ is simply connected. This is a consequence of the functoriality of the main diagram with respect to restriction to invariant opens (see Remark \ref{rk:functoriality-main-diagram}). We obtain a group homomorphisms $\phi$ between the (orbifold) fundamental groups corresponding to $(U, \pi|_U)$ and $(M, \pi)$ taking $\cW_{\hx}(U,\pi|_U)$ to $\cW_{\hx}(M, \pi)$. Because  $B_U^\lin$ is simply connected, it follows that the map $\psi:\oB_U\to \oB$ in diagram \eqref{eq:Weyl:main-diagram-open} is an equivariant embedding. Observe that regular points are obtained by pulling back $B^\reg\subset B$, thus $\psi({(\oB_U)}^\reg)=\psi(\oB_U)\cap {(\oB)}^\reg$.

\item The claim holds when $M= Q/G$ is the Hamiltonian local model with $\hx=(x,\tt_x)$ $Q$-related to $(0,\tt_x)$. This is shown in Example \ref{ex:local:model:Morita}.
\end{enumerate}
\end{proof}







\section{Some Examples Derived from Lie Theory}

In this section we first analyze the structures on the leaf spaces of two fundamental examples from Lie theory: the linear Poisson structure on a compact Lie algebra
and the Cartan-Dirac structure on a compact Lie group. We will see, for example, that our notions of Weyl group and Weyl chambers coincide with the classical ones. We will also relate other Poisson theoretic aspects to classical results in Lie theory, leaving verifications to the reader. 
A few of these results have already appeared in the text, but we find convenient to list them together here in a more detailed manner.
Also, the regular loci of these two examples were discussed in \cite[Section 4.5 and 5.3.2]{CFM-II}. 

In the remaining of the section we will consider PMCTs and DMCTs obtained from $\gg^*$ and $G$ by removing saturated closed subsets, and we will discuss how that affects the Weyl group and other related notions.

\subsection{Coadjoint orbits}\label{ex:regular-coadjoint-orbits}
Let $\gg$ be a compact Lie algebra.
If $G$ is a compact connected Lie group with Lie algebra $\gg$, then  $(\cG,\Omega)=(T^*G,\w_\mathrm{can})$ is a s-proper integration of
$(\gg^*,\pi_{\gg^*})$. We will use root data to describe the integral affine orbifold structure of the leaf space $B=\gg^*/\cF_{\gg^*}$ and of the members of $\cSi(B)$, and the Weyl group of $(\gg^*,\pi_{\gg^*})$. For the relevant background on Lie theory we refer the reader to \cite[Chapter 3]{DK}. We first enumerate the main results. 
\begin{enumerate}[(i)]
\item There is a canonical diffeomorphism  
\[G/T\times_{W} \tt^*\cong \hM,\]
where $T$ is the isotropy Lie group of the fixed base point, a point in the regular part, and $\tt$ its Lie algebra. Also, $W=N(T)/T$ is the Weyl group.
Therefore 
\[B=\tt^*/W,\quad \oB=B^{\aff}=B^\lin \cong \tt^*.\]
\item A (classical) integral affine orbifold atlas for $B$ is given by 
\[(\tt^*,\Lambda_G^\vee)\rtimes W\tto \tt^*,\]
where $\Lambda_G^\vee\subset \tt^*$ is the lattice dual to 
$\Ker(\exp:\tt\to T)$.
\item The Weyl group of $\gg^*$ and the orbifold fundamental group of $B$
agree, and coincide with the (classical) Weyl group:
\[\cW=\pi_1^\orb(B)\cong W.\]
The orbifold fundamental group of $B^\reg$ is trivial, this being reflected in the triviality of the short exact sequence 
\begin{equation}
\label{eq: orbifold-fundamental-group-coadjoint} 
\vcenter{
\xymatrix{ 
1\ar[r]&\cW\ar[r]&  \pi_1^\orb(B) \ar[r]^---{\res_*}  
& 1 \ar[r] & 1 
}
}
\end{equation}
\item The canonical and the infinitesimal stratifications of $B$ agree. Their strata are in bijection with the open faces of a fixed Weyl chamber $\Delta$.  If $\Sigma_{I}$ denotes the stratum determined by the face $\Delta_I$ with underlying vector subspace $\tt^*_I\subset \tt^*$, then an orbifold atlas for $B_{\Sigma_I}$ is 
\[(\mathrm{int}(\Delta_I),\Lambda_{G,I}^\vee)\rtimes W_I.\]
Here $W_I\subset W$ denotes the subgroup generated by the reflections fixing $\Delta_I$ pointwise and $\Lambda^\vee_{G,I}=\Lambda_G^\vee\cap \tt_I^*$. In particular, the non-regular strata are purely ineffective orbifolds.
\end{enumerate}
\medskip

In order to prove the previous assertions, we start by recalling the Weyl resolution of the PMCT $(\gg^*,\pi_{\gg^*})$ from  Example \ref{ex:linear:case}:
\[\hM=\{(\xi,\tt')\,|\, \xi\in \gg^*,\,\tt'\subset \gg_\xi\,\,\mathrm{maximal}\,\,\mathrm{torus}\},\quad \res=\mathrm{pr}_1:\hM\to \gg^*.\]
We have identifications 
\[G/T\times_W \tt^*\cong G\times_{N(T)}\tt^* \cong  \hM,\quad [(gT,\xi)]\mapsto (\Ad^*_g\xi,\Ad_g\tt).\]
The groupoid integrating the resolution is the action groupoid
\begin{equation}
\label{eq:res:grpd:coadjoint}
    \hG=G\ltimes (G/T\times_W \tt^*)\tto G/T\times_W \tt^*
\end{equation} 
associated to left $G$-translations on $G/T$. The associated foliation groupoid is   
\begin{equation}\label{eq:fol-oid-coadjoint}
\cB(\hG)=(G/T\times G/T\times \tt^*)/W\tto G/T\times_W \tt^*,\quad \s=\mathrm{pr}_{12},\,\t=\mathrm{pr}_{13}.
\end{equation}
Observe that the obvious map $\tt^*\hookrightarrow G/T\times_W \tt^*$ gives a complete transversal to the orbit foliation of $\hG$. Upon restriction to this transversal, one obtains a presymplectic groupoid $(\hG|_{\tt^*},i^*\hOmega)$ Morita equivalent to $(\hG,\hOmega)$.  The isotropy bundle of $\hG|_{\tt^*}$ is the trivial bundle with fiber $N(T)$, while the associated foliation groupoid is the action groupoid
\[ \cB(\cG|_{\tt^*})=\tt^*\rtimes W\tto \tt^*.\]
This action groupoid is effective, since the action of $W$ on $\tt^*$ is effective (recall that $G$ is connected). It follows that the foliation groupoid \eqref{eq:fol-oid-coadjoint} is effective, so it is the holonomy groupoid of $\hF$.
The identity component  of the isotropy bundle is the trivial bundle, $\cT(\hG|_{\tt^*})=\tt^*\times T\to \tt^*$, and the pullback of the presymplectic form to this bundle is the canonical duality pairing between $\tt$
and $\tt^*$.
We conclude that the induced integral affine structure on $\tt^*$ has (constant) lattice $\Lambda_G^\vee$ dual to $\Ker(\exp:\tt\to T)$.

In this way we obtain a (classical) integral affine orbifold atlas for the leaf space $B$ (c.f. \cite[Section 4.6.1]{CFM-II}): 
\[(\tt^*,\Lambda_G^\vee)\rtimes W.\]
Being simply connected, this is also the universal orbifold covering space for $B$, therefore
\[\oB=B^{\aff}=B^\lin \cong \tt^*,\]
and \[\cW=\pi_1^\orb(B)=W,\quad \pi_1^\orb(B^\reg)=1.\]
This recovers the results we found in Example \ref{ex:Weyl-group:gg*} using the simple connectivity of principal leaves of the resolution.

The equality of the canonical stratifications can be drawn from classical Lie theory: the infinitesimal orbit types and the orbit types agree because isotropy groups are always connected, since $G$ is connected. 

Next we describe the infinitesimal stratification $\cSi(B)$.
An explicit computation in terms of root spaces (see \cite[Page 153]{DK}) implies that all points of each open face $\interior(\Delta_I)$ share the same isotropy Lie algebra $\pp_I$. It also implies that  
the saturation of $\interior(\Delta_I)\subset \gg^*$ is a stratum $\Si_I$ admitting the  canonical parametrization
\begin{equation}
\label{eq:coadj-triv0}
 G/P_I \times \interior(\Delta_I)  \to \Si_I,\quad (gP_I,\xi)\mapsto \Ad_g^*\xi.
 \end{equation}
Similarly, we have a parametrization of the corresponding strata of the resolution, namely
\[
(G/T)/W_I\times \interior(\Delta_I)\to
\hSi_{I},\quad ([gT],\xi)\mapsto (\Ad_g^*\xi,\Ad_g\tt),\]
where $W_I\subset W$ is the subgroup that fixes pointwise the supporting vector space $\tt_I^*$ of $\Delta_I$. We conclude that $B_{\Si_I}=\interior(\Delta_I)$, as topological spaces.

By Theorem \ref{thm:canonical-orbifold-stratifications} the reduction of  $(\hG,\hOmega)$ along $\hSi_I$ 
makes $B_{\Si_I}$ into an integral affine embedded suborbifold of $B$.
Namely, restricting the groupoid \eqref{eq:fol-oid-coadjoint} to $\hSi_I$, we obtain the foliation groupoid
\[(G/T\times G/T \times \interior(\Delta_I))/W_I\to  (G/T)/W_I\times \interior(\Delta_I). \]
Using the complete transversal $\interior(\Delta_I)\hookrightarrow \hSi_I$, we obtain that this groupoid is Morita equivalent to the (purely ineffective) action groupoid 
\[\interior(\Delta_I)\rtimes W_I\tto \interior(\Delta_I). \]
This gives the desired orbifold structure for $B_{\Sigma_I}$. For its integral affine structure, we observe that the isotropy bundle of the reduced symplectic groupoid trivializes upon restriction to the complete transversal,
with fiber $T/(T\cap [P_I,P_I])$. From
the explicit description of $\pp_I$ we obtain the decomposition 
\[\tt=\zz(\pp_I)\oplus \tt\cap [\pp_I,\pp_I],\quad \tt\cap [\pp_i,\pp_I]=(\tt_I^*)^0. \]
The first identity gives the identification $Z(P_I)\cong T/(T\cap [P_I,P_I])$, and the second one shows that the lattice 
$\Ker(\exp:\zz(\pp_I)\to Z(P_I))$ is identified, via the presymplectic form at units, with
$\Lambda_{G,I}^\vee= \Lambda_G^\vee\cap \tt^*_I$. We now conclude that we have the integral orbifold atlas for $B_{\Si_I}$
\[(\interior(\Delta_I),\Lambda_{G,I}^\vee)\rtimes W_I. \]
Comparing with the integral orbifold atlas for $B$ obtained above, we see that $B_{\Si_I}$ is an embedded suborbifold of $B$.

A different integral affine orbifold structure on $B_{\Si_I}$ is obtained by using the reduction of $(\cG,\Omega)$ along $\Si_I$. Since the leaves of $\cG$ are simply connected, the associated foliation groupoid must be the holonomy groupoid of the simple foliation $G/P_I\times \interior(\Delta_I)$. Therefore, the orbifold structure on $B_{\Si_I}=\interior(\Delta_I)$ arising from this reduction is the usual manifold structure. Note that the integral affine structure is still given by the lattice $\Delta_{G,I}^\vee$.

In conclusion, reduction of $(\hG,\hOmega)$ along $\Si_I$ produces  a purely ineffective integral affine orbifold atlas for $B_{\Si_I}$
whose isotropy reduction is the integral affine orbifold atlas obtained from reduction of $(\cG,\Omega)$ along $\Si_I$. This concurs with item (i) of Proposition 
\ref{prop:canonical-orbifold-stratifications:G}.

\subsection{Conjugacy classes}
\label{ex:conjugacy-classes}
Consider a compact, connected, Lie group $G$, let $\langle\cdot, \cdot\rangle$ be an $\Ad$-invariant inner product on $\gg$, and let $L_G$ be the corresponding
Cartan-Dirac structure on $G$ with twisting the Cartan 3-form. We will describe the Weyl resolution, the integral affine orbifold structure of the leaf space, and the Weyl group and orbifold fundamental group in terms of classical Lie theoretic data, leaving the verifications for the reader; we shall also sketch the description of the members of the canonical stratifications. Our viewpoint shall provide a geometric interpretation of some of these classical results. For the material in Lie theory that we will use we refer the reader to \cite[Chapter 3]{DK} and \cite[Chapter 5]{BtD95}.
We first enumerate the main results:
\begin{enumerate}[(i)]
\item There is a canonical diffeomorphism  
\[G\times_{N(T)} T\cong \hM,\]
where $T\subset G$ is the centralizer of the fixed base point, a point in the principal part -- so $T$ is a maximal torus of $G$.
Therefore 
\[B=T/W,\quad B^\lin \cong T,\quad \oB=B^{\aff}\cong\tt,\]
where $W=N(T)/T$ is the Weyl group.
\item One has a (classical) integral affine orbifold atlas for $B$ given by 
\[(T,\Lambda_G^*)\rtimes W\tto T,\]
where 
\[\Lambda_G^*=\{\lambda\in \tt\,|\, \langle u,\lambda \rangle \in \Z,\,\forall u\in \Lambda_G\},\]
with $\Lambda_G=\ker(\exp:\tt\to T)$.
\item The Weyl group of $(G,L_G)$ coincides with the classical affine Weyl group of $G$, while the orbifold fundamental group of $B$ coincides with the (classical) extended affine Weyl group of $G$: 
\[\cW=W^\aff=W\cdot \Lambda_{\mathrm{R}^\vee},\quad \pi_1^\orb(B)=W^\aff_{G}=W\cdot\Lambda_G,\]
where $\Lambda_{\mathrm{R}^\vee}$ denotes the coroot lattice of the root system of $(\gg,\tt)$. When $G$ is simply connected, the two groups coincide and we recover the results in Example \ref{ex:Weyl-group:G-Dirac}. On the other hand, this also yields another integral affine orbifold atlas for $B$, namely 
\[(\tt,\Lambda_G^*)\rtimes W^\aff_{G}\tto \tt. \]
Moreover, the orbifold fundamental group of $B^\reg$ coincides with $\pi_1(G)$ and all these groups fit into the following commutative diagram 
\begin{equation}
\label{eq: orbifold-fundamental-group-conjugacy} 
\vcenter{
\xymatrix{ 
&1\ar[d] &  1\ar[d] & \\
1\ar[r]&\Lambda_{\mathrm{R}^\vee}\ar[d]\ar[r]&   \Lambda_G\ar[d] \ar[r]   
& \pi_1(G) \ar@{=}[d]\ar[r] & 1 
\\
1\ar[r]&\cW \ar[d] \ar[r]&   \pi_1^\orb(B) \ar[d]\ar[r]_{\res_*}&  \pi_1(G) \ar[r] 
& 1 \\
&W\ar[d]\ar@{=}[r] &  W\ar[d] & 
& 
\\
&1  & 1 &  
}}
\end{equation}
In the two vertical short exact sequences, the kernel and the cokernel correspond to the translational and linear parts of $\cW$ and $\pi_1^\orb(B)$, viewed as groups of affine transformations of $\tt$.

\item According to Proposition \ref{prop:split:orbi:fundamental:group}, the inclusion of the regular part $G^\reg$ in $G$ splits the short exact sequence of the orbifold fundamental group, giving a short exact sequence for $\pi_1(G)$:
\begin{equation}
\label{eq: orbifold-fundamental-group-conjugacy2} 
\vcenter{
\xymatrix{ 
  1\ar[d] &  1\ar[d]\\
   \Lambda_G \ar[d] &  \ar@{_{(}->}[l]\, \Lambda_G\cap\zz(\gg)\ar[d] & \\
  \pi_1^\orb(B) \ar[d]\ar@/^1.2pc/[r]^{\res_*}&  \ \pi_1(G) \ar@{_{(}->}[l]^{i}\ar[d]\ar[r] & 1 \\
  W\ar[d] &  \quad Z\ar[d]\quad \ar@{_{(}->}[l]^{i^\lin}\\
 1 &  1
}}
\end{equation}
The translational part of $\pi_1(G)$ defines a (compact) covering group $\overline{G}$ which factors into a central and a 1-connected semisimple factor
\begin{equation}\label{eq: product covering group} \overline{G}:=G(\gg)/(\Lambda_G\cap\zz(\gg))=H\times G(\gg^\ss). 
\end{equation}
The linear part of $\pi_1(G)$ defines a finite abelian group $Z\subset W$ which acts on the faces of the Weyl alcove $\aa=\zz(\gg)\oplus \aa^\ss$ determined by the choice of base point. Note also that 
\[\overline{G}/Z\to G\]
is an isomorphism.
\item The strata of  $\cS^\inf(B)$ are in bijection with the orbits of action of $Z$ on the faces of the Weyl alcove.  If $\Sigma_{[I]}$ denotes the stratum determined by the face $\aa_I$ of the alcove with underlying vector subspace $\tt_I\subset \tt$, then an orbifold atlas for $B_{\Sigma_{[I]}}$ is 
\[(H\times \interior(\aa_I^\ss),\Lambda_{G,I}^*)\rtimes W_IZ_I\]
where $Z_I$ is the $Z$-stabilizer of $\aa_I$, $W_I\subset W$ is the subgroup generated by the reflections fixing $\tt_I$ pointwise and  $\Lambda^*_{G,I}=\Lambda_G^*\cap\tt_I$. 
\end{enumerate}

In order to justify these assertions, we start by discussing the Weyl resolution of $(G,L_G)$. From Example \ref{ex:Cartan:Dirac:baby} we have
\[\hM=\{(h,T')\,|\, h\in T',\,T'\subset G\,\,\mathrm{maximal}\,\,\mathrm{torus}\},\quad \res=\mathrm{pr}_1:\hM\to G,\]
and, upon fixing a principal point with centralizer the maximal torus $T$,  an identification 
\[G\times_{N(T)}T=G/T\times_W T \cong  \hM,\quad [(g,h)]\mapsto (ghg^{-1},gTg^{-1}).\]
The groupoid integrating the resolution is the action groupoid
\begin{equation}
\label{eq:res:grpd:conjugacy}
    \hG=G\ltimes (G/T\times_W T)\tto G/T\times_W T
\end{equation} 
associated to left $G$-translations on $G/T$. The associated foliation groupoid is   
\begin{equation}\label{eq:fol-oid-conjugacy}
\cB(\hG)=(G/T\times G/T\times T)/W\tto G/T\times_W T,\quad \s=\mathrm{pr}_{12},\,\t=\mathrm{pr}_{13}.
\end{equation}
Observe that the obvious map $T\hookrightarrow G/T\times_W T$ gives a complete transversal to the orbit foliation of $\hG$. Upon restriction to this transversal, one obtains a presymplectic groupoid $(\hG|_T,i^*\hOmega)$ Morita equivalent to $(\hG,\hOmega)$. The background 3-form pulls back to zero under the inclusion $T\hookrightarrow G/T\times_W T$, so $(\hG|_T,i^*\hOmega)$ is untwisted. The isotropy bundle of $\hG|_T$ is the trivial bundle with fiber $N(T)$, while the associated foliation groupoid is the action groupoid
\[ \cB(\cG|_T)=T\rtimes W\tto T.\]
This action groupoid is effective, since the action of $W$ on $T$ is effective. It follows that the foliation groupoid \eqref{eq:fol-oid-conjugacy} is effective, so it is the holonomy groupoid of $\hF$.
The identity component $\cT(\hG|_T)$ is the trivial bundle $T\times T\to T$, and the pull back of the presymplectic form to this bundle is the symplectic form
\[i^*\hOmega_{(h_1,h_2)}((u_1,v_1),(u_2,v_2))=\langle u_1,v_2\rangle-\langle v_1,u_2\rangle.\]
Therefore the induced integral affine structure on $T$ has lattice at $h\in T$ the translation of 
\[\Lambda_G^*=\{\lambda\in \tt\,|\, \langle u,\lambda \rangle \in \Z,\,\forall u\in \Lambda_G\}\subset \tt\]
(c.f. \cite[Section 4.6.2]{CFM-II}).

In this way we obtain the (classical) integral affine orbifold atlas for $B$
\[(T,\Lambda_G^*)\rtimes W,\]
with universal orbifold covering space
\[(\tt,\Lambda_G^*)\rtimes \pi^\orb_1(B).\]
In particular, we have
\[B^\lin\cong T,\quad \oB=B^\aff\cong \tt.\] 
Moreover, we also obtain the decomposition of $\pi_1^\orb(B)$ into translational and linear parts (c.f.~\eqref{eq:pi1-orb-B-gamma})
\[\xymatrix{
1\ar[r] & \Lambda_G\ar[r] & \pi_1^\orb(B) \ar[r]& W \ar[r]&1.
}
\]
This extension of the Weyl group by the lattice of $G$ coincides with $W^\aff_G$, the extended affine Weyl group of $G$. 
 

Next, we also relate $\cW$ to well-known groups in Lie theory.  
Recall that $\exp^{-1}(T\setminus (T\cap G^\reg))$ is the union of the affine root hyperplanes of the root system $(\gg,\tt)$ (hyperplanes where the roots attain integer multiples of $2\pi i$). The Weyl group  $\cW$ is generated by (orthogonal) reflections on the affine root hyperplanes, and therefore it coincides with the classical affine Weyl group $W^\aff$. Its decomposition according to the translational and linear parts of the affine action gives the first column in \eqref{eq: orbifold-fundamental-group-conjugacy} 
\[\xymatrix{
1\ar[r] & \Lambda_{\mathrm{R}^\vee}\ar[r] & \cW\ar[r] & W\ar[r] &1.
}
\]

Proposition 
\ref{prop:split:orbi:fundamental:group} implies that the image of the splitting
\[i:\pi_1^\orb(B^\reg)\to \pi_1^\orb(B)=W^\aff_G\]
are the elements of $W^\aff_G$ leaving the Weyl alcove $\aa$ stable. Since $\aa=\zz(\gg)\oplus \aa^\ss$ and the second factor is a compact polytope, we conclude that the translation part of $\pi_1^\orb(B^\reg)$
is $\zz(\gg)\cap \Lambda_G$. Since the principal leaves of $G^\reg$ are 1-connected, we can apply Lemma \ref{lem: orbifold fundamental group isomorhism} to conclude that we have isomorphisms
\[ \pi_1^\orb(B^\reg)\cong\pi_1(G^\reg)\cong \pi_1(G). \] 
This yields the right column of diagram \eqref{eq: orbifold-fundamental-group-conjugacy2}, as well as (by diagram chasing) the monomorphism
\[i^\lin:Z\to W.\]
This also shows that the group $\overline{G}$ defined by \eqref{eq: product covering group} is a covering group of $G$.



There is a second Lie theoretical description of $i^\lin$. First, we assume that $\gg$ is semisimple and $G$ is the adjoint group of $\gg$.  We denote by $T'\subset G(\gg)$ the maximal torus above $T$. On the one hand, we have a free and transitive action of the Weyl group on the connected components of $T'\cap G(\gg)^\reg$. On the other hand, we have another action by left/right translation of $Z(G(\gg))$ on the connected components of $T'\cap G(\gg)^\reg$. Upon selecting the connected component containing the fixed base point and comparing the two actions, we get 
\[i_\gg^\lin:Z(G(\gg))\to W.\]
For an arbitrary compact group $G$ we have  
\[i^\lin=i^\lin_{\gg^\ss}\circ \mathrm{pr}_2|_Z,\]
where $\mathrm{pr}_2: \overline{G}\to G(\gg^\ss)$ is the projection in the second factor in 
\eqref{eq: product covering group}. The fact that this coincides with the previous definition of $i^\lin$ follows along the same lines as the discussion in \cite[Corollary 3.9.5]{DK}.


Next we describe the infinitesimal stratification $\cSi(B)$, i.e., the one given by infinitesimal orbit types.
The explicit computation of isotropy Lie algebras in terms of root spaces \cite[Page 161]{DK} implies that they are constant on the points which are in the image of the interior of each face $\exp(\interior(\aa_I))$, where the interior of $\aa_I$ is relative to its supporting affine subspace. It follows that if $G=\overline{G}$ then the faces of $\aa$ parametrize the members of $\cSi(B)$. In general one has to account for the action of $Z$ on the faces of $\aa$ (equivalently, of $\aa^\ss$).  To describe this action  observe first that the covering map 
$\zz(\gg)\times \tt^\ss\to H\times \tt^\ss$ sends the Weyl alcove onto $H\times \aa^\ss$, inducing a bijection on faces (the faces of $H\times \aa^\ss$
being the preimage of the faces of $\aa^\ss$ by the second projection). Diagram \eqref{eq: orbifold-fundamental-group-conjugacy} allows to construct an action of 
$Z$ on $H\times \tt^\ss$ by integral affine transformations that preserves $H\times \aa^\ss$: for any element of $Z$ one chooses a representative of it on $\pi_1(G)$, applies $i$ to the representative, and then projects the corresponding (integral affine map) to an integral affine map on  $H\times \aa^\ss$ (which acts trivially on the first component). 
Then the strata of $\cSi(B)$ are parametrized by the orbits of $Z$ acting on the faces of $\aa$.  

Let $[I]$ denote one such orbit and let $\Si_{[I]}\in \cSi(G)$ be its associated stratum. To describe it we denote by $Z_I$ the isotropy group of $\aa_I$ for the action of $Z$ on the faces of $\aa$, by $\pp_I\subset \gg$ the isotropy Lie algebra on points in $\exp(\interior(\aa_I))$ and by $P_I\subset G$ its connected integration. To parameterize $\Sigma_{[I]}$ we take first its preimage by $\kappa:\overline{G}\to G/Z$, obtaining a disconnected union of manifolds
\[ G/P_J\times (H\times \interior(\aa_J^\ss)), \quad J\in [I]. \]
It then follows that we can parameterize the strata by the diffeomorphism
\begin{equation}\label{eq:conjugacy-triv0}
G/P_I\times_{Z_I} (H\times \interior(\aa_I^\ss))\to \Sigma_{[I]},\quad [(gP_I,h,u)]\mapsto g\kappa\big(h\exp_{\overline{G}}(u)\big)g^{-1},
\end{equation}
where the action of $Z_I$ on $G/P_I$ is via $i^\lin$.


Lifting \eqref{eq:conjugacy-triv0} along the resolution map gives a parameterization of the corresponding strata of the resolution, namely
\[\hSi_{[I]}\cong ((G/T)/W_I)\times_{Z_I} (H\times \interior(\aa_I^\ss)),\]
where $W_I\subset W$ is the linear part of the subgroup of $W^\aff$ generated by reflections on affine root hyperplanes containing $\aa_I$, and we use that $Z_I$ normalizes $W_I$.

The orbifold integral affine structure on $B_{\Sigma_{[I]}}\in S^\inf(B)$
comes from the reduction of \eqref{eq:res:grpd:conjugacy} along $\hSi_{[I]}$. Namely, from \eqref{eq:fol-oid-conjugacy} we obtain first the foliation groupoid
\[(G/T\times G/T\times(H\times \interior(\aa_I^\ss))/W_IZ_I \tto ((G/T)/W_I)\times_{Z_I} (H\times \interior(\aa_I^\ss)).\]
This is Morita equivalent to the ineffective orbifold
\[(H\times \interior(\aa_I^\ss))\rtimes W_IZ_I,\]
which gives the desired orbifold atlas for $B_{\Sigma_{[I]}}$. Next, for its integral affine structure, we consider the embedding 
\[H\times \interior(\aa^\ss_I)\to G,\quad (h,u)\mapsto [h\exp_{\overline{G}(u)}],\]
which is a full transversal to the presymplectic foliation. The presymplectic torus bundle over $\hSi_{[I]}$ pull backs to this transversal to a trivial bundle with fiber  isomorphic to $Z(P_I)$. Its Lie algebra $\zz(\pp_I)$ is the vector subspace underlying $\zz(\gg)\times \aa^\ss_I$,
and one checks that contraction by the presymplectic form followed by duality takes $\Lambda_{G,I}=\ker(\exp:\zz(\pp_I)\to Z(P_I))$ to 
$\Lambda_{G}^*\cap \zz(\pp_I)$, which yields the desired integral affine structure on $B_{\Sigma_{[I]}}$.



We finish the discussion on conjugacy classes with three comments about strata:
\begin{itemize}
\item The action of  $Z$ by permutations of the vertices of the alcove is explicitly determined for every simple Lie group (see \cite{IM65}), which
describes  equivalently the action on codimension one faces). Therefore one has an explicit description of the members of $\cSi(B)$ and their integral affine orbifold structures for every given compact Lie group. 
\item The partition of $B_{\Si_{[I]}}$ by canonical strata is given by the orbit type stratification of the action of $Z_I$ on $H\times \interior(\aa_I)$. 
\item The integral affine orbifold atlas on  $B_{\Si_{[I]}}$ obtained from the reduction of $(\cG,\Omega)$ along $\Si_{[I]}$ is
\[(H\times \interior(\aa_I^\ss),\Lambda_G^*\cap \zz(\pp_I))\rtimes Z_I.\]
Notice that this is an isotropy reduction (by $W_I$) of the action groupoid in item (v). 
It is effective if and only if $Z_I$ acts effectively on $\aa^\ss_I$. This is equivalent to $\Si_{[I]}$ have 1-connected principal leaves, necessarily diffeomorphic to $G/P_I$. The holonomy and monodromy groupoids of such a (proper) foliation must coincide.
\end{itemize}

\subsection{Removal of closed saturated subsets}
\label{sec:examples:parts:removed}
Let $(M,\pi)$ be a Poisson manifold of proper type with fixed symplectic integration $(\cG,\Omega)$. The restriction of this symplectic groupoid to any open saturated submanifold $M_0\subset M$ is a symplectic integration of $(M_0,\pi)$. Let $B_0$ denote the leaf pace $\hM_0/\hF$. The relation
between the Weyl groups of  these two PMCTs and of the orbifold fundamental groups of their leaf spaces is described in the following commutative diagram of  split short exact sequences:
\begin{equation}
\label{eq:Weyl:main-diagram-restric}
\vcenter{
\xymatrix@C=15pt@R=10pt{
\hW_0 \ar[rrr] \ar[dd] \ar[dr]& & & \pi_1(\hM_0)\ar[rr]^(.3){\res_*} \ar@{-}[d] \ar[dr]&  & \pi_1(M_0^\reg) \ar@{-}[d] \ar[dr]\\
& \hW\ar[rrr]\ar[dd] & & \ar[d] & 
\pi_1(\hM)\ar[rr]^(.3){\res_*}\ar[dd] & \ar[d] &
\pi_1(M^\reg)\ar[dd]\\
\cW_0 \ar[dr] \ar@{-}[r] & \ar[rr] & & \pi_1^\orb(B_0) \ar[dr] \ar@{-}[r] &\ar[r] &  \pi_1^\orb(B_0^\reg)\ar[dr]\\
  & \cW\ar[rrr] & & & \pi_1^\orb(B)\ar[rr] & & \pi_1^\orb(B^\reg)
}}
\end{equation}
Note that on has also a local diffeomorphism 
\[ \oB_0\to \oB, \]
which is equivariant relative to the group morphism $\pi_1^\orb(B_0)\to \pi_1^\orb(B)$. 

We assume from now on that $M^\reg=M^\reg_0$, so the top and bottom horizontal diagrams become morphisms between split extensions of the groups $\pi_1(M^\reg)$ and $\pi_1^\orb(B^\reg)$, respectively.

\begin{example}
\label{ex:su(2):R}
Consider the Poisson manifold
\[ M:=(\R\times\mathfrak{su}(2))^*\cong\R\times\mathfrak{su}^*(2),\]
with the linear Poisson structure. We can take $T^*(\S^1\times \SU(2))\tto M$ as an integration of this Poisson manifold, so it is of s-proper type. Since the maximal rank of the Poisson structure is $2$, there are only regular and subregular strata. The origin is a subregular point and we let
\[ M_0:=M\setminus \{0\}.\]

It follows from Section \ref{ex:regular-coadjoint-orbits}, that the resolutions of these Poisson manifolds can be identified with
\[ 
\hM=(\S^2\times_{\Z_2}\R^2), \quad
\hM_0=(\S^2\times_{\Z_2}(\R^2\setminus\{0\}),
\]
where the non-trivial element of $\Z_2$ acts on $\S^2$ by the antipodal map and on $\R^2$ by
\[(x,y)\mapsto (x,-y).\]
Here $\R^2=\R\times\tt^*$ where $\tt\subset\su(2)$ is a maximal torus. Henceforth, we identify $\R^2=\C^2$ so the action of $\Z_2$ becomes $z\mapsto \overline{z}$.

The leaf spaces are now
    \[
    \xymatrix{B_0=(\C\setminus\{0\})/\Z_2\
    \ar@{^{(}->}[r] &  \ B=\C/\Z_2,}\]
with orbifold atlas the action groupoids
\[
\C\ltimes (\Z\rtimes\Z_2),\quad \C\rtimes\Z_2,
\]
with projections $p_0:\C\to \C\setminus\{0\})/\Z_2$ and $p:\C\to \C/\Z_2$ given by
\[ p_0(z)=[e^{2\pi z}], \quad p(z)=[z],\]
and where $\Z_2$ acts in both cases by complex conjugation and $n\in \Z\subset\C$ acts by translations $z\mapsto z+in$. Since $M$ and $M_0$ are 1-connected, it follows that the Weyl groups coincide with the orbifold fundamental groups and, using \eqref{eq:pi1-orb-B-gamma}, we obtain:
    \[ \cW_0=\pi_1^\orb(B_0)=\Z\rtimes\Z_2, \quad \cW=\pi_1^\orb(B)=\Z_2.\]
Furthermore the morphism $\pi_1^\orb(B_0)\to \pi_1^\orb(B)$ in the diagram is the projection
    \[ \Z\rtimes\Z_2\to \Z_2. \]

At the level of the universal orbifold covering spaces, both diffeomorphic to $\C$, one now obtains the commutative diagram
\[
\xymatrix{
\oB_0\ar[r]^{\tilde{i}}\ar[d]_{p_0} &  \oB\ar[d]^{p}\\
B_0 \ar@{^{(}->}[r] & B
}
\]
where the top horizontal arrow $\tilde{i}$ is the map
\[ z\mapsto e^{2\pi z}. \]
This map is equivariant relative to the actions of the orbifold fundamental groups.

There is only one reflection in $\cW$, whose  fixed hyperplane is the real axis $\{x=0\}$. Hence we obtain two Weyl chambers $\Delta_\pm$, namely the closed upper and lower half-planes in $\C$. On the other hand, in $\oB_0$ by Proposition \ref{eq: equality regulars} the hyperplanes are the connected components of the preimage $p_0^{-1}(B_0^\sing)=p_0^{-1}(\{y=0\})$, namely
\[ \cH_n:= \big\{x+i\tfrac{n}{2}: x\in\R\big\},\quad (n\in \Z), \]
where the even $\cH_n$ are mapped into the half line $\big\{(x,0):x>0\big\}$ and the odd ones into the negative half line $\big\{(x,0):x<0\big\}$. Hence the chambers in $\oB_0$ the strips 
\[ \Delta_n:=\{x+iy\in\C:\tfrac{n}{2}\leq y\leq\tfrac{n+1}{2}\},\quad (n\in\Z).\]
Notice that the translation $(1, 0)\in \Z\rtimes\Z_2$ sends $\Delta_n$ to $\Delta_{n+2}$, while the reflection $r_0= (0,[1])$ sends $\Delta_n$ to $\Delta_{-n}$; this is consistent with the fact that the Weyl group acts freely and transitively on the set of Weyl chambers.
The set $\Refl\subset \cW_0$ consist of the reflections of the form $r_n(x,y)=\left(x,n-y\right)$, $(n\in\Z)$. The reflections $\Refl_{\Delta_n}=\{r_n,r_{n+1}\}$ give a Coxeter system with solo relations
    \[ r_{n}^2=r_{n+1}^2=1. \]
This describes $\cW$ as the free product $\Z_2*\Z_2\cong \Z\rtimes\Z_2$.

The integral lattice on $\oB$ is  a product lattice $\Lambda=\lambda_1\Z\d x+\lambda_2\Z\d y$, for fixed positive numbers $\lambda_i>0$.  Its pullback by the map $\oB_0\to \oB$ 
is the integral lattice defining the integral affine structure on $\oB_0$. The Weyl groups clearly act by integral affine transformations on these spaces. Moreover, we have the equalities 
\[\oB=B^\aff=B^\lin=\C,\quad B^\aff_0=B^\lin_0=\C\backslash \{0\},\]
and  $\tilde{i}:\oB_0\to \oB$ can be viewed as the developing map $\dev: \oB_0\to \R^2$ (modulo rescaling so that the product lattice $\Lambda$ the standard integer lattice). 
\end{example}

\begin{example}
    Consider the Cartan-Dirac structure on $M=\SU(2)$ and let $M_0=\SU(2)\backslash\{-I\}$. Then it follows from Examples \ref{ex:Weyl-group:G-Dirac} and \ref{ex:Weyl-group:G-Dirac-minus} that the orbit spaces are 
    \[
    \xymatrix{B_0=(\S^1\backslash \{-1\})/\Z_2\
    \ar@{^{(}->}[r] &  \ B=\S^1/\Z_2,}\]
    while at the level of universal covering spaces one has an inclusion
    \[ \oB_0=\left(-\tfrac{1}{2},\tfrac{1}{2}\right)\hookrightarrow \oB=\R,\]
    with orbifold covering projection 
    \[ \oB\to B,\quad t\mapsto [e^{2\pi it}]. \]
    On the other hand, the Weyl groups and the orbifold fundamental groups coincide  
    \[ \cW_0=\pi_1^\orb(B_0)=\Z_2, \quad \cW=\pi_1^\orb(B)=\Z\rtimes\Z_2\]
    and the morphism $\pi_1^\orb(B_0)\to \pi_1^\orb(B)$ in the diagram is the inclusion
    \[ \Z_2\hookrightarrow \Z\rtimes\Z_2,\quad a\mapsto (0,a). \]
    Note that $\Z_2$ acts by $x\mapsto -x$ and $\Z$ acts by translations. Finally, the Weyl chambers in $\oB$ are the intervals $\Delta_n:=[\tfrac{n}{2},\tfrac{n+1}{2}]$, ($n\in\Z$). Intersecting with $\oB_0$ one obtains the two Weyl chambers of $\oB_0$, namely $\left(-\tfrac{1}{2},0\right]$ and $\left[0,\tfrac{1}{2}\right)$. The set $\Refl\subset \cW$ consist of the reflections of the form $r_n(t)=n-t$, $(n\in\Z)$. The reflections $\Refl_{\Delta_n}=\{r_n,r_{n+1}\}$ give a Coxeter system with solo relations
    \[ r_{n}^2=r_{n+1}^2=1. \]
    This describes $\cW$ as the free product $\Z_2*\Z_2\cong \Z\rtimes\Z_2$.
    
    Note that the integral lattice in the universal covering spaces is determined by the choice of $\Ad$-invariant inner product in $\su(2)$, which is unique up to scaling. As we vary the inner product we obtain all lattices $\Lambda=\lambda_0\Z\,\d t$, $(\lambda_0>0)$. The Weyl group clearly acts by integral affine transformations for any such choice.
    \end{example}

In the previous two examples, we saw that by removing a part of the subregular strata may create or destroy topology on the leaf spaces, and may reduce or increase the size of the Weyl group. We will now see that when one removes parts of lower strata one can only increase the size of the Weyl group. 

\begin{proposition}
    \label{prop:remove:part}
    Let $(M,\pi)$ be a Poisson manifold of proper type and assume that $M_0\subset (M,\pi)$ is an open saturated set such that
    \[ M_0^\reg=M^\reg\quad \textrm{and}\quad M_0^\subreg=M^\subreg. \]
    Then the following holds:
    \begin{enumerate}[(i)]
    \item The canonical maps $\cW_0\to \cW$, $\pi_1^\orb(B_0)\to \pi_1^\orb(B)$ are surjective and have a common kernel $\cK$.
    \item $\cK$ is isomorphic to the fundamental group of the pre-image of $B_0$ by the canonical projection $\oB\to B$.
    \item $\cK= 0$ whenever $M_0$ contains all the points $x\in M$ with the property that the semisimple part of $\gg_x$ is 2-dimensional.
    \end{enumerate}
\end{proposition}
\begin{proof}
The hypothesis imply that $\hM_0$ contains the regular and subregular strata of $\res^{-1}(\cSi(M,\pi))$. By
Lemma \ref{lem:stratif-lemma-4}
we deduce that the inclusion induces a surjection $\pi_1(\hM_0)\to \pi_1(\hM)$. Therefore the top diagonal arrows in Diagram \eqref{eq:Weyl:main-diagram-restric} fit into an epimorphism of short exact sequences.
Its kernel is then another short exact sequence. Since its rightmost group is trivial, the first two kernels are identified. Because the horizontal arrows in  \eqref{eq:Weyl:main-diagram-restric} are epimorphisms, the same result holds for the  short exact sequences at the level of leaf spaces. This shows that 
$\cW_0$ and $\pi_1^\orb(B_0)$ are extensions of $\cW$ and $\pi_1^\orb(B)$ by the same group $\cK$, so (i) holds. 

To prove item (ii), we use that for any open subset of $B$, its linear holonomy cover is an open subset of $B^\lin$ (see Proposition \ref{rk:functoriality-main-diagram}). This reduces statement (ii) from one for (good) orbifolds to one for open subsets in manifolds and ordinary fundamental groups, as we did in proof  of Proposition \ref{pro:chamber-covering}. Therefore item (ii) follows. 

To prove item (iii) we observe that the assumption on isotropy Lie algebras is equivalent to the codimension 2 strata of $\res^{-1}(\cSi(M,\pi))$ being contained in  $\hM_0$. Then the map $\pi_1(\hM_0)\to \pi_1(\hM)$ is an isomorphism, and so all diagonal maps in \eqref{eq:Weyl:main-diagram-restric} are trivial. 
\end{proof}

\begin{remark} If in the previous proposition $(M,\pi)$ is source-proper and has  1-connected principal leaves, then also the canonical maps $\cW_0\to \cW$, $\pi_1^\orb(B_0)\to \pi_1^\orb(B)$ are surjective and have kernel $\cK$. In this case, $\cK$ also coincides with the fundamental group of the pre-image of $\hM_0$ by the universal covering projection of $\hM$. Indeed, if the principal leaves as 1-connected, then by Lemma \ref{lem: orbifold fundamental group isomorhism} the vertical arrows in diagram \eqref{eq:Weyl:main-diagram-restric} are isomorphisms. Therefore the kernel of 
$\pi_1^\orb(B_0)\to \pi_1^\orb(B)$ is identified with the kernel of $\pi_1(\hM_0)\to \pi_1(\hM)$. Since s-properness implies that $\hM^\orb$
is the universal covering space of $\hM$, then $\cK$ coincides with the fundamental group of the preimage  $p^{-1}(M_0)\subset \hM^\orb=\widetilde{\hM}$.
\end{remark}   

\begin{example}
Let $\gg$ be a semi-simple, rank 2, compact Lie algebra. We consider the linear Poisson manifold $M=\gg^*$. It is strong s-proper with source 1-connected integration $(T^*G,\omega_\can)$, where $G$ is the 1-connected Lie group integrating $\gg$. Besides the regular and subregular strata, we have one further stratum that consists of the origin. We let
\[ M_0:=M\setminus \{0\}.\]

It follows from Section \ref{ex:regular-coadjoint-orbits}, that the resolution of this Poisson manifold can be identified with
\[ 
\hM=(G/T\times_W\R^2), \quad
\hM_0=(G/T\times_W(\R^2\setminus\{0\}).
\]
where $\R^2$ appears as the dual $\tt^*$ of a maximal torus $\tt\subset\gg$ and $W$ is the classical Weyl group. It is a known fact from Lie theory that the simple, rank 2, compact Lie algebras and their Weyl groups are:
\medskip
\begin{center}
\begin{tabularx}{.9\textwidth} { 
  | >{\centering\arraybackslash}X 
  | >{\centering\arraybackslash}X
  | >{\centering\arraybackslash}X
  | >{\centering\arraybackslash}X 
  | >{\centering\arraybackslash}X | }
\hline
$\gg$ &$\su(2)\oplus\su(2)$ & $\su(3)$ & $\mathfrak{so}(5)$ & $\gg_2$ 
\\
\hline
$W$  & $D_2$ & $D_3$    & $D_4$ & $D_6$
\\
\hline
\end{tabularx}
\end{center}
Here the Weyl group $D_2$ acts on $\C=\R^2$ by reflection in the coordinate axes and in the other three cases $D_k$ acts on $\C$ as the symmetry group of the regular polygon whose vertices are the $k$-th roots of unity. 
The leaf spaces are now
    \[
    \xymatrix{B_0=(\C\setminus\{0\})/D_k\
    \ar@{^{(}->}[r] &  \ B=\C/D_k,}\]
with orbifold atlas the action groupoids
\[
\C\ltimes (\Z\rtimes \Z_2),\quad \C
\rtimes D_k,
\]
with projections, respectively, 
\[ p_0(z)=[e^{2\pi z}], \quad p(z)=[z].\]
In the first action groupoid, $\Z_2$ acts by complex conjugation and the $\Z$ action is generated by  $z\mapsto z+\tfrac{i}{k}$.

The situation is entirely similar to Example \ref{ex:su(2):R} and, as there, we find that  the Weyl groups coincide with the orbifold fundamental groups and are given by:
    \[ \cW_0=\pi_1^\orb(B_0)=\Z\rtimes \Z_2, \quad \cW=\pi_1^\orb(B)=D_k\]
The morphism $\pi_1^\orb(B_0)\to \pi_1^\orb(B)$ in diagram \eqref{eq:Weyl:main-diagram-restric} is the projection
    \[ \Z\rtimes \Z_2\to D_k=\Z_k\rtimes\Z_2,\quad (n,a)\mapsto (n\,\mathrm{mod}\,k,a). \]
Its kernel $\cK=\Z$ is identified via  the map $\tilde{i}:\oB_0\to\oB$ with the fundamental group of 
\[p^{-1}(B_0)=\C\backslash \{0\}\subset \oB=\C.\]
One can check that, in contrast to Example \ref{ex:su(2):R}, $\cW_0$ is a non-split extension of $\cW$ in all 4 cases.

The Weyl chambers are the strips (see Figure \ref{fig:Weyl})
\[ \Delta_n:=\{x+iy\in\C:\tfrac{n}{2k}\leq y\leq\tfrac{n+1}{2k}\},\quad (n\in\Z).\]
The translation $(1,0)\in D_k$ sends $\Delta_n$ to $\Delta_{n+2k}$, whereas the reflection $(0,[1])\in D_k$
sends $\Delta_n$ to $\Delta_{-n}$.

\begin{figure}[h!]
 \begin{center}
          \setlength\epsfxsize{2cm}
         \leavevmode
\includegraphics[scale=0.9]{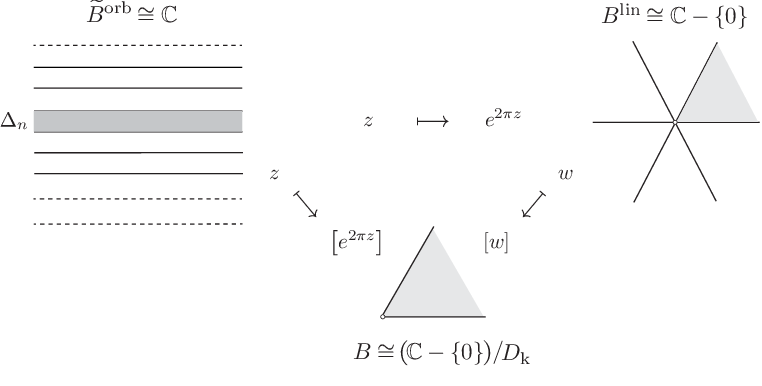}
 \end{center}
 \caption{Orbifold cover and Weyl chambers}
 \label{fig:Weyl}
 \end{figure}

Observe finally that for all four Lie algebras the corresponding $B_0$ are isomorphic orbifolds, since one has isomorphisms of orbifold atlases (Morita equivalences) 
\[(\C\setminus \{0\})\rtimes (\Z_k\rtimes \Z_2)\cong (\C\setminus \{0\})/\Z_k\rtimes \Z_2\cong (\C\setminus \{0\})\rtimes \Z_2.\]
However, they are non-isomorphic (integral) affine orbifolds since, for instance, the linear holonomy is $D_k$. 
\end{example}

\section{The Linear Variation and the Volume Polynomial}
\label{sec:variation:volume}

In this section we will assume that we have a Poisson manifold $(M,\pi)$ of $\s$-proper type. We will apply the theory developed in \cite{CFM-II} to its resolution $(\hM,\hL)$, in order to obtain statements about the variation of the cohomology classes of the symplectic leaves of $(M,\pi)$ and the variation of the symplectic volumes.

\subsection{The linear variation of cohomology}
\label{sec:linear:variation}

As described in \cite[Section 5]{CFM-II} one can interpret the cohomology classes of the of leafwise presymplectic form
$\hF$ as a section of a vector bundle over $\hM$ whose fibers are the second cohomology groups of the s-fibers of $\Hol(\hM,\hF)$. We denote the latter vector bundle by
\[ \hH\to \hM, \quad \hH_{\hx}:=H^2(\Hol(\hx, -)).\]
The section provided by the presymplectic forms will be denoted by
\[ \varpi:\hM\to \hH, \quad \varpi_{\hx}:=[\t^*\widehat{\omega}_{\hS_\hx}], \]
where $\widehat{\omega}_{\hS_\hx}$ denotes the $\hS_\hx$ denotes the presymplectic leaf containing $\hx$ and $\t:\Hol(\hM,\hF)\to\hM$ is the target map.

The vector bundle $\hH$ is equipped with several structures:
\begin{enumerate}[(i)]
    \item An integral affine structure defined by the image of the integral cohomology $H^2(\Hol(\hx,-),\Z)$. This gives a canonical flat connection on $\hH$, the Gauss-Manin connection, and hence a representation of the groupoid $\Pi_1(\hM)$ on $\hM$;
    \item A representation of the groupoid $\Hol(\hM,\hF)$, where the action is via the map induced in cohomology by right translations.
\end{enumerate}
The two groupoids in the previous items fit into a commutative diagram
\[\xymatrix{
\Mon(\hM, \hF) \ar[r]\ar[d] & \Pi_1(\hM)\ar[d] \\
\Hol(\hM, \hF) \ar[r] & \Pi_1(\Hol(\hM, \hF)) 
} \]
where we are making use of the fundamental group from section \ref{sec:The orbifold fundamental group} and the canonical maps (\ref{eq:fund-gpd-canonical-maps}). 

\begin{proposition}
\label{prop:rep:pi:hol}
    There is a representation of $\Pi_1(\Hol(\hM, \hF))$ on $\hH$ compatible with the representations of $\Hol(\hM,\hF)$ and $\Pi_1(\hM)$. Furthermore, the section $\varpi$ is $\Hol(\hM,\hF)$-invariant.
\end{proposition}

\begin{proof}
We observe that:
\begin{itemize}
    \item[(C)] given an element $[\gamma]\in\Hol(\hM,\hF)$ its action on $\hH$ coincides with the action of $[\gamma]$ via the Gauss-Manin connection.
\end{itemize}
 Using this, one obtains:
\begin{enumerate}[(a)]
    \item the kernel $K$ of the map $\pi_1(\hM,\hx_0)\to \pi_1(\Hol(\hM,\hF),\hx_0)$ acts trivially on $\hH$. By \eqref{eq:F-homotopy}  an element in $K$ is represented by a loop homotopic to a concatenation of paths of the form $\alpha_i \beta_i \alpha_i^{-1}$, where each $\beta_i$ is a leafwise path with trivial holonomy. Hence, by (C), each $\beta_i$ acts trivially by the Gauss-Manin connection, and so does $\alpha_i \beta_i \alpha_i^{-1}$.
    \item By (a) the action of $\Pi_1(\hM)$ descends to an action of $\Pi_1(\Hol(\hM,\hF))$. By (C) this action is also compatible with the action of $\Hol(\hM,\hF)$.
\end{enumerate}
The $\Hol(\hM,\hF)$-invariance of $\varpi$ follows from its definition and the fact that the target is invariant under right translations.
\end{proof}

The previous proposition shows that we can view $\hH$ has a orbivector bundle over the classical orbifold $B$, which is equipped with an action of $\pi_1^\orb(B)$, and has a canonical section $\varpi$. Recall that $B$ can be presented by various atlases of a discrete group acting on a smooth manifold:
$B^\lin\rtimes\Gamma^\lin\tto B^\lin$, $B^\aff\rtimes \Gamma^\aff\tto B^\aff$ and $\oB\rtimes \pi_1^\orb(B)\tto\oB$. For each of these one obtains an equivariant vector bundle  and an equivariant section, representing $\hH$ and $\varpi$, which will be denoted by the same letter. In particular the fiber at the base point will be denoted by $\hH_0$. Note that it coincides with the 2nd cohomology group of the generic symplectic leaf $S_{x_0}$. The value of $\varpi$ at this leaf will be denoted $\varpi_0$.

We define the \textbf{transport map} (denoted $\Var$ in \cite{CFM-II}) using the representation of $\Pi_1(\hM)$ on $\hH$ by: 
\[ \cT_{\varpi}: \oB\to \hH_{0}, \quad [\gamma]\mapsto \gamma^*\varpi_{\gamma(1)} \]
where we view $\oB$ as the universal cover of $B^{\lin}$, hence $\gamma$ is a path in this space starting at the fixed base point. It follows from the next theorem that this maps descends to a similar map in $B^\aff$.

Recall that we have the (based) developing map 
\[ \dev:B^\aff\to \tt_0, \]
where we use the identification of the tangent space of $B^\aff$ at the base point with $\tt_0$. Also, denoting by $\nabla$ the Gauss-Manin connection on $\hH\to B^\aff$, we have the \textbf{linear variation map} defined by
\[ \Ilin: \tt_0\to \hH_0,\quad v\mapsto \nabla_v\varpi, \]
where $\hH_0$ is the fiber of $\hH$ at the fixed base point. We also consider the \textbf{affine variation map}, defined by
\[ \Iaff:=\varpi_{0}+\Ilin. \]
The linear variation of the cohomology classes of the symplectic leaves of $(M,\pi)$, can finally be stated as follows.

\begin{theorem}
\label{thm:linear:variation}
    The transport map descends to $B^\aff$ yielding the following commutative diagram:
    \[ \xymatrix{ 
B^{\aff} \ar[rr]^-{\mathcal{T}_\varpi} \ar[rd]_-{\dev} &    & \hH_{0}      \\
 & \tt_{0} \ar[ru]_-{\Iaff}                             &            
 }, 
\]
\end{theorem}

\begin{remark}
The reader may wonder why the previous result expresses the linear variation of the cohomology classes of the symplectic leaves of $(M,\pi)$. For that, observe that the resolution map yields for each symplectic leaf $S_x$, a map in cohomology:
\[
\xymatrix{
H^2(S_x)\ar[r]^{\res^*} & H^2(\hS_{\hx})\ar[r]^---{\t^*} & H^2(\Hol(\hM,\hF)(\hx,-))\simeq \hH_0} \]
where the last isomorphism is obtained by choosing a path connecting $\hx$ to the base point $\hx_0$. So the theorem states that the variation (of the pullback) of the cohomology classes $[\omega_{S_x}]$ becomes linear in any integral affine chart. This will also be clear in the proof below.


A more explicit description can be obtained by fixing a $\Z$-basis $\mathfrak{b}_{\Lambda}$ for $\Lambda_{\hx_0}\subset \tt_0$. This induces an identification of $\cG_{x_0}^{0}$ with the standard torus $\T^r$ via the exponential map. 
On the other hand, denoting $\G= \cG_0$,  one has the principal $G$-bundle 
\[ \t: \cG(x_0, -)\to S_{x_0}\]
and the associated Chern-Weyl  homomorphism is defined on 
\[ S(\gg^*)^G\cong S(\tt_0^*)^{G/G^0}= S(\tt_0^*)\cong \R[X_1, \ldots , X_r]. \]
Here the last identification is the one induced by the basis $\mathfrak{b}_{\Lambda}$, and the middle equality derives from the fact that $x_0$ is in the principal part. Therefore our principal bundle has associated Chern classes, denoted 
\[ c_1, \ldots, c_r\in H^2(S_{x_0}).\]
Finally, the developing map will have coordinates $\dev^{k}$ w.r.t. the basis $\mathfrak{b}_{\Lambda}$. 

All together, the linear variation formula boils down to a commutative diagram 
\[\xymatrix{
\oB  \ar[rr]^-{\Varb} \ar[rd]_-{{\dev_0}}  &    &   H^2(S)        \\
 & \R^r  \ar[ru]!<+10pt,0pt>_-{\quad(v^i)\mapsto[\omega_0]+\sum_i v^i c_i}                                     & &             
}\]
More explicitly, for any path $\gamma$ in $\hM$ starting at $x_0$, one has  
\[ \gamma^*[\omega_{\gamma(1)}]- [\omega_{0}]=  \dev^{1}(\gamma)c_1+ \ldots + \dev^{r}(\gamma)c_r.\]

In the strong s-proper case, \cite[Corollary ]{CFM-II} shows that the classes $c_1,\dots,c_r$ form a primitive family, i.e., we have
\[  \mathrm{Span}_{\Z}(c_1, \ldots , c_r)= \mathrm{Span}_{\R}(c_1, \ldots , c_r)\cap H^2(S_{x_0}, \Z).\]
\end{remark}

\begin{proof}
The result will follow from \cite[Theorem 5.2.4]{CFM-II} applied to the resolution of $(M, \pi)$, viewed as a regular Dirac manifold of $\s$-proper type.

The aforementioned theorem provides the commutative diagram
\begin{equation}\label{eq: diagram variation PMCTII} 
\vcenter{\xymatrix{ 
\Pi_1(\hM) \ar[rr]^-{\mathcal{T}_{\varpi}} \ar[rd]_-{\dev} &    & \hH_\cB    \\
& \nu^{\aff}(\hF) \ar[ru]_-{\Iaff}                                     &            
}}
\end{equation}
The bundle $\hH_\cB$ where the transport map takes values differs slightly from our cohomology bundle: it is constructed from the atlas $\cB=\cB(\hG)$ defining the orbifold structure on $B$, rather than the classical one.  More precisely, it is the bundle over $\hM$ whose fiber is the second cohomology of the source fibers $\cB(\hx,-)$ of the given orbifold atlas.  In this case $\nu^{\aff}(\hF)= \htt$, seen as an affine bundle, and $\Iaff$ is defined as a above, but considering $\varpi$ as section of $\Gamma(\hH_\cB)\to\hM$.

To obtain the commutative diagram of the statement we take the following steps:
\begin{enumerate}[(i)]
\item Replace the $\pi_1(\cB)$-representation $\hH_\cB$ by the $\pi_1(\Hol(\hM, \hF))$-representation $\hH$: first the covering maps $\cB(\hx,-)\to \Hol(\hM,\hF)(\hx, -)$ induce an injection
\[ \hH\to \hH_{\cB} \]
with image the fixed-point sets for the action of the groups of deck transformations. This inclusion preserves the integral affine structure given by integral cohomology and therefore is $\Pi_1(\hM)$-equivariant. The aforementioned covers arise from the short exact sequence of foliation groupoids
\[
\xymatrix{1\ar[r] & \cK \ar[r] & \cB\ar[r] &  \Hol(\hM,\hF)\ \ar[r] & 1}.
\]
We also have a commutative diagram of fundamental groupoids 
\[
\xymatrix{  & \Pi_1(\hM)\ar[dl]\ar[dr]\\
\Pi_1(\cB) \ar[rr] & & \Pi_1(\Hol(\hM,\hF))}
\]
where the diagonal arrows are surjective. As in the proof of Proposition \ref{prop:rep:pi:hol}, the representations of $\Pi_1(\cB)$ and $\Pi_1(\Hol(\hM,\hF))$ on $\hH_\cB$ and $\hH$ are completely determined by the action of $\Pi_1(\hM)$. It follows that the inclusion $\hH\to\hH_\cB$ is equivariant. 

Finally, notice that the injection $\hH\to \hH_{\cB}$ preserves the sections $\varpi$.

\item Drop down to $B^\aff$: 
the developing map descends to $\Pi_1^\aff(\hM)$, and, hence, so does the whole diagram.
Restricting to the s-fiber above $\hx_0$ one obtains
\[ \xymatrix{ 
\hM^{\aff} \ar[rr]^-{\cT_{\varpi}} \ar[rd]_-{\dev} &    & \hH_0      \\
& \tt_0 \ar[ru]_-{\Iaff}                                     &            
}\]
The last diagram also descends to $B^\aff$ giving the commutative diagram in the statement.
\end{enumerate}

\end{proof}

A Poisson manifold of $\s$-proper type around a symplectic leaf is described by the linear local model (see Section \ref{sec:ex:local-model}). The linear variation of cohomology for such a linear local model is described explicitly in the following example. 

\begin{example}\label{example:D-H}
We described in Example \ref{ex:linear local model resolution} the resolution of the linear local model. Fixing the data of a principal $G$-bundle $p:P\to (S,\w_S)$, a connection 1-form $\theta\in \Omega^1(P,\gg)$, and  choice of maximal torus $T\subset G$, 
the resolution $\res: \hM\to M$ of the linear Poisson model is
\[\res:((P\times V)/N(T),\widehat{L}_{\pi^\theta_\lin})\to ((P\times U)/G,\pi^\theta_{\lin}),\quad [(p,\xi)]_{N(T)}\to [(p,\xi)]_G,\]
where the Dirac and Poisson tensors are built out of reduction  of the forms 
\begin{equation}\label{loc-mod-2-form:resolution}
p^*\omega_S-\d\langle \theta^T,\cdot\rangle \in \Omega^2(P\times \tt^*),\quad  p^*\omega_S-\d\langle \theta,\cdot\rangle \in \Omega^2(P\times \gg^*).
\end{equation}
Here $\theta^T:=\mathrm{pr}_\tt\circ \theta$ and
$V$ is the intersection with $\tt^*$ of a small saturated neighborhood $U$ of $0\in \gg^*$ (so that $\pi_\lin^\theta$ is Poisson).

We can move in the big diagram one step up. For simplicity, we assume that the centralizer $Z_G(T)=T$. We obtain that
\[ \hM^\aff=\hM^\lin:=(P\times U)/T=P/T\times U, \]
with Dirac structure still induced by the form \eqref{loc-mod-2-form:resolution}. Notice that now this form is the pullback under the projection $P\times\tt^*\to P/T\times\tt^*$ of the form
\begin{equation}
\label{loc-mod-2-form:affine:cover}
\underline{\omega}:=q^*\omega_S-\d\langle \theta^T,\cdot\rangle \in \Omega^2(P/T\times \tt^*), 
\end{equation}
where $q:P\to P/T$ is the quotient map. In other words, the Dirac structure on $\hM^\lin$ has presymplectic leaves $P/T\times{\xi}$ with presymplectic forms the restriction of the globally defined closed 2-form \eqref{loc-mod-2-form:affine:cover}.
    
Fixing a basis for $\tt$, the principal torus bundle $q:P\to P/T$ has Chern classes $[c_i]\in H^2(P/T)$ represented by the components of the curvature form of $\theta^T$, i.e.,
\[ -\d\theta^T=q^*(c_1,\dots,c_r). \]
The presence of the minus sign is due to our convention that Lie algebras and Lie algebroids are constructed using right-invariant vector fields. 

The closed 2-form \eqref{loc-mod-2-form:affine:cover} becomes
\[ \underline{\omega}:=q^*\omega_S+\sum_{i=1}^r c_i\,\xi^i,
\]
for $\xi=(\xi^1,\dots,\xi^r)\in\tt^*$.
The linear variation of the cohomology class of the symplectic forms can be seen as a cohomological version of the last formula. To see this, recall that the symplectic leaves of 
$(M,\pi^\theta_\lin)$ are parametrised by $\xi\in \tt^*$. More precisely, the leaf $(S_\xi,\omega_\xi)$ corresponding to $\xi$ 
is 
\[ S_\xi=(P\times\{\xi\})/G_\xi\]
with the symplectic form obtained by descending 
\[  \underline{\omega}|_{G/T\times \{\xi\}}
\in \Omega^2(G/T\times \{\xi\})\]
along the bundle map $q_{\xi}: G/T\times \{\xi\}\to G/G_{\xi}\times \{\xi\}$. 
Therefore the previous formula evaluated at a specific $\xi$ becomes
\[ q_{\xi}^*\omega_{S_{\xi}}= q^*\omega_{S}+
\sum_{i= 1}^{r}  c_i\, \xi^{i} \in \Omega^2(G/T).\]
Now, observe that the developing map is just the projection
\[ \dev:\hM^\aff=P/T\times U\to \tt^*, \]
so we find that
\[ \Iaff(\xi)=\sum_{i= 1}^{r}  c_i\, \xi^{i}. \] 
\end{example}


\subsection{The symplectic volume polynomial}
\label{sec:volume:polynomial}

We now study how the symplectic volumes of the leaves vary. For that we need to pass from second degree cohomology to top-degree cohomology. 
The discussion in the previous section on the bundle $\hH$ applies word by word to the similar bundles obtained using instead the cohomology in any fixed degree. In particular, it applies to the top degree of the regular leaves, which is (see the discussion following Theorem \ref{thm:Weyl-intro})
\[ 2 k:= \dim(M)-r, \quad \text{where} \quad r= \corank \pi. \]  
Therefore one obtains a line bundle 
\[ \hLL\to \hM, \quad \hLL_{\hx}= H^{2k}(\Hol(\hx, -)),\]
which is a representation of $\Pi_1(\Hol(\hM, \hF))$. In particular, the fiber at $\hx_0$  becomes 
\[ \hLL_{0}:= H^{2k}(S_{x_0})\in \Rep(\pi^{\orb}_{1}(B)).\]
Furthermore, the role of the section $\varpi$ is replaced by
\[ \sigma_{\vol}:= \frac{\varpi^{k}}{k!}\in \Gamma(\hLL).\]
%
This section is $\Hol(\hM, \hF)$-invariant dues to the fact that the map
$\hH\to \hLL$, $u \mapsto u^{k}/k!$,
is $\Pi_1(\Hol(\hM, \hF))$-equivariant.
This allows us to interpret $\sigma_{\vol}$ as a section of an orbi-vector bundle over $B$. To work over smooth manifolds one can, as before, pullback $\hLL$ together with $\sigma_{\vol}$ to $\hM^{\lin}$, $\hM^{\aff}$, etc, via the canonical projections onto $\hM$, or drop them to $B^{\lin}$, $B^{\aff}$, etc. 
When doing so we will keep the same notations $\hLL$ and $\sigma_{\vol}$.

Next we explain how to encode $\sigma_{\vol}$ in terms of a polynomial. There will be two ways of achieving this and both are relevant to us:
\begin{enumerate}[(i)]
    \item We will show that $\hLL$ trivializes over $B^\aff$. This gives a polynomial function 
    \[ \VVO:B^\aff\to \R. \]
    \item The tensor product $\hLL\otimes\hLL$ trivializes already over $B^{\lin}$. We obtain a function whose pullback to $B^{\aff}$ coincides with $\VV^2_0$. 
\end{enumerate}

In the statement of the next result, by the singular part of 
$B^{\aff}$ we mean the complement of the pre-image of $B^{\reg}$ by the covering projection $B^{\aff}\to B$. 

\begin{theorem}\label{thm:volume-polynomial}
There exists a function 
\[ \VVO: B^{\aff}\to \mathbb{R}\]
with the following properties:
\begin{enumerate}[(i)]
\item it is a polynomial function on the integral affine manifold $B^{\aff}$;
\item its zero-set is precisely the singular part of $B^{\aff}$; 
\item at regular points $\tilde{b}\in B^\aff$,
\[ 
\VVO(\widetilde{b})= \pm \, \iota_0(b) \vol(S_b, \omega_{S_b})
\]
where $\iota_0: B^\reg\to \mathbb{N}$ is the function that for each $b\in B$ gives the order of the holonomy group of $S_b$;
\item for all $w\in \cW$ and $\widetilde{b}\in \oB$ one has 
\[ \VVO(\widetilde{b}\, w)= \delta(w)\, \VVO(\widetilde{b}), \]
where $\delta:\cW\to\Z_2$ is the parity character;
\item the square of $\VVO$ descends to an orbifold smooth function 
\[ \VVO^{2}: B\to \R.  \]
\end{enumerate}
\end{theorem}

\begin{proof}
Consider $\hLL$ as a line bundle over $\oB$. It can be trivialised 
\[ \hLL\to \oB\times \hLL_0, \]
using the Gauss-Manin connection: writing $\widetilde{b}=[\gamma]$, where $\gamma$ a path in $B^{\lin}$ from $\hx$ to $\hx_0$, parallel transport with respect to this connection yields the desired trivialization
\[ T_{\gamma}: \hLL_{\hx}\to \hLL_{0}.\]
Next, if we identify $\hLL_0$ with $\R$ using a primitive integral covector 
$\lambda_0\in \hLL_{0}^{*}$, we set
\begin{equation}\label{eq:last-moment}
\VVO([\gamma]):= \lambda_0\left( T_{\gamma}(\sigma_{\vol}(\gamma(1))\right)
\end{equation}
Therefore, under the isomorphism above, $\sigma_{\vol}$ is identified with the smooth function $\VVO:\oB\to \R$.

Since $\sigma_{\vol}=\frac{\varpi^{k}}{k!}$, the linear variation from Theorem \ref{thm:linear:variation} implies that $\VVO$ is a polynomial function. It follows (see also Section \ref{sec:compact:completeness}) that there exist a polynomial $P$ in $\R^r$ such that
\[ \VVO=P\circ\dev. \]
Since $\dev$ descends to $B^\aff$, the function $\VVO$ factors through the projection
\[
\xymatrix{
\oB \ar[d] \ar[r]^{\dev} & \R^r\ar[rr]^{P} & &\R \\
B^\aff \ar[ur]_{\dev}\ar@/_/[rrru]_{\VVO}}
\]
so it can be viewed as a polynomial function $\VVO:B^\aff\to\R$, proving (i).

We now prove item (iii). For $\widetilde{b}=[\gamma]$ with $\gamma: [0, 1]\to M^{\lin}$ a path from $\hx_0$ to some $\hx$ in the regular part, to compute $\VVO([\gamma])$
we have to apply the parallel transport
\[ T_{\gamma}: H^{2k}(\Hol({\hx},-))\to H^{2k}(\Hol({\hx_0},-))\]
to the Liouville element
\[ \frac{\omega_{\hx}^{k}}{k!}\in H^{2k}(\Hol({\hx},-))\]
and then, up to a sign, pair it with the fundamental class for $\Hol({\hx_0},-)$. Up to a sign, that is the same as pairing the Liouville element at $\hx$ with the fundamental class of $\Hol({\hx},-)$, i.e., the symplectic volume of 
\[ (\Hol({\hx},-), \omega_{\hx}).\]
In turn, this manifold is a covering of the symplectic leaf $S_{\hx}$, with group of deck transformations the holonomy group of $S_{\hx}$. Therefore (iii) follows.

We now move to the proofs of items (iv) and (v). Note that the triviality of the integral affine line bundle $\hLL\to\hM$ is controlled by the character
\[ \delta_{\hLL}: \pi_1(\hM)\to \Z_2\]
determined by
\[ T_{\gamma}=\delta_{\hLL}([\gamma])\cdot \id, \]
where $T_{[\gamma]}$ denotes parallel transport  w.r.t. the Gauss-Manin connection. 
Furthermore, this descends to 
\[ \delta_{\hLL}: \pi_1^{\orb}(B)\to \Z_2.\]
On the other hand, it follows from (\ref{eq:last-moment}) that we have the formula
\begin{equation}
    \label{eq:volume:parity:bundle:character}
    \VVO(\widetilde{b}\, u)= \delta_{\hLL}(u)\, \VVO(\widetilde{b}),
\end{equation}
for all $\tilde{b}\in \oB$, $u\in\pi_1^{\orb}(B)$. This implies that (v) holds.

To prove (iv) note that $\delta_{\hLL}$ is trivial over the image of the splitting given in Proposition \ref{prop:split:orbi:fundamental:group} since the section $\sigma_{\vol}\in\Gamma(\hLL)$ is nowhere vanishing over the regular part. Therefore, $\delta_{\hLL}$ is determined by its values on $\cW$ where we claim that it coincides with the parity character of $\cW$
\begin{equation}\label{eq:parity-identified} 
\delta_{\hLL}|_{\cW}=\delta_{\cW}. 
\end{equation}
This identity, together with \eqref{eq:volume:parity:bundle:character}, implies that (iv) holds. In order to prove (\ref{eq:parity-identified}) it is enough to check it for an abstract reflection $\tau_{\hx}=p_*\hat{\tau}_{\hx}$ associated to subregular point $\hx$, since the conjugate of these generate $\cW$. Working at the level of $\hM$, and denoting by $y= 1_{\hx}\in \hM^\lin$ the point represented by the constant path at $\hx$, 
we will prove that $\hat{\tau}_{\hx}$ acts on the fiber 
\[\hLL_{y}=H^{2k}(\Hol(\hx,-)),\]
as minus the identity, i.e., changing the orientation of the fiber.
The element $\hat{\tau}_{\hx}$
belongs to the group of Deck transformations of the cover $\Hol(\hx,-)\to \hS_{\hx}$. Therefore the statement that we have to prove is equivalent to show that $\hat{\tau}_{\hx}$ is orientation reversing. Consider the following commutative diagram
\[
\xymatrix{\Hol(\hx,-)\ar[d]^\t\ar[dr]^q& \\
\hS_{\hx}\ar[r]^{\res} & S_x}
\]
We observe that:
\begin{enumerate}[(a)]
    \item \emph{The fiber  $q^{-1}(\hx)$ is a finite union of spheres.} Indeed, $\res^{-1}(x)\cong \mathbb{RP}^2$
and therefore $q^{-1}(\hx)$ is a (possibly non-connected) covering space of $\mathbb{RP}^2$. Both the total space $\Hol(\hx,-)$ and the base $S_x$ are orientable, and thus so the fiber.
    \item \emph{The action of $\hat{\tau}_{\hx}$ on the sphere $\mathbb{S}^2_y$ in $q^{-1}(\hx)$ through $y$ is the antipodal map.}  Indeed, the restriction
\[q:\mathbb{S}^2_y\to \mathbb{RP}^2\]
is a covering map whose Deck transformation group corresponds to the generator of the fundamental group $\hat{\tau}_{\hx}$.
\end{enumerate}
Now, since $\hat{\tau}_{\hx}$ is represented by a loop in the fiber of $\res$, its action on $\Hol(\hx,-)$ lifts the identity on $S_x$. This, combined with (b), implies that $\hat{\tau}_{\hx}$ reverses the orientation on $\Hol(\hx,-)$.

Finally, item (ii) follows from the other items.
\end{proof}

\begin{remark} 
The proof also shows that $\hLL$ is trivializable over $B^{\aff}$. Indeed, from the proof this is equivalent to $\delta_{\hLL}|_{K^{\aff}}= 1$. This follows by observing that, since $\VVO$ descends to $B^{\aff}$, we must have 
\[\VVO(\widetilde{b}\, u)=  \VVO (\widetilde{b}), \]
whenever $u\in K^{\aff}$. Since $\VVO$ is not identically zero, we obtain from \eqref{eq:volume:parity:bundle:character} that $\delta_{\hLL}= 1$ on $K^{\aff}$.
\end{remark}

\section{The Duistermaat-Heckman and the Weyl Integration Formulas}
\label{sec:DH:formula}

In this section we discuss measure theoretic aspects of Poisson manifolds of s-proper type, extending the Duistermaat-Heckman formula given in \cite{CFM-II} to non-regular Poisson manifolds.  

We will be working with Radon measures and, when working on manifolds,
we will focus mainly on \textit{geometric measures}, by which we mean measures associated to densities. In general, Radon measures will be denoted by $\mu$ and positive densities will be denoted by $\rho$. A measure induced by a density $\rho$ will be denoted by $\mu_{\rho}$. Also, we denote by the symbol $\cD$ density bundles associated with a vector bundle. We refer to \cite{CM19} and \cite[Section 3.13]{DH} for details on measures and densities.

For the rest of the section we fix such a Poisson manifold $(M, \pi)$ together with an s-proper integration $(\G, \Omega)\tto (M, \pi)$ and we consider the leaf space $B$ together with the canonical projection 
\[ p: M\to B.\]
We define two measures on $B$. The first one arises from the fact that $B$ has an integral affine orbifold structure and, therefore, it 
carries a natural measure (see below). We call it the \textbf{Lebesgue measure} and we denote it by 
\begin{equation}\label{eq:measure-aff-B} 
\boldsymbol{\mu^{\aff}_{B}}.
\end{equation}
The second measure is obtained from the Liouville measure $\mu_{\Omega}$ on $\G$, induced by the density associated to the Liouville form
$\frac{\Omega^{\top}}{\top !}\in \Omega^{\top}(\G)$. Noticing that the map $p\circ\s=p\circ\t:\cG\to B$ is proper, one can push $\mu_\Omega$ to a measure on the leaf space, denoted
\[ \boldsymbol{\mu^{DH}_{B}}= p_*s_{*}(\mu_{\Omega}). \]
We call it the \textbf{Duistermaat-Heckman measure} on $B$. 

The last ingredient for stating the main theorem in this section is the \textbf{square volume polynomial} $\VV^2: B\to\R$. It coincides with the function 
\[ \VV^2(b)=\begin{cases}(\iota(b)\vol(S_b,\omega_{S_b}))^2, & b\in B^\reg\\
0, & b\notin B^\reg
\end{cases}
\]
where $\iota:B\to\mathbb{N}$ counts the number of connected components of the isotropy group of $\cG$ at a/any $x\in S_b$ (see also Remark \ref{rem:volume:other:orbifold:structure}).

\begin{theorem} 
\label{thm:measures:DH:Aff}
If $(\G,\Omega)$ is an s-connected, s-proper integration of $(M, \pi)$ then
\begin{equation}\label{eq:DH-measures}
\boldsymbol{\mu^{DH}_{B}}= \VV^2\cdot \boldsymbol{\mu^{\aff}_{B}},
\end{equation}
where $\VV^2: B\to\R$ is the square volume polynomial.
\end{theorem}  

\begin{remark}
\label{rem:measures:DH:Aff}
It may be instructive to the reader to notice that $\boldsymbol{\mu^{DH}_{B}}$ and $\boldsymbol{\mu^{\aff}_{B}}$ are associated with the two interpretations of $B$ as a leaf space of $(M, \cF)$ and as a leaf space of $(\hM, \hF)$, respectively. Recall also from Section \ref{sec:volume:polynomial} that, although both $\iota$ and the symplectic volumes of the leaves of $(M, \pi)$ yield functions $\iota$ and $\vol$ defined on the entire $B$, the reason why $\VV^2$ appears in the final formula is its smoothness.
\end{remark}

\begin{remark}
\label{rem:volume:other:orbifold:structure}
There is very little difference between the function $\VV^2$ and the function $\VV_0^2$ introduced in the previous section: one is a rescaling of the other. The scaling factor arises from the fact that we use two orbifold structures on $B$: in this section we use the one whose atlas is the foliation groupoid $\cB= \cB(\hG)$ in (\ref{eq:Weyl-short-exact-sequence}), while in the previous section we used the classical orbifold structure, i.e., whose atlas is the holonomy groupoid $\Hol(\hF)$.

The fact that $\iota(b)/\iota_0(b)$ is (locally) constant can be seen as an instance of a general remark: for any ($s$-connected) proper foliation groupoid $\cB$, the kernel of the canonical map $\cB\tto \Hol(\cF)$ is a bundle of groups of (locally) constant cardinality. This follows, e.g., from the local normal form for proper Lie groupoids, which reduces the statement to one about linear actions of finite groups. 


In our case, the resulting rescaling factor can be read off at any point $x\in M$: it is the cardinality of $Z_{\cG_x}(T)/T$, where $T\subset \cG_x$ is a maximal torus (compare with Section \ref{ex:model main diagram}). At a regular point $x$, this is the subgroup of $\pi_0(\cG_x)$ consisting of elements which act trivially on $\cG_x^0$.
\end{remark}


We now discuss the measure (\ref{eq:measure-aff-B}) in detail.  The key ingredient in its construction is the density arising from the integral affine lattice $\Lambda\subset \htt$ given by
\begin{equation}\label{eq:measure-aff-nu} 
\rho^{\aff}_{\nu}\in\Gamma(\hM, \cD_\nu), \quad \rho^{\aff}_{\nu}|_x:=|\lambda_1\wedge \cdots \wedge\lambda_n|
\end{equation}
where $\lambda_{1}, \ldots , \lambda_{n}$ is any basis of $\Lambda_x\subset \htt_x= \nu_{x}^{*}$. This density has the property that it is invariant under the holonomy of $\hF$ and, therefore, can be thought of as a density on the orbifold $B$. 
Next, we need the generalisation to orbifolds of the construction 
of measures associated to densities. If we use \'etale orbifold atlases, this is
rather straightforward, as explained in the discussion leading to Proposition 6.1.3  in \cite{CFM-II} and produces (\ref{eq:measure-aff-B}) starting from $\rho^{\aff}_{\nu}$. However, it is useful to have a concrete description for all, not necessarily \'etale,  orbifold atlases, as explained also in \cite{CFM-II}. The next proposition provides a summary of what we need here.

\begin{proposition}
\label{prop:measure:affine}
Let $B$ be an orbifold with atlas $\cB\tto M\to B$, where $\cB$ is a proper foliation groupoid with underlying foliation denoted $\cF$ and normal bundle $\nu$. 
Then any positive density $\rho_{\nu}\in\Gamma(\cD_\nu)$ that is invariant under the holonomy of $\cF$ gives rise to a measure $\mu_{ \rho_{\nu}}$ on $B$ such that
for any $f\in \mathcal{C}^{\infty}_c(M)$, 
\[ \int_M f(x) \,\d\mu_{\rho_M}(x) =\int_B \left(\iota(b) \int_{O_b}f(y)\,\d\mu_{O_b}(y) \right) \, d \mu_{ \rho_{\nu}}(b),\]
where $\mu_{\rho_M}$ and $\mu_{O_b}$ are measures associated to densities on $M$ and on the orbits $O_b$ obtained from any decomposition 
\[ \rho_{\nu}=\rho^*_{\cF}\otimes\rho_M\in\mathcal{D}_{\cF^*}\otimes \mathcal{D}_{TM}\cong \mathcal{D}_{\nu}\]
as follows: 
\begin{enumerate}[(a)]
    \item $\rho_M$ is a density on $M$;
    \item $\rho^*_{\cF}$ is the dual of a strictly positive density $\rho_{\cF}\in \mathcal{D}_{\cF}$;
    \item $\mu_{O_b}$ is associated with the density $\rho_{\cF}|_{O_b}\in \cD_{TO_b}$;
    \item $\iota(b)$ is the number of elements of the isotropy groups $\cB_x$ with $x\in O_b$.
\end{enumerate}
Furthermore, this property characterizes uniquely $\mu_{ \rho_{\nu}}$.
\end{proposition}

\begin{remark}\label{remark:measure:affine}
The density $\rho^*_{\cF}$, combined with the isomorphisms $t^*\cF|_{S_x}\cong T\cB(x, -)$ induced by right translations, give rise to densities on the $s$-fibers $\cB(x, -)$. The associated measures will be denoted $\mu_{\cB(x, -)}$. It follows that the numbers $\iota(b)$ can be absorbed into the integration as follows:
\[\iota(b) \int_{O_b}f(y)\,\d\mu_{O_b}(y) = \int_{\cB(x, -)} f(t(g)) \d\mu_{\cB(x, -)}(g),\]
where $x$ is any point in the orbit $O_b$. 

When $\cB\tto M$ is $\s$-proper, the fact that the $\mu_{\cB(x, -)}$ are induced by a smooth family of densities implies that the volumes can be arranged into a smooth function 
 \begin{equation}\label{eq:fiber-volumes} \vol^{\rho_{\cF}}: B\to \R, \quad b= p(x)\mapsto \vol\left(\cB(x, -), \mu_{\cB(x, -)}\right)= \iota(b) \vol(O_b, \mu_{O_b}) .
 \end{equation}
Then the proposition implies that, in the $\s$-proper case, $\mu_{ \rho_{\nu}}$ could be defined by 
\begin{equation}\label{eq:induce-measure-s-proper-case} 
\mu_{ \rho_\nu}= \frac{1}{\vol^{\rho_{\cF}}}\,  p_{*}(\mu_{\rho_M}).
\end{equation}
\end{remark}
\medskip

\begin{proof}[Proof of Theorem \ref{thm:measures:DH:Aff}]
We will apply Proposition \ref{prop:measure:affine} to the foliation groupoid
$\cB\tto \hM$ that serves as orbifold atlas for $B$ (see Section \ref{sec:The integral affine structure on the leaf space}).

The Liouville form $\frac{\Omega^{\top}}{\top !}\in \Omega^{\top}(\G)$ can be pushed down via the proper map $\s: \G\to M$ to a density
\[ \boldsymbol{\rho^{DH}_{M}}:=\int_{s} \frac{|\Omega^{\top}|}{\top !} \in \Gamma(\mathcal{D}_{TM}).\]
The associated measure, which is also the pushforward of the Liouville measure, is denoted $\mu^{DH}_{M}$ and we call it the \textbf{Duistermaat-Heckman measure} on $M$. We can use $\rho^{DH}_{M}$ to construct a canonical decomposition of the density \eqref{eq:measure-aff-nu}, as in Proposition \ref{prop:measure:affine}. To that end we use the exact sequence of vector bundles over $\hM$
\[ \xymatrix{0 \ar[r] & \hnu^*\ar[r] & \hL \ar[r] & \hF\ar[r] & 0},\]
together with the isomorphism $\hL\cong \res^*T^*M$ -- see  (\ref{eq:injective:IM:form}) -- to identify
\[ \cD_{\hF}\cong \cD_{\hL}\otimes \cD_{\hnu}\cong \res^*\cD_{T^*M}\otimes \cD_{\hnu}.\]
Making use of $\rho^{DH}_{M}$ and its dual $\rho^{DH, \vee}_{M}\in \Gamma(\cD_{T^*M})$,
we can now define 
\[ \rho^{\aff}_{\hF}:= \left(\rho^{DH, \vee}_{M}\circ \res\right)\otimes \rho^{\aff}_{\nu}\in \Gamma(\cD_{\hF}). \] 
Similarly, using the identification $\cD_{T\hM}\cong  \cD_{\hF}\otimes  \cD_{\hnu}$ arising from the sequence \[\xymatrix{0
\ar[r]& \hF\ar[r]& T\hM\ar[r]& \hnu\ar[r]& 0},\] 
we define
\[  \rho^{\aff}_{\hM}:= \rho^{\aff}_{\hF}\otimes \rho^{\aff}_{\nu}\in \Gamma(\cD_{T\hM}).\]
By construction, these two densities provide a decomposition of $\rho^{\aff}_{\nu}$ as in Proposition \ref{prop:measure:affine}. 
In particular, $\rho^{\aff}_{\hF}$ induces measures on the leaves $\hS_b$ and on the $s$-fibers $\cB(\hx, -)$, denoted now
\begin{equation}\label{eq:induced-Affine-measures-on-leaves}
\mu_{\hS_b}^{\aff}, \quad \mu_{\cB(\hx, -)}^{\aff}. 
\end{equation} 
Notice that these measure are defined for every point in $B$, regular or not. 
For the volume function \eqref{eq:fiber-volumes} associated with $\rho^{\aff}_{\hF}$ one finds:

\begin{lemma} 
The function $\vol^{\rho^{\aff}_{\hF}}: B\to \R$ is identically $1$. 
\end{lemma}

\begin{proof} 
The function $\vol^{\rho^{\aff}_{\hF}}$ is smooth so it suffices to prove it is $1$ on $B^{\reg}$. Therefore we may assume that $M$ is regular and make
use of the results from \cite{CFM-II}. We will prove something stronger, namely that 
\begin{equation}\label{eq:useful-for-Duistermaat-Heckman}
(\iota \cdot \vol)\cdot\, \rho^{\aff}_{\hF}= \frac{|\omega_{\cF_\pi}^{\top}|}{\top!} ,
\end{equation}
where $\vol(b)$ is now the symplectic volume of $S_b$. In the regular case, the density $\rho^{\aff}_{\hF}$ coincides witj $\rho^{DH, \vee}_{M}\otimes \rho^{\aff}_{\nu}$. Hence, we can apply \cite[Lemma 6.3.3 ]{CFM-II} to rewrite the left hand side of the last equality as 
\[ \left(\frac{1}{\iota \cdot \vol}\cdot\, \rho^{DH}_{M}\right)^{\vee}\otimes \rho^{\aff}_{\nu}= \rho_{M}^{\vee}\otimes \rho^{\aff}_{\nu},\]
with $\rho_M= \frac{|\omega_{\cF_\pi}^{\top}|}{\top!}  \otimes \rho^{\aff}_{\nu} $ (which is equation (6.2) in loc.~cit.). Noticing that now $\res=\id$, we are left with showing that under the isomorphism $\cD_{\cF}\cong \cD_{T^*M}\otimes \cD_{\nu}$ we have
\[ 
\rho_{M}^{\vee}\otimes \rho^{\aff}_{\nu}= \frac{|\omega_{\cF_\pi}^{\top}|}{\top!} 
\] 
The leafwise symplectic form $\omega_{\cF_{\pi}}$, gives an identification $\cF_{\pi}\cong\cF_{\pi}^*$ and from the decomposition induced by
$\xymatrix{0\ar[r]& \cF\ar[r]& TM\ar[r]& \nu\ar[r]& 0}$,
we see that the last identity becomes precisely the definition of $\rho_M$ -- again equation (6.2) in \cite{CFM-II}. 
\end{proof}

Denoting by $\boldsymbol{\mu^{\aff}_{\hM}}$ the measure on $\hM$ induced by $\rho^{\aff}_{\hM}$, the previous lemma combined with Remark \ref{remark:measure:affine} (cf.~equation (\ref{eq:induce-measure-s-proper-case})) implies that 
\[ \boldsymbol{\mu^{\aff}_{B}}= p_{*}(\boldsymbol{\mu^{\aff}_{\hM}}).\]
Therefore, to complete the proof of theorem  we only need to prove the following lemma.

\begin{lemma} \label{lemma:the-DH-lemma}
One has the following equality of measures on $M$:
\[ \boldsymbol{\mu^{DH}_{M}}= \VV^2\cdot  \res_*(\boldsymbol{\mu^{\aff}_{\hM}})\]
\end{lemma}

\begin{proof} 
We have to prove that $\int_M f \rho^{DH}_{M}= \int_{\hM} \VV^2\cdot \res^*(f) \rho^{\aff}_{M}$ for all $f\in \cC_{c}^{\infty}(M)$, 
Since both sides involve integration of \textit{smooth} densities (see Remark \ref{rem:measures:DH:Aff}) nothing changes if we remove from $M$ and $\hM$ 
subspaces of (Lebesque) measure zero. We deduce that it suffices to prove the identity from the statement 
over $\hM^{\reg}\cong M^{\reg}$. Therefore we may now assume that $M$ is regular, and prove that 
$\rho^{DH}_{M}= (\iota \cdot \vol)^2\cdot\,  \rho^{\aff}_{M}$. 
But this is just a restatement of Lemma 6.3.3 from \cite{CFM-II}, as the equation (\ref{eq:useful-for-Duistermaat-Heckman}) and the definition of
$\rho_M$ ((6.2) in loc.cit. again) show. 
\end{proof}

This completes the proof of Theorem \ref{thm:measures:DH:Aff}.
\end{proof}

From Proposition \ref{prop:measure:affine} we deduce the following \emph{Weyl integration formula}:

\begin{corollary} 
For all $f\in \cC_c^{\infty}(M)$ one has
\[ \int_M f(x) \, \boldsymbol{\d \mu^{DH}_{M}} (x) =  \int_B \left(\int_{\cB(\hx, -)}f(\res(t(g)))\,\d\mu_{\cB(\hx, -)}^{\aff}(g) \right) \VV^2(b) \,  \boldsymbol{\d \mu^{\aff}_{B} } (b).\]
\end{corollary}

\begin{corollary} One has a unitary isomorphism of $L^2$-spaces:
\[ L^2(B, \mu_{B}^{DH})\cong L^2(B, \mu_{B}^{\aff}), \quad \phi\mapsto \phi\cdot |\VV|.\]
\end{corollary}

\begin{example} 
Let us consider the linear Poisson structure on the dual of a compact Lie algebra $M= \gg^*$. Fixing an integrating groupoid $\cG= T^*G\cong G\ltimes \gg^*$, the resolution $\hgg=G/T\times_W\tt^*$ has the integration $\hG=G\ltimes\hgg$, and the induced foliation groupoid $\cB$ is (see \ref{eq:fol-oid-coadjoint}):
\[  (G/T\times G/T\times \tt^*)/W \tto (G/T\times\tt^*)/W. \]
In particular, we obtain an identification
\begin{equation}
 \label{eq:example:identif-s-fibers-cB}
 \cB(\hx, -) \cong G/T. 
 \end{equation}

The measures and densities above become the following:
\begin{itemize}
    \item $\rho_{\gg^*}^{DH}=\d X^*$, the constant density on $\gg^*$ determined by the Haar density on $G$; i.e., if one $G$-translates  its dual $\d X\in \cD_{\gg}$ then one obtains the Haar density of $G$.
    \item $\rho_{\tt^*}^\aff=\frac{1}{|W|} \d Y^*$,
    where $\d Y^*$ is the constant density on $\tt^*$ associated to the lattice $\Lambda\subset \tt$. Note that the dual of $\d Y^*$ produces the Haar density of $T$. 
    \item  $\rho_{\cB(\hx,-)}^{\aff}=\d(gT)$,  the quotient 
    of the Haar density of $G$ modulo the Haar density of $T$, where one uses identification (\ref{eq:example:identif-s-fibers-cB}).
\end{itemize}
Moreover, in this case one has $\VV^2(Y^*)=|\det(\ad Y^*)_{\gg/\tt}|$, as in \cite{DK}, and therefore our integral formula becomes
\[ \int_{\gg^*} f(X^*) \, \d X^* =  \frac{1}{|W|}\int_{\t^*} \left(\int_{G/T} f(\Ad_{g}(Y^*))\,\d (gT) \right) |\det(\ad Y^*)_{\gg/\tt}| \,  \d Y^*.\]
Comparing with \cite[Corollary 3.14.2]{DK}, notice that we are using more specific densities 
on $G$ and $T$ for which the constant $c$ there equals $1$.

\end{example}

\begin{remark}
The Duistermaat-Heckman formula discussed here can be seen as a formula associated to the Hamiltonian $(\cG, \Omega)$-space $\t: (\cG, \Omega)\to M$. It implies a similar formula for other free Hamiltonian $(\G, \Omega)$-spaces $\mu: (P, \Omega)\to M$ by applying it to the corresponding gauge groupoids. The more general version, for Hamiltonian spaces that are not necessarily free, will be discussed somewhere else (see also \cite{Mol22,Zwaan23}).
\end{remark}

 \begin{remark}
 While the push-forward of measures is well-behaved, the situation is more subtle when pushing forward densities along maps that are not submersions.
 That is the reason that the equality in the statement of Lemma \ref{lemma:the-DH-lemma} is one of measures and cannot be formulated as an equality of densities on $M$.

 On the other hand, the lemma can be understood as the equality of densities on $\hM$
 \[ \res^*(\rho^{DH}_{M})= \VV^2 \cdot \rho^{\aff}_{\hM}\]
 For the left hand side we use that densities can be pulled back along maps that are local diffeomorphisms on dense open sets, like our $\res: \hM\to M$ is. The price to pay is that the pull-back of smooth densities may be only continuous (but that is not a problem for inducing measures). However, the equality says that the left hand side is also smooth. 
 

 \end{remark}

\section{Compactness and Completeness}
\label{sec:compact:completeness}

In this section, we establish structural properties of compact PMCTs and complete, s-proper PMCTs. We achieve this by combining the theory developed in the previous sections with a new ingredient: the factorization properties of polynomials in integral affine manifolds and orbifolds, which we apply to the volume polynomial of Poisson manifolds of s-proper type. For example, we will show that:

\begin{theorem}\label{thm:s-proper-is-regular}
Any Poisson manifold of compact type is regular. 
Equivalently, any compact symplectic groupoid is regular.
\end{theorem}

For a compact symplectic manifold all powers of the cohomology class of the symplectic form must be non-trivial. The theorem says that compactness in multiplicative symplectic geometry severely constrains the groupoid multiplication. Equivalently, it severely constrains the underlying Poisson geometry.

\begin{remark} 
Theorem \ref{thm:s-proper-is-regular} is false for twisted Dirac manifolds of compact type. Equivalently, there exist compact twisted presymplectic groupoids which are non-regular. We will discuss this failure at the end of this section.
\end{remark}

If the Markus Conjecture holds for integral affine orbifolds (see Section \ref{sec:open:Markus}), then PMCTs of compact type are complete. For complete, s-proper Poisson manifolds we obtain the following structural theorem:

\begin{theorem}\label{thm:s-proper-complete}
If $(M,\pi)$ is a complete, s-proper Poisson manifold, then it has a unique minimal infinitesimal stratum $\Sigma_0$. Its Weyl group coincides with the classical Weyl group of the isotropy Lie algebra $\gg_x$, for any $x\in\Sigma_0$.
\end{theorem}

This result -- and its more detailed version given later in this section -- suggests that all the information about a complete, s-proper Poisson manifold is encoded in the first jet of $(M,\pi)$ along the minimal stratum $\Sigma_0$.

The previous results provide deep insight into the Poisson geometry of PMCTs and its relation to Lie theory. Hence, one can use them to deduce some classical results in Poisson geometry related to Lie theory. For example, we have the following somewhat surprising application of Theorem \ref{thm:s-proper-is-regular}.

\begin{corollary} 
There exists no Lie bialgebra $(\gg,\gg^*)$ with both $\gg$ and $\gg^*$ semisimple Lie algebras of compact type.
\end{corollary}



\subsection{Polynomial functions on integral affine orbifolds}
\label{sec:polynomial:functions}
An affine manifold $N$ has a well-defined subalgebra of polynomial functions 
\[\Pol(N)\subset C^\infty(N).\]
Namely, we say that $P\in C^\infty(N)$ is a {\bf polynomial} if its pullback to every
affine chart of $N$ is a polynomial. Assuming that $N$ is connected, the degree of a polynomial is the degree on its pullback to any affine chart. It is well-defined since affine transformations on $\R^q$ preserve the degree of polynomials in $\R[x_1,\dots,x_q]$.

The developing map of an affine manifold provides useful information on its algebra of polynomial functions.

\begin{proposition}\label{pro:UFD}
    For an affine manifold $N$, its developing map $\dev:N^\aff\to\R^q$ gives an isomorphism of rings
    \[  (\dev^*)^{-1}\circ p^*: \Pol(N)\to (\R[x^1,\dots,x^q])^{\Gamma^\aff}, \]
    where $p:N^\aff\to N$ is the affine holonomy cover and $\Gamma^\aff$ the affine holonomy group. 
\end{proposition}

\begin{proof}
For the proof we will use the abbreviated notations $\widetilde{N}:=N^\aff$ and $\Gamma:=\Gamma^\aff$. We will show that both 
\[p^*: \Pol(N)\to \Pol(\widetilde{N})^{\Gamma},\quad 
\dev^*:\R[x_1,\dots,x_q]^\Gamma\to \Pol(\widetilde{N})^{\Gamma}\]
are isomorphisms.

For the first map we note that pullback by the covering map is an isomorphism
     $p^*: C^\infty(N)\to C^\infty(N)^{\Gamma}$. 
It sends polynomials to polynomials, and any invariant polynomial comes from one in $N$. Hence the result follows.

For the second map  we recall that an affine manifold has a canonical real analytic structure, and that real analytic functions on a (connected) manifold agree if and only if they agree on an open subset.
We apply this first to  $\R[x_1,\dots,x_q]$ to deduce 
that two polynomials on $\R^q$ agree if and only if they agree in the image of the developing map. This implies that 
\[\dev^*:\R[x_1,\dots,x_q]\to \Pol(\widetilde{N}) \]
is injective. To prove surjectivity we restrict any given 
$P\in \Pol(\widetilde{N})$ to a connected open subset $U\subset \widetilde{N}$ on which the developing map is a diffeomorphism onto its image. Then $P|_U$ is the pullback by $\dev$ of the restriction to $\dev(U)$ of a polynomial $P_0\in \R[x_1,\dots,x_q]$. Therefore, $\dev^*P_0=P$.
Finally, since $\dev$ is equivariant with respect to the actions of $\Gamma$ and $\Gamma$, the restriction to the respective subalgebras of invariants  
\[\dev^*:\R[x_1,\dots,x_q]^\Gamma\to \Pol(\widetilde{N})^{\Gamma}\]
is an isomorphism.
\end{proof}

Every hyperplane in $\R^n$ is the zero set of a degree one polynomial. The polynomial is unique up to rescaling by a non-zero real number. We say that the hyperplane is integral if the polynomial can be chosen to have rational degree one coefficients. In this case, up to sign, it is the zero  set of a unique polynomial whose degree one coefficients are coprime integers.


For an integral affine manifold $N$, by an {\bf integral hyperplane} of $N$ we mean a connected, codimension one, integral affine submanifold. A degree one polynomial on $N$ is called {\bf primitive} if its pullback via \emph{some} integral affine chart is a degree one polymonial for which non-constant coefficients are coprime integers. Since an integral change of coordinates has linear part a matrix with determinant $\pm 1$, it follows that this holds for any integral affine chart.

\begin{proposition}\label{pro:monic}
    Let $N$ be a integral affine manifold with trivial affine holonomy. Then any $P\in\Pol(N)$ has polynomial factorization
    \[ P=L_1\cdots L_k\cdot Q,\]
    where the $L_i$'s are primitive degree one polynomials with non-empty vanishing locus and the vanishing locus of $Q$ does not contain an integral affine hyperplane. The factorization is unique up to sign changes and the ordering of the primitive degree one polynomials.
\end{proposition}

\begin{proof} 
Since the affine holonomy is trivial, by Proposition \ref{pro:UFD}, every $P\in\Pol(N)$ has a unique factorization into irreducible polynomials (up to units and order). An irreducible polynomial on $N$ of degree greater than one cannot vanish on a hyperplane: if it were to vanish, its restriction to a chart around a point in the hyperplane would have a degree one factor, and so using the developing map, we see that we could factor out a linear term.  Therefore, in the factorization of $P$ only the linear terms can vanish on a hyperplane. We can single out those vanishing on integral hyperplanes (if any), and normalize them to be primitive (we may get the same primitive factor more than once). If $Q$ denotes the product of the remaining polynomials in the factorization of $P$ then, by construction, it does not vanish on any integral affine hyperplane. The uniqueness in the statement is a consequence of the uniqueness of factorizations in $\Pol(N)$.
\end{proof}

For a classical integral affine orbifold $B$ by an {\bf integral hyperplane} we mean a connected, codimension one, integral affine suborbifold. A polynomial on $B$ is a function whose pullback to $\oB$ (an integral affine manifold) is a polynomial. If the pullback polynomial is primitive of degree one we say that the original polynomial on the orbifold is primitive of degree one. 

\begin{proposition}\label{pro:monic-orbifold}
    Let $B$ be a classical integral affine orbifold. Then for  any  $P\in\Pol(B)$ there exists a finite orbifold cover $q:\widetilde{B}\to B$ 
    and a factorization:
    \begin{equation}
    \label{eq:factorization:finite:cover}
    q^*P=L_1\cdots L_k\cdot Q,
    \end{equation}
    where the $L_i$'s are primitive degree one polynomials with non-empty vanishing locus and the vanishing locus of $Q$ does not contain an integral affine hyperplane. 
\end{proposition}

\begin{proof}
    We will call a factorization as in the statement a \emph{primitive factorization}. By Proposition \ref{pro:monic}, the the pullback of $P\in\Pol(B)$ to the universal covering space $p:\oB\to B$ has a primitive factorization:
    \[ \widetilde{P}=p^*P=L'_1\cdots L'_k\cdot Q'. \]
    We may further assume that there are no degree one primitive polynomials with opposite sign and that all repeated primitive polynomials are in contiguous positions. The polynomial $\widetilde{P}$ is $\pi_1^\orb(B)$-invariant, so given any element $a\in\pi_1^\orb(B)$, by the properties of the vanishing sets of the factors and uniqueness of factorization we have:
    \[ a^*\widetilde{P}:=\eps_1(a)L'_{\sigma_1(a)}\cdots \eps_k(a)L'_{\sigma_k(a)}\cdot \eps(a) Q',\]
    where \[
    (\eps_1,\dots,\eps_k,\sigma):\pi_1^\orb(B)\to \Z^k_2\rtimes \mathrm{S}_k,\quad \eps:\pi_1^\orb(B)\to \Z_2\]
    define representations of $\pi_1^\orb(B)$ on the group of signed permutations and on $\Z_2$, respectively. Let us denote by $H\subset \pi_1^\orb(B)$ the kernel of the homomorphism
    \[ \left((\eps_1,\dots,\eps_k,\sigma),\eps\right):\pi_1^\orb(B)\to (\Z_2^{k}\rtimes \mathrm{S}_k)\times \Z_2.\]
    Since $H$ has finite index we obtain the finite orbifold cover $q:\widetilde{B}:=\oB/H\to B$, together with a polynomial defined on it with the properties in the statement.
\end{proof}

It is a non-trivial classical fact that every polynomial on a compact integral affine \emph{manifold} is constant (see \cite{GH86}). Note that this fact is obvious if the manifold is known to be complete and this still holds for compact, complete orbifolds.  Without the assumption of completeness we are able to prove the following result which suffices for our purpose.

\begin{corollary}\label{cor:no-polynomial} 
A non-trivial polynomial on a compact integral affine orbifold $B$ cannot vanish on an integral  hyperplane. 
\end{corollary}

\begin{proof} 
Let $P\in \Pol(B)$ be a non-trivial polynomial that vanishes on an integral  hyperplane. We apply Proposition \ref{pro:monic-orbifold} to $P$,  obtaining a finite orbifold cover $q:\widetilde{B}\to B$ where $q^*P$ has a factorization \eqref{eq:factorization:finite:cover} with $k\ge 1$. Since $\widetilde{B}$ is compact, any such primitive factor $L$ has a critical point and we find a contradiction: if we consider the pullback $L'$ of $L$ to the orbifold universal covering space, we have by Proposition \ref{pro:UFD}
\[L'=\dev^*L_0,\]
where $L_0\in \R[x_1,\dots,x_q]$ is a degree one polynomial. But a degree one polynomial in $\R^q$ has no critical points, and, therefore, $L$ cannot have critical points. 


\end{proof}

\subsection{Consequences for PMCTs}
We will now use the volume polynomial to draw important consequences for PMTC's.

\subsubsection{PMCTs of compact type}
We give a quick proof that every Poisson manifold of compact type must be regular. 

\begin{theorem}\label{thm:compact PMCT is regular}
If $(M,\pi)$ is a compact Poisson manifold that admits a proper integration, then $\pi$ is regular.
\end{theorem}

\begin{proof}
Observe that if $(M,\pi)$ is proper with leaf space $B$, then it follows from Corollary \ref{cor:subregular:leaf:space} that the subregular infinitesimal strata are integral hyperplanes in the classical integral affine orbifold $B$. This subregular locus is non-empty iff $\pi$ is non-regular.

So assume that $(M,\pi)$ is proper, compact and non-regular. First, we apply  Theorem \ref{thm:Weyl-intro} to construct the Weyl resolution of $(M,\pi)$. By Theorem \ref{thm:int:affine:leaf:space}, the leaf space $B$ of $(M,\pi)$ inherits
an integral affine orbifold structure. Second, we apply Theorem \ref{thm:volume-polynomial}
to deduce that the square of the volume defines a polynomial $\mathcal{V}_0^2:B\to \R$
whose zero set is the singular part of $B$. Third, because $(M,\pi)$ is non-regular a subregular stratum $B_\Sigma\in \cS(B)$ exists and is an integral affine suborbifold of $B$ contained in its singular part. Therefore we have produced a polynomial on $B$ with properties that contradict the statement of Corollary \ref{cor:no-polynomial}. 
\end{proof}

\begin{corollary} 
There exists no Lie bialgebra $(\gg,\gg^*)$ with both $\gg$ and $\gg^*$ semisimple Lie algebras of compact type.
\end{corollary}

\begin{proof}
Any Lie bialgebra integrates to Poisson-Lie group structures $(G,\pi_G)$ and $(G^*,\pi_{G^*})$ on the 1-connected integrations of $\gg$ and $\gg^*$. Under the assumption in the statement, both groups are compact. Therefore (see, e.g., \cite[Section 11.4]{LPV13}) the 1-connected integration of the double $\mathfrak{d}:=\gg\Join \gg^*$ is a double Lie group $D:=G\Join G^*$ carrying two distinct symplectic groupoid structures that integrate $(G,\pi_G)$ and $(G^*,\pi_{G^*})$. Since $D$ is compact, by Theorem \ref{thm:compact PMCT is regular}, $(G,\pi_G)$ and $(G^*,\pi_{G^*})$ must be regular. This forces $\gg=\gg^*=\{0\}$.
\end{proof}

\subsubsection{Complete, s-proper PMCTs}
Recall that a Poisson manifold of proper type is complete if its leaf space $B$ is a complete integral affine orbifold, i.e., if the developing map $\dev:\oB\to \R^q$ is a diffeomorphism. Notice also that the connected components of the preimages of the subregular strata in the orbifold universal covering space $\oB$ are contained in the hyperplanes defined the reflections of the Weyl group of the Poisson manifold. The latter are closed integral affine hyperplanes. 

This allows us to show the following structural theorem for this class of Poisson manifolds. In the statement we use the decomposition $\pi_1^\orb(B)=\cW(M,\pi)\rtimes\pi_1^\orb(B^\reg)$ from Proposition \ref{prop:split:orbi:fundamental:group}. 

\begin{theorem}
\label{thm:s:proper:complete}
    Let $(M,\pi)$ be a complete, s-proper, Poisson manifold, and denote its leaf space by $B$. Then:
    \begin{enumerate}[(i)]
    \item There exists a unique minimal infinitesimal stratum $(\Si_0,\pi_0)$, which is a regular, complete, s-proper Poisson manifold with $\pi_1(\Sigma_0)=\pi_1(M)$;
    \item Let $\gg$ denote the isotropy group of any point $x\in\Si_0$. Then
    \[ \cW(M,\pi)\cong W,\] 
    where $W$ is the classical Weyl group of $\gg$ relative to a maximal torus $\tt$;
    \item There is an isomorphism of integral affine manifolds 
    \[ \oB\cong \oB_{\Sigma_0}\times \tt^*_\ss, \]
    where $\tt_\ss$ is a maximal torus of $\gg_\ss:=[\gg,\gg]$.
    Under this isomorphism, the action of $\pi_1^\orb(B)=\cW(M,\pi)\rtimes\pi_1^\orb(B^\reg)$  has the following properties:
    \begin{enumerate}[-]
        \item $\cW(M,\pi)$ acts trivially on the first factor and via the classical action on $\tt^*_\ss$;
        \item $\pi_1^\orb(B^\reg)$ leaves $\oB_{\Sigma_0}\times \{0\}$ invariant and its action on it factors through the action of $\pi_1^\orb(B_{\Sigma_0})$.
    \end{enumerate}
    \end{enumerate}
\end{theorem}





\begin{proof}[Proof of Theorem \ref{thm:s:proper:complete}]
        Consider the pullback of the square of the volume polynomial to $\oB\simeq \R^q$
    \[p^*\mathcal{V}_0^2:\R^q\to \R\]
    This is a polynomial has a function that vanishes precisely on the singular part of $\oB\simeq \R^q$. By Proposition \ref{pro: equality regulars} the singular part agrees with the union of the hyperplanes of the reflections of the Weyl group of the Poisson manifold of proper type. The factorization of $p^*\mathcal{V}_0^2$ provided by Proposition \ref{pro:monic} shows that the number of these hyperplanes is bounded by the number of degree one primitive factors $L_1,\dots,L_k$, and hence are finite in number. It follows that the number of chambers of the corresponding hyperplane arrangement in $\R^q$ is also finite, and hence the Weyl group is finite.

    Now, applying Lemma \ref{lem:invariant:decomposition}, one obtains a $\cW$-invariant decomposition of integral affine subspaces
    \begin{equation}
        \label{eq:decomp:leaf:space}
        \oB=(\oB)^\cW\oplus V,
    \end{equation}
    where one can identify the fixed point set with the intersection of all root hyperplanes
    \[ (\oB)^\cW=\bigcap_{r\in\cW}\cH_r. \] 
    Note that this fixed-point set $(\oB)^\cW$ is stable under the action of $\pi_1^\orb(B)$ since this action preserves the set of hyperplanes.
    It follows that $(\oB)^\cW$ is the preimage of an infinitesimal stratum $B_{\Si_0}$ under the covering projection $p:\oB\to B$. Moreover, all other infinitesimal stratum have dimension greater than $\Si_0$. By Theorem \ref{thm:inf-strat}, we conclude that item (i) holds, with the exception of the statement about $\pi_1(\Sigma_0)$.

    If $x\in \Si_0$, then $\res^{-1}(x)\cong G/N(T)$ where $G$ is a compact Lie integrating the isotropy Lie algebra $\gg_x$ and $T\subset G$ is a maximal torus. Hence, item (ii) follows from the last statement in Proposition \ref{prop:Weyl:group:isotropy at point}, since $\hx$ is a fixed point of $\cW$. It also follows that the $\cW$-invariant decomposition \eqref{eq:decomp:leaf:space} can be written as
    \[ \oB=(\oB)^\cW\oplus \tt^*_\ss,\]
    that the action of $\cW=W$ on $\tt^*_\ss$ is the classical action.

    Since $B$ is assumed complete, we have
    \[ \oB=B^\aff,\quad \oM=\hM^\aff, \quad \Gamma^\aff=\pi_1^\orb(B),\quad\hW\cong\cW\cong W,\]
    where the last isomorphisms follow from Proposition \ref{pro:non-trivial Weyl group elements} and item (ii). The main diagram \eqref{eq:Weyl:main-diagram} then yields:
    \begin{equation}
    \label{eq:main:diag:complete}
\vcenter{
\xymatrix{
*\txt{$p^{-1}(\hSi_0)\quad$ \\$\quad$ }\ar@{^{(}->}[d] \ar@<3pt>[r]&
*\txt{$\quad\hSi_0\quad$\\$\quad$}\ar@{^{(}->}[d]\ar@<3pt>[r] &  
*\txt{$\quad\Si_0\quad$\\$\quad$}\ar@{^{(}->}[d]\\
\oM=\hM^{\aff} \ar[d]_q\ar[r]^---{p}&    \hM \ar[d]\ar[r]^{\res} &  M \ar[d] & \\
\oB=B^{\aff}=\R^q \ar[r] &    B  \ar@{=}[r] & B \\
*\txt{$\quad$\\ $(\oB)^\cW\quad$}
\ar@{^{(}->}[u] \ar[r] 
&  
*\txt{$\quad$\\ $\quad$ $B_{\Si_0}\quad$}
\ar@{^{(}->}[u]
\ar@{=}[r]
& 
*\txt{$\quad$\\\quad $B_{\Si_0}\quad$} \ar@{^{(}->}[u]
}}
\end{equation}
Here the map $p:\oM\to\hM$ is a covering map with covering group $\pi_1^\orb(B)$ and 
\[ p^{-1}(\hSi_0)=q^{-1}((\oB)^\cW)).\]

We now prove item (iii). Recall that, by Theorem \ref{thm:canonical-orbifold-stratifications}, $B_{\Si_0}$ is a suborbifold of $B$, which therefore has the (ineffective) atlas
\[ (\oB)^\cW\rtimes\pi_1^\orb(B)\tto (\oB)^\cW. \]
Effectivization of this action amounts to quotient by the subgroup $N\subset \pi_1^\orb(B)$ consisting of elements that fix $(\oB)^\cW$. This results on the effective atlas for $B_{\Si_0}$
\[ (\oB)^\cW\rtimes\pi_1^\orb(B)/N\tto (\oB)^\cW. \]
Since $(\oB)^\cW$ is simply connected,  this implies -- cf.~\eqref{eq:pi1-orb-B-gamma} -- that
\[  \pi_1^\orb(B_{\Si_0})\cong \pi_1^\orb(B)/N, \]
and that we can identify
\[ (\oB)^\cW\cong \oB_{\Sigma_0}.\]
Obviously, $\cW\subset N$, so the splitting in Proposition \ref{prop:split:orbi:fundamental:group} gives a surjective morphism
\[ \pi_1^\orb(B^\reg)\cong \pi_1^\orb(B)/\cW \to \pi_1^\orb(B)/N\cong \pi_1^\orb(B_{\Si_0}). \]
In particular, the action of $\pi_1^\orb(B^\reg)$ on $\oB_{\Sigma_0}$ factors through the action of $\pi_1^\orb(B_{\Si_0})$.


Finally, to prove the first identity in item (i), we claim that:

\begin{lemma}
    The inclusion $i:\hSi_0\hookrightarrow \hM$ induces an isomorphism
\[ \pi_1(\hSi_0)\cong \pi_1(\hM). \]
\end{lemma}

To prove this lemma, observe that 
 $p^{-1}(\hSi_0)$ is connected and $\pi_1^\orb(B)$-invariant, since it equals $q^{-1}((\oB)^\cW))$ and $q:\oM\to\oB$ is a $\pi_1^\orb(B)$-equivariant. It follows that
\begin{equation}
    \label{eq:covering:hS_0}
    p:p^{-1}(\hSi_0)\to \hSi_0,
\end{equation} 
is a covering with covering group $\pi_1^\orb(B)$.
Also, $q:\oM\to\oB$ is a trivial fibration since the base $\oB\cong\R^q$ is contractible (by completeness) and the fibers are compact and connected (by s-properness). It follows that $p^{-1}(\hSi_0)$ is a deformation retract of $\oM$. 
Therefore, the top left square of \eqref{eq:main:diag:complete} induces a diagram of fundamental groups
\[ 
\xymatrix{
1\ar[r] & \pi_1(p^{-1}(\hSi_0)) \ar[d]_{i_*}\ar[d]\ar[r] & \pi_1(\hSi_0)\ar[d]_{i_*} \ar[r] & \pi_1^\orb(B) \ar@{=}[d] \ar[r] & 1\\
1\ar[r] & \pi_1(\oM) \ar[r] & \pi_1(\hM) \ar[r] & \pi_1^\orb(B) \ar[r] & 1
}
\]
where the left vertical arrow is an isomorphism, so the lemma holds.

Now, to prove (iii), notice that the restriction of the resolution $\res:\hSi_0\to\Si_0$ is a submersion with compact connected fibers diffeomorphic to $G/N(T)$. It follows that we have a commutative diagram
\[ 
\xymatrix{
\ar[r]&\pi_2(\Si_0)\ar[r] & \pi_1(\res^{-1}(x)) \ar[d]\ar[r] & \pi_1(\hSi_0)\ar[d]_{i_*} \ar[r]^{\res_*} & \pi_1(\Sigma_0) \ar[d]_{i_*} \ar[r] & 1\\
&1\ar[r] & \cW \ar[r] & \pi_1(\hM) \ar[r]^{\res_*} & \pi_1(M) \ar[r] & 1
}
\]
and we already know that the left two vertical arrows are isomorphisms. Hence, we obtain the first isomorphism of item (iii):
\[ \pi_1(\Si_0)\cong \pi_1(M). \]

\end{proof}

\subsection{A class of complete, s-proper, PMCTs}
\label{sec:example:s-proper:complete}

We consider a class of Poisson manifolds obtained by fixing the following data:
\begin{enumerate}[(a)]
    \item A Lie algebra $\gg$, semi-simple of compact type;
    \item A finite group $\Gamma\subset \Aut(\gg)$;
    \item A compact symplectic manifold $(S,\omega_S)$ with fundamental group $\Gamma$.
\end{enumerate}
Letting $q:\widetilde{S}\to S$ denote the universal covering space with symplectic form $\omega_{\widetilde{S}}:=q^*\omega_S$, the class of Poisson manifolds we are interested take the form 
\[ M=\widetilde{S}\times_\Gamma \gg^*, \]
where the Poisson structure $\pi_M$ is the quotient of the product $\omega_{\widetilde{S}}^{-1}\times\pi_{\gg^*}$.
\begin{description}
    \item[Claim] $(M,\pi_M)$ is a complete, s-proper, Poisson manifold. 
\end{description}


We will prove this claim by showing that $(M,\pi_M)$ is a (global) local linear model, as in Section \ref{sec:ex:local-model}. The base of its principal bundle $P$ will be the compact symplectic manifold $S$ and the structural group $G$ will be the union of all connected components of $\Aut(\gg)$ intersecting $\Gamma$. Classical Lie theory grants that $G$ is a possibly disconnected compact Lie group integrating $\gg$.

To see how $M$ arises from such linear local model, consider the symplectic manifold $\widetilde{S}\times T^*G$ with symplectic form $\pr_1^*\omega_{\widetilde{S}}+\pr_2^*\omega_\can$. It carries a symplectic action of $\Gamma$  defined by
\[ \gamma\cdot (s,\alpha):=(s\gamma^{-1},\gamma\alpha), \]
where the action on the factor $T^*G$ is via the lift of the left action of $G$ on itself. On the other hand, it carries a symplectic action of $G$ given by
\[ g\cdot(s,\alpha):=(s,\alpha g^{-1}), \]
where now we use the lift of the right action of $G$ on itself. Note that these are free and proper commuting actions, and their orbit spaces make up the following commutative diagram
\[
\xymatrix{
\widetilde{S}\times T^*G \ar[r]^{\Gamma\,\circlearrowright }\ar[d]_{G\,\circlearrowright } &  \widetilde{S}\times_\Gamma T^*G {}\save[]+<+2cm,0cm>*\txt{$\cong\widetilde{S}\times_\Gamma G\times \gg^*$}\restore
\ar[d]^{\circlearrowright\,G} \\
\widetilde{S}\times\gg^*\ar[r]_{\Gamma\,\circlearrowright } & \widetilde{S}\times_\Gamma\gg^*
}
\]
The isomorphism on the upper right corner comes from trivializing the cotangent bundle using left translations. Since the actions are symplectic, there are induced symplectic/Poisson structures on each space of the previous diagram:
\[
\xymatrix{
(\pr_1^*\omega_{\widetilde{S}}+\pr_2^*\omega_\can)^{-1}\ar[r]\ar[d] & (\omega^{\theta}_\lin)^{-1}\ar[d] \\
\omega_{\widetilde{S}}^{-1}\times\pi_{\gg^*}\ar[r] & \pi_M
}
\]
Here, on the right upper corner, we have the symplectic form
\[
\omega^{\theta}_\lin:= p^*\omega_S-\d\langle \theta,\cdot\rangle \in \Omega^2(P\times \gg^*),
\]
where $\theta$ denotes the flat connection on the (connected) principal $G$-bundle
\[ P:=\widetilde{S}\times_\Gamma G\to S,\]
whose pullback to $\widetilde{S}\times G$ is the trivial connection. This is precisely the closed, $G$-invariant, 2-form \eqref{loc-mod-2-form} of the local model, which in this case is globally symplectic since the connection is flat. This shows that $(M,\pi_M)$ is a linear local model and proves the claim.

Since $(M,\pi_M)$ is a linear local model, it follows from Example \ref{ex:local:model:Morita} that it has
\begin{enumerate}[(a)]
    \item complete leaf space $B$ with orbifold fundamental group
    \[ \pi_1^\orb(B)\cong N(T)/Z_G(T),\]
    where $T\subset G$ is a maximal torus, while its Weyl group is
    \[ \cW(M,\pi)\cong W=(N(T)\cap G^0)/T. \]
    \item $\oB=\tt^*$, where $\tt$ is the Lie algebra of $T$, and the $\pi_1^\orb(B)$-action is the adjoint action of $N(T)/Z_G(T)$.
\end{enumerate}

To connect with the results of Theorem \ref{thm:s-proper-complete}, note that the minimal stratum of $(M,\pi)$ consists of a single symplectic leaf,
\[ \Sigma_0=\widetilde{S}\times_\Gamma \{0\}\cong S,  \]
since any other leaf has dimension strictly larger than $S$. Any point in this stratum has isotropy isomorphic to $\gg$, so (a) and (b) agree with the results of Theorem \ref{thm:s-proper-complete}. 

Finally, observe that the split short exact sequence of Proposition \ref{prop:split:orbi:fundamental:group} becomes
\[ \xymatrix{
1\ar[r]& (N(T)\cap G^0)/T\ar[r] & N(T)/Z_G(T) \ar[r] & G/Z_G(T)G^0 \ar[r] & 1
}\]
)The group $\pi_1^\orb(B^\reg)\cong G/Z_G(T)G^0$ is non-trivial if and only if some element in $\Gamma$ acts non-trivially on $\gg^*/G^0$. In this case $\pi_1^\orb(B^\reg)$ is non-trivial while $\pi_1^\orb(B_{\Sigma_0})$ is trivial. For a concrete example, one can take any $\Gamma\subset \Aut(\mathfrak{so}(2m))$, with $m\geq 4$, which is not contained in the connected components of the identity.

\subsection{Remarks on Theorem \ref{thm:s-proper-is-regular} and twists.}
The crucial ingredient in the proof of Theorem \ref{thm:s-proper-is-regular} is the polynomial nature of the symplectic volume function. In particular, this theorem does not hold for general twisted Dirac structures, as shown for instance by the example of the Cartan-Dirac structures on compact Lie groups $G$. For these, as we saw in
Section \ref{ex:conjugacy-classes},
the twisted presymplectic leaves are the conjugacy classes, so as long as $G$ is not abelian, the twisted Dirac structure is never regular. 

Lying between Poisson and general twisted Dirac structures one has the class of $\phi$-twisted Poisson structures. For the latter, the Poisson condition for $\pi$ is replaced by $[\pi, \pi]= \pi^{\sharp}(\phi)$. Theorem \ref{thm:s-proper-is-regular} fails even for $\phi$-twisted Poisson structures. 

\begin{example}
    A $\phi$-twisted Poisson can be seen as $\phi$-twisted Dirac structures $L$ satisfying $L\cap TM= \{0\}$.  It follows that, at the global level, $\phi$-twisted Poisson structures correspond to $\phi$-twisted presymplectic groupoids $(\Sigma,\Omega)$ with \emph{non-degenerate} presymplectic form $\Omega$. Hence, for example, the Cartan-Dirac structure in the case of $G=\SU(2)$ becomes a non-regular twisted Poisson structure that integrates to a compact groupoid.    
\end{example}

The foliation of the Cartan-Dirac structure on $\SU(2)=\S^3$ in the previous example consists of two zeros at the north and south poles (i.e., $I$ and $-I$) and a family of 2-spheres. Hence, the twist $\phi$ (the Cartan 3-form), actually vanishes on the leaves, so $[\pi,\pi]=0$. However, as an untwisted Poisson manifold, $(\S^3,\pi)$ is non-integrable: the symplectic area of the spheres has a critical point.

Generalizing this example, assume that $\pi$ is an honest Poisson structure on $M$ and that one has a
closed 3-form $\phi\in\Omega^3(M)$ that vanishes on its symplectic leaves, that is, satisfying
\[ \phi(\pi^{\sharp}\alpha, \pi^{\sharp}\beta, \pi^{\sharp}\gamma)= 0,\]
for all $\alpha,\beta,\gamma\in\Omega^1(M)$.
We call such a 3-form a \textbf{fake twist of $\pi$}. It allows us to interpret $\pi$ both as a Poisson structure as well as a twisted one. The example above shows that a fake twist can turn a possibly non-regular (even non-integrable!) Poisson structure to one of compact type.

The regular case is easier to handle and sheds light on the difficulty of the non-regular case as well. The key remark is that, starting with a regular Poisson structure $\pi$ and a fake twist $\phi$, when interpreted as a $\pi$-twisted Poisson structure, the corresponding foliation remains the same. Actually, the resulting algebroid of the $\phi$-twisted $\pi$ is still $T^*M$, with the same anchor $\rho=\pi^\sharp$, but with the modified Lie bracket
\begin{equation}\label{eq:twisted-Poisson-bracket} 
[\alpha, \beta]_{\pi,\phi}=  [\alpha, \beta]_{\pi}+ i_{\pi^{\sharp}(\alpha)}i_{\pi^{\sharp}(\beta)} \phi.
\end{equation} 
Therefore, for compactness to be achieved from the twisted perspective, the original symplectic foliation must be of compact type and must admit a transverse integral affine structure. In fact, this is all one needs.

\begin{proposition}
Let $(M,\pi)$ be a regular Poisson manifold whose foliation $\cF_\pi$ is of $\cC$-type and admits a transverse integral affine structure. Then there exists a fake twist $\phi$ such that $\pi$ is of $\cC$-type when viewed as a $\phi$-twisted Poisson structure.  
\end{proposition}

\begin{proof}
Let $\widetilde{\omega}\in\Omega^2(M)$ be an extension of the leafwise symplectic form with the property that 
$\widetilde{\omega}(\pi^{\sharp}(\alpha),X)= \alpha(X)$ for all $X$ and $\alpha$ (it exists, since $\pi$ is regular) and set $\phi=\d\widetilde{\omega}$. Note that the extension $\widetilde{\omega}$ determines a splitting of the short exact sequence:
\[ \xymatrix{0\ar[r] & \nu^*(\cF_\pi)\ar[r] & T^*M \ar[r] & T\cF_\pi \ar[r] \ar@/^0.5pc/[l]^{i_{\cdot}\widetilde{\omega}}& 0} \]
The middle term will be interpreted as the 
the algebroid $A^\phi$ corresponding to  $\pi$ viewed as a $\phi$-twisted Poisson, as discussed above, with Lie bracket (\ref{eq:twisted-Poisson-bracket}). We claim that the splitting of the sequence becomes an algebroid splitting of $A^\phi$. In other words,
\[ i_{[X,Y]}\widetilde{\omega}=[i_X\widetilde{\omega},i_Y\widetilde{\omega}]_{\pi,\phi}, \]
for any $X,Y\in\X(\cF_\pi)$. To see this we can assume that $X=\pi^{\sharp}(\alpha)$, $Y=\pi^{\sharp}(\beta)$, so that $i_X\widetilde{\omega}=\alpha$ and $i_Y\widetilde{\omega}=\beta$. Then:
\begin{align*} 
    [i_X\widetilde{\omega},i_Y\widetilde{\omega}]_{\pi,\phi}
    &=[\alpha,\beta]_{\pi}+i_{\pi^{\sharp}(\alpha)}i_{\pi^{\sharp}(\beta)} \phi\\
    &=\Lie_{\pi^{\sharp}(\alpha)}\beta-i_{\pi^{\sharp}(\beta)}\d\alpha-i_{\pi^{\sharp}(\beta)}i_{\pi^{\sharp}(\alpha)} \d\widetilde{\omega}\\
    &=\Lie_{X}i_{Y}\widetilde{\omega}-i_{Y}\Lie_{X}\widetilde{\omega}
    =i_{[X,Y]}\widetilde{\omega}
\end{align*}

We deduce that the algebroid $A^\phi$ is isomorphic to the semi-direct product algebroid $T\cF_\pi\ltimes\nu^*(\cF_\pi)$, which is just the algebroid associated with the Dirac structure defined by $\cF_\pi$. This algebroid admits an integration of $\cC$-type, namely $\Hol(\cF_\pi)\ltimes \nu^*(\cF_\pi)/\Lambda$, where $\Lambda$ is the transverse integral affine structure (see \cite[Section 4.4]{CFM-I}
\end{proof}

\section{Open Problems} 
\label{sec:open}

In this closing section we list some questions and open problems that arise naturally from our work and that we believe are interesting and important to further understand the geometry of PMCTs and DMCTs. 

\subsection*{Finite Coxeter groups of non-Lie type and the Weyl group}
\label{sec:open:Coxeter}
Theorem \ref{thm:Davis} shows that the Weyl group of a Poisson manifold of proper type is a Coxeter group. The examples discussed in Section \ref{sec:examples:parts:removed} show that such a Poisson manifold can have a Weyl group that is not isomorphic to either a classical Weyl group of a compact Lie algebra or an affine Weyl group of some compact Lie group. We note that all the examples in that section have Weyl groups of infinite order.  

Recall that finite Coxeter groups are classified, and among them, there exists the class of reflection groups of Lie type, i.e., those that arise as Weyl groups of some compact Lie algebra. This raises the following problem.  

\begin{problem}
Is there a Poisson manifold $(M,\pi)$ of proper type with a finite Weyl group $\cW(M,\pi)$ that is not of Lie type?  
\end{problem}

For example, can one find an example of a Poisson manifold of proper type with Weyl group isomorphic to $D_5$?

\begin{remark}
    After the submission of this paper and posting it in the arXiv, Joshua Mundiger send us a solution to this problem, showing that the answer is negative -- see Appendix \ref{sec:appendix}.
\end{remark}

\subsection*{Duistermaat-Heckman theory for Dirac manifolds of compact types}  
Our treatement of Duistermaat-Heckman theory in Sections \ref{sec:variation:volume} and \ref{sec:DH:formula} focussed entirely on the Poisson case. However, one can wonder what happens in the general Dirac setting. On the one hand, Weyl integration formulas have versions both for Lie algebras and for Lie groups, see e.g. \cite{DK}. On the other hand,
the work of Alekseev, Meinrenken and Woodward \cite{AMW2} develops Duistermaat-Heckman theory for Lie-group valued moment maps. The discussion in \cite{ABM} shows that a generalization of the theory to the general Dirac setting is not straightforward -- but also provides some ideas on how to proceed. 

\begin{problem}
Extend the Duistermaat-Heckmann theory to the setting of Dirac manifolds of compact types. 
\end{problem}

\subsection*{Reductive holomorphic Poisson structures} The complex counterparts of compact Lie algebras (respectively, compact Lie groups) are complex reductive Lie algebras (respectively, complex reductive Lie groups). Many results that hold for compact Lie algebras and groups have analogues in the context of complex reductive Lie algebras and groups. Since many of the results in this paper can be interpreted as generalizations of classical results for compact Lie algebras and Lie groups to Poisson geometry, it is natural to ask the following:

\begin{problem}
Define the notion of a reductive holomorphic Poisson/Dirac manifold and establish an analogue of the theory of Poisson manifolds of proper type in the holomorphic context.
\end{problem}

For example, reductive holomorphic Poisson (respectively, Dirac) manifolds should include, as examples,  holomorphic symplectic manifolds and holomorphic linear Poisson structures on the duals of complex reductive Lie algebras (respectively, the Cartan-Dirac structures on complex reductive Lie groups). The work of Evens and Lu on the Grothendieck-Springer resolution for complex semi-simple Poisson-Lie groups \cite{EL07} should also provide some clues to solve this problem.

\subsection*{Markus conjecture for integral affine orbifolds}
\label{sec:open:Markus}

Recall that for an affine manifold the Markus conjecture states that a compact affine manifold is complete iff it has a parallel volume (see, e.g., \cite{GH86})). For integral affine manifolds a parallel volume always exist, and so the conjecture says that every compact integral affine manifold is complete. 

Our results concerning polynomial functions on integral affine orbifolds in Section \ref{sec:polynomial:functions} suggest that the conjecture should hold for such orbifolds.

\begin{problem}
Show that every compact integral affine manifold is complete. 
\end{problem}

It is proved in \cite{GH86} that every polynomial on a compact integral affine \emph{manifold} is constant. This result would follow easily from the Markus conjecture. Similarly, solving the above question would imply that every polynomial on a compact integral affine \emph{orbifold} is constant. Corollary \ref{cor:no-polynomial} provides some evidence for this conjecture.

For us, a positive answer to this problem would have the consequence that every compact PMCT (or DMCT) is complete. It is an intriguing question if, conversely, one can use Poisson geometry to give a proof of the Markus conjecture.

\subsection*{Classification of complete s-proper PMCTs } 

Theorem \ref{thm:s:proper:complete} provides structural results for complete, s-proper, PMCTs. The results in that theorem -- see also Section \ref{sec:example:s-proper:complete} -- suggest that the following characterization of such PMCTs should hold:

\begin{problem}
Given a complete, s-proper, Poisson manifold $(M,\pi)$ show that it is isomorphic to
\[ (\widetilde{\Si}_0\times \gg^*_\ss)/\pi_1(\Si_0), \]
where $\widetilde{\Si}_0$ is the universal covering space of the minimal strata and $\gg_\ss=[\gg,\gg]$ is the semisimple part of the isotropy Lie algebra $\gg$ of any point in $\Si_0$. 
\end{problem}

Notice that one can think of this as a global linearization result around the minimal strata. A special case of this is the classical linearization result of Ginzburg-Weinstein \cite{GW92}. In fact, let $(G^*,\pi_{G^*})$ be the dual of a compact, semi-simple, Poisson-Lie group $(G,\pi_G)$. Since the dressing action of $G$ on $G^*$ is complete, it follows that  $G^*$ is a s-proper, complete, Poisson manifold. The minimal strata $\Si_0$ consists of the identity of $G^*$. So if the result above holds, one obtains a Poisson  isomorphism
\[ G^*\cong\gg^*,\]
Hence, one recovers the global linearization  result of Ginzburg and Weinstein.

\subsection*{Multiplicity free spaces and symplectic gerbes}
\label{sec:open:gerbes}

Recall that, given a Poisson manifold $(M,\pi)$, a \emph{symplectic realization} consists of a symplectic manifold $(X,\omega)$ and a surjective submersion $\mu:(X,\omega)\to (M,\pi)$ that is a Poisson map. It is called \emph{isotropic} when its fibers are isotropic submanifols. 
Regular Poisson manifolds are characterized among Poisson manifolds by the existence of isotropic realizations.

In \cite{DazordDelzant}, Dazord and Delzant showed that an isotropic realization $\mu:(X,\omega)\to(M,\pi)$, with $\mu$ proper, induces a transverse affine structure on  $\cF_\pi$.  Conversely, given a transverse integral affine structure $\Lambda$ on $\cF_\pi$, they showed that the existence of a proper isotropic realization of $(M,\pi)$ inducing $\Lambda$ is obstructed by a certain class.

In \cite{CFM-II}, we extended the Dazord-Delzant theory. In one direction, we showed that a proper isotropic realization $\mu:(X,\omega)\to (M,\pi)$ of a  Poisson manifold comes with a canonical proper symplectic integration of $(M,\pi)$. It is characterized as the smallest symplectic integration that acts on the realization, and its isotropy bundle has connected component of the identity $T^*M/\Lambda$, where $\Lambda$ is the transverse integral affine structure determined by the given realization. In the opposite direction, we showed that any regular, proper symplectic groupoid defines a symplectic gerbe over its leaf space, which is classified by a symplectic version of the Dixmier-Douady class. Moreover, we proved that the latter class pulls back to the class of Dazord and Delzant, and that it is trivial if and only if there exists a complete isotropic realization providing a symplectic Morita equivalence with a symplectic torus bundle.

One would like to extend these results beyond the regular case. The natural generalization of isotropic realizations is \emph{multiplicity-free} symplectic realizations: a symplectic realization $\mu:(X,\omega)\to M$ of $(M,\pi)$ such that each isotropy Lie algebra $\gg_x$ acts transitively on the fibers of $\mu$. However, one must relax the condition that $\mu$ be a submersion; otherwise one remains in the regular case.

\begin{problem}
Develop the theory of (non-singular) multiplicity-free realizations of Poisson manifolds of proper type.  
\end{problem}  

One possible approach is to consider pre-symplectic ``isotropic realizations" of the Weyl resolution of $(M,\pi)$. In fact, one expects that a suitable version of the Weyl resolution can be constructed for the entire realization. The correct notion should allow to extend the above notions and results to the non-regular case, including versions of the Dazord-Delzant class, symplectic gerbes, and their Dixmier-Douady classes for non-regular PMCTs.

\appendix
\section{Finite Weyl Groups of PMCTs are of Lie type}
\label{sec:appendix}
\smallskip
\begin{center}by \textsc{Joshua Mundinger}\end{center}
\medskip

The purpose of this appendix is to show that if $N$ is a simply connected integral affine manifold and $W$ is a finite integral affine reflection group on $N$ (Definition \ref{def:integral affine reflection group}), then $W$ is a Weyl group. This solves Open Problem 1.

\begin{lemma}[\cite{Humphreys90}, p. 39]\label{lemma: condition for Weyl}
  Suppose that $W$ is a finite Coxeter group, generated by reflections $s_\alpha$ with $m_{\alpha\beta} = ord(s_\alpha s_\beta)$. Then $W$ is a Weyl group if and only if $m_{\alpha\beta} \in \{2,3,4,6\}$ for all $\alpha \neq \beta$.\qedhere
\end{lemma}

Recall from §\ref{sec:geometric:reflections} that a \textbf{geometric reflection} on a connected manifold $N$ is a smooth involution $r:N \to N$ such that the fixed-point set $N^r$ is a codimension one submanifold separating $N$. By \cite[Lemma 10.1.3]{Davis08}, $N \setminus N^r$ has exactly two components which are exchanged by $r$.

\begin{lemma}\label{lem: reflecting hyperplanes intersect}
    Let $N$ be a simply connected manifold and let $s$ and $t$ be geometric reflections on $N$. If $st$ has finite order, then $N^s \cap N^t$ is not empty.
\end{lemma}
\begin{proof}
    Suppose $N^s \cap N^t$ is empty.
    Let $U = N \setminus N^s$ and $V = N \setminus N^t$.
    Since $s$ and $t$ are geometric reflections on a connected manifold, 
    $U$ has two components $U_1$ and $U_2$ exchanged by $s$, and $V$ has two components $V_1$ and $V_2$ exchanged by $t$.
    Consider the Mayer-Vietoris sequence associated to open covering $N = U \cup V$. Since $H_1(N) = 0$, $N \setminus (N^s \cup N^t) = U \cap V$ has exactly three components, say $W_1, W_2, W_3$, which are each contained in a unique pairwise intersection $U_i \cap V_j$ for some $i$ and $j$.
    Without loss of generality,
    \begin{align*}
    W_1 &\subseteq U_1 \cap V_1, \\
    W_2 &\subseteq U_2 \cap V_1,\\
    W_3 &\subseteq U_2 \cap V_2.
    \end{align*}
    Thus $U_1 \cap V_2$ is empty, so since $U_1$ and $V_2$ are connected, $U_1 \subseteq V_1$ and $V_2 \subseteq U_2$.
    Thus $W_1 = U_1$ and $W_3 = V_2$.

    \begin{figure}[h]
  \includegraphics[width=2.5in]{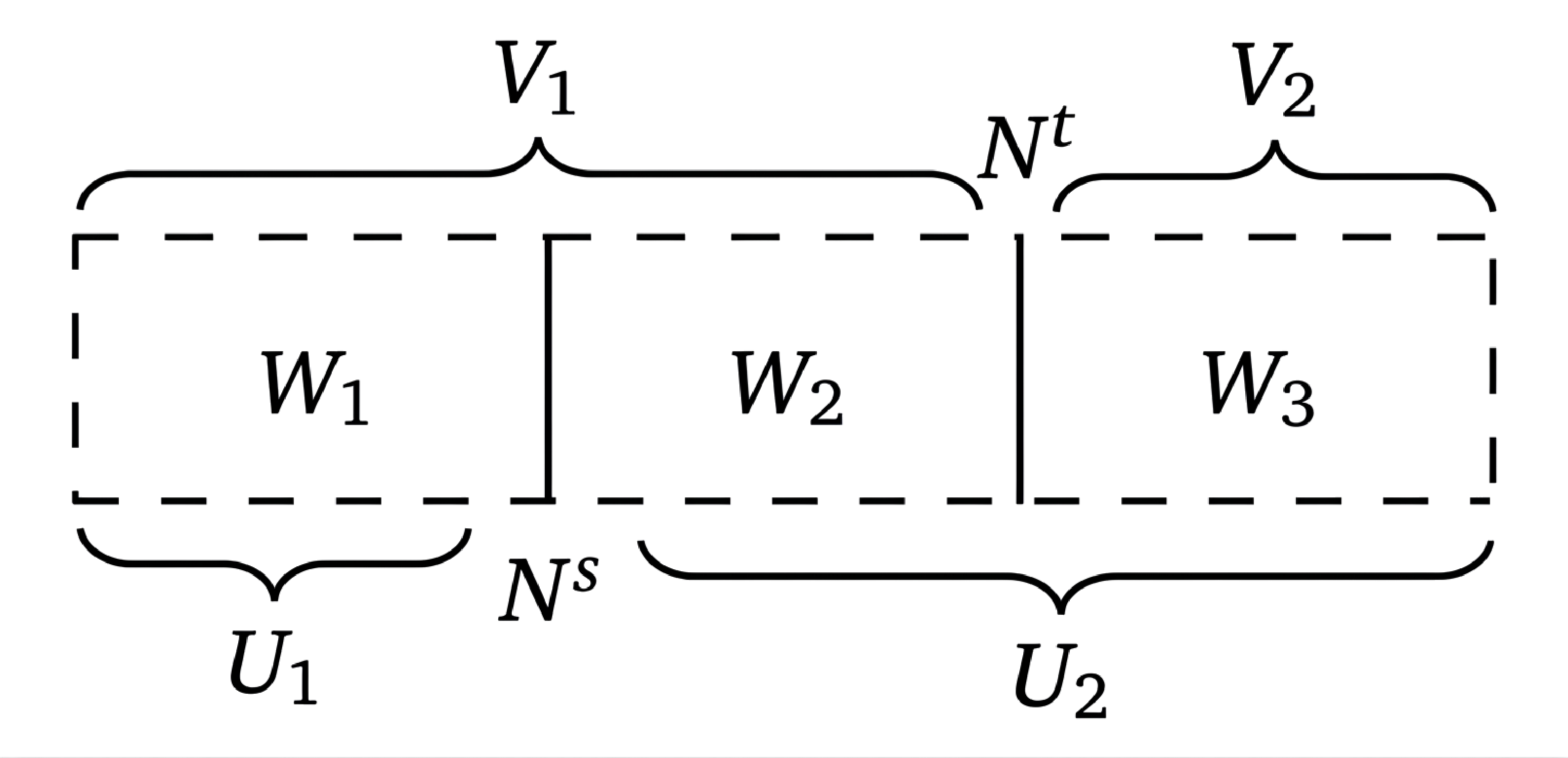}
  \caption{The connected components of $N \setminus (N^s \cup N^t)$.}
  \end{figure}
  
  We proceed by ping-pong. Now $st(W_2) \subseteq s(t(V_1)) \subseteq s(V_2) \subseteq s(U_2) = U_1 = W_1$.
  Similarly,
  $st(W_1) \subseteq s(t(V_1)) = s(V_2) \subseteq s(U_2) = U_1 = W_1$.
  Hence $(st)^n(W_2) \subseteq W_1$ for all $n > 0$, so $st$ has infinite order.
\end{proof}

\begin{theorem}\label{thm: finite-integral-affine-sc-is-Weyl}
  Let $N$ be a simply connected integral affine manifold and $W$ be a finite integral affine reflection group on $N$.
  Then $W$ is a Weyl group. 
\end{theorem}
\begin{proof}
    
  Say $\{s_\alpha\}_{\alpha \in I}$ is a generating set of reflections for $W$ and $H_\alpha = N^{s_\alpha}$.
  Let $\alpha$ and $\beta$ be distinct elements of $I$.
  Since $W$ is finite, Lemma \ref{lem: reflecting hyperplanes intersect} implies that $H_\alpha \cap H_\beta \neq \varnothing$. Let $x \in H_\alpha \cap H_\beta$.
  Define $\sigma_\alpha = \d_x s_\alpha$, $\sigma_\beta = \d_xs_\beta$. Then $\sigma_\alpha$ and $\sigma_\beta$ are reflections on $V = T_x N$. 
  Since $W$ acts by integral affine transformations, both $\sigma_\alpha$ and $\sigma_\beta$ fix some lattice $L \subseteq V$.
  Thus $\tr(\sigma_\alpha\sigma_\beta) \in \mathbb Z$.
  On the other hand, $\sigma_\alpha\sigma_\beta$ is a product of two reflections, so it has a fixed subspace of codimension 2 and is a rotation by some angle $\theta$ on a complementary plane. 
  Then $\tr(\sigma_\alpha\sigma_\beta) = 2 \cos(\theta) + (\dim V - 2)$, so $2 \cos \theta \in \mathbb Z$.
  Thus $\theta = 2\pi / m$ for $m \in \{2,3,4,6\}$ and $(\sigma_\alpha \sigma_\beta)^m = 1$.
  Since $s_\alpha$ and $s_\beta$ are integral affine transformations on $N$, $(s_\alpha s_\beta)^m$ is the identity on a neighborhood of $x$.
  Again since $(s_\alpha s_\beta)^m$ is integral affine, following paths in $N$ from $x$ shows that $(s_\alpha s_\beta)^m$ is the identity at all points of $N$.
  Hence $m_{\alpha\beta} = m \in \{2,3,4,6\}$. Lemma \ref{lemma: condition for Weyl} implies $W$ is a Weyl group.
\end{proof}

The hypothesis of simple connectivity in Theorem \ref{thm: finite-integral-affine-sc-is-Weyl} is necessary:
\begin{example}
    For all $m > 1$, the dihedral group $D_m$ acts on $S^1 \cong \R/m\Z$ by integral affine reflections, but $D_m$ is not a Weyl group unless $m \in \{2,3,4,6\}$. 
    For consider $\tilde N = \R$ with standard integral affine structure. 
    Let $\Gamma = D_\infty = \langle s,t \mid s^2 = t^2 = 1\rangle $ be the infinite dihedral group, generated by reflections $s,t$ about the points $0$ and $1$ of $\tilde N$. If $N = \tilde N / \langle (st)^m\rangle$, then $N$ is an integral affine manifold and the dihedral group $D_m = \Gamma / \langle (st)^m\rangle$ acts on $N$ as a finite integral affine reflection group. 
\end{example}

\section*{List of Notations}

The following notations are used throughout this paper:
{\small
\begin{description}
    \item[$-(M,\pi)$] manifold $M$ with Poisson bivector $\pi\in\X^2(M)$ (section \ref{sec:intro});
    \item[$-\cF_{\pi}$] symplectic foliation of $(M,\pi)$. Its leaves are denoted $(S,\omega_S)$ (section \ref{sec:intro});
    \item[$-(\cG,\Omega)$] symplectic groupoid integrating $(M,\pi)$; its source/target maps are denoted $\s,\t:\cG\to M$ (section \ref{sec:intro});
    \item[$-(\gg^*,\pi_{\gg^*})$] linear Poisson structure on the dual of a Lie algebra $\gg$ (section \ref{sec:ex:simple-local-model});
    \item[$-(Q/G,\pi_\red)$] Hamiltonian local model associated with Hamiltonian $G$-space $\mu: (Q, \omega)\to \gg^*$ (section \ref{sec:ex:simple-local-model});
    \item[$-(\cM,\pi_\lin)$, $(\cM_\lin,L_\lin)$] Poisson and Dirac linear local models  (section \ref{sec:ex:local-model});
    \item[$-(\cG_\lin,\Omega_\lin)$] canonical (pre-)symplectic groupoid integrating the linear local models (sections \ref{sec:ex:local-model} and \ref{sec:Dirac:local-model});
    \item[$-\gg_x(M,\pi)$] isotropy Lie algebra of $(M,\pi)$ at $x\in M$ (section \ref{sec:canonical:stratifications});
    \item[$-\cS(M,\pi)$, $\cSi(M,\pi)$] canonical stratification and canonical infinitesimal stratification of $(M, \pi)$ (sections \ref{sec:The canonical infinitesimal stratification of PMCTs} and \ref{sec:The canonical stratification of PMCTs});
    \item[$-\Sc$, $\Si$] strata of $\cS(M,\pi)$ and $\cSi(M,\pi)$, respectively (sections \ref{sec:The canonical infinitesimal stratification of PMCTs} and \ref{sec:The canonical stratification of PMCTs});
    \item[$-\Sc_k(M,\pi)$, $\Si_k(M,\pi)$] unions of codimension $k$ strata of $\cS(M,\pi)$ and $\cSi(M,\pi)$, respectively (sections \ref{sec:The canonical infinitesimal stratification of PMCTs} and \ref{sec:The canonical stratification of PMCTs});
    \item[$-M^{\reg}$, $M^{\princ}$] regular and principal parts of $(M,\pi)$ (sections \ref{sec:The canonical infinitesimal stratification of PMCTs} and \ref{sec:The canonical stratification of PMCTs});
    \item[$-\rho^{\Hol}: \pi_1(S, x)\to \GL(\zz(\gg_x))$] linear holonomy representation of $(M,\pi)$ at $x$ (section \ref{sec:The canonical infinitesimal stratification of PMCTs});
    \item[$-\partial_x: \pi_2(S, x)\to Z(G(\gg_x))$] monodromy map of $(M,\pi)$ at $x$ (section \ref{sec:The canonical infinitesimal stratification of PMCTs});
    \item[$-(\cT,\Omega_{\cT})$] presymplectic torus bundle, often the connected component of the isotropy of a symplectic groupoid (sections \ref{sec:stratifications:DMCT} and \ref{sec:stratifications:leaf:space}).
    \item[$-\res:(\hM,\hL)\to (M,\pi)$] Weyl resolution of a PMCT (section \ref{sec:Weyl:resolution});
    \item[$-\hF$] presymplectic foliation of $(\hM,\hL)$. Its leaves are denoted $(\hS,\widehat{\omega}_{\hS})$ (section \ref{sec:Weyl:resolution});
    \item[$-(\hG,\hOmega)$] presymplectic groupoid integrating $(\hM,\hL)$ (section \ref{sec:Weyl:resolution});
    \item[$-B=B(M,\pi)$] leaf space of a PMCT $(M, \pi)$ (section \ref{sec:The smooth structure}); 
    \item[$-\cS(B)$, $\cSi(B)$] canonical stratification and canonical infinitesimal stratification of the leaf space $B$ (section \ref{sec:stratifications:leaf:space});
    \item[$-B_\Sc$, $B_\Si$] leaf spaces of the strata $\Sc$ and $\Si$ (section \ref{sec:stratifications:leaf:space});
    \item[$-B^\lin$, $B^\aff$, $\oB$, $\hM^\lin$, $\hM^\aff$, $\hM^\orb$] linear holonomy cover, affine holonomy cover and orbifold universal cover of leaf space $B$, and corresponding covering spaces of the resolution $\hM$ (section \ref{sec:main:diagram});
    \item[$-\htt\to \hM$] tautological line bundle (section \ref{sec:main:diagram});
    \item[$-\Hol(\hF)$, $\Mon(\hF)$] holonomy and monodromy groupoids of $\hF$ 
    (section \ref{sec:main:diagram});
    \item[$-\Lambda$] transverse integral affine structure to $\hF$ (section \ref{sec:main:diagram});
    \item[$-h^{\lin}: \Pi_1(\hM) \to \GL_\Lambda(\,\htt^*\,)$, $h^{\aff}: \Pi_1(\hM) \to \Aff_\Lambda(\,\htt^*\,)$] linear holonomy and affine holonomy actions (section \ref{sec:main:diagram});
    \item[$-K^\lin$, $K^\aff$ $\Gamma^\lin$, $\Gamma^\aff$] kernels and images of $h^\lin$ and $h^\aff$ at the base point $\hx_0\in\hM$
    (section \ref{sec:main:diagram});
    \item[$-\dev$, $\dev_0$] developing map and developing map based at a point (section \ref{sec:main:diagram});
    \item[$-\pi_1^\orb(B,b_0)$] orbifold fundamental group of $B$ based at $b_0$ (section \ref{sec:The orbifold fundamental group});
    \item[$-\pi_1(\cG,x)$] fundamental group of the Lie groupoid $\cG\tto M$ based at $x\in M$ (section \ref{sec:The orbifold fundamental group});
    \item[$-\hW_{\hx}$, $\cW_{\hx}$] kernel of $\res_*: \pi_1(\hM,\hx)\to\pi_1(M,x)$ and  Weyl group of $(M,\pi)$ based at $\hx\in\hM$ (section \ref{sec:Weyl:group:PMCT});
    \item[$-B^\reg$, $(B^\lin)^\reg$,$(B^\aff)^\reg$,$(\oB)^\reg$] regular locus in leaf spaces $B$, $B^\lin$, $B^\aff$ and $\oB$ (section \ref{sec:abstract:reflections})
    \item[$-r$] integral affine geometric reflection (section \ref{sec:geometric:reflections});
    \item[$-\delta:\Gamma\to\Z_2$] parity character of $\Gamma$ (section \ref{sec:geometric:reflections});
    \item[$-\Refl\subset \cW$] subset of geometric reflections (section \ref{sec:Coxeter});
    \item[$-\cH_r$] fixed-point set (hyperplane) defined by $r\in \Refl$ (section \ref{sec:Coxeter});
    \item[$-\Delta$] a chamber in $\oB$ (section \ref{sec:Coxeter});
    \item[$-\Refl_\Delta$] subset of simple reflections associted to chamber $\Delta$ (section \ref{sec:Coxeter});
    \item[$-\hH\to \hM$] bundle of 2nd cohomology of the fibers of $\hF$ (section \ref{sec:linear:variation});
    \item[$-\varpi$] section of $\hH$ induced by presymplectic forms on leaves of $\hF$  (section \ref{sec:linear:variation});
    \item[$-\cT_\varpi$] transport map  (section \ref{sec:linear:variation});
    \item[$-\Ilin$, $\Iaff$] linear and affine variation maps (section \ref{sec:linear:variation});
    \item[$- \hLL\to \hM$] line bundle of top cohomology of the fibers of $\hF$ (section \ref{sec:volume:polynomial});
    \item[$- \sigma_{\vol}$] section of $\hLL$ induced by $\tfrac{\varpi^k}{k!}$ (section \ref{sec:volume:polynomial});    
    \item[$-\VVO:B^\aff\to \R$] symplectic volume polynomial (section \ref{sec:volume:polynomial});
    \item[$-\delta_{\hLL}: \pi_1(\hM)\to \Z_2$] character associated to $\hLL$ (section \ref{sec:volume:polynomial});
    \item[$-{\mu^{\aff}_{B}}$, ${\mu^{DH}_{B}}$] Lebesgue and Duistermaat-Heckman measures on $B$ (section \ref{sec:DH:formula});
    \item[$-\VV^2: B\to\R$] square volume polynomial (section \ref{sec:DH:formula});
    
\end{description}
}
We also use some standard notations. For a linear subspace $W\subset V$ we denote by $W^0\subset V^*$ its annihilator, while $V^\Gamma$ denotes the fixed-point set of a $\Gamma$-action. For a Lie group $G$ we denote by $G^0$ the component containing the identity and by $Z(G)$ its center, while $\zz(\gg)$ denotes the center of a Lie algebra $\gg$. We also denote by $Z_G(H)$ and by $N(H)$  the centralizer and the normalizer of a subgroup $H\subset G$. For a compact Lie group $G$, we denote a maximal torus by $T$ and by $W=N(T)/T$ its Weyl group. At the Lie algebra level, a maximal torus is denoted by $\tt\subset\gg$ and the Grassmannian of maximal torus is denoted by $\cT(\gg)$. We use the symbol $\nu(\cdot)$ to denote normal spaces and normal bundles to submanifolds and foliations.

\bibliographystyle{abbrv}

\begin{thebibliography}{10}

\bibitem{ABM}
A.~Alekseev, H.~Bursztyn, and E.~Meinrenken.
\newblock Pure spinors on {L}ie groups.
\newblock {\em Ast\'{e}risque}, (327):131--199 (2010), 2009.

\bibitem{AAM98}
A.~Alekseev, A.~Malkin, and E.~Meinrenken.
\newblock Lie group valued moment maps.
\newblock {\em J. Differential Geom.}, 48(3):445--495, 1998.

\bibitem{AMW2}
A.~Alekseev, E.~Meinrenken, and C.~Woodward.
\newblock Duistermaat-{H}eckman measures and moduli spaces of flat bundles over
  surfaces.
\newblock {\em Geom. Funct. Anal.}, 12(1):1--31, 2002.

\bibitem{AlvarezTese}
D.~{\'{A}}lvarez.
\newblock {\em Integrability of quotients in Poisson and Dirac geometry}.
\newblock PhD thesis, Instituto de Matem{\'a}tica Pura e Aplicada (IMPA), 2019.

\bibitem{BF14}
O.~Brahic and R.~L. Fernandes.
\newblock Integrability and reduction of {H}amiltonian actions on {D}irac
  manifolds.
\newblock {\em Indag. Math. (N.S.)}, 25(5):901--925, 2014.

\bibitem{BH99}
M.~R. Bridson and A.~Haefliger.
\newblock {\em Metric spaces of non-positive curvature}, volume 319 of {\em
  Grundlehren der mathematischen Wissenschaften [Fundamental Principles of
  Mathematical Sciences]}.
\newblock Springer-Verlag, Berlin, 1999.

\bibitem{BtD95}
T.~Br\"{o}cker and T.~tom Dieck.
\newblock {\em Representations of compact {L}ie groups}, volume~98 of {\em
  Graduate Texts in Mathematics}.
\newblock Springer-Verlag, New York, 1995.
\newblock Translated from the German manuscript, Corrected reprint of the 1985
  translation.

\bibitem{BCWZ}
H.~Bursztyn, M.~Crainic, A.~Weinstein, and C.~Zhu.
\newblock Integration of twisted {D}irac brackets.
\newblock {\em Duke Math. J.}, 123(3):549--607, 2004.

\bibitem{BR03}
H.~Bursztyn and O.~Radko.
\newblock Gauge equivalence of {D}irac structures and symplectic groupoids.
\newblock {\em Ann. Inst. Fourier (Grenoble)}, 53(1):309--337, 2003.

\bibitem{CF2}
M.~Crainic and R.~L. Fernandes.
\newblock Integrability of {P}oisson brackets.
\newblock {\em J. Differential Geom.}, 66(1):71--137, 2004.

\bibitem{CFM-I}
M.~Crainic, R.~L. Fernandes, and D.~Mart\'{\i}nez~Torres.
\newblock Poisson manifolds of compact types ({PMCT} 1).
\newblock {\em J. Reine Angew. Math.}, 756:101--149, 2019.

\bibitem{CFM-II}
M.~Crainic, R.~L. Fernandes, and D.~Mart\'{\i}nez~Torres.
\newblock Regular {P}oisson manifolds of compact types.
\newblock {\em Ast\'{e}risque}, 413, 2019.

\bibitem{CFM21}
M.~Crainic, R.~L. Fernandes, and I.~M\u{a}rcu\c{t}.
\newblock {\em Lectures on {P}oisson geometry}, volume 217 of {\em Graduate
  Studies in Mathematics}.
\newblock American Mathematical Society, Providence, RI, [2021] \copyright
  2021.

\bibitem{CraMe}
M.~Crainic and J.~N. Mestre.
\newblock Orbispaces as differentiable stratified spaces.
\newblock {\em Lett. Math. Phys.}, 108(3):805--859, 2018.

\bibitem{CM19}
M.~Crainic and J.~N. Mestre.
\newblock Measures on differentiable stacks.
\newblock {\em J. Noncommut. Geom.}, 13(4):1235--1294, 2019.

\bibitem{CS}
M.~Crainic and I.~Struchiner.
\newblock On the linearization theorem for proper {L}ie groupoids.
\newblock {\em Ann. Sci. \'{E}c. Norm. Sup\'{e}r. (4)}, 46(5):723--746, 2013.

\bibitem{Davis83}
M.~W. Davis.
\newblock Groups generated by reflections and aspherical manifolds not covered
  by {E}uclidean space.
\newblock {\em Ann. of Math. (2)}, 117(2):293--324, 1983.

\bibitem{Davis08}
M.~W. Davis.
\newblock {\em The geometry and topology of {C}oxeter groups}, volume~32 of
  {\em London Mathematical Society Monographs Series}.
\newblock Princeton University Press, Princeton, NJ, 2008.

\bibitem{DazordDelzant}
P.~Dazord and T.~Delzant.
\newblock Le probl\`eme g\'en\'eral des variables actions-angles.
\newblock {\em J. Differential Geom.}, 26(2):223--251, 1987.

\bibitem{HoFe}
M.~del Hoyo and R.~L. Fernandes.
\newblock Riemannian metrics on {L}ie groupoids.
\newblock {\em J. Reine Angew. Math.}, 735:143--173, 2018.

\bibitem{DH}
J.~J. Duistermaat and G.~J. Heckman.
\newblock On the variation in the cohomology of the symplectic form of the
  reduced phase space.
\newblock {\em Invent. Math.}, 69(2):259--268, 1982.

\bibitem{DK}
J.~J. Duistermaat and J.~A.~C. Kolk.
\newblock {\em Lie groups}.
\newblock Universitext. Springer-Verlag, Berlin, 2000.

\bibitem{EL07}
S.~Evens and J.-H. Lu.
\newblock Poisson geometry of the {G}rothendieck resolution of a complex
  semisimple group.
\newblock {\em Mosc. Math. J.}, 7(4):613--642, 766, 2007.

\bibitem{FM22}
R.~L. Fernandes and I.~M\u{a}rcu\c{t}.
\newblock Poisson geometry around {P}oisson submanifolds.
\newblock {\em J. Eur. Math. Soc. (JEMS)}, 28(3):1213--1311, 2026.

\bibitem{GW92}
V.~L. Ginzburg and A.~Weinstein.
\newblock Lie-{P}oisson structure on some {P}oisson {L}ie groups.
\newblock {\em J. Amer. Math. Soc.}, 5(2):445--453, 1992.

\bibitem{GH86}
W.~M. Goldman and M.~W. Hirsch.
\newblock Affine manifolds and orbits of algebraic groups.
\newblock {\em Trans. Amer. Math. Soc.}, 295(1):175--198, 1986.

\bibitem{Humphreys90}
J.~E. Humphreys.
\newblock {\em Reflection groups and {C}oxeter groups}, volume~29 of {\em
  Cambridge Studies in Advanced Mathematics}.
\newblock Cambridge University Press, Cambridge, 1990.

\bibitem{IM65}
N.~Iwahori and H.~Matsumoto.
\newblock On some {B}ruhat decomposition and the structure of the {H}ecke rings
  of {${p}$}-adic {C}hevalley groups.
\newblock {\em Inst. Hautes \'Etudes Sci. Publ. Math.}, 25:5--48, 1965.

\bibitem{LPV13}
C.~Laurent-Gengoux, A.~Pichereau, and P.~Vanhaecke.
\newblock {\em Poisson structures}, volume 347 of {\em Grundlehren der
  mathematischen Wissenschaften [Fundamental Principles of Mathematical
  Sciences]}.
\newblock Springer, Heidelberg, 2013.

\bibitem{Mar}
D.~Mart\'{\i}nez~Torres.
\newblock A {P}oisson manifold of strong compact type.
\newblock {\em Indag. Math. (N.S.)}, 25(5):1154--1159, 2014.

\bibitem{MestreMartin}
J.~N. Mestre and M.~Weilandt.
\newblock Suborbifolds and groupoid embeddings.
\newblock {\em North-West. Eur. J. Math.}, 6:1--18, i, 2020.

\bibitem{Moerdijk03}
I.~Moerdijk.
\newblock Lie groupoids, gerbes, and non-abelian cohomology.
\newblock {\em $K$-Theory}, 28(3):207--258, 2003.

\bibitem{MoerdijkMrcun02}
I.~Moerdijk and J.~Mr{\v{c}}un.
\newblock On integrability of infinitesimal actions.
\newblock {\em Amer. J. Math.}, 124(3):567--593, 2002.

\bibitem{MM05}
I.~Moerdijk and J.~Mr\v{c}un.
\newblock Lie groupoids, sheaves and cohomology.
\newblock In {\em Poisson geometry, deformation quantisation and group
  representations}, volume 323 of {\em London Math. Soc. Lecture Note Ser.},
  pages 145--272. Cambridge Univ. Press, Cambridge, 2005.

\bibitem{Mol22}
M.~Mol.
\newblock {\em Hamiltonian actions of compact type from a Poisson point of view
  Hamiltonian actions of compact type from a Poisson point of view}.
\newblock PhD thesis, Utrecht University,
  https://dspace.library.uu.nl/handle/1874/421046, 2022.

\bibitem{Pflaum}
M.~J. Pflaum.
\newblock {\em Analytic and geometric study of stratified spaces}, volume 1768
  of {\em Lecture Notes in Mathematics}.
\newblock Springer-Verlag, Berlin, 2001.

\bibitem{PPT14}
M.~J. Pflaum, H.~Posthuma, and X.~Tang.
\newblock Geometry of orbit spaces of proper {L}ie groupoids.
\newblock {\em J. Reine Angew. Math.}, 694:49--84, 2014.

\bibitem{Samelson69}
H.~Samelson.
\newblock Orientability of hypersurfaces in {$R\sp{n}$}.
\newblock {\em Proc. Amer. Math. Soc.}, 22:301--302, 1969.

\bibitem{Schwarz75}
G.~W. Schwarz.
\newblock Smooth functions invariant under the action of a compact lie group.
\newblock {\em Topology}, 14(1):63--68, 1975.

\bibitem{Vorobjev01}
Y.~Vorobjev.
\newblock Coupling tensors and {P}oisson geometry near a single symplectic
  leaf.
\newblock In {\em Lie algebroids and related topics in differential geometry
  ({W}arsaw, 2000)}, volume~54 of {\em Banach Center Publ.}, pages 249--274.
  Polish Acad. Sci. Inst. Math., Warsaw, 2001.

\bibitem{Vorobjev05}
Y.~Vorobjev.
\newblock Poisson equivalence over a symplectic leaf.
\newblock In {\em Quantum algebras and {P}oisson geometry in mathematical
  physics}, volume 216 of {\em Amer. Math. Soc. Transl. Ser. 2}, pages
  241--277. Amer. Math. Soc., Providence, RI, 2005.

\bibitem{SW01}
P.~\v{S}evera and A.~Weinstein.
\newblock Poisson geometry with a 3-form background.
\newblock {\em Progr. Theoret. Phys. Suppl.}, 144:145--154, 2001.

\bibitem{We3}
A.~Weinstein.
\newblock Linearization of regular proper groupoids.
\newblock {\em J. Inst. Math. Jussieu}, 1(3):493--511, 2002.

\bibitem{Xu04}
P.~Xu.
\newblock Momentum maps and {M}orita equivalence.
\newblock {\em J. Differential Geom.}, 67(2):289--333, 2004.

\bibitem{Zu}
N.~T. Zung.
\newblock Proper groupoids and momentum maps: linearization, affinity, and
  convexity.
\newblock {\em Ann. Sci. \'{E}cole Norm. Sup. (4)}, 39(5):841--869, 2006.

\bibitem{Zwaan23}
L.~Zwaan.
\newblock Duistermaat-{H}eckman measures for {H}amiltonian groupoid actions.
\newblock Preprint arXiv 2311.02491, 2023.

\bibitem{Zwaan23b}
L.~Zwaan.
\newblock Poisson manifolds of strong compact type over 2-tori.
\newblock {\em Pacific J. Math.}, 325(2):353--374, 2023.

\end{thebibliography}

\end{document}